\documentclass[11pt]{amsart}
\usepackage{amsmath} 
\usepackage{amssymb} 
\usepackage{amscd} 
\usepackage{amsxtra}
\usepackage{mathtools}
\usepackage{epic,eepic} 

\usepackage[margin=2.5cm]{geometry} 

\usepackage{hyperref}
\usepackage{xcolor} 
\hypersetup{
     colorlinks,
     linkcolor={red!50!black},
     citecolor={blue!50!black},
     urlcolor={blue!80!black}
   }
\usepackage{graphicx}
\usepackage[all]{xy}
\usepackage{calligra}
\usepackage{mathrsfs}
\usepackage{booktabs}
\usepackage{cite}

\numberwithin{equation}{section} 

\newtheorem{theorem}{Theorem}[section]
\newtheorem{lemma}[theorem]{Lemma}
\newtheorem{proposition}[theorem]{Proposition}
\newtheorem{proposition-definition}[theorem]{Proposition-Definition} 
\newtheorem{conjecture}[theorem]{Conjecture}
\newtheorem{corollary}[theorem]{Corollary}
\newtheorem{assumption}[theorem]{Assumption} 
\newtheorem{question}[theorem]{Question} 

\theoremstyle{definition} 
\newtheorem{definition}[theorem]{Definition}
\newtheorem{remark}[theorem]{Remark}
\newtheorem{example}[theorem]{Example}
\newtheorem{notation}[theorem]{Notation}
 
\newtheorem{problem}[theorem]{Problem} 
\newtheorem{recall}[theorem]{Recall} 

\newtheorem*{question*}{Question} 
\newtheorem*{remark*}{Remark} 

\newcommand{\C}{\mathbb C}
\newcommand{\R}{\mathbb R}
\newcommand{\Z}{\mathbb Z}
\newcommand{\Q}{\mathbb Q}
\newcommand{\N}{\mathbb N} 
\newcommand{\T}{\mathbb T}
\newcommand{\F}{\mathbb F}
\newcommand{\U}{\mathbb U} 
 
\newcommand{\D}{\mathbb D}  
\newcommand{\A}{\mathbb A} 
\newcommand{\hA}{\widehat{\A}} 
\newcommand{\bL}{\mathbb L}
\newcommand{\LL}{\boldsymbol{\mathsf L}}  
\newcommand{\LLo}{\LL^\circ}
\newcommand{\hLL}{\widehat{\LL}}
\newcommand{\hLLo}{\hLL{}^\circ}
\newcommand{\Proj}{\mathbb P}
\newcommand{\FM}{\mathbb{FM}}

\newcommand{\bx}{\mathbf{x}}
\newcommand{\by}{\mathbf{y}} 
\newcommand{\tby}{\tilde{\by}} 
\newcommand{\ba}{\mathbf{a}} 
\newcommand{\bq}{\mathbf{q}}
\newcommand{\tbq}{\tilde{\bq}} 
\newcommand{\bp}{\mathbf{p}} 
\newcommand{\bt}{\mathbf{t}}
\newcommand{\bs}{\mathbf{s}} 
\newcommand{\bY}{\boldsymbol{Y}} 
\newcommand{\bT}{\boldsymbol{T}}
\newcommand{\bN}{\mathbf{N}}
\newcommand{\bOmega}{\boldsymbol{\Omega}}
\newcommand{\bOmegao}{\bOmega_{\circ}}
\newcommand{\AOmega}{\cA\bOmega} 
\newcommand{\AOmegao}{\AOmega_\circ}
\newcommand{\bTheta}{\boldsymbol{\Theta}} 
\newcommand{\bThetao}{\bTheta_{\circ}} 
\newcommand{\ATheta}{\cA\bTheta}
\newcommand{\AThetao}{\ATheta_\circ} 
\newcommand{\unit}{\boldsymbol{1}} 
\newcommand{\Nabla}{\boldsymbol{\nabla}} 
 
\newcommand{\hNabla}{\widehat{\Nabla}} 
\newcommand{\Nablacc}{\Nabla^{\rm cc}}

\newcommand{\bD}{\mathbf{D}} 
\newcommand{\bxi}{\boldsymbol{\xi}}
\newcommand{\bh}{\boldsymbol{h}} 

\newcommand{\NE}{\operatorname{NE}}
\newcommand{\ev}{\operatorname{ev}}
\newcommand{\Hom}{\operatorname{Hom}}
\DeclareMathOperator{\sHom}{\mathscr{H}\text{\kern -3pt {\calligra\large om}}\,}
\newcommand{\Pic}{\operatorname{Pic}}
\newcommand{\Ker}{\operatorname{Ker}}
\newcommand{\Image}{\operatorname{Im}}
\newcommand{\rank}{\operatorname{rank}} 
\newcommand{\id}{\operatorname{id}}
\newcommand{\End}{\operatorname{End}}
\newcommand{\Res}{\operatorname{Res}}

\newcommand{\Spf}{\operatorname{Spf}}   
\newcommand{\Sym}{\operatorname{Sym}}
\newcommand{\pr}{\operatorname{pr}} 
\newcommand{\KS}{\operatorname{KS}}
\newcommand{\Cont}{\operatorname{Cont}}
\newcommand{\Aut}{\operatorname{Aut}} 
\newcommand{\Auteq}{\operatorname{Auteq}} 
\newcommand{\Tr}{\operatorname{Tr}} 
\newcommand{\filt}{\operatorname{filt}} 
\newcommand{\diag}{\operatorname{diag}} 
 
\newcommand{\PT}{\operatorname{PT}} 
\newcommand{\hot}{\operatorname{h.\! o.\! t.}}
 
\newcommand{\For}{\operatorname{For}} 
\newcommand{\emb}{\operatorname{emb}}
\newcommand{\sdet}{\operatorname{sdet}} 
\newcommand{\Toric}{\operatorname{\mathfrak{Toric}}}  
\newcommand{\Crep}{\operatorname{\mathfrak{Crep}}}
\newcommand{\Amp}{\operatorname{Amp}}
\newcommand{\Vol}{\operatorname{Vol}} 
\newcommand{\Mir}{\operatorname{Mir}} 
\newcommand{\gr}{\operatorname{gr}} 
\newcommand{\Gr}{\operatorname{Gr}} 
\newcommand{\dist}{\operatorname{dist}} 
\newcommand{\cc}{{\operatorname{cc}}}
\newcommand{\Mon}{\operatorname{Mon}} 
\newcommand{\tr}{{\rm T}}

\newcommand{\CR}{\operatorname{CR}} 

\newcommand{\cN}{\mathcal{N}}
\newcommand{\cA}{\mathcal{A}}
\newcommand{\tcA}{\widetilde{\cA}} 
\newcommand{\cU}{\mathcal{U}}
\newcommand{\cV}{\mathcal{V}} 
\newcommand{\cC}{\mathcal{C}} 
\newcommand{\tcC}{\widetilde{\cC}} 
\newcommand{\cO}{\mathcal{O}}
\newcommand{\bcO}{\boldsymbol{\cO}}
\newcommand{\AO}{\cA\bcO} 

\newcommand{\cT}{\mathcal{T}}
\newcommand{\cL}{\mathcal{L}}
\newcommand{\cH}{\mathcal{H}} 
\newcommand{\tcH}{\widetilde{\cH}}
\newcommand{\cZ}{\mathcal{Z}} 

\newcommand{\cM}{\mathcal{M}} 
\newcommand{\cMo}{\cM^{\circ}} 
\newcommand{\cMss}{\cM_{\rm ss}} 
\newcommand{\hcM}{\widehat{\cM}}
\newcommand{\cMbar}{\overline{\cM}}
\newcommand{\cF}{\mathcal{F}} 
 
\newcommand{\cW}{\mathcal{W}}
\newcommand{\cS}{\mathcal{S}} 
\newcommand{\cR}{\mathcal{R}} 

\newcommand{\tnabla}{\widetilde{\nabla}}

\newcommand{\hF}{\widehat{F}}
\newcommand{\hG}{\widehat{G}} 
\newcommand{\hUU}{\widehat{\U}}
\newcommand{\hcZ}{\widehat{\cZ}} 
\newcommand{\hcF}{\widehat{\cF}}
\newcommand{\sfF}{\mathsf{F}}
\newcommand{\AF}{\cA\sfF}

\newcommand{\hGamma}{\widehat{\Gamma}} 
\newcommand{\hp}{\hat{p}}
\newcommand{\hq}{\hat{q}} 
\newcommand{\hpsi}{\hat{\psi}} 
\newcommand{\hR}{\widehat{R}} 
\newcommand{\hx}{\hat{x}} 
\newcommand{\hs}{\hat{s}} 
\newcommand{\hc}{\hat{c}} 

\newcommand{\ts}{\tilde{s}} 
\newcommand{\te}{\tilde{e}} 
\newcommand{\tg}{\tilde{g}} 

\newcommand{\ty}{\tilde{y}} 
\newcommand{\tkappa}{\tilde{\kappa}} 
\newcommand{\tcM}{\widetilde{\cM}}

\newcommand{\hC}{\widehat{C}} 
\newcommand{\hcW}{\widehat{\cW}}
\newcommand{\hotimes}{\mathbin{\widehat\otimes}}
\newcommand{\hdelta}{\hat{\delta}}

\newcommand{\sfP}{\mathsf{P}}
\newcommand{\sfPss}{\sfP_{\rm ss}} 
\newcommand{\sfQ}{\mathsf{Q}} 
\newcommand{\tsfQ}{\widetilde{\sfQ}} 
\newcommand{\sfR}{\mathsf{R}} 
\newcommand{\sx}{\mathsf{x}}
\newcommand{\sy}{\mathsf{y}}

\newcommand{\Fock}{\mathfrak{Fock}}
\newcommand{\Fockan}{\operatorname{\mathfrak{AFock}}}
\newcommand{\Fockanrat}{\operatorname{\Fockan_{\rm rat}}}
\newcommand{\ov}{\overline}
\newcommand{\surj}{\twoheadrightarrow} 
  
\newcommand{\fm}{\mathfrak{m}}

\newcommand{\iu}{\mathtt{i}} 
\newcommand{\Phiss}{\Phi_{\rm ss}}
\newcommand{\Phissu}{\Phi_{{\rm ss},u}}
\newcommand{\Css}{C_{\rm ss}} 
\newcommand{\wave}{\mathscr{C}}
\newcommand{\wavess}{\wave_{\rm ss}} 

\newcommand{\Wall}{\mathsf{W}} 

\newcommand{\oi}{{\ov{\imath}}} 
\newcommand{\oj}{{\ov{\jmath}}} 

\newcommand{\Mbar}{\ov{M}}
\newcommand{\cAabs}{\mathcal{A}^{\text{\rm \tiny abs}}}
\newcommand{\Picc}{\Pi_{\rm cc}} 
\newcommand{\Lambdacc}{\Lambda_{\rm cc}} 
\newcommand{\Phicc}{\Phi^{\rm cc}}
\newcommand{\varthetacc}{\vartheta_{\rm cc}}

\renewcommand{\projlim}{\varprojlim}
\renewcommand{\injlim}{\varinjlim} 

\def\pair#1#2{\langle #1,#2\rangle}
\def\Pair#1#2{\left\langle #1,#2\right\rangle}
\def\parfrac#1#2{\frac{\partial{#1}}{\partial #2}}
\def\corr#1{\left\langle #1 \right\rangle}

\title{A Fock sheaf for Givental quantization}
\author{Tom Coates}
\email{t.coates@imperial.ac.uk}
\address{Department of Mathematics, Imperial College London, 180 Queen's Gate, 
London SW7 2AZ, United Kingdom} 
\author{Hiroshi Iritani}
\email{iritani@math.kyoto-u.ac.jp}
\address{Department of Mathematics, Graduate School of Science, 
Kyoto University, Kitashirakawa-Oiwake-cho, 
Sakyo-ku, Kyoto, 606-8502, Japan}

\begin{document}
\begin{abstract}
We give a global, intrinsic, and co-ordinate-free quantization formalism for Gromov--Witten invariants and their B-model counterparts, which simultaneously generalizes the quantization formalisms described by Witten, Givental, 
and Aganagic--Bouchard--Klemm. 
Descendant potentials live in a Fock sheaf, consisting of local functions on Givental's Lagrangian cone that satisfy the $(3g-2)$-jet condition of Eguchi--Xiong; they also satisfy a certain anomaly equation, which generalizes the Holomorphic Anomaly Equation of Bershadsky--Cecotti--Ooguri--Vafa.  We interpret Givental's formula for the higher-genus potentials associated to a semisimple Frobenius manifold in this setting, showing that, in the semisimple case, there is a canonical global section of the Fock sheaf.  This canonical section automatically has certain modularity properties.  When $X$ is a variety with semisimple quantum cohomology, a theorem of Teleman implies that the canonical section coincides with the geometric descendant potential defined by Gromov--Witten invariants of $X$.  We use our formalism to prove a higher-genus version of Ruan's Crepant Transformation Conjecture for compact toric orbifolds.  When combined with our earlier joint work with Jiang, this shows that the total descendant potential for compact toric orbifold $X$ is a modular function for a certain group of autoequivalences of the derived category of $X$.
\end{abstract}

\maketitle 

\let\oldtocsection=\tocsection
\let\oldtocsubsection=\tocsubsection

\renewcommand{\tocsection}[2]{\hspace{0em}\oldtocsection{#1}{#2}}
\renewcommand{\tocsubsection}[2]{\hspace{1em}\oldtocsubsection{#1}{#2}}
\tableofcontents

\section{Introduction} 

Givental's quantization formalism~\cite{Givental:quantization,Givental:symplectic} has been an essential ingredient in many recent advances in Gromov--Witten theory.  These include the Quantum Lefschetz theorem~\cite{Coates--Givental,Tseng,CCIT:computing}, the Abelian/non-Abelian correspondence~\cite{Bertram--Ciocan-Fontanine--Kim,Ciocan-Fontanine--Kim--Sabbah}, connections to integrable systems~\cite{Givental:nKdV,Milanov,Faber--Shadrin--Zvonkine} and birational geometry~\cite{CIT:wall-crossings,Coates--Ruan,Coates--Iritani--Jiang,Iritani:Ruan,Iritani:integral,Brini--Cavalieri--Ross,Brini--Cavalieri}, 
the Landau--Ginzburg/Calabi--Yau 
correspondence~\cite{Chiodo--Ruan,Milanov--Ruan,Krawitz--Shen}, 
the study of relations in the tautological ring~\cite{Lee:1,Lee:2,Pandharipande--Pixton--Zvonkine}, and the theory of quasimaps~\cite{CFK:1,CFK:2,CFK:3,Cheong--Ciocan-Fontanine--Kim}.  The quantization formalism suggests, roughly speaking, that the Gromov--Witten theory of a target space $X$ is controlled by linear symplectic geometry in a certain symplectic vector space\footnote
{In the main body of the text, we use the space of $L^2$-functions on $S^1$ 
(see \S \ref{subsec:Givental-symplecticvs}) 
or a certain nuclear space (see equation \ref{eq:nuclear-Giventalsp}) 
instead of $\C(\!(z^{-1})\!)$. } 
\[
\cH^X = H^\bullet(X;\C)\otimes \C(\!(z^{-1})\!)
\]
which can be thought of as the localized $S^1$-equivariant Floer cohomology of the loop space of $X$~\cite{Givental:homological,Morava,Iritani:Floer}.  Genus-zero Gromov--Witten invariants of $X$ determine and are determined by a Lagrangian cone $\cL_X \subset \cH^X$ with very special geometric properties.  Natural operations in Gromov--Witten theory correspond to symplectic linear transformations $\U$ of $\cH^X$: their effect on genus-zero Gromov--Witten invariants is recorded by the effect of $\U$ on $\cL_X$, and their effect on higher-genus Gromov--Witten invariants is (or is expected to be) recorded by the action of the \emph{quantized} symplectic transformation $\hUU$ on the total descendant potential $\cZ_X$ for $X$, which is a generating function for all Gromov--Witten invariants of $X$.  That is, the total descendant potential $\cZ_X$, which is the mathematical counterpart of the partition function in type IIA string theory, should be thought of as an element of the Fock space arising from the geometric quantization of the Givental space $\cH^X$.

The symplectic transformation $\U$ is visible at the level of genus-zero Gromov--Witten invariants, and so the quantization formalism is a powerful ``genus zero controls higher genus'' principle.  One of the most striking instances of this is Givental's formula~\cite{Givental:quantization} for the total descendant potential of a target space with generically semisimple quantum cohomology:
\begin{equation}
  \label{eq:Givental_formula}
  \cZ_X = e^{F^1(t)} \widehat{S_t^{-1}} \widehat{\Psi} \widehat{R_t} \Big(\cZ_{\rm pt}^{\otimes N}\Big)
\end{equation}
Here $\cZ_{\rm pt}$ is the total descendant potential for a point (the Kontsevich--Witten $\tau$-function~\cite{Witten:2D,Kontsevich}); $N$ is the rank of $H^\bullet(X;\C)$; $S_t$, $\Psi$, and $R_t$ are linear symplectomorphisms defined in terms of genus-zero Gromov--Witten invariants of $X$; and $F^1(t)$ is the genus-one non-descendant Gromov--Witten potential.   The formula \eqref{eq:Givental_formula} gives a closed-form expression for higher-genus Gromov--Witten invariants of $X$ in terms of genus-zero Gromov--Witten invariants of $X$ and higher-genus Gromov--Witten invariants of a point.  It was conjectured by Givental and proven by him in the toric case~\cite{Givental:semisimple}; it was proven for arbitrary  generically semisimple Frobenius manifolds by Teleman~\cite{Teleman}, using the classification of two-dimensional semisimple Family Topological Field Theories.

Since we have this powerful genus-zero controls higher genus principle, and since genus-zero Gromov--Witten theory is in many cases reasonably well understood (for instance via mirror symmetry), it is surprising that this has not, to date, led to a better understanding of higher genus Gromov--Witten theory.  One reason for this is that Givental's quantization formalism is rather difficult to use in practice: the quantized operators involved are defined as infinite sums over Feynman diagrams, and there are delicate questions of convergence.  In particular there is not a ``Fock space'' on which Givental's quantized operators act: the well-definedness of $\hUU \cZ$ is proven on a case-by-case basis, using properties both of the particular transformation $\U$ and the particular generating function $\cZ$.

\medskip

The idea that the partition function should be regarded as a state in a Fock space arising from geometric quantization was originally proposed by 
Witten~\cite{Witten:background} in the context of the B-model for a 
Calabi--Yau $3$-fold. Witten used this idea to give an interpretation 
of the holomorphic anomaly of 
Bershadsky--Cecotti--Ooguri--Vafa~\cite{BCOV:HA,BCOV:KS}. 
It has had a number of important consequences, including the discovery that 
the higher-genus Gromov--Witten potentials of local Calabi--Yau $3$-folds 
should be quasi-modular forms~\cite{ABK,Brini-Tanzini,ASYZ}.  
Building on these works \cite{BCOV:HA,BCOV:KS,Witten:background}, 
Aganagic--Bouchard--Klemm \cite{ABK} (see also~\cite{ADKMV})
described a concrete quantization procedure, in the context 
of the Calabi--Yau B-model, which is free of many of the technical complexities 
of Givental's quantization.  
We refer to this as Witten quantization. 
In their setting, the relevant symplectic vector space is 
a finite dimensional space given by the middle 
cohomology group of the Calabi--Yau $3$-fold. 
The quantized operators involved are defined as \emph{finite} 
sums over Feynman diagrams, and so convergence and 
well-definedness of the results are manifest.  

\medskip

This paper grew out of an attempt to understand and 
unify Givental quantization and Witten quantization.  
We give a global, intrinsic, and co-ordinate-free quantization formalism, 
which reduces to Givental quantization whenever 
they both make sense 
and which reduces to Witten quantization in the Calabi--Yau 3-fold case.   
We now give a brief summary of this paper. 

\medskip 

\noindent 
{\bf Construction of a Fock Sheaf.} 
Let $\cM$ be a complex manifold. We start with 
a locally free $\cO_{\cM}[\![z]\!]$-module $\sfF$ 
of finite rank equipped with a flat connection: 
\[
\nabla \colon \sfF \to \Omega^1_\cM \otimes z^{-1} \sfF 
\]
and a $\nabla$-flat, symmetric, non-degenerate and 
``$z$-sesquilinear'' pairing: 
\[
(\cdot,\cdot)_\sfF \colon (-)^*\sfF \otimes_{\cO_{\cM}[\![z]\!]} 
\sfF \to \cO_{\cM}[\![z]\!]
\]
where $(-)^* \sfF$ means $\sfF$ on which $z$ acts by 
$-z$. 
We call the triple $(\sfF,\nabla,(\cdot,\cdot)_\sfF)$ 
a \emph{cTP structure},
extending the terminology of Hertling~\cite{Hertling:ttstar};
see Definition~\ref{def:cTP}.  
A cTP structure arises from geometry as the Dubrovin connection 
associated to quantum cohomology or as the Gauss--Manin 
connection associated to deformation of complex manifolds 
or singularities. 
We need to assume that our cTP structure satisfies a certain miniversality 
condition (see Assumption \ref{assump:miniversal}). 
We regard $\sfF$ as an infinite-dimensional vector bundle over $\cM$ 
and write $\LL$ for the total space of the vector bundle $z\sfF\to \cM$. 
The total space $\LL$ is an analogue of Givental's Lagrangian cone $\cL_X$. 
As a polarization for geometric quantization, we consider an 
$\cO_\cM$-submodule  
$\sfP$ of $\sfF[z^{-1}]$ such that:
\begin{itemize} 
\item $\sfP$ is opposite to $\sfF$, i.e.~$\sfF \oplus \sfP = \sfF[z^{-1}]$; 
\item $\sfP$ is isotropic with respect to the 
symplectic form $\Omega$ on $\sfF[z^{-1}]$ defined by 
$\Omega(s_1,s_2) := \Res_{z=0}(s_1,s_2)_{\sfF} \, dz$; 
\item $\sfP$ is parallel ($\nabla \sfP \subset \Omega^1_\cM \otimes \sfP$) 
and closed under $z^{-1}$ ($z^{-1} \sfP \subset \sfP$). 
\end{itemize} 
We call $\sfP$ an \emph{opposite module}\footnote
{We also consider $\sfP$ which does not satisfy the 
third condition: in this case $\sfP$ is called a 
\emph{pseudo-opposite module}. 
\label{foot:pseudo-opposite}}. 
This serves as a splitting of the (semi-infinite) Hodge filtration 
and has been used to construct a Frobenius manifold structure 
(or flat structure) in the context of singularity theory \cite{SaitoK,SaitoM}. 
In terms of Givental's symplectic space, $\sfP$ corresponds to a 
Lagrangian subspace $P\subset \cH^X$ which is transversal to $\cL_X$. 
The opposite module $\sfP$ defines an affine flat structure 
on the total space $\LL$: we write 
$\Nabla$ for the corresponding flat connection on $T\LL$.  
Given an open set $U\subset \cM$ 
and an opposite module $\sfP$ over $U$, the Fock space 
$\Fock(U;\sfP)$ consists of collections
\[
\wave = \left\{ C^{(g)}_{\mu_1 \dots \mu_n} : 
\text{$g\ge 0$, $n\ge 0$,  $2g -2 + n> 0$} \right\} 
\]
of meromorphic symmetric tensors $C^{(g)}_{\mu_1 \dots\mu_n} 
\in (T^*\LL)^{\otimes n}$ 
over $\LL|_U$, called the genus-$g$, $n$-point correlation functions. 
We require that these tensors satisfy:
\begin{itemize} 
\item the Jetness condition
$C^{(g)}_{\mu_1 \dots \mu_n} = 
\Nabla_{\mu_1} C^{(g)}_{\mu_2 \dots \mu_n}$; 
\item   
the Eguchi--Xiong $(3g-2)$-jet 
condition (see equation~\ref{eq:jetcondition} below 
or~\cite{Eguchi--Xiong,Getzler:jetspace,Dubrovin--Zhang}); 
\item 
the Dilaton Equation (this is a homogeneity condition);  
\item a certain pole order condition along a discriminant divisor 
in $\LL|_U$;
\end{itemize} 
see Definition \ref{def:localFock}. The genus-zero correlation functions are given by the 
Yukawa coupling and its derivatives, 
which are determined by the cTP structure itself. 
We glue these Fock spaces to give a sheaf of sets on $\cM$ via a 
\emph{transformation rule} 
\[
T(\sfP,\widehat{\sfP}) \colon \Fock(U;\sfP) \to \Fock(U;\widehat{\sfP}) 
\] 
defined for two opposite modules $\sfP,\widehat{\sfP}$ over $U$. 
The element $\{\hC^{(h)}_{\mu_1\dots\mu_m}\} 
\in \Fock(U;\widehat{\sfP})$ corresponding to 
the element $\{C^{(g)}_{\mu_1\dots\mu_n}\} 
\in \Fock(U;\sfP)$ 
is given by a Feynman rule: each $\hC^{(g)}_{\mu_1\dots\mu_m}$ 
is expressed as a finite sum over connected stable graphs, with vertex terms given 
by the $C^{(h)}_{\mu_1,\dots,\mu_n}$, $h\le g$, 
and propagator defined geometrically in terms 
of the two opposite modules $\sfP$, $\widehat{\sfP}$. 
The construction of a Fock sheaf will be given in \S \ref{sec:global_theory}. 
We also refer the reader to \S \ref{sec:globalquantization:motivation} 
for a more informal account.

\medskip 
\noindent 
{\bf Comparison to Givental and Witten Quantization.} 
Our transformation rule is given as a finite sum over Feynman graphs  
and is a direct generalization of Witten quantization to infinite dimensions. 
A key point is the use of a certain \emph{algebraic co-ordinate system} 
on the total space $\LL$. 
Every ingredient in the Feynman rule has a polynomial/rational 
expression in the algebraic co-ordinate system, 
and this fact makes evident that the Feynman rule is well-defined 
in infinite dimensions. On the other hand, when we restrict 
correlation functions to the formal neighbourhood $\hLL$ 
of the fiber $\LL_t$ of $\LL$ at a point $t\in \cM$ and write the 
Feynman rule in a \emph{flat} co-ordinate system on $\hLL$, 
our transformation rule coincides exactly with the action of 
Givental's quantized operator on tame functions. 
\begin{theorem}[see Theorem \ref{thm:transformationrule-Giventalquantization} 
for a more precise formulation] 
The transformation rule $T(\sfP,\sfP')$ matches with the action of Givental's 
upper-triangular loop group over the formal neighbourhood 
$\hLL$ of the fiber at each point on the base space. 
\end{theorem} 
In \S \ref{subsec:global_L2}, we will adapt the transformation 
rule to an $L^2$-setting. 
There we work with the $L^2$-subspace $L^2(\LL)\subset \LL$,  
which is an infinite-dimensional Hilbert manifold, 
and describe a transformation rule for holomorphic correlation 
functions on $L^2(\LL)$. 
We see that one can define the quantized operator $\hUU$ 
for any linear symplectic transformation $\U$ that satisfies a 
certain ``trace class'' condition 
(see Definitions~\ref{def:Darboux_close}, \ref{def:transformation_L2}). 
This gives a uniform definition of quantization without insisting that 
 $\U$ be upper-triangular or lower-triangular. 
We also show in Remark~\ref{rem:very_detailed_remark} 
that this $\hUU$ coincides with Givental's quantized operator 
whenever $\U$ is sufficiently close to the identity. 

\medskip 
\noindent 
{\bf A Global Section of the Fock Sheaf in the Semisimple Case.}
There is a simple and attractive interpretation of 
Givental's formula \eqref{eq:Givental_formula} in our setting.  
Suppose that the flat connection $\nabla$ of the cTP structure 
is extended in the $z$-direction with poles of order $2$ 
along $z=0$ such that the pairing $(\cdot,\cdot)_\sfF$ 
is flat in the $z$-direction: this is called a \emph{cTEP structure}  
(see Definition \ref{def:cTP}). 
cTP structures that come from geometry, such as quantum cohomology, are often 
cTEP structures, not just cTP structures. We say that a cTEP structure is 
\emph{tame semisimple} if the residue $\cU \in \End(\sfF/z\sfF)$ 
of $\nabla_{z\partial_z}$ is semisimple with distinct eigenvalues. 
We prove the following: 

\begin{theorem}[Definition \ref{def:Giventalwavefcn}, 
Theorem \ref{thm:formalization-Giventalwave}] 
There exists a canonical global section of 
the Fock sheaf associated to a tame semisimple cTEP structure, 
which coincides with the potential given by 
Givental's formula \eqref{eq:Givental_formula} 
in the formal neighbourhood of each point of the base space. 
We call this global section the \emph{Givental wave function}. 
\end{theorem}

We observe, via the Levelt--Turrittin formal decomposition of $\nabla_{z\partial_z}$, 
that any tame semisimple cTEP structure of rank $N+1$ is locally isomorphic to 
the cTEP structure associated with the quantum cohomology of $N+1$~points; 
moreover the isomorphism is unique up to (signed) permutation of $N+1$~points 
(Proposition \ref{prop:semisimpletriv}). 
This shows that a tame semisimple cTEP structure admits 
a canonical \emph{semisimple opposite module} 
$\sfPss$ (Definition \ref{def:semisimple_opposite}). 
Then the Gromov--Witten potential $\cZ^{\otimes N}_{\rm pt}$ 
of $N$~points defines an element of $\Fock(\cM,\sfPss)$: 
this is the Givental wave function above. 
Teleman's theorem~\cite{Teleman} can be rephrased in our language as: 
\begin{theorem}[see Theorem~\ref{thm:Teleman}]  
When the quantum cohomology of $X$ is generically semisimple, 
the total descendant Gromov--Witten potential of $X$, when viewed as a section 
of the Fock sheaf, coincides with the Givental wave function. 
\end{theorem} 
This is just a rephrasing of Givental's formula \eqref{eq:Givental_formula} 
that says that $\cZ_X$ and $\cZ_{\rm pt}^{\otimes N}$ are related by 
a quantized symplectic operator; in our formalism, 
this quantized operator arises as the ``transition function''  
$T(\sfPss, \sfP_{\rm std})$ of the Fock sheaf between the semisimple 
opposite module $\sfPss$ and the standard opposite module $\sfP_{\rm std}$ 
(see Example \ref{ex:Amodel-opposite}) of  
the quantum cohomology of $X$.

Since the Givental wave function is canonically associated to 
a semisimple cTEP structure, it is automatically ``modular'' in the following sense. 
Opposite modules $\sfP$ arising from geometry are typically 
\emph{not} monodromy invariant, and therefore the 
presentation $\wave_\sfP$ of the Givental wave function 
with respect to the polarization $\sfP$ is not single-valued 
in general. 
Regarding it as a function on the universal cover of $\cM$, 
we have the following transformation property with respect to 
a deck-transformation $\gamma \in \pi_1(\cM)$:  
\begin{equation} 
\label{eq:deck_trans} 
\gamma^\star \wave_\sfP  = T(\sfP,\gamma^\star \sfP) \wave_\sfP 
\end{equation} 
Since $\cM$ typically arises as a moduli space of 
complex structures, the universal cover of $\cM$ should be 
regarded as an analogue of a Hermitian symmetric space. 
Thus we refer to the property \eqref{eq:deck_trans} as 
modularity; see also Remark~\ref{rem:modularity}.  

\medskip 
\noindent 
{\bf The Crepant Transformation Conjecture in the Toric Case.} 
Combining our global quantization formalism with mirror symmetry, 
we deduce in \S\ref{sec:toric_mirror_symmetry} 
a higher-genus version of Ruan's celebrated Crepant Transformation Conjecture 
for compact toric orbifolds.  
We fix a convex lattice polytope $\Delta$ in $\Z^n$ 
containing the origin in its interior and consider 
the set $\Crep(\Delta)$ of weak-Fano 
toric orbifolds having $\Delta$ as the fan polytope. 
(See \S \ref{sec:GIT} for the precise conditions that 
we impose on $\Delta$.) 
Toric orbifolds from $\Crep(\Delta)$ are $K$-equivalent to each other; 
moreover they are derived equivalent.  
These toric orbifolds have the same Landau--Ginzburg models 
as mirrors \cite{Givental:ICM,Hori--Vafa}, each of them corresponding 
to a different limit point of the B-model moduli space. 
The Landau--Ginzburg mirror produces a generically 
semisimple cTEP structure over the B-model moduli space  
\cite{Sabbah:hypperiod,Douai--Sabbah:I, Douai--Sabbah:II, Barannikov:projective, 
CIT:wall-crossings, Iritani:integral, Reichelt--Sevenheck}. 
Therefore the Fock sheaf $\Fock_{\rm B}$ associated to the B-model 
admits a global section given by the Givental wave function. 
Mirror symmetry for toric orbifolds \cite{CCIT:mirror,Iritani:integral} 
and Teleman's theorem \cite{Teleman} 
immediately imply: 
\begin{theorem}[see Theorem \ref{thm:CTC_higher_genus}]
\label{thm:CTC_toric_introd}  
There exists a global section of the B-model Fock sheaf $\Fock_{\rm B}$ 
which restricts to the total descendant Gromov--Witten potential $\cZ_X$ 
of $X \in \Crep(\Delta)$ (viewed as a section of $\Fock_{\rm B}$) 
in a neighbourhood of the large radius limit point of $X$. 
In other words, the Gromov--Witten potentials $\cZ_X$ with $X\in \Crep(\Delta)$ 
are analytically continued to each other as sections of the B-model Fock sheaf. 
\end{theorem} 

This establishes the higher-genus Crepant Transformation Conjecture 
for toric orbifolds in $\Crep(\Delta)$. 
Using the $L^2$-formalism in \S\ref{subsec:global_L2}, 
we recover an earlier formulation of the higher-genus Crepant Transformation 
Conjecture~\cite{Bryan--Graber,CIT:wall-crossings,Coates--Ruan} 
as follows\footnote{Note that we do not require any Hard Lefschetz 
hypothesis here~\cite{Bryan--Graber}.}:

\begin{theorem}[see Corollary \ref{cor:CTC_higher_genus}] 
\label{thm:CTC_toric_introd_2} 
Let $X_1$, $X_2$ be compact weak-Fano toric orbifolds from $\Crep(\Delta)$. 
There exists a linear symplectic transformation $\U_\gamma \colon \cH^{X_1} \to 
\cH^{X_2}$ depending on a path $\gamma$ on the B-model moduli space 
connecting the two large radius limit points 
such that, under analytic continuation along $\gamma$, we have 
\[
\cZ_{2} \propto \hUU_\gamma \cZ_{1}
\]
where $\cZ_{i}$ is the total descendant Gromov--Witten 
potential of $X_i$. 
\end{theorem} 

In our recent joint work with Jiang~\cite{Coates--Iritani--Jiang}, 
we computed the symplectic transformation $\U_\gamma$ 
explicitly for a certain path $\gamma$, and showed that 
it arises from a composition of Fourier--Mukai transformations 
$\FM \colon D^b(X_1) \cong D^b(X_2)$ 
via the $\hGamma$-integral structure in quantum cohomology 
\cite{Iritani:integral,Katzarkov--Kontsevich--Pantev}. 
This means that we have the following commutative diagram: 
\[
\xymatrix{ 
D^b(X_1) \ar[r]^{\FM} \ar[d] & D^b(X_2) \ar[d] \\ 
\tcH^{X_1} \ar[r]^{\U_\gamma} & \tcH^{X_2}
} 
\]
where the vertical arrow is the map defining the $\hGamma$-integral 
structure and $\tcH^{X_i}$ is a multi-valued variant of Givental's symplectic space 
(see \cite{Coates--Iritani--Jiang}). 
Note that an auto-equivalence of $D^b(X)$ induces a symplectic transformation 
of the Givental space $\cH^X$ via the $\hGamma$-integral structure, 
and we expect that the total descendant potential $\cZ_X$ of $X$ should 
be modular with respect to the group of autoequivalences. 
Our joint work with Jiang \cite{Coates--Iritani--Jiang} implies: 

\begin{theorem}[see Corollary \ref{cor:modularity_autoeq}] 
\label{thm:CTC_toric_introd_3} 
The total descendant potential $\cZ_X$ for a compact weak Fano toric 
orbifold $X$ is modular with respect to a certain non-trivial subgroup 
of the group of autoequivalences of the bounded derived category of 
coherent sheaves on $X$.
\end{theorem} 

\begin{remark} 
\label{rem:analyticity_in_z} 
We remark on the analyticity of the genus-zero data with respect to $z$ and its role in 
the above Theorems \ref{thm:CTC_toric_introd}--\ref{thm:CTC_toric_introd_3}.
The B-model cTEP structure above can be in fact analytified in the $z$-direction 
and lifted to a TEP structure (see Definition \ref{def:TP}) 
globally over the B-model moduli space 
\cite{Sabbah:hypperiod,Douai--Sabbah:I, Douai--Sabbah:II, Barannikov:projective, 
CIT:wall-crossings, Iritani:integral, Reichelt--Sevenheck}. 
Mirror symmetry implies that this global TEP structure restricts to 
the quantum cohomology TEP structure of each $X \in \Crep(\Delta)$ 
on a neighbourhood of the large radius limit point of $X$: this 
is the content of the \emph{genus-zero} crepant transformation conjecture 
\cite{CIT:wall-crossings, Iritani:Ruan, Coates--Ruan} 
(which was proved in \cite{Coates--Iritani--Jiang} for the most general set-up 
for toric stacks). 
In Theorem \ref{thm:CTC_toric_introd}, we do not need  
this lift to a TEP structure 
since the analytic structure of the Fock sheaf depends only on 
the underlying cTEP structure. 
On the other hand, in order to define a semi-infinite period map 
(see \S\ref{subsec:Lagrangian_TP} and \S\ref{subsec:Kaehler}) 
of the genus-zero data, we need its analyticity in $z$. 
This analyticity enables us to compare Givental's symplectic 
spaces $\cH^{X_1}, \cH^{X_2}$ via analytic continuation 
along the path $\gamma$. The symplectic 
transformation $\U_\gamma \colon \cH^{X_1} \to \cH^{X_2}$ 
in Theorem \ref{thm:CTC_toric_introd_2} arises in this way 
and matches up the Lagrangian cones encoding the information of 
the genus-zero theory: 
\[
\cL_{X_2} = \U_\gamma \cL_{X_1}.  
\]
\end{remark}

\medskip 

\noindent 
{\bf Anomaly Equation.} 
In our global quantization formalism, we also allow polarizations
$\sfP$ which are not parallel along $\cM$ (see footnote~\ref{foot:pseudo-opposite} on page~\pageref{foot:pseudo-opposite}). 
In this case, the connection $\Nabla$ on the tangent bundle 
$T\LL$ is not flat, and correlation functions $C^{(g)}_{\mu_1\dots\mu_n}$ 
fail to satisfy the Jetness condition. 
We have instead the following anomaly equation.   
\begin{theorem}[Theorem \ref{thm:anomaly}]  
Correlation functions under a non-parallel polarization $\sfP$ 
satisfy the anomaly equation: 
\[
C^{(g)}_{\mu_1 \dots \mu_n} 
= \Nabla_{\mu_1} C^{(g)}_{\mu_2\dots \mu_n} 
+ \frac{1}{2} 
\sum_{\substack{\{2,\dots,n\} = I \sqcup J \\ 
k+ l = g}} 
C^{(k)}_{\mu_I, \alpha} 
{\Lambda}_{\mu_1}^{\alpha\beta}
C^{(l)}_{\mu_J, \beta}  
+ \frac{1}{2} 
C^{(g-1)}_{\mu_2\dots\mu_n\alpha\beta} 
{\Lambda}^{\alpha\beta}_{\mu_1}
\]
where ${\Lambda}^{\alpha\beta}_{\mu}$ is 
a tensor which measures the deviation of $\sfP$ from  
being parallel. 
\end{theorem} 

In \S \ref{sec:HAE}, we consider a cTP structure equipped 
with a \emph{real structure}, which we call a \emph{TRP structure}. 
We impose the condition that the TRP structure is pure 
(see Definition \ref{def:pure_TRP}). 
Quantum cohomology or its B-model counterpart are often equipped 
with a natural real structure and give examples of pure TRP structures 
\cite{Sabbah:FL, Hertling:ttstar,Iritani:ttstar}.  
For a pure TRP structure, we can consider a polarization 
which is obtained as the complex conjugate of $\sfF$. 
Such a polarization is called the \emph{K\"{a}hler 
polarization} or \emph{holomorphic polarization} in the context of 
geometric quantization. 
This complex-conjugate polarization is intrinsic 
to the TRP structure, and therefore if we have a single-valued section 
of the Fock sheaf (such as the Givental wave function), 
its presentation with respect to this gives a single-valued \emph{function}. 
This should be a useful and important property.  
A drawback of this polarization is that the corresponding correlation functions 
are not holomorphic. 
The anti-holomorphic dependence is described precisely 
by the holomorphic anomaly equation. 
\begin{theorem}[Proposition \ref{prop:HAE}] 
Correlation functions under the complex conjugate polarization 
satisfy the following holomorphic anomaly equation: 
\[
0  = \partial_{\ov\mu_1} C^{(g)}_{\mu_2\dots\mu_n} 
+ \frac{1}{2} 
\sum_{\substack{\{2,\dots,n\} = I \sqcup J \\ 
k+ l = g}} 
C^{(k)}_{\mu_I, \alpha} {\Lambda}_{\, \ov\mu_1}^{\alpha\beta}
C^{(l)}_{\mu_J, \beta}  
+ \frac{1}{2} 
C^{(g-1)}_{\mu_2 \dots \mu_n\alpha\beta} 
{\Lambda}^{\alpha\beta}_{\, \ov{\mu}_1}  
\]
where $\Lambda^{\alpha\beta}_{\ov{\mu}}$ is a tensor 
associated to the TRP structure. 
\end{theorem} 
This is analogous to the holomorphic anomaly equation 
of Bershadsky--Cecotti--Ooguri--Vafa~\cite{BCOV:HA, BCOV:KS}. 
Given a parallel polarization $\sfP$ of the TRP structure, 
we introduce a positive scalar function on the base $\cM$, 
called the \emph{half-density metric}. 
This can be thought of heuristically as a Hermitian metric on the half-density line bundle 
``$\det(T^*\LL)^{1/2}$'' of $\LL$; see Definition \ref{def:half_density_metric}. 
The genus-one potential can be viewed as a holomorphic 
section of $\det(T^*\LL)^{1/2}$, and the holomorphic 
anomaly equation at genus one is a formula for 
the curvature of $\det(T^*\LL)^{1/2}$ 
(see equation \ref{eq:genus_one_HAE} and \cite{BCOV:HA}). 
Singularities and global properties of this metric will be the subject 
of a future study. 

\section*{Relation to Other Work} 
In this paper we focus on the construction of a Fock sheaf 
and its fundamental properties, but we only discuss how to construct a canonical section of the Fock sheaf in the semisimple case (where we use Givental's formula). 
To give a section of the Fock sheaf in general, we certainly need more data 
from geometry. An approach based on Calabi--Yau categories and 
Topological Quantum Field Theory has been proposed by Costello~\cite{Costello:TCFT_CYcat, Costello:partitionfcn_TFT}, 
Kontsevich--Soibelman~\cite{Kontsevich-Soibelman:notes_Ainfinity} 
and Katzarkov--Kontsevich--Pantev~\cite{KKP:TQFT}. 
Another approach based on renormalization and BCOV theory 
has been developed by Costello--Li~\cite{Costello--Li, SiLi:BCOV_elliptic, 
SiLi:VHS_Frob_gauge}; using a chain-level version of Givental's symplectic space, they construct a mathematical version 
of the higher-genus B-model. These works should give a canonical 
global section of the Fock sheaf. 
We also remark that the approach based on Givental's formula 
(as in this paper) has been taken by several authors  
\cite{Milanov--Ruan, Krawitz--Shen, 
Milanov--Ruan--Shen, LiLiSaitoShen:exc_unimodular}; 
in particular Milanov--Ruan \cite{Milanov--Ruan} showed that the Gromov--Witten 
potential of an elliptic orbifold $\Proj^1$ is a quasi-modular form 
using Givental's formula.

\section*{Plan of the Paper} 
We begin by fixing notation for various objects in Gromov--Witten theory (\S\ref{sec:nota}).  We give an informal sketch of our quantization framework in~\S\ref{sec:globalquantization:motivation}, and give the rigorous construction in~\S\ref{sec:global_theory}.  In~\S\ref{sec:Givental} we explain the precise connection between our quantization formalism and Givental's. \S\ref{sec:GW-wavefunction}~describes how the Gromov--Witten potential fits into our framework.  \S\ref{sec:semisimple_case}~treats the semisimple case; in particular we explain how Givental's formula \eqref{eq:Givental_formula} gives rise to a global section of the Fock sheaf.  In~\S\ref{sec:mirror_symmetry} we give two applications of our formalism to mirror symmetry, proving the higher-genus Crepant Transformation Conjecture for toric orbifolds in~\S\ref{sec:toric_mirror_symmetry} and discussing mirror symmetry for Calabi--Yau manifolds in \S\ref{sec:CY}.  In \S\ref{sec:HAE} we describe how the Holomorphic Anomaly Equation of 
Bershadsky--Cecotti--Ooguri--Vafa arises from the anomaly equation for curved polarizations given in \S\ref{subsec:curved}.

\section*{Acknowledgements}
We thank Hsian-Hua Tseng for many useful conversations 
on Givental quantization. 
The definition of Fock spaces for ancestor
potentials (\S\ref{sec:ancestor_Fock}) was originally worked out 
in a joint project with him, and we are grateful to him for allowing 
us to present this formulation here. 
We thank Si Li, Tony Pantev and Daniel Sternheimer for helpful 
discussions. 
We also thank the anonymous referee for the careful reading of 
our manuscript and for many helpful suggestions. 
This research was supported by a Royal Society University Research 
Fellowship; the Leverhulme Trust; ERC Starting Investigator 
Grant number~240123; EPSRC Mathematics Platform grant EP/I019111/1; 
Inoue Research Award for Young Scientists, 
EPSRC EP/E022162/1, 
JSPS Kakenhi Grant Number 19740039, 22740042, 23224002, 24224001 
and 25400069.

\section{Notation in Gromov--Witten Theory}
\label{sec:nota} 
We use the same notation as~\cite{CI:convergence}. 
Let $X$ be a smooth projective variety and let $H_X$
be the even part of $H^\bullet(X;\Q)$.

\subsection{Gromov--Witten Invariants}
\label{sec:GW}
Let $X_{g,n,d}$ denote the moduli space of $n$-pointed
genus-$g$ stable maps to $X$ of degree $d \in H_2(X;\Z)$.
Write
\begin{align}
  \label{eq:correlator}
  \corr{a_1 \psi_1^{l_1},\ldots,a_n
    \psi_n^{l_n}}^{X}_{g,n,d}
  =
  \int_{[X_{g,n,d}]^{\text{vir}}}
  \prod_{i=1}^{n} \ev_i^*(a_i) \cup \psi_i^{l_i}
\end{align}
where $a_1,\ldots,a_n \in H_X$; $\ev_i \colon X_{g,n,d} \to X$ is the
evaluation map at the $i$th marked point; $\psi_1,\ldots,\psi_n \in
H^2\big(X_{g,n,d};\Q\big)$ are the universal cotangent line classes;
$l_1, \ldots, l_n$ are non-negative integers; and the integral denotes
cap product with the virtual fundamental class~\cite{Behrend--Fantechi,Li--Tian}.  The right-hand
side of \eqref{eq:correlator} is a rational number, called a
\emph{Gromov--Witten invariant} of $X$ (if $l_i= 0$ for all $i$) or a
\emph{gravitational descendant} (if any of the $l_i$ are non-zero).

\subsection{Bases for Cohomology and Novikov Rings}
\label{subsec:bases}
Fix bases $\phi_0,\ldots,\phi_N$ and $\phi^0,\ldots,\phi^N$ for $H_X$
such that:
\begin{equation}
  \label{eq:basisproperties}
  \begin{minipage}{0.9\linewidth}
    \begin{itemize}
    \item $\phi_0$ is the identity element $\unit \in H_X$
    \item $\phi_1,\ldots,\phi_r$ is a nef $\Z$-basis for the free part of
 $H^2(X;\Z) \subset H_X$ 
    \item each $\phi_i$ is homogeneous
    \item $(\phi_i)_{i=0}^{i=N}$ and $(\phi^j)_{j=0}^{j=N}$ are dual with
      respect to the Poincar\'e pairing
    \end{itemize}
  \end{minipage}
\end{equation}
Note that $r$ is the rank of $H_2(X)$. 
Define the \emph{Novikov ring} 
$\Lambda = \Q[\![Q_1,\ldots,Q_r]\!]$ and, for $d \in H_2(X;\Z)$, write $Q^d = Q_1^{d_1} \cdots Q_r^{d_r}$ where $d_i = d\cdot \phi_i$. 
 
\subsection{Quantum Cohomology}
\label{sec:convergence}
Let $t^0,\ldots,t^N$ be the co-ordinates of $t \in H_X$ defined by the basis
$\phi_0,\ldots,\phi_N$, so that $t = t^0 \phi_0
+ \ldots + t^N \phi_N$.  
Define the \emph{genus-zero Gromov--Witten potential} 
$F^0_X \in \Lambda[\![t^0,\dots,t^N]\!]$ by
\[
F^0_X = \sum_{d \in \NE(X)} \sum_{n = 0}^\infty {Q^d \over n!}
\corr{\vphantom{\big\vert}t, \ldots,t}^X_{0,n,d}
\]
where the first sum is over the set $\NE(X)$ of degrees of effective
curves in $X$.  
This is a generating function for genus-zero
Gromov--Witten invariants.  The \emph{quantum product} $\ast$ is
defined in terms of the third partial derivatives of $F^0_X$:
\begin{equation}
  \label{eq:bigQC}
  \phi_i \ast \phi_j = \sum_{h=0}^{N} 
  {\partial^3 F^0_X \over \partial t^i \partial t^j \partial t^h}
  \phi^h
\end{equation}
The product $\ast$ is bilinear over $\Lambda$, and defines a formal
family of algebras on $H_X \otimes \Lambda$ parameterized by
$t^0,\ldots,t^N$.  This is the \emph{quantum cohomology} or
\emph{big quantum cohomology} of $X$.

We have defined big quantum cohomology as a formal family of algebras,
i.e. in terms of the ring of formal power series
$\Q[\![Q_1,\ldots,Q_r]\!][\![t^0,\ldots,t^N]\!]$.  
In many cases however, 
including the examples discussed in~\cite{CI:convergence},
the genus-zero Gromov--Witten potential $F^0_X$ converges 
to an analytic function.  
By this we mean the following.  The Divisor
Equation~\cite[\S2.2.4]{Kontsevich--Manin} implies that
\[
F^0_X \in \Q[\![t^0,Q_1 e^{t^1}, \ldots, 
Q_r e^{t^r}, t^{r+1},t^{r+2},\ldots,t^N]\!]
\]
and one can often show, for example by using mirror symmetry, that
$F^0_X$ is the power series expansion of an analytic function:
\[
F^0_X \in \Q\Big\{t^0,Q_1 e^{t^1}, \ldots, Q_r e^{t^r},
t^{r+1}, t^{r+2}, \ldots, t^N\Big\}
\]
We can then set $Q_1 = \cdots = Q_r = 1$, obtaining an analytic
function
\[
F^0_X \in \Q\Big\{t^0, e^{t^1}, \ldots, e^{t^r},
t^{r+1}, t^{r+2},\ldots,t^N\Big\}
\]
of the variables $t^0,\ldots, t^N$ defined in a region
\begin{align}
  \label{eq:LRLnbhd}
  \begin{cases}
    |t^i| < \epsilon_i & \text{$i=0$ or $r<i\leq N$} \\
    \Re t^i \ll 0 & 1 \leq i \leq r
  \end{cases}
\end{align}
We refer to the limit point
\begin{align*}
  \begin{cases}
    t^i =0  & \text{$i=0$ or $r<i\leq N$} \\
    \Re t^i \to -\infty & 1 \leq i \leq r
  \end{cases}
\end{align*}
as the \emph{large-radius limit point}.  When $F^0_X$ converges to an
analytic function in the sense just described, the quantum product
$\ast$ then defines a family of algebra structures on $H_X$ that
depends analytically on parameters $t^0,\ldots,t^N$ in the
neighbourhood \eqref{eq:LRLnbhd} of the large-radius limit point.

\subsection{The Dubrovin Connection}
\label{subsec:Dubrovin_conn}

Consider $H_X \otimes \Lambda$ as a scheme over $\Lambda$ and let
$\cM$ be a formal neighbourhood of the origin in $H_X\otimes \Lambda$.  The
\emph{Euler vector field} $E$ on $\cM$ is
\begin{equation}
  \label{eq:Eulerfield}
  E = {t^0 \parfrac{}{t^0}} + \sum_{i=1}^r \rho^i \parfrac{}{t^i} +
  \sum_{i=r+1}^N \big(1 - \textstyle\frac{1}{2}{\deg \phi_i} \big) 
t^i \parfrac{}{t^i}
\end{equation}
where $c_1(X) = \rho^1 \phi_1 + \cdots + \rho^r \phi_r$.  The
\emph{grading operator} $\mu \colon H_X \to H_X$ is defined by:
\begin{equation} 
\label{eq:gradingoperator}
\mu(\phi_i) = \deg \phi_i - \textstyle\frac{1}{2} \dim_{\C} X
\end{equation} 
Let $\pi \colon \cM \times \A^1 \to \cM$ 
denote projection to the first factor.  
The \emph{extended Dubrovin connection} 
is a meromorphic flat connection $\nabla$ 
on $\pi^* T\cM \cong H_X \times (\cM\times \A^1)$, defined by:
\begin{align*}
& \nabla_{\parfrac{}{t^i}} = \parfrac{}{t^i} 
  - \frac{1}{z}\big(\phi_i {\ast}\big)
&& 0 \leq i \leq N \\
& \nabla_{z \parfrac{}{z}} = z \parfrac{}{z} + \frac{1}{z} 
\big({E\ast}\big) + \mu 
&& \text{where $z$ is the co-ordinate on $\A^1$.} 
\end{align*}
Together with the pairing on $T\cM$ induced by the Poincar\'e pairing,
the Dubrovin connection equips $\cM$ with the structure of a formal
Frobenius manifold with extended structure
connection~\cite{Manin}.

The Dubrovin connection admits a canonical fundamental solution 
(see e.g.~\cite[Proposition~2]{Pandharipande:afterGivental},~\cite[Proposition~2.4]{Iritani:integral}) 
$L\in \End(H_X)\otimes \Lambda[\![t]\!][\![z^{-1}]\!]$, 
defined by
\begin{equation}
\label{eq:fundamentalsolution}
L(t,z) v = v + 
\sum_{d \in \NE(X)} \sum_{n=0}^{\infty} \sum_{\epsilon = 0}^{N} 
\frac{Q^d}{n!} 
  \corr{\frac{v}{z-\psi},t,\ldots,t,\phi^\epsilon}^X_{0,n+2,d}
  \phi_\epsilon
\end{equation}
where $v \in H_X$. The expression $v/(z-\psi)$ in the correlator 
should be expanded as the series $\sum_{n=0}^\infty 
{v \psi^n}{z^{-n-1}}$.  
This satisfies $\nabla_{\partial/\partial t^i} ( L(t,z) v ) =0$ for 
all $i=0,\dots,N$. 
The fundamental solution also satisfies the unitarity property
\begin{align*}
  \big(L(t,-z) v, L(t,z) w\big)_{H_X} = (v,w)_{H_X}
  &&
  \text{for all $v$,~$w \in H_X$}
\end{align*}
where $(\cdot,\cdot)_{H_X}$ denotes 
the Poincar\'{e} pairing on $H_X$. 
Hence the inverse fundamental solution 
$M(t,z):= L(t,z)^{-1}$ is identified with the adjoint of $L(t,-z)$: 
\begin{equation} 
\label{eq:inversefundamentalsolution} 
M(t,z) v= v + 
\sum_{d \in \NE(X)} \sum_{n=0}^{\infty} \sum_{\epsilon = 0}^{N} 
\frac{Q^d}{n!} 
\corr{\frac{\phi^\epsilon}{-z-\psi},t,\ldots,t,v}^X_{0,n+2,d}
\phi_\epsilon
\end{equation} 
The Divisor Equation~\cite[Theorem~8.3.1]{AGV} for descendant 
invariants shows that 
\begin{equation} 
\label{eq:M-divisoreq} 
M(t,z) v = e^{-\delta/z} \left( 
v +  
\sum_{d \in \NE(X)} \sum_{n=0}^{\infty} \sum_{\epsilon = 0}^{N} 
\frac{e^{d\cdot \delta}Q^d}{n!} 
\corr{\frac{\phi^\epsilon }{-z-\psi}, t' ,\ldots, t', 
v}^X_{0,n+2,d}
\phi_\epsilon \right)
\end{equation} 
where $t =\delta + t'$, $\delta \in H^2(X)$, $t' \in  
\bigoplus_{p\neq 1} H^{2p}(X)$. 
This form will be helpful when we specialize the Novikov variables $Q_i$ to $1$ 
in the fundamental solutions.  

If the genus-zero Gromov--Witten potential $F^0_X$ converges to an
analytic function, as discussed in Section~\ref{sec:convergence} above, 
then the extended Dubrovin connection with $Q_1= \cdots = Q_r =1$ 
depends analytically on $t$ in a neighbourhood \eqref{eq:LRLnbhd} 
of the large-radius limit point 
and defines an analytic Frobenius manifold 
with extended structure connection. 
The fundamental solution with $Q_1=\cdots = Q_r=1$ then depends
analytically on both $t$ and $z$, where $t$ lies in the neighbourhood
\eqref{eq:LRLnbhd} and $z$ is any point of $\C^\times$.

\subsection{Gromov--Witten Potentials}
\label{subsec:GW-potentials} 

We introduce various generating functions 
for Gromov--Witten invariants. 
They belong to certain rings of formal 
power series (in infinitely many variables), 
for which we refer the reader to~\cite[\S 2.5]{CI:convergence}. 

Let $(t_0, t_1, t_2, \ldots)$ 
be an infinite sequence of elements of $H_X$
and write $t_n = t_n^0 \phi_0 + \cdots + t_n^N \phi_N$. 
The
\emph{genus-$g$ descendant potential} 
\begin{equation} 
\label{eq:genus_g_descendantpot}
\cF^g_X := \sum_{d \in \NE(X)} \sum_{n = 0}^{\infty} 
\sum_{l_1=0}^\infty
\cdots
\sum_{l_n=0}^\infty 
{Q^d \over n!}
\corr{t_{l_1} \psi_1^{l_1},\ldots,t_{l_n}\psi_n^{l_n}}^X_{g,n,d}
\end{equation} 
is a generating function for genus-$g$ gravitational descendants of $X$.
The \emph{total descendant potential}
\begin{equation} 
\label{eq:totaldescendantpot}
\cZ_X := \exp \Bigg(\sum_{g=0}^\infty \hbar^{g-1} \cF^g_X\Bigg)
\end{equation}
is a generating function 
for all gravitational descendants of $X$. 

Consider now the morphism $p_m \colon X_{g,m+n,d} \to \Mbar_{g,m}$ that
forgets the map and the last $n$ marked points, and then stabilises
the resulting prestable curve.  Write $\psi_{m|i} \in
H^2(X_{g,n+m,d};\Q)$ for the pullback along $p_m$ of the $i$th
universal cotangent line class on $\Mbar_{g,m}$, and write
\begin{multline}
  \label{eq:mixedcorrelator}
  \corr{a_1 \bar{\psi}_1^{k_1},\ldots,a_m
    \bar{\psi}_m^{k_m}:b_1 \psi_{m+1}^{l_1},\ldots,b_n
    \psi_{m+n}^{l_n}}^{X}_{g,m+n,d} \\
  =
  \int_{[X_{g,m+n,d}]^{\text{vir}}}
  \prod_{i=1}^{m} 
  \Big(\ev_i^*(a_i) \cup \psi_{m|i}^{k_i}\Big)
  \cdot
  \prod_{j=1}^{n} \Big(\ev_{m+j}^* (b_{j}) \cup
  \psi_{m+j}^{l_{j}} \Big)
\end{multline}
where $a_1,\ldots,a_m \in H_X$; $b_1,\ldots,b_n \in
H_X$; and $k_1,\ldots,k_m, l_1,\ldots,l_n$ are non-negative
integers.

As above, consider $t \in H_X$ with $t = t^0\phi_0 + \cdots + t^N
\phi_N$ and a sequence $(y_0,y_1,y_2,\dots)$ 
of elements in $H_X$ with $y_n = y_n^0 \phi_0 + \cdots + y_n^N \phi_N$. 
The \emph{genus-$g$ ancestor potential}
is 
\begin{equation}
  \label{eq:ancestor}
  \bar{\cF}^g_X := 
  \sum_{d \in \NE(X)}
  \sum_{n=0}^\infty
  \sum_{m=0}^\infty
  \sum_{l_1=0}^\infty
  \cdots
  \sum_{l_m=0}^\infty
  {Q^d \over n!m!}
 \corr{y_{l_1} \bar{\psi}_{1}^{l_1},\ldots,y_{l_m} \bar{\psi}_{m}^{l_m} : 
    \overbrace{t,\ldots, t}^n}^X_{g,m+n,d}
\end{equation}
and the \emph{total ancestor potential} is:
\begin{equation}
  \label{eq:ancestorpotential}
  \cA_X := \exp \Bigg(\sum_{g=0}^\infty \hbar^{g-1} \bar{\cF}^g_X\Bigg)
\end{equation}
We will often want to emphasize the dependence of the ancestor
potentials on the variable $t$, writing $\bar{\cF}^g_t$ for
$\bar{\cF}^g_X$ and $\cA_t$ for $\cA_X$.  Note that the ancestor
potentials \eqref{eq:ancestor} do not contain terms with $g=0$ and
$m<3$, or with $g=1$ and $m=0$, as in these cases the space
$\Mbar_{g,m}$ is empty and so the map $p_m\colon X_{g,m+n,d} \to
\Mbar_{g,m}$ is not defined.

Let $(t_0, t_1,t_2,\ldots)$ and $(y_0, y_1, y_2,\ldots)$ 
be infinite sequences of elements of $H_X$ 
with $t_n= t_n^0 \phi_0 + \cdots + t_n^N \phi_N$ 
and $y_n = y_n^0 \phi_0 + \cdots + y_n^N \phi_N$. 
Define the \emph{genus-$g$ jet potential}
\begin{multline*}
  \cW^g_X := \sum_{d \in \NE(X)}
  \sum_{m = 0}^{\infty}
 \sum_{k_1=0}^\infty
  \cdots
  \sum_{k_m=0}^\infty \\
  \sum_{n=0}^\infty 
  \sum_{l_1=0}^\infty
  \cdots
  \sum_{l_n=0}^\infty 
  {Q^d \over n! m!}
  \corr{y_{k_1} \bar{\psi}_1^{k_1},\ldots,y_{k_m}\bar{\psi}_m^{k_m}:
    t_{l_1} \psi_{m+1}^{l_1},\ldots,t_{l_n} \psi_{m+n}^{l_n}
  }^X_{g,m+n,d}
\end{multline*} 
We write $\cW_X = \sum_{g=0}^\infty \hbar^{g-1} \cW^g_X$. 
The \emph{total jet potential} is:
\begin{equation} 
\label{eq:jetpotential-GW} 
\exp(\cW_X) = \exp\Bigg(\sum_{g=0}^\infty \hbar^{g-1} \cW^g_X
\Bigg)
\end{equation}

The co-ordinates $(t_0,t_1,t_2,\dots)$ are used for the descendant 
potentials and the co-ordinates $(y_0, y_1, y_2, \dots)$ are used 
for the ancestor potentials. 
We sometimes also use the co-ordinates 
$(q_0,q_1,q_2,\dots)$ with $q_n = q_n^0 \phi_0 + \cdots + q_n^N \phi_N$ 
related to  $(t_0,t_1,t_2,\dots)$ or $(y_0,y_1,y_2,\dots,)$ 
by the identification:
\begin{align*}
  q_n^i = - \delta_{n,1} \delta_{i,0} + t_n^i 
&&   q_n^i = - \delta_{n,1} \delta_{i,0} + y_n^i
\end{align*}
This identification is called the \emph{Dilaton shift}. 
See also \S\ref{subsec:dilatonshift} below. 

\subsection{The Orbifold Case} 
The discussion in this paper applies to the case 
where $X$ is a smooth algebraic orbifold or Deligne--Mumford stack,
rather than a smooth algebraic variety.  
The discussion above goes through in this 
situation with minimal changes, as follows:
\begin{itemize}
\item We take $H_X$ to be the even part\footnote{Here we mean the even
    part of the rational cohomology of the inertia stack $IX$ with
    respect to the usual grading on $H^\bullet(IX)$, not the age-shifted
    grading; cf.~\cite[\S2.2]{Coates--Iritani--Jiang}} of the Chen--Ruan orbifold cohomology
  $H^\bullet_{\CR}(X;\Q)$ rather than the even part of the ordinary
  cohomology $H^\bullet(X;\Q)$.
\item We replace:
  \begin{itemize}
  \item the usual grading on $H^\bullet(X)$;
    by the age-shifted grading on $H^\bullet_{\CR}(X)$
  \item the Poincar\'e pairing on $H^\bullet(X)$ by the orbifold
    Poincar\'e pairing on $H^\bullet_{\CR}(X)$.
  \end{itemize}
  Note that $H^2(X) \subset H^2_{\CR}(X)$, and so definition
  \eqref{eq:basisproperties} makes sense in the orbifold context.
\item We define correlators \eqref{eq:correlator} and
  \eqref{eq:mixedcorrelator} using orbifold Gromov--Witten invariants
 ~\cite{AGV} rather than usual Gromov--Witten invariants.  
  There are two small differences:
  \begin{itemize}
  \item a subtlety in the definition of $\ev_k^*$, 
  discussed in~\cite{AGV},~\cite[\S2.2.2]{CCLT};
  \item the degree $d$ of an orbifold stable map $f:\Sigma \to X$ 
    lies in $H_2(|X|;\Z)$, where $|X|$ is the coarse moduli 
    space of $X$. 
  \end{itemize}
\end{itemize}
Having made these changes, the discussion in
\S\S\ref{sec:GW}--\ref{subsec:GW-potentials} 
applies to orbifolds as well.  
In this context, the family of algebras 
$\big(H_X \otimes\Lambda,\ast\big)$ 
is called \emph{orbifold quantum cohomology}~\cite{Chen--Ruan:GW}. 

\section{Global Quantization: Motivation}
\label{sec:globalquantization:motivation}

In this section, as an introduction to global quantization, we review
Givental's symplectic formalism~\cite{Givental:quantization, 
Coates--Givental, Givental:symplectic} 
from the viewpoint of geometric
quantization.  This section is not logically necessary, and can safely
be skipped by the impatient reader, but provides motivation and
context for the rest of the paper.  Roughly speaking one can think
of our Fock space as obtained from the quantization of Givental's
infinite-dimensional symplectic space $\cH$, and of the total
descendant potential $\cZ_X$ as an element of the Fock space.  The aim
of this section is to give an informal account of the ideas behind the
rigorous construction, which is given in \S\ref{sec:global_theory} 
and \S\ref{subsec:global_L2}.

\subsection{Givental's Symplectic Vector Space} 
\label{subsec:Givental-symplecticvs} 
Givental's quantization is based on the 
Hilbert space  
\[
\cH = H_X \otimes_{\Q} L^2(S^1,\C) 
\]
equipped with the symplectic form:
\[
\Omega(f(z), g(z)) = 
\frac{1}{2\pi\iu} 
\int_{S^1} (f(-z), g(z))_{H_X} \, dz
\]
Here $L^2(S^1,\C)$ denotes the space of 
complex-valued $L^2$ functions on $S^1$ 
and $(\alpha,\beta)_{H_X} = \int_X \alpha \cup \beta$ 
is the Poincar\'{e} pairing. The co-ordinate $z$ 
on $S^1$ coincides with the variable that appeared in the 
Dubrovin connection (see \S\ref{subsec:Dubrovin_conn}). 
We call $(\cH, \Omega)$ the \emph{Givental space} for $X$. 
Each element $f(z)\in \cH$ has 
a Fourier expansion 
\[
f(x) = \sum_{n=0}^\infty q_n z^n + \sum_{n=0}^\infty p_n (-z)^{-n-1} 
\]
with $q_n, p_n \in H_X\otimes \C$. 
We have the decomposition $\cH = \cH_+ \oplus \cH_-$ 
where:
\begin{align} 
\label{eq:cH+-}
\cH_+ & = \Biggl\{\bq = \sum_{n=0}^\infty q_n z^n \in \cH \Biggr \} &
\cH_- & = \Biggl\{ \bp = \sum_{n=0}^\infty p_n (-z)^{-n-1} \in \cH
\Biggr\} \\
\intertext{These are maximally isotropic subspaces. We set}
q_n & = \sum_{i=0}^N q_n^i \phi_i &
p_n & = \sum_{i=0}^N p_{n,i} \phi^i 
\end{align}
and regard $\{\text{$q_n^i$, $p_{n,i}$} :
\text{$0 \leq n < \infty$, $0\le i \le N$}\}$ as a complex 
co-ordinate system on $\cH$. 
These are holomorphic Darboux co-ordinates,
in the sense that:
\[
\Omega = \sum_{n=0}^\infty \sum_{i=0}^N dp_{n,i} 
\wedge dq_n^i
\]
\subsection{Dilaton Shift} 
\label{subsec:dilatonshift} 
Let us denote by $\cF^{g}$ the genus-$g$ descendant 
Gromov--Witten potential \eqref{eq:genus_g_descendantpot} 
with Novikov variables specialized\footnote{In order to make sense of this specialization, we need 
a certain convergence assumption for $\cF^g$~\cite[\S8.1]{CI:convergence}. 
This technical point will be explained in 
Definition~\ref{def:F^g_Xan} below.} to 1 (i.e.~$Q_1= \cdots = Q_r=1$). 
We can regard $\cF^g$ as a holomorphic function 
on an open subset $U$ of $\cH_+$
\[
\cF^{g} \colon U \to \C 
\]
via the \emph{Dilaton shift}
\begin{equation*} 
\bq = \bt  - z \unit  
\end{equation*} 
where $\unit \in H_X$ is the identity element and we set:
\begin{align*}
  \bt = \sum_{n = 0}^\infty \sum_{i=0}^N t_n^i \phi_i z^n && 
  \bq = \sum_{n = 0}^\infty \sum_{i=0}^N q_n^i \phi_i z^n
\end{align*}
The open subset $U$ contains a point 
$-z \unit + t$ with $t$ in a neighbourhood 
\eqref{eq:LRLnbhd} of the large radius limit point 
and $\cF^g|_{-z \unit +t}$ gives the non-descendant 
genus-$g$ Gromov--Witten potential $F^g(t)$ with the Novikov variables specialized to~$1$.

\subsection{Lagrangian Submanifold and TP Structure} 
\label{subsec:Lagrangian_TP}
Here we introduce the Givental cone for $X$, a submanifold $\cL$ of $\cH$ 
which encodes all information about
the genus-zero Gromov--Witten theory of $X$.  
Define $\cL$ to be the following submanifold 
of $\cH$:  
\[
\cL = \left\{\bq + \bp \, \Bigg |\, p_{n,i} 
= \parfrac{\cF^{0}}{q_n^i}(\bq) \right\}
\]
where we set 
$\bp = \sum_{n=0}^\infty \sum_{i=0}^N 
p_{n,i} \phi^i  (-z)^{-n-1}$. 
This is Lagrangian since 
it is the graph of the differential $d\cF^0$. 
Moreover, it has the following 
special geometric properties~\cite{Coates--Givental, Givental:symplectic}: 
\begin{itemize} 
\item $\cL$ is a cone, i.e.~it is preserved by 
scalar multiplication. 
\item $T_f$, the tangent space of $\cL$ at 
$f\in \cL$, is tangent to $\cL$ exactly along 
$z T_f$. This means: 

(i) $z T_f \subset \cL$; 

(ii) For $g \in z T_f$, we have $T_g = T_f$; 

(iii) $T_f \cap \cL = z T_f$. 
\end{itemize}  
The Lagrangian submanifold $\cL$ is a submanifold-germ 
around the unique family of points on $\cL$ 
of the form: 
\begin{align*} 
t \longmapsto J(t,-z) = -z \unit + t + \bp_t && \bp_t \in \cH_-
\end{align*} 
with $t\in H_X^\C$ in a neighbourhood \eqref{eq:LRLnbhd} 
of the large radius limit point, and the above properties 
should be understood in the sense of germs. 
The set of all tangent spaces to $\cL$ forms a finite dimensional family: 
every tangent space coincides with $T_t = T_{J(t,-z)}\cL$ at 
$J(t,-z)$ for a unique $t\in H_X^\C$. 
The point $J(t,-z)$ on $\cL$ is called the \emph{$J$-function}. 
Moreover we can recover the Lagrangian submanifold 
$\cL$ as the union of tangent spaces: 
\[
\cL = \bigcup_{f\in \cL} z T_f 
= \bigcup_{t\in H_X^\C} z T_t
\]
The special geometric properties of $\cL$ can be rephrased 
as \emph{Griffiths transversality} for the family $\{T_t\}$ 
of semi-infinite subspaces; $\{T_t\}$ is an example of 
Barannikov's \emph{variation of semi-infinite Hodge structure} ~\cite{Barannikov}. 

We saw that $\cL$ is ruled by infinite-dimensional 
spaces $z T_t$. The ruling structure can be understood 
via the identification of $\cL$ with the total space of a certain 
infinite-dimensional vector bundle, as follows. 
Consider the 
vector bundle $\cH \times H_X^\C \to H_X^\C$ 
endowed with the (non-extended) Dubrovin connection 
$\nabla = d - \frac{1}{z} 
\sum_{i=0}^N (\phi_i *) dt^i$ 
(see \S\ref{subsec:Dubrovin_conn}). 
The inverse
$M(t,z)=L(t,z)^{-1}$ of the fundamental solution 
(see equation~\ref{eq:inversefundamentalsolution}) 
defines an isomorphism 
of flat bundles: 
\[
M \colon \left(
\cH \times H_X^\C \to H_X^\C,   
\nabla \right) \stackrel{\sim}{\longrightarrow}
\left( 
\cH \times H_X^\C \to  
H_X^\C, d \right)
\]
Here Novikov variables have (again) been specialized to 1 in 
$\nabla$, $L(t,z)$ and $M(t,z)$. 
We have:
\begin{align*}
  M(t,z) (-z \unit) = J(t,-z) &&
  M(t,z) z \cH_+ = z T_t
\end{align*}
Therefore the Lagrangian submanifold $\cL$ 
is obtained as the projection to the fiber $\cH$ 
of the image of the subbundle $z \cH_+\times H_X^\C$ 
under the map $M$: 
\[
\cL = (\pr_{\cH} \circ M)
\left(z \cH_+ \times H_X^\C\right)  
\]
and the bundle structure $z \cH_+ \times H_X^\C \to H_X^\C$ 
gives the ruling on $\cL$. See Figure~\ref{fig:ruling}. 

\begin{figure}[tbp] 
\includegraphics[bb=48 540 552 732]{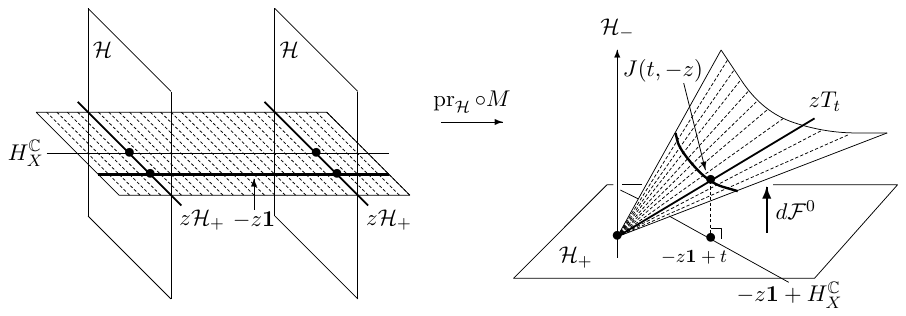}
\caption{The subbundle $z\cH_+\times H_X^\C \to H_X^\C$ 
(left) 
and the ruled Lagrangian submanifold $\cL$ 
(right). By identifying all the fibers $\cH$ 
by the $\nabla$-parallel transportation 
in the left picture, we get the picture 
on the right. 
The zero section collapses to the origin  
and the section $-z\unit$ goes to the $J$-function $J(t,-z)$.}
\label{fig:ruling}
\end{figure} 
Using this identification, we can introduce 
two different co-ordinate systems on $\cL$. 
\begin{description} 
\label{desc:flat-vs-alg} 
\item[flat co-ordinates] 
$(q_0, q_1,q_2, \cdots) 
\mapsto (\bq, d \cF^0(\bq))$.  These are the co-ordinates 
given by the projection to $\cH_+$; here $q_n = \sum_{i=0}^N q_n^i \phi_i 
\in H_X^\C$. 

\item[algebraic co-ordinates]
$(t,x_1,x_2,\dots,) 
\mapsto M(t,z) (x_1 z + x_2 z^2 + x_3 z^3 + \cdots)$.  These are
the co-ordinates coming from the standard co-ordinates 
on $z\cH_+ \times H_X^\C$; here $t, x_n \in H_X^\C$. 
\end{description} 
We saw that the Lagrangian submanifold $\cL$ can be identified with 
the total space of the infinite-dimensional vector bundle 
$z\cH_+ \times H_X^\C \to H_X^\C$. 
This infinite-dimensional vector bundle arises 
from the \emph{finite dimensional} vector bundle 
\[
F= H_X^\C \times (H_X^\C\times \C_z)  
\to H_X^\C \times \C_z
\]
as its push-forward $\pi_*(z \cO(F))$ 
along the projection $\pi \colon H_X^\C \times \C_z \to H_X^\C$; 
here $\C_z$ denotes the complex plane with co-ordinate $z$. 
The finite dimensional vector bundle $F$ over $H_X^\C \times \C_z$ 
is endowed with a flat connection $\nabla$ and a $\nabla$-flat 
pairing $(\cdot,\cdot)_F$. 
The structure $(F,\nabla,(\cdot,\cdot)_F)$ here is given the name \emph{TP structure} in Definition~\ref{def:TP} below;
this terminology is borrowed from Hertling~\cite{Hertling:ttstar}. 
The global quantization formalism in \S\ref{sec:global_theory} 
is based on a closely related structure called a \emph{cTP structure}, for `complete TP' structure: 
we replace the Lagrangian submanifold $\cL$ above with the total space of a cTP structure. 
The use of algebraic co-ordinates will be important there. 

\begin{remark} 
  Both TP structures and variations of semi-infinite Hodge structure 
  are generalizations of the notion of \emph{Variation of Hodge Structure} (VHS), and
  in fact reduce to it when we deal with the small quantum cohomology
  of a 3-dimensional Calabi--Yau manifold.  These structures originate from K.~Saito's theory of primitive forms~\cite{SaitoK}, and have been
  rediscovered in the context of integrable systems, string theory, and
  mirror symmetry~\cite{Dubrovin:2DTFT,Barannikov,Katzarkov--Kontsevich--Pantev}.
\end{remark}

\subsection{Geometric Quantization} 
\label{subsec:geometricquantization} 
The quantization of a real symplectic manifold $H$ 
is given by a Hilbert space $\Fock(H)$ called 
the \emph{Fock space}, and an assignment of an operator 
$\hF$ acting on the Fock space $\Fock(H)$ to 
a smooth function $F\colon H \to \R$ 
such that 
\[
[\hF_1, \hF_2] = 
\iu\hbar \widehat{\{F_1, F_2\}} + O(\hbar^2) 
\]
where $\{\cdot, \cdot \}$ 
is the Poisson bracket and 
$\hbar$ is a formal variable. 
In geometric quantization, the construction 
of the Fock space depends on the choice 
of a polarization $P$, i.e.~an integrable 
Lagrangian subbundle of $TH\otimes \C$. 
To emphasize the dependence on $P$, 
we denote by $\Fock(H;P)$ the Fock space associated to $P$. 
We illustrate this in the following example.

\begin{example}[\!\cite{Kirillov,Woodhouse}] 
\label{ex:quantization}
Take $H$ to be the symplectic vector space $\R^{2n}$ 
with co-ordinates $(p_\mu, q^\mu)$, $\mu = 1,\dots,n$. 
Let $\omega =\sum_{\mu=1}^n dp_\mu \wedge dq^\mu$ be 
the symplectic form on $H$. 
The \emph{prequantum line bundle} is a Hermitian  
line bundle $L\to H$ endowed with a Hermitian 
connection $\nabla$ such that the curvature 
$\nabla^2$ equals $-\iu \omega/\hbar$, where 
$\hbar$ is a positive real parameter in this example. 
We take the following prequantum line bundle: 
\begin{align*}
  L = H \times \C &&  \nabla = d - 
\frac{\iu}{2\hbar} \sum_{\mu=1}^n (p_\mu dq^\mu - q^\mu dp_\mu)
\end{align*}
The connection $\nabla$ here is 
$Sp(H)$-invariant. 
For $F\in C^\infty(H,\R)$, we define 
\[
\hF : = \iu \hbar \nabla_{X_F} + F 
\]
where $X_F$ is the Hamiltonian vector field of $F$ 
(i.e.\ $\iota_{X_F} \omega = dF$). 
This operator acts on the space $C^\infty(H,L)$ of sections of 
$L$ and we have $[\hF, \hG]= \iu\hbar 
\widehat{\{ F, G\}}$. 
This is called \emph{prequantization}. 
However, $C^\infty(H,L)$ is too big and we 
need to take a smaller subspace. 
Let $P\subset H\otimes \C$ be a Lagrangian subspace. 
We can view $P$ as a subbundle of $TH\otimes \C$ 
which is invariant under translation. 
The space of \emph{polarized sections} of $L$ is 
defined to be 
\[
\Gamma_P(H,L) = \{ s\in C^\infty(H,L) : 
\text{$\nabla_V s = 0$ for all $V \in P$}\}
\]
Note that $[\nabla_{V_1}, \nabla_{V_2} ] = 0$ for 
$V_1,V_2 \in P$ 
because $P$ is Lagrangian and 
$\nabla^2 = - \iu \omega/\hbar$. 
There are two important special cases: 
\begin{itemize} 
\item When $P\subset H$, it is called the 
\emph{real polarization}. In this case, $\Gamma_P(H,L)$ 
is the space of sections of $L$ which are covariantly 
constant along each leaf $v + P$, $v\in H$. 

\item When $P \oplus \ov{P} = H\otimes \C$, it is called 
a \emph{K\"{a}hler} or \emph{holomorphic polarization}. 
This corresponds to the choice of a complex structure 
$I_P$ on $H$ such that $\omega(v_1,v_2) = 
\omega(I_P v_1, I_P v_2)$ and $P = (H\otimes \C)^{0,1}$. 
In this case, $\Gamma_P(H,L)$ is the space of 
holomorphic sections of $L$ (with respect to $I_P$). 
\end{itemize}
Suppose that $P$ is non-negative, i.e.~that $\iu \omega(v,\overline{v}) \ge 0$
for all $v\in P$. Then one can introduce a certain $L^2$-metric on
the space of polarized sections~\cite{Kirillov} and
define the Fock space $\Fock(H;P)$ to be the Hilbert space
of $L^2$-polarized sections.
If the flow generated by $X_F$ 
preserves the polarization $P$ as a subbundle 
of $TH\otimes \C$, then the operator $\hF$ preserves 
the subspace $\Gamma_P(H,L)$ and acts on the 
Fock space (possibly as an unbounded operator). 
In particular, the quantizations of the 
linear functions $p_\mu, q^\mu$ 
act on $\Fock(H,P)$ and satisfy the canonical 
commutation relation: $
[\hq^\mu, \hp_\nu] = \iu\hbar \delta^\mu_\nu$. 
Thus $\Fock(H,P)$ becomes an irreducible unitary 
representation of the Heisenberg algebra. 
Because an irreducible unitary representation of the Heisenberg 
algebra is unique up to isomorphism 
(the Stone--von~Neumann Theorem), 
Schur's Lemma shows that there exists 
an
isomorphism 
\[
T_{P,P'} \colon \Fock(H,P) \stackrel{\sim}{\longrightarrow} \Fock(H,P')
\]
of representations of the Heisenberg algebra. 
This isomorphism $T_{P,P'}$ is unique up to scalar multiplication.
For example, when $P$ is the subbundle 
spanned by $\partial/\partial p_\mu$ and 
$P'$ is spanned by $\partial/\partial q^\mu$, 
the isomorphism $T_{P,P'}$ is given by the Fourier 
transformation 
\begin{equation} 
\label{eq:FT}
\psi(q) \longmapsto 
\hpsi(p) = \frac{1}{(2\pi \hbar)^{n/2}}
\int_{\R^n} e^{-\iu pq/\hbar} \psi(q) \, dq 
\end{equation} 
where we identify elements of $\Fock(H,P)$ (respectively of $\Fock(H,P')$) 
as functions of the $q^\mu$ (respectively of the $p_\mu$) by 
restriction to $p_\mu=0$ (respectively to $q^\mu=0$). 
The transformation $T_{P,P'}$ is known as a 
\emph{Segal--Shale--Weil representation} 
or \emph{Bogoliubov transformation}.  
\end{example} 

We regard the Givental space $\cH$ as a complexification 
of a real symplectic vector space $\cH_\R$ and 
try to apply the above scheme to it. 
However, since $\cH$ is infinite-dimensional, 
the quantization has many difficulties. 
For example, it is known that there are uncountably 
many irreducible representaions of the infinite-dimensional Heisenberg algebra~\cite{Florig--Summers,Garding--Wightman}, and so 
the argument in Example~\ref{ex:quantization} fails in our situation. 
The following heuristic discussion will be only used 
as a motivation\footnote{ 
For example, we \emph{do not} construct 
the Fock space as a representation of the Heisenberg 
algebra. Our Fock space is not even a vector space.}. 
Consider the following prequantum line bundle: 
\begin{align*}
  L = \cH \times \C &&
  \nabla = d - \frac{1}{2\hbar} 
  \sum_{n=0}^\infty \sum_{i=0}^N  (p_{n,i} dq_n^i -  q_n^i dp_{n,i})    
\end{align*}
Here we dropped the imaginary unit since we will ignore 
the metric. 
As the standard polarization of $\cH$, we take  
$P = \cH_-$ which is spanned by 
$\partial/\partial p_{n,i}$. 
In this case, a polarized section $s$ of $L$ 
should take the form 
\begin{align*}
  s = \exp\left(-\frac{1}{2\hbar} \bq\cdot \bp\right) f(\bq) &&
  \bq \cdot \bp = \sum_{n=0}^\infty \sum_{i=0}^N q_n^i  p_{n,i}
\end{align*}
for some holomorphic function $f$ on $\cH_+$. 
Following Givental's convention~\cite{Givental:quantization}, 
we define the quantized operator of 
a linear function $F\colon \cH \to \C$ as: 
\[
\hF := \frac{1}{\sqrt{\hbar}} \big( - \hbar 
\nabla_{X_F} + F\big)
\]
Then it is easy to check that the actions  of 
$p_{n,i}$,~$q_n^i$ on 
polarized sections are given by:
\begin{align*} 
\hp_{n,i} 
\left( e^{-\frac{1}{2\hbar} \bq\cdot \bp} f(\bq) \right) 
& = e^{-\frac{1}{2\hbar} \bq\cdot \bp} 
\left( 
\sqrt{\hbar} \parfrac{}{q_n^i} f(\bq) 
\right) \\ 
\hq_n^i 
\left( e^{-\frac{1}{2\hbar} \bq \cdot \bp} f(\bq) \right) 
& = e^{-\frac{1}{2\hbar} \bq \cdot \bp} 
\left ( 
\frac{q_n^i}{\sqrt{\hbar}} f(\bq)\right)
\end{align*} 
These give the Schr\"{o}dinger representation. 
By Dilaton shift we regard the 
total descendant potential $\cZ = \exp(\sum_{g=0}^\infty 
\hbar^{g-1} \cF^g)$ of $X$ (see equation~\ref{eq:totaldescendantpot}) 
as a function on $\cH_+$. (Here again 
all Novikov variables $Q_1,\dots, Q_r$ have been specialized to $1$.) 
We regard the total descendant potential as a polarized section of $L$ 
by the following extension: 
\begin{equation}
\label{eq:pol_section_Z} 
\ov\cZ(\bq,\bp) = \exp\left( 
- \frac{1}{2\hbar}\bq \cdot \bp\right) \cZ(\bq) 
\end{equation} 
Let us consider the restriction of $\ov\cZ$ to the 
Lagrangian submanifold $\cL$. Note that  
$\cF^0$ is homogeneous of degree two in $\bq$ 
since $\cL$ is a cone. Therefore:
\begin{align*} 
\cZ '(\bq) = 
\ov\cZ(\bq,\bp)\Big|_{(\bq,\bp)\in \cL}  & = \exp 
\left( -\frac{1}{2\hbar} \bq \cdot d\cF^0\right )\cZ(\bq) 
= \exp\left( -\frac{1}{\hbar} \cF^0 \right) \cZ(\bq) \\  
&= \exp\left(
\cF^1(\bq) + \hbar \cF^2(\bq) + \hbar^2 \cF^3(\bq)+ 
\cdots  \right)
\end{align*} 
where the genus $0$ potential cancelled in the 
second line. 
Therefore \emph{we can forget about the genus-zero 
potential after restricting to $\cL$}.  
Moreover the original polarized section can be reconstructed 
from this restriction if we know the submanifold 
$\cL$. Therefore we shall define the Fock space 
to be the set of certain functions $\cZ' \colon \cL \to \C$ 
over $\cL$ of the form 
\[
\cZ' = \exp\left(\sum_{g=1}^\infty \hbar^{g-1} \cF^g\right)
\] 
(without genus-zero term).   Different choices of polarization give different ways
of extending functions $\cZ'$ on $\cL$ to $\cH$.

\begin{remark} 
\label{rem:lower_triangular}
Givental~\cite{Givental:quantization} 
defined the quantized operator $\hUU$ 
on the Fock space for a linear 
symplectic transformation $\U \in Sp(\cH)$. 
In particular, if $\U(z)$ is an element of the 
loop group $LGL(H_X^\C)$ satisfying 
$\U(-z)^\dagger \U(z) =\unit$, it defines a 
symplectic transformation $\U$ of $\cH$ and 
its quantization $\hUU$. 
(Here $\U(-z)^\dagger$ is the adjoint with 
respect to the Poincar\'{e} pairing on $H_X$.) 
This is called the \emph{Givental group action} 
on the Fock space. 
The above interpretation in geometric quantization 
immediately explains the lower triangular part 
of the Givental group action. 
If the Fourier expansion of $\U(z)$ contains 
only non-positive powers in $z$, then $\U=\U(z)$ preserves the 
standard polarization $\cH_-$. In this case, we can 
make it act on polarized sections $s$ via the 
``co-ordinate change": $s\mapsto \hUU s 
:= s \circ \U(z)^{-1}$. 
For the polarized section $\ov\cZ$ 
in \eqref{eq:pol_section_Z}, 
we have: 
\[
(\hUU \ov\cZ)(\bp,\bq) = 
\exp\left(-\frac{1}{2\hbar} \bp\cdot \bq\right)
\exp\left(\frac{1}{2\hbar} W(\bq,\bq)\right)  
\cZ\left([\U(z)^{-1} \bq(z)]_+\right)
\] 
where $[\U(z)^{-1}\bq(z)]_+ \in \cH_+$ 
denotes the non-negative part as a $z$-series and 
$W(\bq,\bq)$ is the quadratic form defined by: 
\begin{equation} 
\label{eq:form-W} 
W(\bq,\bq)= \Omega([\U(z)^{-1} \bq(z)]_+, [\U(z)^{-1} \bq(z)]_-)
\end{equation} 
This coincides with Givental's formula~\cite[Proposition~5.3]{Givental:quantization}
for $\hUU$. 
\end{remark} 

\subsection{Ancestor-Descendant Relation}
\label{subsec:AncDec} 
We have seen that the total descendant potential $\cZ$ gives rise
to a polarized section $\ov\cZ$ which restricts over $\cL$ to the potential $\cZ'$
which does not contain the genus-zero term.  
A result of Kontsevich--Manin~\cite[Theorem~2.1]{Kontsevich--Manin:relations},
reformulated\footnote{See Coates--Givental~\cite[Appendix~2]{Coates--Givental} for 
a proof.} by Givental~\cite[\S5]{Givental:quantization},  
tells us that $\cZ'$ coincides with the total ancestor potential 
\eqref{eq:ancestorpotential}
with the zeroth variable $y_0$ set to be zero:
\[
\cZ'(\bq) = e^{F^1(t)} \cA_t 
\Bigr |_{y_0 =0,\, y_1 = x_1 + \unit,\, 
y_2 = x_2,\, y_3 = x_3, \dots} 
\]
where $(t,x_1,x_2,\dots)$ are 
the \emph{algebraic co-ordinates} 
from page \pageref{desc:flat-vs-alg}. 
Set $\bx=\sum_{n=1}^\infty x_n z^n$, and notice that $\bq$ and $(t,\bx)$ are 
related by $\bq = [M(t,z) \bx]_+$. 
In other words, for $g\ge 1$, we have:
\[
\cF^g(\bq) = \delta_{g,1} F^1(t) + \bar\cF^g_{t}
\Bigr |_{y_0 =0,\, y_1 = x_1 + \unit,\, 
y_2 = x_2,\, y_3 = x_3, \dots}
\]
Strictly speaking, this relation is an equality of 
formal power series over the Novikov ring.  
We shall explain how to make sense of it as an equality of analytic functions in Theorem~\ref{thm:ancdec-analytic} below,
where we discuss the specialization to $Q_1 = \cdots =Q_r =1$. 

\subsection{Transformation Rule and the Fock Sheaf}  
\label{subsec:motivation-Focksheaf}
We have so far discussed only local situations, since 
the Lagrangian submanifold $\cL$ is given \emph{a~priori} as a germ. 
Assume that $\cL$ is analytically continued 
to a certain global submanifold. 
We would like to construct a Fock sheaf over $\cL$ 
by gluing local Fock spaces. 
The essential point here 
is that \emph{a transversal polarization may not exist globally}. 
Take an open covering $\{U_\alpha\}$ of $\cL$. 
If each $U_\alpha$ is sufficiently small then
we can choose a polarization $P_\alpha\subset \cH$ 
which is transversal to $\cL$ over $U_\alpha$, i.e.\ 
$P_\alpha \pitchfork T_x\cL$ for $x\in U_\alpha$. 
We take a Lagrangian subspace $S \subset \cH$ 
transversal to $P_\alpha$. 
By the identification $\cH = S \oplus P_\alpha 
\cong T^* S$, we can express 
$U_\alpha \subset \cL$ as the graph of the differential 
$d\cF^0$ of a quadratic function:
\[
\cF^0 \colon S \to \C
\] 
This defines the genus-zero potential over $U_\alpha$. 
(Here we identify $U_\alpha$ with a subset of $S$ via the 
projection $\cH \to S$ along $P_\alpha$.) 
The third derivatives $C_{\mu\nu\rho}^{(0)} = 
\partial_\mu \partial_\nu \partial_\rho 
\cF^0$ in linear coordinates $\{x^\mu\}$ on $S$ 
define a well-defined cubic tensor on $\cL$,
independent of the choice of $(S,P_\alpha)$. 
The tensor $\sum C^{(0)}_{\mu\nu\rho} 
dx^{\mu} \otimes dx^{\nu} \otimes dx^{\rho}$ 
is called the \emph{Yukawa coupling}. 

We define the local Fock space $\Fock(U_\alpha; P_\alpha)$ 
to be the set of functions $\cZ' \colon U_\alpha \to \C$ 
of the form $\cZ'= 
\exp(\sum_{g=1}^\infty \hbar^{g-1} \cF^g)$ 
without the genus-zero term. 
When $U_\alpha \cap U_\beta \neq \varnothing$, 
we shall define a \emph{transformation rule}\footnote
{Our transformation rule $T_{\alpha\beta}$ 
is defined up to a scalar multiple, due to the ambiguity 
at genus one,
and $T_{\gamma \alpha} T_{\beta\gamma} T_{\alpha\beta}  = 
c_{\alpha\beta\gamma} \id$ for some constant 
$c_{\alpha\beta\gamma}$.  
Later we ignore the constant ambiguity and 
work with the genus-one one-form $d\cF^1$ 
rather than the potential function $\cF^1$. 
See Definition~\ref{def:localFock} and Proposition~\ref{prop:cocycle}.} 
(gluing map) 
\[
T_{\alpha\beta} \colon 
\Fock(U_\alpha\cap U_\beta;P_\alpha) \to \Fock(U_\alpha\cap U_\beta;P_\beta) 
\]
induced by the change of polarizations. 
This defines a sheaf of Fock spaces---\emph{the Fock sheaf}---over $\cL$. 
\begin{figure}[tbp]
\centering 
\includegraphics[bb=162 593 436 737]{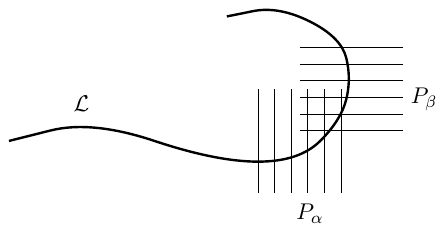}
\caption{We need to take a different polarization 
on each chart.}
\end{figure}
Moreover, if there exists a linear symplectic transformation 
$\U \in Sp(\cH)$ which leaves the global Lagrangian 
submanifold invariant ($\U(\cL) = \cL$), then $\U$ acts 
on sections of the Fock sheaf by pull-back along $\U$ 
followed by the transformation rule induced from 
the difference of polarizations. 
(In the context of mirror symmetry, such an automorphism 
$\U$ of $\cL$ arises from the monodromy 
of the mirror family.)
For the Fock sheaf so constructed, we can ask the following questions. 

\begin{question}
Does the total descendant potential extend to a global section 
of the Fock sheaf? 
\end{question}

\noindent and if so:

\begin{question}
Let $\U \in Sp(\cH)$ be a symplectic transformation preserving $\cL$.  
Is the global section invariant under $\U$?  
(This should imply a ``modularity" of the Gromov--Witten potential.)
\end{question} 
The transformation rule $T_{\alpha\beta}$  
is described as follows. 
Let $T_x$ denote the tangent space of $\cL$ at $x\in \cL$. 
For each $x\in U_\alpha\cap U_\beta$, we have 
$T_x \oplus P_\alpha = T_x \oplus P_\beta = \cH$. 
Since $T_x, P_\alpha,P_\beta$ are Lagrangian subspaces, 
we can identify $P_\alpha,P_\beta$ with the dual space of $T_x$ 
via the symplectic form $\Omega$. 
Take $\varphi \in T_x^*$ and let $v_\alpha(\varphi)\in P_\alpha$, 
$v_\beta(\varphi)\in P_\beta$ be the corresponding 
vectors. Then $v_\beta(\varphi) - v_\alpha(\varphi)$ is 
symplectic-orthogonal to $T_x$, and thus belongs to $T_x$. 
Thus we have a map:
\begin{align*}
  \Delta(x) \colon T_x^* \longrightarrow T_x &&
  \varphi \longmapsto v_\beta(\varphi) - v_\alpha(\varphi)
\end{align*}
We can regard $\Delta(x)$ as an element of $T_x\otimes T_x$. 
Then $\Delta$ defines a symmetric bivector field 
on $U_\alpha\cap U_\beta$. 
The polarization $P_\alpha$ defines an affine flat structure 
on $U_\alpha$, via the open embedding to the 
vector space $U_\alpha \hookrightarrow \cH/P_\alpha$. 
Let $\{x^\mu\}$ be a flat co-ordinate system on $U_\alpha$. 
Write $\Delta = \Delta^{\mu\nu} 
\partial_{x^\mu}\otimes \partial_{x^\nu}$. 
For $\cZ_\alpha' = 
\exp(\sum_{g = 1}^\infty \hbar^{g-1} \cF_\alpha^g) \in 
\Fock(U_\alpha\cap U_\beta;P_\alpha)$, we define:
\begin{equation} 
\label{eq:transf-rule_expository}
(T_{\alpha\beta} \cZ'_\alpha)(x)  
:= e^{\frac{1}{2}
\int C^{(0)}_{\mu\nu\rho}(x) 
\Delta^{\mu\nu}(x) dx^\rho} 
\exp\left(\frac{\hbar}{2} \Delta^{\mu\nu}(x) 
\partial_{y^\mu} \partial_{y^\nu}\right) 
\cZ_\alpha''(x;y) \biggr |_{y=0}
\end{equation} 
where 
\[
\cZ_\alpha''(x;y) = \cZ_\alpha'(x+y) 
e^{\frac{1}{\hbar} \left( 
\cF^0(x+y) - \cF^0(x) - (\partial_{\mu} \cF^0(x)) y^{\mu} 
- \frac{1}{2} (\partial_{\mu} \partial_{\nu} \cF^0)(x) 
y^{\mu} y^\nu \right)}
\]
\begin{remark} 
Take Lagrangian subspaces $S_\alpha, S_\beta \subset \cH$ transversal 
to $P_\alpha, P_\beta$ respectively.  
These define genus-zero potentials $\cF^0_\alpha$, 
$\cF^0_\beta$ as above, and we set 
$\cZ_\alpha = \exp(\cF^0_\alpha/\hbar) \cZ_\alpha'$, 
$\cZ_\beta = \exp(\cF^0_\beta/\hbar) \cZ_\beta'$ 
for $\cZ_\alpha' \in \Fock(U_\alpha;P_\alpha)$, 
$\cZ_\beta' \in \Fock(U_\beta;P_\beta)$. 
The definition \eqref{eq:transf-rule_expository} 
originates from the asymptotic 
expansion as $\hbar \to 0$ of the Fourier-type transformation 
(cf.~equation~\ref{eq:FT}) 
\[
\cZ_\beta(x) = \int_{S_\alpha} 
e^{-G_{\alpha\beta}(x,x')/\hbar}\cZ_\alpha(x') \, dx' 
\]
which would make rigorous sense if $\cH$ were finite dimensional. 
In the finite dimensional case, this integral 
representation and its asymptotic expansion 
were used by Aganagic--Bouchard--Klemm~\cite[equation~2.8]{ABK}
to describe the transformation of topological 
string partition functions. 
Here $G_{\alpha\beta}(x,x')$ 
is a quadratic function (the so-called \emph{generating function}) 
defined by  
\[
d G_{\alpha\beta}(x,x') =  \sum y_\mu dx^\mu - \sum y_\mu' d{x'}^{\mu}
\]
where $(x^\mu,y_\mu)$ and $({x'}^{\mu}, y'_\mu)$ 
are Darboux co-ordinate systems on $\cH$ compatible 
with the decompositions $\cH = S_\alpha \oplus P_\alpha$ 
and $\cH=S_\beta \oplus P_\beta$ respectively.  
\end{remark} 

\begin{remark}
The discussion here is far from being rigorous.  
For instance, in the previous Remark we assumed that $(x^\mu, {x'}^\mu)$ 
form a co-ordinate system on $\cH$, which would hardly be true in our 
infinite-dimensional setting. 
In \S\S\ref{subsec:Fockspace}--\ref{subsec:transformation}, we set 
up a correct function space for the Fock space 
and show that the transformation rule 
is indeed well-defined.  
We will give another formulation based on $L^2$-topology 
in \S\ref{subsec:global_L2}, which is more similar to the 
exposition here. 
\end{remark} 

\section{Global Quantization: General Theory} 
\label{sec:global_theory} 

We now construct a rigorous version of the structure sketched out in~\S\ref{sec:globalquantization:motivation}.  Let $\cM$ be a complex manifold and let $\cO_{\cM}$ denote 
the analytic structure sheaf. 
The space $\cM$ will be the base space of a (c)TP structure. 
Examples include the cohomology of a smooth 
projective variety ($A$-model TP structure) and the base 
space of an unfolding of singularities ($B$-model TP structure). 
The discussion in this section also applies 
to the case where $\cM$ is replaced with 
the formal neighbourhood of a point on it, and 
in particular applies to formal Frobenius manifolds (such as those associated to the A-model or B-model). 

\subsection{TP and TEP Structure} 
\label{subsec:TP_TEP}
A TP structure is a certain coherent sheaf with extra structure
over $\cM\times \C$. 
Fix a co-ordinate $z$ on the complex line $\C$. 
Let $(-)\colon \cM \times \C \to \cM \times \C$ 
be the map sending $(t, z)$ to $(t,-z)$ 
and let $\pi \colon \cM\times \C \to \cM$ be 
the projection.

\begin{definition}
\label{def:TP} 
\ 
\begin{enumerate}
\item A \emph{TP structure} $(\cF, \nabla, (\cdot,\cdot)_{\cF})$ 
  with base $\cM$ consists of 
  a locally free $\cO_{\cM\times \C}$-module $\cF$ of rank $N+1$, and 
  a flat connection $\nabla$ with pole along $z=0$ 
  \[ 
  \nabla \colon \cF \to 
  \pi^* \Omega^1_{\cM} \otimes_{\cO_{\cM\times \C}} \cF(\cM\times \{0\})  
  \]  
  so that for $f\in \cO_{\cM\times \C}$, $s\in \cF$,
  and tangent vector fields $v_1,v_2\in \Theta_\cM$:
  \begin{align*}
    \nabla (f s) = df \otimes s + f \nabla s &&
    [\nabla_{v_1}, \nabla_{v_2}] = \nabla_{[v_1,v_2]}  
  \end{align*}
  together with a non-degenerate pairing 
  \[
  (\cdot,\cdot)_{\cF} \colon 
  (-)^* \cF \otimes_{\cO_{\cM\times \C}} \cF \to \cO_{\cM\times \C}
  \]
  which satisfies
  \begin{align*} 
    \begin{split} 
      ((-)^*s_1,s_2)_{\cF} 
      & = (-)^* ((-)^* s_2, s_1)_{\cF}   \\ 
      d ((-)^* s_1, s_2)_{\cF}  
      & = ((-)^* \nabla s_1, s_2)_{\cF} 
      + ((-)^*s_1,\nabla s_2)_{\cF}  
    \end{split} 
  \end{align*}
  for $s_1,s_2 \in \cF$. 
  Here $\cF(\cM\times \{0\})$ denotes the sheaf of 
  sections of $\cF$ with poles of order at most~1 
  along the divisor $\cM\times \{0\} \subset \cM \times \C$. 

\item A \emph{TEP structure} is a TP structure 
  such that the connection $\nabla$ is extended in 
  the $z$-direction with a pole of order 2 along $z=0$. 
  More precisely, it is a TP structure 
  $(\cF,\nabla,(\cdot,\cdot)_{\cF})$ 
  equipped with a $\pi^{-1}\cO_{\cM}$-module map  
  \[
  \nabla_{z\partial_z} \colon 
  \cF \to \cF(\cM\times \{0\}) 
  \]
  such that for $f\in \cO_{\cM\times \C}$, 
  $s$,~$s_1$,~$s_2 \in \cF$, and $v\in \Theta_\cM$:
  \begin{gather*} 
    \nabla_{z\partial_z} (fs) = z(\partial_z f) s + 
    f \nabla_{z\partial_z} s \qquad 
    [\nabla_v, \nabla_{z\partial_z} ] = 0 \\   
    z\partial_z ((-)^* s_1, s_2)_{\cF} 
    = ((-)^* \nabla_{z\partial_z} s_1, s_2)_{\cF} 
    + ((-)^* s_1, \nabla_{z\partial_z} s_2)_{\cF}
  \end{gather*} 
  Combining the $\cM$-direction with the $z$-direction, 
  we can view $\nabla$ as a map:
  \[
  \nabla \colon \cF \to 
  (\pi^*\Omega_{\cM}^1 \oplus \cO_{\cM\times \C} z^{-1}dz) 
  \otimes_{\cO_{\cM\times \C}} \cF(\cM\times \{0\}) 
  \]
\end{enumerate}
\end{definition} 

\begin{remark} 
These notions are due to Hertling~\cite{Hertling:ttstar}. 
TEP here stands for Twister, Extension and Pairing. 
Definitions similar to the one above 
were given in~\cite[Definition~2.6]{CIT:wall-crossings},~\cite[Definition~2.1]{Iritani:ttstar}. 
\end{remark} 

\begin{example} 
\label{ex:AmodelTP}
  An important class of examples of TEP structure is
  provided by the quantum cohomology of a projective algebraic variety
  $X$.  If the genus-zero Gromov--Witten potential $F^0_X$ converges
  in the sense of \S\ref{sec:convergence} then, as discussed there, we
  can specialize Novikov variables, setting $Q_1 = \cdots = Q_r = 1$,
  and regard $F^0_X$ as an analytic function on an open subset
  \eqref{eq:LRLnbhd} of $H_X\otimes \C$.  Denote this open set by
  $\cM_{\rm A}$. Then the Dubrovin connection (see
  \S\ref{subsec:Dubrovin_conn}) for the quantum cohomology of $X$
  defines a TEP structure over the analytic space $\cM =
  \cM_{\rm A}$, which we call the \emph{A-model TEP structure} for
  $X$, by setting:
\begin{equation}
    \label{eq:AmodelTP}
    \begin{minipage}{0.92\linewidth}
      \begin{itemize} 
      \item $\cF = H_X \otimes_\Q \cO_{\cM\times \C}$;
      \item $\nabla = d - \frac{1}{z} \sum_{i=0}^N (\phi_i *) dt^i
        +(\frac{1}{z^2} (E *) + \frac{1}{z} \mu) dz $;  
      \item $(\alpha(-z), \beta(z))_{\cF} = \int_X \alpha(-z) \cup
        \beta(z)$;
      \end{itemize}
    \end{minipage}
\end{equation}
where $E$ is the Euler vector field \eqref{eq:Eulerfield} and 
$\mu$ is the grading operator \eqref{eq:gradingoperator}. 
In the case where the genus-zero Gromov--Witten potential is not
known to converge, the same procedure defines a TEP structure over the formal 
neighbourhood of the origin in $H_X\otimes \Lambda$, 
viewed as a formal scheme over $\Lambda$. 
\end{example}

\subsection{cTP and cTEP Structure} 
\label{subsec:cTP} 
A cTP structure is a certain coherent sheaf with extra structure
over $\cM\times \hA^1$, where $\hA^1 = \Spf \C[\![z]\!]$ 
denotes the formal neighbourhood of zero in $\C$. 
A sheaf of modules over $\cM \times \hA^1$ 
is the same thing as a sheaf of $\cO_{\cM}[\![z]\!]$-modules. 
Let $(-)\colon \cM \times \hA^1 \to \cM \times \hA^1$ 
be the map sending $(t, z)$ to $(t,-z)$ as before. 
For an $\cO_{\cM}[\![z]\!]$-module $\sfF$, 
the structure of an $\cO_\cM[\![z]\!]$-module 
on the pull-back $(-)^* \sfF$ is defined by 
$f(z) (-)^* \alpha = (-)^* f(-z) \alpha$ for 
$f(z) \in \cO_{\cM}[\![z]\!]$ and $\alpha \in \sfF$. 
\begin{definition}
\label{def:cTP} 
\ 
\begin{enumerate}
\item A \emph{cTP structure} $(\sfF, \nabla, (\cdot,\cdot)_{\sfF})$ 
  with base $\cM$ consists of 
  a locally free $\cO_\cM[\![z]\!]$-module $\sfF$ of rank $N+1$, and
  a flat connection $\nabla$ with pole along $z=0$ 
  \[ 
  \nabla \colon \sfF \to 
  \Omega^1_{\cM} \otimes_{\cO_{\cM}} z^{-1} \sfF 
  \]  
  so that for $f\in \cO_\cM[\![z]\!]$, $s\in \sfF$, 
  and tangent vector fields $v_1,v_2\in \Theta_\cM$:
  \begin{align*}
    \nabla (f s) = df \otimes s + f \nabla s && 
    [\nabla_{v_1}, \nabla_{v_2}] = \nabla_{[v_1,v_2]} 
  \end{align*}
  together with a pairing 
  \[
  (\cdot,\cdot)_{\sfF} \colon 
  (-)^* \sfF \otimes_{\cO_{\cM}[\![z]\!]} \sfF \to \cO_{\cM}[\![z]\!] 
  \]
  which satisfies
  \begin{align*} 
    \begin{split} 
      ((-)^*s_1,s_2)_{\sfF} 
      & = (-)^* ((-)^* s_2, s_1)_{\sfF}   \\ 
      d ((-)^* s_1, s_2)_{\sfF}  
      & = ((-)^* \nabla s_1, s_2)_{\sfF} 
      + ((-)^*s_1,\nabla s_2)_{\sfF}  
    \end{split} 
  \end{align*}
  for $s_1,s_2 \in \sfF$. 
  The pairing $(\cdot,\cdot)_{\sfF}$ is assumed to be 
  non-degenerate in the sense that the induced 
  pairing on $\sfF_0 := \sfF/z\sfF$ 
  \[
  (\cdot,\cdot)_{\sfF_0} \colon \sfF_0\otimes_{\cO_{\cM}} 
  \sfF_0 \to \cO_{\cM} 
  \]
  is non-degenerate. 
  We regard $z^{-1} \sfF$ as a subsheaf of $\sfF[z^{-1}] : = 
  \sfF \otimes_{\cO_{\cM}[\![z]\!]} \cO_{\cM}(\!(z)\!)$. 

\item A \emph{cTEP structure} is a cTP structure 
  such that the connection $\nabla$ is extended in 
  the $z$-direction with a pole of order 2 along $z=0$. 
  More precisely, it is a cTP structure 
  $(\sfF,\nabla,(\cdot,\cdot)_{\sfF})$ 
  equipped with an $\cO_{\cM}$-module map  
  \[
  \nabla_{z\partial_z} \colon 
  \sfF \to z^{-1} \sfF
  \]
  such that for $f\in \cO_{\cM}[\![z]\!]$, 
  $s$,~$s_1$,~$s_2 \in \sfF$, and $v\in \Theta_\cM$:
  \begin{gather*} 
    \nabla_{z\partial_z} (fs) = z(\partial_z f) s + 
    f \nabla_{z\partial_z} s \qquad 
    [\nabla_v, \nabla_{z\partial_z} ] = 0 \\   
    z\partial_z ((-)^* s_1, s_2)_{\sfF} 
    = ((-)^* \nabla_{z\partial_z} s_1, s_2)_{\sfF} 
    + ((-)^* s_1, \nabla_{z\partial_z} s_2)_{\sfF}
  \end{gather*} 
  Combining the $\cM$-direction with the $z$-direction, 
  we can view $\nabla$ as a map:
  \[
  \nabla \colon \sfF \to 
  (\Omega_{\cM}^1 \oplus \cO_{\cM} z^{-1}dz) 
  \otimes_{\cO_{\cM}} z^{-1} \sfF
  \]
\end{enumerate}
\end{definition} 

\begin{remark} 
\label{rem:completion-of-TP} 
A TP structure (respectively a TEP structure) in Definition~\ref{def:TP} 
gives rise to a cTP structure (respectively a cTEP structure)   
by restriction to the formal neighbourhood of $z=0$ 
in $\cM\times \C$. 
In particular, the A-model TEP structure in 
Example~\ref{ex:AmodelTP} defines 
the \emph{A-model cTEP structure} 
over the formal neighbourhood of $z=0$. 
On the other hand, we do not know if every cTP  structure
admits an extension to $\cM \times \C$.
The first letter ``c" of cTP stands for ``complete".
\end{remark}

\begin{remark} 
In the remainder of this section we work with cTP structures, without extending the connection to the $z$-direction. 
Consequently the framework that we construct applies 
to cases, such as \emph{equivariant} quantum cohomology, where the flat connection cannot 
be extended to the $z$-direction.
An extension to the $z$-direction will play an important role 
when we construct a \emph{semisimple opposite module} 
in \S\ref{subsec:semisimpleopposite}, and in certain
 \emph{Virasoro symmetries} of the Fock space.  
\end{remark} 

\subsection{Total Space of a cTP Structure} 
\label{subsec:totalspace} 
We begin by studying the geometry of the total space 
of a cTP structure. The total space of a 
cTP structure is an algebraic analogue of Givental cone $\cL$ discussed in~\S\ref{subsec:Lagrangian_TP}. 

Let $(\sfF,\nabla, (\cdot,\cdot)_\sfF)$ be a cTP structure. 
Set $\sfF[z^{-1}] := \sfF\otimes_{\cO_{\cM}[\![z]\!]} 
\cO_{\cM}(\!(z)\!)$. 
This is a locally free $\cO_{\cM}(\!(z)\!)$-module.  
The pairing $(\cdot,\cdot)_{\sfF}$ induces a 
symplectic pairing 
\[
\Omega \colon \sfF[z^{-1}] \otimes_{\cO_{\cM}} 
\sfF[z^{-1}] \to \cO_\cM
\]
defined by:
\begin{equation}
\label{eq:symplecticform} 
\Omega(s_1, s_2) =  
\Res_{z=0} ((-)^*s_1,s_2)_{\sfF} \,dz
\end{equation} 
Note that this is anti-symmetric: $\Omega(s_1,s_2) = 
- \Omega(s_2,s_1)$. 
We define the dual modules 
$(z^n \sfF)^\vee$, $n\in \Z$, and $\sfF[z^{-1}]^\vee$ as:
\begin{align}
\label{eq:dualmodules} 
\begin{split}  
(z^n \sfF)^\vee &:= 
\injlim_{l} \sHom_{\cO_{\cM}}(z^n \sfF/z^l \sfF,\cO_{\cM})
\\
\sfF[z^{-1}]^\vee &:= \projlim_n 
\injlim_l \sHom_{\cO_{\cM}}(z^{-n}\sfF/z^l \sfF,\cO_{\cM})
\end{split}  
\end{align}  
There is a sequence of natural projections:
\[
\sfF[z^{-1}]^\vee \surj \cdots \surj (z^{-2} \sfF)^\vee 
\surj (z^{-1} \sfF)^\vee \surj \sfF^\vee \surj (z\sfF)^\vee 
\surj \cdots
\] 
The dual $(z^n\sfF)^\vee$ has the structure of an 
$\cO_{\cM}[\![z]\!]$-module such that the action of $z$ is nilpotent. 
It is locally isomorphic to 
$(\cO_{\cM}(\!(z)\!)/\cO_{\cM}[\![z]\!])^{\oplus (N+1)}$ 
as an $\cO_\cM[\![z]\!]$-module, where $N+1$ is the rank of $\sfF$. 
Also $\sfF[z^{-1}]^\vee$ is a locally free $\cO_{\cM}(\!(z)\!)$-module. 
The dual flat connection $\nabla^\vee$ is defined 
by:
\begin{align}
  \label{eq:dual_flat_connection}
  \nabla^\vee \colon 
  (z^{-1}\sfF)^\vee \to \sfF^\vee\otimes_{\cO_{\cM}}\Omega^1_{\cM} &&
  \pair{\nabla^\vee \varphi}{s} := d\pair{\varphi}{s} - 
  \pair{\varphi}{\nabla s}
\end{align}
The symplectic pairing gives an isomorphism 
\begin{align*}
  \sfF[z^{-1}] \cong \sfF[z^{-1}]^\vee &&
  s \mapsto \iota_s \Omega = \Omega(s, \cdot)
\end{align*}
which in turn induces the dual symplectic 
pairing $\Omega^\vee$ on $\sfF[z^{-1}]^\vee$: 
\begin{equation}
\label{eq:dualsymp} 
\Omega^\vee \colon \sfF[z^{-1}]^\vee \otimes_{\cO_{\cM}} 
\sfF[z^{-1}]^\vee \to \cO_{\cM} 
\end{equation}

\begin{definition} 
\label{def:totalspace}
The \emph{total space} $\LL$ 
of a cTP structure $(\sfF,\nabla, (\cdot,\cdot)_{\sfF})$ 
is the total space of 
the infinite-dimensional vector bundle 
associated to $z\sfF$. 
As a set, $\LL$ consists of all pairs 
$(t,\bx)$ such that $t\in \cM$ and $\bx\in z \sfF_t$. 
Let $\pr \colon \LL \to \cM$ denote the natural projection. 
We endow $\LL$ with the structure of a ringed space 
so that we can regard it as a ``fiberwise algebraic variety''
over $\cM$. 
For a connected open set $U\subset \cM$ such that 
$\sfF|_U$ is a free $\cO_U[\![z]\!]$-module, 
the ring of regular functions on $\pr^{-1}(U)$ 
is defined to be the polynomial ring over $\cO(U)$:  
\begin{equation}
\label{eq:strsheaf_prinvU}
\bcO(\pr^{-1}(U)) := \Sym_{\cO(U)}^\bullet 
\Gamma(U,(z\sfF)^\vee)
\end{equation} 
A basis of open sets of $\LL$ 
is given by the complements in $\pr^{-1}(U)$ of the zero-loci  
of regular functions in $\bcO(\pr^{-1}(U))$ 
for all such open sets $U\subset \cM$. 
For a general open set $V \subset \LL$ in this topology, 
$\bcO(V)$ is the ring\footnote
{When $V = \pr^{-1}(U)$ for a connected open set $U$ 
such that $\sfF|_U$ can be trivialized, 
one can check that $\bcO(V)$ coincides with 
the original definition \eqref{eq:strsheaf_prinvU}.  
More generally, for the complement $D(h)$ of the zero-locus 
of $h\in \bcO(\pr^{-1}(U))$, $\bcO(D(h))$ is the 
localization of the polynomial ring $\bcO(\pr^{-1}(U))$ by $h$: 
\[
\bcO(D(h)) = \bcO(\pr^{-1}(U))_h
\]
\begin{proof} 
Each element $r\in \bcO(D(h))$ can be 
locally written as $r=f/g$ for some $f,g \in \bcO(\pr^{-1}(U'))$ 
with $U'\subset U$. Then by the standard argument 
using Hilbert's Nullstellensatz, we can see that 
for each $t\in U'$, there exists $m\in \N$ such that 
$h^m (f/g)$ restricted to the fiber  
$z \sfF_t$ is a polynomial on $z \sfF_t$ 
($m$ here can depend on $t$). 
On the other hand, it is clear that 
we can take $m$ to be $\deg(g)$. 
Then $r h^{\deg(g)}$ is a polynomial 
in fiber variables with coefficients in holomorphic functions 
on the base $U$.
\end{proof} 
} of $\C$-valued functions 
which can be written locally as quotients $f/g$ 
of some polynomials $f, g \in \bcO(\pr^{-1}(U))$. 

Let $U\subset \cM$ be a connected 
open set such that $\sfF|_U$ is a free $\cO_U[\![z]\!]$-module. 
Then by \eqref{eq:strsheaf_prinvU}, 
$\bcO(\pr^{-1}(U))$ is graded by the degree of polynomials: 
\[
\bcO^n(\pr^{-1}(U)) = \Sym^n_{\cO(U)} \Gamma(U, (z\sfF)^\vee)
\]
The $\cO_{\cM}$-module $(z \sfF)^\vee$ 
has the increasing filtration 
$(z\sfF)_l^\vee = \sHom_{\cO_{\cM}}(z\sfF/z^{l+2}\sfF,\cO_{\cM})$.  
This induces the exhaustive increasing filtration on 
$\bcO(\pr^{-1}(U))$: 
\begin{align*}
  & \bcO_l(\pr^{-1}(U)) =
  \cO(U) + 
  \sum_{n=1}^\infty \sum_{\substack{l_1,\dots,l_n \ge 0 \\ 
      l_1 + \cdots + l_n \le l}} 
  \Gamma\big(U, (z \sfF)_{l_1}^\vee (z\sfF)_{l_2}^\vee 
  \cdots (z\sfF)^\vee_{l_n}\big)
  && l\ge 0  \\
  & \bcO_{-1}(\pr^{-1}(U)) := \{0\}
\end{align*}
such that 
\[
\{0\} \subset 
\bcO_0(\pr^{-1}(U)) \subset 
\bcO_1(\pr^{-1}(U)) \subset 
\bcO_2(\pr^{-1}(U)) \subset \cdots
\] 
and $\bcO_{l_1}(\pr^{-1}(U)) \bcO_{l_2}(\pr^{-1}(U))
\subset \bcO_{l_1+l_2}(\pr^{-1}(U))$. 
\end{definition} 

Let $U\subset \cM$ be a connected open set  
such that $\sfF|_U$ is a free $\cO_U[\![z]\!]$-module. 
Take a trivialization 
$\sfF|_U \cong \C^{N+1} \otimes \cO_U[\![z]\!]$.  
This induces a trivialization 
$\sfF[z^{-1}]|_U \cong \C^{N+1} \otimes 
\cO_U(\!(z)\!)$ and defines a dual frame 
$x_n^i \in \sfF[z^{-1}]^\vee$, 
$n\in \Z$, $0\le i\le N$, by: 
\begin{align} 
  \label{eq:frame_x}
  x_n^i \colon \sfF[z^{-1}]\Bigr|_U \cong 
  \C^{N+1} \otimes \cO_U(\!(z)\!) 
  & \longrightarrow \cO_U &&
  \sum_{m\in \Z} \sum_{j=0}^N a_m^j e_j z^m 
  \longmapsto a_n^i
\end{align}  
Here $e_i$, $0\le i\le N$, denotes the standard basis 
of $\C^{N+1}$. 
By restricting $x_n^i$ to $z \sfF$, we obtain
fiber co-ordinates $x_n^i$, $n\ge 1$, $0\le i\le N$, on $\LL|_U$. 
Assume that $\dim \cM = N+1 = \rank \sfF$ 
and let $t^0,\dots, t^N$ be a local co-ordinate system on $U$. 
We call $\{t^i, x_n^i : 0\le i\le N, \, n\ge 1\}$ 
an \emph{algebraic local co-ordinate system} on $\LL$. 
This corresponds to the algebraic co-ordinate system on 
the Lagrangian submanifold $\cL$ discussed on page \pageref{desc:flat-vs-alg}. 
In \S\ref{subsec:flatstronL} below, 
we will introduce a \emph{flat co-ordinate system} 
on the formal neighbourhood (or an $L^2$ or nuclear neighbourhood) of the fiber $\LL_t =\pr^{-1}(t)$.  
We have:
\[
\bcO(\pr^{-1}(U)) = 
\cO(U)\left[x_n^i : 0\le i\le N, n\ge 1\right]
\]
The grading is given by the degree as polynomials in the variables $x_n^i$. 
The filtration is given by:
\begin{equation} 
\label{eq:filtration} 
\bcO_l(\pr^{-1}(U)) = \left \{\sum_{n=0}^\infty  
\sum_{\substack{l_1,\dots,l_n\ge 0 \\ 
l_1 + \cdots + l_n \le l}} 
\sum_{i_1,\dots,i_n \ge 0} 
f^{l_1,\dots,l_n}_{i_1,\dots,i_n}(t) 
x_{l_1+1}^{i_1}\cdots x_{l_n+1}^{i_n} : 
f^{l_1,\dots,l_n}_{i_1,\dots,i_n}(t)
\in \cO(U)\right \}
\end{equation} 
Under this trivialization, 
we present the flat connection $\nabla$ as 
\begin{equation} 
\label{eq:presentation_nabla} 
\nabla s = d s - \frac{1}{z} \cC(t,z) s 
\end{equation} 
where $\cC(t,z) = \sum_{i=0}^N \cC_i(t,z) dt^i 
\in \End(\C^{N+1})\otimes \Omega_U^1[\![z]\!]$ and 
$s \in \C^{N+1} \otimes \cO_U[\![z]\!] 
\cong \sfF|_U$. 
The residual part $\cC(t,0) = (-z\nabla)|_{z=0}$ 
defines a section of $\End(\sfF_0) \otimes \Omega_U^1$, which is independent of the choice of trivialization. 
In the case of the A-model TEP structure in Example~\ref{ex:AmodelTP},  
we have $\cC(t,z) = \cC(t,0) = \sum_{i=0}^N (\phi_i*) dt^i$ 
with respect to the standard trivialization. 

\begin{definition}[The open subset $\LLo\subset \LL$] 
\label{def:possibilityofdilaton} 
Define the following open subsets: 
\begin{align*}
\sfF_{0,t}^\circ & := 
\{ x_1 \in \sfF_{0,t} :
\text{$T_t \cM \to \sfF_{0,t}$, $\ v\mapsto \cC_v(t,0) x_1$ is an isomorphism}\}  \\  
\LLo & := \{ (t,\bx) \in \LL : \text{$t\in \cM$, $\bx \in z \sfF_t$, $(\bx/z)|_{z=0} \in \sfF_{0,t}^\circ$} \}
\end{align*}
We set $\sfF_0^\circ = \bigcup_{t\in \cM} \sfF_{0,t}^\circ$. 
This is an open subset of the total space of $\sfF_0$. 
\end{definition} 

Henceforth we assume that our 
cTP structure $(\sfF,\nabla, (\cdot,\cdot)_\sfF)$ 
is \emph{miniversal}, which means: 
\begin{assumption} 
\label{assump:miniversal} 
At every point $t\in \cM$, 
$\sfF_{0,t}^\circ$ is a non-empty Zariski 
open subset of $\sfF_{0,t}$. 
\end{assumption} 

This assumption implies in particular that $\dim \cM = \rank \sfF$. 
Miniversality holds for the cTP structure defined
by quantum cohomology because $\phi_0= \unit$ 
is a section of $\sfF_0^\circ$. 
Using an algebraic local co-ordinate system $\{t^i,x_n^i\}$ 
on $\LL$, we can write $\LLo$ as 
the complement of the zero-locus 
of the degree $N+1$ polynomial $P(t,x_1)$ defined by
\begin{equation} 
\label{eq:discriminant}
P(t,x_1) := (-1)^{N+1}\det\big(\cC_0(t,0) x_1, \cC_1(t,0) x_1, \dots, \cC_N(t,0) x_1\big) 
\in \cO(U)[x_1^0,\dots,x_1^N] 
\end{equation} 
where $\cC_i(t,z) = \cC_{\partial/\partial t^i}(t,z)$. 
We call $P$ the \emph{discriminant}. 
More invariantly, we can think of 
$P(t,x_1) dt^0\wedge \dots \wedge dt^N$ 
as a section of the line bundle 
$\pr^*(\det(\sfF_0) \otimes K_\cM)$ over $\LL$. 
In the case of the A-model TEP structure in Example~\ref{ex:AmodelTP}, 
we have $P(t,x_1) = \det(-x_1*_t)$ under the standard 
trivialization. 
The ring of regular functions over 
$\pr^{-1}(U)^\circ := \pr^{-1}(U) \cap \LLo$ 
is 
\[
\bcO(\pr^{-1}(U)^\circ)  
= \cO(U)[\{x_n^i\}_{n\ge 1,0\le i\le N}, 
P(t,x_1)^{-1}].  
\]
Since $P(t,x_1)$ is homogeneous in $x_1$ and 
lies in the zeroth filter, $\bcO(\pr^{-1}(U)^\circ)$ 
inherits the grading and the filtration. 
Since we will almost always deal with open sets of the form 
$\pr^{-1}(U)$ or $\pr^{-1}(U)^\circ$,  
we will omit the domain $\pr^{-1}(U)$ or 
$\pr^{-1}(U)^\circ$ from the notation, writing 
e.g.~$\bcO^n$, $\bcO_l$ 
for $\bcO^n(\pr^{-1}(U))$, $\bcO_l(\pr^{-1}(U))$ 
(or for $\bcO^n(\pr^{-1}(U)^\circ)$, 
$\bcO_l(\pr^{-1}(U)^\circ)$). 
We also write $\bcO^n_l := \bcO^n \cap \bcO_l$. 

The sheaf $\bOmega^1$ of one-forms on $\LL$ 
is defined on a local co-ordinate chart as 
\[
\bOmega^1 = \bigoplus_{i=0}^N  
\bcO  d t^i \oplus  
\bigoplus_{n =1}^\infty 
\bigoplus_{i=0}^N  \bcO d x_n^i   
\]
and then glued in the obvious way. 
The grading and the filtration on $\bOmega^1$ 
are determined by:
\begin{align}
  \label{eq:grading_filtration_coord}
  \deg(dt^i) = 0 &&
  \deg(d x_n^i) = 1 &&
  \filt(dt^i) =-1 &&
  \filt(d x_n^i) = n-1  
\end{align}
Here $\filt(y)$ denotes the least number $m$
such that $y$ belongs to the $m$th filter. 
We have:
\begin{equation*} 
(\bOmega^1)^n_l = \bigoplus_{i=0}^N \bcO^n_{l+1} 
d t^i \oplus \bigoplus_{l_1+l_2\le l} \bigoplus_{i=0}^N 
\bcO_{l_1}^{n-1} d x_{l_2+1}^i
\end{equation*} 
More generally, we set:
\[
\big((\bOmega^1)^{\otimes m}\big)_l^n = 
\sum_{l_1+\dots+l_m \le l} \sum_{n_1+\dots+n_m=n} 
(\bOmega^1)_{l_1}^{n_1} \otimes \cdots 
\otimes (\bOmega^1)_{l_m}^{n_m}
\]
The sheaf $\bTheta$ of tangent vector fields 
on $\LL$ 
is the dual of $\bOmega^1$: 
\[
\bTheta := 
\sHom_{\bcO}(\bOmega^1, \bcO) 
\underset{\text{(locally)}}{=} \prod_{i=0}^N 
\bcO  \partial_i \times 
\prod_{n=1}^\infty \prod_{i=0}^N 
\bcO \partial_{n,i}
\]
where $\partial_i := \partial/\partial t^i$,  
$\partial_{n,i} := \partial/\partial x_n^i$. 
Note that $\bOmega^1$ is the direct \emph{sum} 
whereas $\bTheta$ is the direct \emph{product}. 

\subsection{Yukawa Coupling and Kodaira--Spencer Map} 
\label{subsec:Yukawa} 
Recall from \S\ref{subsec:motivation-Focksheaf} that the Yukawa coupling 
is the third derivative of the genus-zero potential. 
In terms of an algebraic local co-ordinate system, 
this has the following simple definition. 
We start by noting that the flatness of $\nabla$ implies 
\begin{align*}
  [\cC_i(t,0), \cC_j(t,0)] = 0 && \text{for all $i$,~$j$.}
\end{align*} 
Also the flatness of the pairing implies:
\[
(\cC_i(t,0) s_1, s_2)_{\sfF_0} = (s_1, \cC_i(t,0) s_2)_{\sfF_0}
\]
Hence the operators $\cC_i(t,0)$ together yield a structure similar to 
a Frobenius algebra. (In order to define a Frobenius algebra 
structure on $\sfF_0$, one needs to choose an identity element from 
$\sfF_0^\circ$.) 
\begin{definition}
\label{def:Yukawa} 
The \emph{Yukawa coupling} is 
a symmetric cubic tensor 
\[
\bY = \sum_{i,j,h} C^{(0)}_{ijh} d t^i\otimes d t^j \otimes d t^h 
\in ((\bOmega^1)^{\otimes 3})^2_{-3}
\]
on $\LL$ 
defined in local co-ordinates by:
\begin{align*}
C^{(0)}_{ijh}(t,\bx)   
= \big(\cC_i(t,0) x_1, \cC_j(t,0) \cC_h(t,0) x_1\big)_{\sfF_0} &&
x_1 = (\bx/z)|_{z=0}
\end{align*}
The tensor $\bY$ is the pull-back of a cubic tensor on 
$\sfF_0$. 
\end{definition} 
Let $\pr\colon \LL\to \cM$ denote the natural projection. 
We define the pull backs of the sheaves $z^n\sfF$, $\sfF[z^{-1}]$, 
$(z^n\sfF)^\vee$, $\sfF[z^{-1}]^\vee$ 
to $\LL$ as follows\footnote
{Note that the standard pull-back 
$\pr^{-1}(z^n \sfF) \otimes_{\pr^{-1}\cO_{\cM}} \bcO$ 
of $z^n\sfF$ is different from the above definition 
of $\pr^*(z^n\sfF)$. 
We take the completion with respect to the $z$-adic topology.}. 
\begin{align}
\label{eq:pullbacks}
\begin{split}  
\pr^* (z^n\sfF) & := 
\projlim_{l} \pr^* (z^n\sfF/z^l\sfF) \cong 
(\pr^{-1} z^n\sfF)  
\otimes_{\pr^{-1}\cO_{\cM}[\![z]\!]}\bcO[\![z]\!] \\
\pr^* \sfF[z^{-1}] 
& := \projlim_{l} \pr^*(\sfF[z^{-1}]/z^l \sfF) \cong 
(\pr^{-1} \sfF[z^{-1}])\otimes_{\pr^{-1}\cO_\cM(\!(z)\!)} \bcO(\!(z)\!) \\ 
\pr^* (z^n \sfF)^\vee & := (\pr^{-1} (z^n\sfF)^\vee)  
\otimes_{\pr^{-1} \cO_{\cM}} \bcO \qquad \text{(the standard definition)} \\ 
\pr^* \sfF[z^{-1}]^\vee 
& := \projlim_l \pr^* (z^{-l}\sfF)^\vee 
\cong (\pr^{-1} \sfF[z^{-1}]^\vee)
\otimes_{\pr^{-1}\cO_{\cM}(\!(z)\!)} \bcO(\!(z)\!)  
\end{split}
\end{align} 
These are locally free modules over, respectively, $\bcO[\![z]\!]$, 
$\bcO(\!(z)\!)$, $\bcO(\!(z)\!)/\bcO[\![z]\!]$, and
$\bcO(\!(z)\!)$. 
The pull-back $\pr^*\sfF$ admits a flat connection $\tnabla:=\pr^*\nabla$:   
\begin{equation} 
\label{eq:tnabla}
\tnabla \colon \pr^*\sfF \to \bOmega^1
\hotimes \pr^*(z^{-1}\sfF)
\end{equation} 
where $\hotimes$ is the completed tensor product 
$\bOmega^1 \hotimes \pr^*(z^{-1}\sfF) 
=\projlim_n (\bOmega^1 \otimes 
\pr^*(z^{-1}\sfF/z^n \sfF))$. 
A local trivialization $\sfF|_U \cong 
\C^{N+1} \otimes \cO_U[\![z]\!]$ induces 
a trivialization 
$\pr^*\sfF|_{\pr^{-1}(U)} 
\cong \C^{N+1}\otimes \bcO[\![z]\!]$. 
Under this trivialization, we can write, using notation as in \eqref{eq:presentation_nabla}: 
\begin{align*}
\tnabla_i =\tnabla_{\partial/\partial t^i} 
= \partial_i - \frac{1}{z} \cC_i(t,z) &&
\text{$\tnabla_{n,i} = \tnabla_{\partial/\partial x_n^i} 
= \partial_{n,i}$,  $n\ge 1$} 
\end{align*} 
The trivialization also 
induces a trivialization 
$\pr^*\sfF[z^{-1}]|_{\pr^{-1}(U)} 
\cong \C^{N+1} \otimes \bcO(\!(z)\!)$.  
We denote by\footnote
{We denote the frame of $\sfF[z^{-1}]^\vee$ by 
$\{x_n^i\}$ 
and the frame of $\pr^*\sfF[z^{-1}]^\vee$ by 
$\{\varphi_n^i\}$ so that the co-ordinates on $\LL$ 
and the frame of $\pr^*\sfF[z^{-1}]^\vee$ are 
not confused.} 
$\{\varphi_n^i : \text{$n\in \Z$, $0\le i\le N$}\}$ the local frame of 
$\pr^* \sfF[z^{-1}]^\vee$ defined by
\begin{align} 
\label{eq:frame_varphi} 
\varphi_n^i \colon \pr^*\sfF[z^{-1}]\Bigr|_{\pr^{-1}(U)}  
\cong \C^{N+1} \otimes \bcO(\!(z)\!) 
\to \bcO &&
\sum_{m\in \Z} \sum_{j=0}^N 
a_m^j e_j z^m \mapsto a_n^i  
\end{align} 
(cf.~equation~\ref{eq:frame_x}).
The \emph{tautological section} 
$\bx$ of $\pr^* (z\sfF)$ is defined by 
\[
\bx(t,\bx) = \bx   
\] 
where $(t,\bx)$ denotes the point $\bx \in z \sfF_t$ 
on $\LL$. 
\begin{definition} 
\label{def:KS} 
The \emph{Kodaira--Spencer map} 
$\KS \colon \bTheta \to \pr^*\sfF$ 
is defined by:
\begin{align*} 
\KS(v) & = \tnabla_v \bx && v\in \bTheta
\intertext{The \emph{dual Kodaira--Spencer map} 
  $\KS^* \colon \pr^*\sfF^\vee 
  \to \bOmega^1$ is defined by:}
  \KS^*(\varphi) & = \varphi(\tnabla \bx) 
  && \varphi \in \pr^*\sfF^\vee
\end{align*}
The maps 
$\KS$ and $\KS^*$ are 
isomorphisms over $\LLo\subset \LL$. 
\end{definition} 
In terms of the Lagrangian submanifold $\cL$ in 
\S\ref{subsec:Lagrangian_TP},  
the Kodaira--Spencer map corresponds to 
the differential $d\iota $ of the embedding 
$\iota \colon \cL \hookrightarrow \cH$;
see also \S \ref{subsec:global_L2}, \S\ref{subsec:Kaehler}. 

\begin{notation}
\label{nota:[]_^}
For $\C^{N+1}$-valued power series 
$f = \sum_{n\in \Z} \sum_{i=0}^N a_n^i e_i z^n$ 
in $z$, we write $[f]_n^i = a_n^i$. 
Here $e_0,\dots, e_N$ are the standard 
basis of $\C^{N+1}$. 
\end{notation}

\begin{remark} 
In algebraic local co-ordinates $\{t^i,x_n^i\}$ on $\LL$, we have:
\begin{equation}
\label{eq:KS_coord}
\begin{aligned} 
  \KS(\partial_i) &= -\cC_i(t,z) (\bx/z) \\
  \KS(\partial_{n,i}) &= e_i z^n && n\ge 1\\  
  \KS^*(\varphi_n^i) & =[d \bx - z^{-1}\cC(t,z) \bx]_{n}^i && n\ge 1\\  
\end{aligned} 
\end{equation} 
Here $\bx= \sum_{n=1}^\infty x_n z^n$ and 
$x_n = \sum_{i=0}^N x_n^i e_i$. (Note that 
$\KS^*(\varphi_0^i) = -[\cC(t,0)x_1]^i$.) 
These formulae make clear that 
$\KS$ and $\KS^*$ are isomorphisms over 
$\LLo$. 
\end{remark} 

\begin{lemma} 
\label{lem:Yukawa-KS} 
The Yukawa coupling $\bY$ can be written as 
$ (\id \otimes (\KS^*)^{\otimes 2}) \pr^*\Upsilon$  
for the following section $\Upsilon \in \Omega^1_\cM \otimes \sfF^\vee 
\otimes \sfF^\vee$: 
\begin{align*}
  \Upsilon(X, v,w) = \left([v], \cC_X(t,0)[w]\right)_{\sfF_0} &&
  \text{$X\in \Theta_\cM$, $v, w \in \sfF$}
\end{align*}
\end{lemma} 
\begin{proof} 
Note that $\KS^*(\varphi_0^i) = - \sum_j [\cC_j(t,0) x_1]^i dt^j$ by
\eqref{eq:KS_coord}. Therefore:
\begin{align} 
\label{eq:Yukawa-KS}
\begin{split} 
\bY & = \sum_{i=0}^N \sum_{j=0}^N \sum_{h=0}^N
\big(\cC_j(t,0) x_1, \cC_i(t,0) \cC_h(t,0)x_1\big)_{\sfF_0} 
dt^i \otimes dt^j \otimes dt^h \\ 
& = \sum_{i=0}^N \sum_{j=0}^N \sum_{h=0}^N
\sum_{f=0}^N \sum_{g=0}^N 
\big(e_f, \cC_i(t,0) e_g\big)_{\sfF_0} 
[\cC_j(t,0) x_1]^f [\cC_h(t,0) x_1]^g  dt^i \otimes dt^j \otimes dt^h \\ 
& = \sum_{i=0}^N \sum_{f=0}^N \sum_{g=0}^N
\big(e_f, \cC_i(t,0) e_g\big)_{\sfF_0} 
dt^i \otimes \KS^*(\varphi_0^f) \otimes \KS^*(\varphi_0^g)
\end{split}
\end{align}  
The conclusion follows. 
\end{proof}

\subsection{Opposite Modules and Frobenius Manifolds} 
\label{subsec:opp-Frob} 
We now introduce the notion of opposite module.  
In the construction of the Fock space, an opposite module 
plays the role of a \emph{polarization}: see \S\ref{subsec:geometricquantization}.  
The content in this section is an adaptation of~\cite[\S2.2.2]{CIT:wall-crossings} 
to the setting of miniversal cTP structures $(\sfF, \nabla, (\cdot,\cdot)_{\sfF})$.  
Opposite modules were first used in singularity theory by M.~Saito~\cite{SaitoM} 
in order to construct K.~Saito's \emph{flat structure}~\cite{SaitoK}, 
or Dubrovin's \emph{Frobenius manifold structure}~\cite{Dubrovin:2DTFT}, 
on the base space of miniversal unfolding of a singularity.
A closely related discussion can be found in Sabbah 
\cite[VI, \S 2]{Sabbah:isomonodromic} and Hertling \cite[\S 5.2]{Hertling:ttstar}.

\begin{definition} 
\label{def:opposite} 
A \emph{pseudo-opposite module} $\sfP$ for a cTP structure 
$(\sfF, \nabla, (\cdot,\cdot)_\sfF)$ 
is an $\cO_\cM$-submodule $\sfP$ of $\sfF[z^{-1}]$ 
satisfying the following two conditions: 
\renewcommand{\labelenumi}{(\roman{enumi})}
\begin{description} 
\item[(Opp1)] (opposedness) $\sfF[z^{-1}] = \sfF \oplus \sfP$; and
\item[(Opp2)] (isotropy) $\Omega(\sfP,\sfP) = 0$.  
\end{description}
A pseudo-opposite module $\sfP$ is said to be 
\emph{parallel} if it satisfies
\begin{description} 
\item[(Opp3)] $\nabla$ preserves $\sfP$, i.e.~$\nabla \sfP \subset \Omega_{\cM}^1\otimes \sfP$.
\end{description}  
If $\sfP$ satisfies (Opp1--Opp3) and
\begin{description} 
\item[(Opp4)] ($z^{-1}$-linearity) $z^{-1} \sfP \subset \sfP$
\end{description} 
then it is called an \emph{opposite module}. 
When a pseudo-opposite module fails to satisfy 
the parallel condition (Opp3), it is said to be \emph{curved}. 

Suppose that $(\sfF,\nabla,(\cdot,\cdot)_\sfF)$ 
is a cTEP structure. 
An opposite module $\sfP$ for the underlying 
cTP structure is said to be \emph{homogeneous} 
if it satisfies 
\begin{description} 
\item[(Opp5)] (homogeneity) $\nabla_{z\partial_z} \sfP \subset \sfP$. 
\end{description} 
The notion of a (pseudo-)opposite module is local.  
For an open set $U\subset \cM$, $\sfP$ is called a 
\emph{(pseudo-)opposite module over} $U$ 
if it is a (pseudo-)opposite module 
of the restriction $(\sfF, \nabla, (\cdot,\cdot)_\sfF)|_U$. 
\end{definition} 

\begin{example} 
\label{ex:Amodel-opposite} 
The A-model cTEP structure 
(see Example~\ref{ex:AmodelTP} and 
Remark~\ref{rem:completion-of-TP}) associated to a smooth 
projective variety $X$ admits 
a standard opposite module $\sfP_{\rm std}$ defined by:
\[
\sfP_{\rm std} = H_X \otimes_\Q 
z^{-1}\cO_{\cM_{\rm A}}[z^{-1}]
\]
Moreover this opposite module is homogeneous. 
See Remark~\ref{rem:opposite_conformalcase} below 
for the relationship between a homogeneous opposite module
and a Frobenius manifold structure. 
\end{example} 

A pseudo-opposite module 
$\sfP$ is necessarily a locally free $\cO_{\cM}$-module 
with a countable basis 
because $\sfP \cong \sfF[z^{-1}]/\sfF$ by opposedness (Opp1). 
We observe that an opposite module exists at least 
in the formal neighbourhood of any point $t$ in $\cM$. 
\begin{lemma}
\label{lem:existence_opposite}
There exists an opposite module $\sfP$ in the formal neighbourhood 
$\hcM$ of every point $t\in \cM$. 
Here an opposite module in the formal 
neighbourhood means an $\cO_{\hcM}$-submodule $\sfP$ of 
$\widehat{\sfF[z^{-1}]}= \projlim_n \sfF[z^{-1}]/\fm_t^n \sfF[z^{-1}]$ 
satisfying the conditions (Opp1)-(Opp4) in Definition~\ref{def:opposite}
with $\sfF[z^{-1}]$ and $\cM$ there replaced by 
$\widehat{\sfF[z^{-1}]}$ and $\hcM$, 
where $\fm_t$ is the maximal ideal of the local ring $\cO_{\cM,t}$. 
\end{lemma} 
\begin{proof} 
The fiber $\sfF_t$ at $t$ is a 
free $\C[\![z]\!]$-module of rank $N+1$. 
We claim that there exists a basis $e_0,\dots,e_N$ 
of $\sfF_t$ over $\C[\![z]\!]$ such that 
$(e_i, e_j)_{\sfF}$ is
independent of $z$. 
Take any basis $e'_0,\dots,e'_N$ of 
$\sfF_t$. By transforming the basis by an 
element in $GL(N+1,\C)$, one can assume that 
$(e_i',e_j') = c_i \delta_{ij} + O(z)$ for some 
non-zero element $c_i\in \C$.  
After a further change of basis  
$[e_0',\dots,e_N'] = [e_0,\dots, e_N] A(z)$ 
with $A(z) = I + A_1 z + A_2 z^2 + \cdots$, 
we can assume that $(e_i, e_j)_{\sfF} = c_i \delta_{ij}$. 
Once we have such a basis, we can define 
a free $\C[z^{-1}]$-submodule $\sfP_t$ of 
$\sfF_t[z^{-1}]$ by 
$\sfP_t = \bigoplus_{i=0}^N \C[z^{-1}] z^{-1} e_i$. 
This is opposite to $\sfF_t$, $z^{-1}$-linear and 
isotropic with respect to $\Omega$. 
Next we extend it to the formal neighbourhood of $t$. 
Let $s_0,\dots,s_N\in \fm_t$ be a regular system of 
parameters of the local ring $\cO_{\cM,t}$. 
We extend the basis $e_0,\dots,e_N$ of $\sfF_t$ 
to a frame $\te_0,\dots,\te_N$ of $\sfF$ 
over an open neighbourhood of $t$. 
This trivializes $\sfF$ in the formal neighbourhood:
$\sfF|_{\hcM} 
= \bigoplus_{i=0}^N \C[\![z,s_0,\dots,s_N]\!] \te_i$. 
In this frame, we can solve for a flat section 
$f_i(s) \in \widehat{\sfF[z^{-1}]} = 
\bigoplus_{i=0}^N \C(\!(z^{-1})\!)[\![s_0,\dots,s_N]\!] \te_i$ 
such that $f_i(0) = e_i$. 
Then $\sfP = \bigoplus_{i=0}^N 
\C[z^{-1}][\![s_0,\dots,s_N]\!] z^{-1}f_i$ 
is parallel with respect to $\nabla$ and 
gives an opposite module over the formal neighbourhood $\hcM$.  
\end{proof} 

\begin{proposition}[cf.\  
{\cite[\S2.2.2]{CIT:wall-crossings},~\cite[Lemma~3.8]{Iritani:Ruan}}]
\label{prop:flat_trivialization}
For an open set $U\subset \cM$ and 
an opposite module $\sfP$ over $U$, 
the following hold.
\begin{itemize}
\item[(i)] The natural maps $\sfF_0 = \sfF/z\sfF \leftarrow 
\sfF \cap z \sfP \rightarrow z \sfP/\sfP$ 
are isomorphisms of $\cO_U$-modules. 

\item[(ii)] We have $\sfF = (\sfF \cap z \sfP )\otimes \C[\![z]\!] 
\cong (z\sfP/\sfP) \otimes \C[\![z]\!]$, which we call a \emph{flat trivialization}. 
Note that $z\sfP/\sfP$ is a locally free coherent $\cO_U$-module 
with a flat connection, 
and let $\nabla^0 \colon z\sfP/\sfP 
\to \Omega^1_U \otimes_{\cO_U} (z\sfP/\sfP) $ denote the 
flat connection induced by $\nabla$. 

\item[(iii)] Under the flat trivialization, the connection 
$\nabla$ takes the form 
\[
\nabla = \nabla^0 - \frac{1}{z} \cC(t) 
\]
where $\cC(t)$ is a $z$-independent $\End(z\sfP/\sfP)$-valued 
one-form. 

\item[(iv)] Under the flat trivialization, the pairing 
$(\cdot,\cdot)_{\sfF}$ induces and can be recovered from 
a $z$-independent symmetric pairing 
\[
(\cdot,\cdot)_{z\sfP/\sfP} \colon (z\sfP/\sfP) 
\otimes (z\sfP/\sfP) \to \cO_U 
\]
which is flat with respect to $\nabla^0$. 

\item[(v)] Assume that there exists a section $\zeta$ of 
$\sfF \cap z\sfP$ over $U$ which is flat with respect to $\nabla^0$ 
in the flat trivialization and whose image 
under $\sfF \to \sfF_0 = \sfF/z\sfF$ lies in $\sfF_0^\circ$. 
(This assumption implies the miniversality of  
$(\sfF,\nabla,(\cdot,\cdot)_\sfF)$.)
We call such a section $\zeta$ a \emph{primitive section} 
associated to $\sfP$. 
Then the base $U$ carries the structure of a \emph{Frobenius manifold without Euler vector field}.
It consists of:
\begin{itemize} 
\item A flat symmetric $\cO_U$-bilinear 
metric $g: \Theta_U \otimes_{\cO_U} \Theta_U \to \cO_U$, defined by:
\[
g(v_1,v_2) = 
(z\nabla_{v_1} \zeta, z\nabla_{v_2} \zeta)_{\sfF}
\]

\item A commutative and associative product  
$* \colon \Theta_U \otimes_{\cO_U} \Theta_U
\to \Theta_U$, defined by:
\[
z\nabla_{v_1} z\nabla_{v_2} \zeta = -z\nabla_{v_1*v_2} \zeta
\]

\item A flat identity vector field $e\in \Theta_U$ for the product $*$, 
defined by:
\[
- z \nabla_e \zeta= \zeta
\]
\end{itemize} 
such that the connection $\nabla^\lambda_v = 
\nabla^{\rm LC}_v - \lambda (v *)$ on the 
tangent sheaf $\Theta_U$ is a flat pencil of 
connections with parameter $\lambda$. 
Here $\nabla^{\rm LC}$ denotes the Levi-Civita 
connection for the metric $g$. 
\end{itemize}
The same statements {\rm (i)-(v)} here hold,
with $U$ replaced by $\hcM$,  for $\sfP$ an opposite module over
the formal neighbourhood $\hcM$ of $t\in \cM$ 
(in the sense of Lemma~\ref{lem:existence_opposite}). 
\end{proposition} 
\begin{proof}
The proof is similar to that 
in~\cite[\S2.2.2]{CIT:wall-crossings}.
For (i), the injectivity of the maps 
$\sfF \cap z\sfP \to \sfF/z\sfF$, 
$\sfF\cap z\sfP \to z\sfP/\sfP$ follows 
from opposedness $\sfF[z^{-1}] 
= \sfF \oplus \sfP = z \sfF \oplus z \sfP$. 
For a local section $s\in \sfF/z\sfF$, 
take a local lift $\ts\in \sfF$. 
By opposedness, one can write $\ts = s' + s''$ 
with $s' \in z\sfF$ and $s'' \in z \sfP$. 
Now $s'' = \ts -s' \in \sfF \cap z\sfP$ 
and the image of $s''$ in $\sfF/z\sfF$ equals $s$. 
A similar argument shows the surjectivity of 
$\sfF\cap z\sfP \to z\sfP/\sfP$. 
For (ii), we need to show that any local section 
$s\in \sfF$ has a unique expression 
$s = \sum_{n=0}^\infty s_n z^n$ 
with $s_n \in \sfF \cap z\sfP$. 
The zeroth term $s_0$ is given as the unique lift of 
$[s]\in \sfF/z \sfF$ to $\sfF \cap z\sfP$ 
(which exists by (i)). Then $s-s_0\in z \sfF$. 
The first term $s_1 z $ is given as the unique lift of 
$[s - s_0]\in z\sfF/z^2 \sfF$ to $z \sfF \cap z^2 \sfP$. 
Then $s-s_0 -s_1 z \in z^2 \sfF$.  
Repeating this, we get the desired expression. 
For (iii), take a section $s\in \sfF \cap z\sfP$. 
Then $\nabla s = \Omega^1_U\otimes (z^{-1}\sfF \cap z \sfP)$ 
because $\nabla(\sfF)\subset \Omega^1_U \otimes z^{-1} \sfF$ 
and $\nabla(z\sfP) \subset \Omega^1_U \otimes z \sfP$. 
By opposedness $\sfF[z^{-1}] = 
\sfF \oplus \sfP$,  
we have $z^{-1} \sfF \cap z\sfP = 
(z^{-1} \sfF \cap \sfP) \oplus (\sfF \cap z\sfP)$. 
With respect to this decomposition, we can write 
$\nabla s = z^{-1} \cC(t) s \oplus \nabla^0 s$. 
For (iv), it suffices to show that $(s_1,s_2)_{\sfF}$ 
is independent of $z$ for $s_1,s_2 \in \sfF \cap z\sfP$.  
Because $\sfP$ is isotropic and $z^{-1}$-linear, 
we have $(\sfP,\sfP)_{\sfF} \subset z^{-2}\cO_{\cM}[z^{-1}]$. 
Therefore $(s_1,s_2)_{\sfF} \in  
(z\sfP,z\sfP)_{\sfF} \subset \cO_{\cM}[z^{-1}]$. 
On the other hand $(s_1,s_2)_{\sfF} \in \cO_{\cM}[\![z]\!]$. 
The $\nabla^0$-flatness of $(\cdot,\cdot)_{z\sfP/\sfP}$ 
follows from the $\nabla$-flatness of $(\cdot,\cdot)_{\sfF}$ 
and (iii). For (v), one needs to show that the isomorphism 
$\Theta_U \ni v \mapsto -z \nabla_v \zeta = \cC_v(t) \zeta 
\in \sfF \cap z\sfP$ translates the given structures on $\sfF$ 
into the Frobenius manifold structure. The details here
are left to the reader. 
\end{proof} 

\begin{example} 
The standard trivialization \eqref{eq:AmodelTP}  
of the A-model TEP structure  
is the flat trivialization associated to 
the standard opposite module $\sfP_{\rm std}$ 
in Example~\ref{ex:Amodel-opposite}. 
\end{example} 

\begin{remark} 
\label{rem:product}
The product $*$ in Proposition~\ref{prop:flat_trivialization}(v) 
does not depend on the choice of opposite module 
$\sfP$. In fact, the tangent sheaf $\Theta_\cM$ 
of the base space $\cM$ of a miniversal cTP structure 
carries a natural product $*$ such that the embedding 
\begin{align*}
  \Theta_\cM \to \End_{\cO_\cM}(\sfF_0) && v \mapsto z \nabla_v  
\end{align*}
becomes a homomorphism of $\cO_{\cM}$-algebras.  
The product $*$ endows $\cM$ with 
the structure of an $F$-manifold~\cite{Hertling--Manin:weak},
since it arises from a Frobenius manifold structure 
at least infinitesimally by Lemma~\ref{lem:existence_opposite}
(cf.~\cite[\S2.2]{CIT:wall-crossings},~\cite[\S3.2]{Iritani:Ruan}). 
\end{remark} 

\begin{remark} 
\label{rem:opposite_extension}
Let $\pi \colon \cM\times \Proj^1 \to \cM$ 
denote the projection.  An opposite module $\sfP$ gives rise to 
an extension of $\sfF$ (regarded as a sheaf on 
$\cM \times \hA^1$) to a locally free sheaf 
$\cF^{(\infty)}$ over $\cM \times \Proj^1$ 
such that $\pi_* \cF^{(\infty)} = \sfF \cap z\sfP$.
The sheaf $\cF^{(\infty)}$ gives a free $\cO_{\Proj^1}$-module 
when restricted to each fiber $\{t\} \times \Proj^1$. 
\end{remark} 

\begin{remark} 
\label{rem:opposite_conformalcase} 
Let $(\sfF,\nabla,(\cdot,\cdot)_{\sfF})$ be a 
cTEP structure. 
Under the miniversality assumption (Assumption~\ref{assump:miniversal}), there is an
\emph{Euler vector field} $E$ on the base 
which is uniquely characterized by the condition 
that $\nabla_{z\partial_z} + \nabla_E$ has no 
poles along $z=0$ (cf.~\cite[\S3.2]{Iritani:Ruan}), i.e.~that:
\[
(\nabla_{z\partial_z} + \nabla_E) \sfF \subset \sfF
\]
Assume that we have a homogeneous opposite module 
$\sfP$ over $U$ and also that there exists 
a primitive section $\zeta$ associated to $\sfP$, 
in the sense of Proposition~\ref{prop:flat_trivialization} (v), 
which satisfies 
\[
(\nabla_{z\partial_z} + \nabla_E)\zeta = -\frac{\hc}{2} \zeta
\]
for some $\hc\in \C$. 
Then the structures $(g,*,e)$ 
in Proposition~\ref{prop:flat_trivialization}(v) 
together with the Euler vector field $E$ define
a Frobenius manifold structure~\cite[Definition~1.2]{Dubrovin:2DTFT} on $U$ 
with conformal dimension $\hc$ 
(cf.~\cite[Proposition 2.12]{CIT:wall-crossings}). 
They satisfy the following additional properties:
\begin{align}
\label{eq:Euler_compatibility} 
\begin{split}  
\left(\nabla^{\rm LC}\right)^2 E & = 0 \\ 
E g(v_1,v_2) &= g([E,v_1],v_2) + g(v_1, [E, v_2]) + (2-\hc) g(v_1,v_2)\\ 
[E, v_1 * v_2] & = [E, v_1]* v_2 + v_1 * [E,v_2] + v_1 * v_2
\end{split} 
\end{align}
Conversely, any conformal Frobenius manifold determines 
a TEP structure: see Definition~\ref{def:cTEP-of-Frobeniusmfd}.  
For the A-model cTEP structure, in the convergent case, the standard 
opposite module $\sfP_{\rm std}$ in Example~\ref{ex:Amodel-opposite} 
gives rise to the standard Frobenius manifold structure 
on the set $\cM_{\rm A}$ defined in equation~\ref{eq:LRLnbhd}.
\end{remark} 

\subsection{Connection on the Total Space $\LLo$}
\label{subsec:conn-LL}

Recall from \S\ref{subsec:motivation-Focksheaf} that 
a polarization $P$ which is transversal to $\cL$ 
defines an affine flat structure on $\cL$ via 
the projection $\cL \to \cH/P$ along $P$. 
In a similar manner, we construct a flat structure on $\LL$ 
associated to a parallel pseudo-opposite module $\sfP$. 
The choice of $\sfP$ also defines the 
genus-zero potential in \S\ref{subsec:flatstronL}. 

The connection $\tnabla$ 
on $\pr^*\sfF$  in \eqref{eq:tnabla} extends $z^{-1}$-linearly 
to the flat connection $\tnabla \colon 
\pr^* \sfF[z^{-1}]\to 
\bOmega^1 \hotimes \pr^* \sfF[z^{-1}]$ where 
$\bOmega^1 \hotimes \pr^* \sfF[z^{-1}] 
:= \projlim_{l} (\bOmega^1 \otimes 
\pr^* (\sfF[z^{-1}]/z^l\sfF))$. 
Define the dual flat connection 
$\tnabla^\vee \colon \pr^* \sfF[z^{-1}]^\vee   
\to \bOmega^1 \hotimes \pr^*\sfF[z^{-1}]^\vee$  by:
\begin{align} 
\label{eq:tnablavee-def}
\pair{\tnabla^\vee \varphi}{s} 
:= d\pair{\varphi}{s} - \pair{\varphi}{\tnabla s} &&
\text{$s \in \pr^*\sfF[z^{-1}]$,
$\varphi\in \pr^* \sfF[z^{-1}]^\vee$}
\end{align} 
where $\bOmega^1 \hotimes \pr^*\sfF[z^{-1}]^\vee 
:= \projlim_l (\bOmega^1 \otimes 
\pr^* (z^{-l}\sfF)^\vee)$. 
Under a local trivialization of $\sfF$ and 
the associated frame 
$\{\varphi_n^i :  \text{$n\in \Z$, $0\le i\le N$}\}$ 
of $\pr^*\sfF[z^{-1}]^\vee$, 
we can write (see Notation~\ref{nota:[]_^}):
\begin{equation} 
\label{eq:tnablavee}
\tnabla^\vee \varphi_n^i = \sum_{l\in \Z} \sum_{j=0}^N 
\left[\cC(t,z) e_j z^{l}\right ]_{n+1}^i \varphi_l^j
\end{equation} 
where the summand in the right-hand side vanishes for $l\ge n+2$. 
This induces the flat connection 
$\tnabla^\vee \colon \pr^* (z^n \sfF)^\vee 
\to \bOmega^1 \otimes \pr^* (z^{n+1} \sfF)^\vee$ 
for each $n\in \Z$ 
such that the following diagram commutes: 
\begin{equation} 
  \label{eq:tnablavee-CD}
  \begin{aligned}
    \xymatrix{
      \pr^* \sfF[z^{-1}]^\vee  \ar[r]^-{\tnabla^\vee} \ar[d] & \bOmega^1 \hotimes \pr^*\sfF^\vee[z^{-1}] \ar[d] \\ 
      \pr^* (z^n \sfF)^\vee \ar[r]^-{\tnabla^\vee} & \bOmega^1 \otimes \pr^* (z^{n+1}\sfF)^\vee
    }
  \end{aligned}
\end{equation} 

\begin{definition} 
\label{def:Nabla} 
Let $\sfP$ be a pseudo-opposite module for
a cTP structure $(\sfF,\nabla,(\cdot,\cdot)_\sfF)$. 
Let $\Pi \colon \sfF[z^{-1}] = \sfF \oplus \sfP \to \sfF$
denote the projection along $\sfP$. 
Set $\bOmegao^1 := \bOmega^1|_{\LLo}$ 
and $\bThetao = \bTheta|_{\LLo}$. 
Consider the maps:
\[
\xymatrix{
  \pr^*\sfF \ar[rr]^-{\tnabla} && \bOmega^1 \hotimes \pr^*(z^{-1}\sfF) \ar[rr]^-{\id\otimes \Pi} && \bOmega^1 \hotimes \pr^*\sfF
}
\]
\[
\xymatrix{
  \pr^*\sfF^\vee \ar[rr]^-{\Pi^*} && \pr^* (z^{-1}\sfF)^\vee \ar[rr]^-{\tnabla^\vee} && \bOmega^1 \otimes \pr^*\sfF^\vee
}
\]
Via the (dual) Kodaira--Spencer isomorphisms 
$\KS\colon \bThetao \cong \pr^*\sfF$ 
and $\KS^* \colon \pr^*\sfF^\vee \cong \bOmegao^1$ 
over $\LLo$, 
these maps induce respectively the connections 
\begin{align*} 
\Nabla  \colon \bThetao \longrightarrow \bOmegao^1\hotimes \bThetao &&
\Nabla \colon \bOmegao^1 \longrightarrow \bOmegao^1 \otimes \bOmegao^1 
\end{align*} 
on the tangent and the cotangent sheaves on $\LLo$.  These induced connections are dual to 
each other. 
Here $\bOmegao^1 \hotimes \bThetao 
= \projlim_n (\bOmegao^1 \otimes (\bThetao/{\bThetao}_n))$ 
with ${\bThetao}_n:= \KS^{-1}(\pr^*(z^n \sfF)) 
\subset \bThetao$. 
The connection on $\bOmegao^1$ induces 
a connection on $n$-tensors: 
\begin{align*}
  \Nabla \colon (\bOmegao^1)^{\otimes n} 
  \to \bOmegao^1 \otimes (\bOmegao^1)^{\otimes n} &&
  n\ge 0 
\end{align*}
For $n=0$, this denotes the exterior derivative. 
When we want to emphasize the dependence on the choice of 
$\sfP$, we will write $\Pi_\sfP$, $\Nabla^\sfP$ for $\Pi$, $\Nabla$. 
\end{definition} 

\begin{proposition} 
\label{prop:Nabla-torsionfree}
The connection $\Nabla = \Nabla^\sfP 
\colon \bThetao \to \bOmegao^1 \hotimes \bThetao$ 
associated to a pseudo-opposite module $\sfP$ is torsion-free. 
If $\sfP$ is parallel, then $\Nabla$ is flat. If $\sfP$ is parallel, then the dual connection $\Nabla \colon \bOmegao^1 \to 
\bOmegao^1 \otimes \bOmegao^1$ is also flat. 
\end{proposition} 
\begin{proof} 
For $v_1,v_2 \in \bThetao$ and 
the tautological section $\bx$, we have 
\[
\Nabla_{v_1} v_2 - \Nabla_{v_2}v_1 
= \KS^{-1} \Pi(\tnabla_{v_1} \tnabla_{v_2} \bx 
- \tnabla_{v_2} \tnabla_{v_1} \bx) 
= \KS^{-1} \Pi(\tnabla_{[v_1,v_2]}\bx) 
= \Nabla_{[v_1,v_2]}  
\]
by the definition of the Kodaira--Spencer map 
and the flatness of $\tnabla$. 
This shows that $\Nabla$ is torsion-free. 

Suppose that $\sfP$ is parallel. 
To prove the flatness of $\Nabla$, 
it suffices to show that the connection  
$(\id\otimes \Pi) \circ \tnabla \colon 
\pr^* \sfF \to \bOmega^1 \hotimes \pr^* \sfF$ 
on $\pr^* \sfF$ is flat. Therefore 
it suffices to prove that 
the connection 
$(\id\otimes\Pi) \circ \nabla 
\colon \sfF \to \Omega^1_{\cM} \otimes \sfF$ 
on $\sfF$ is flat. 
Under the decomposition  
$\sfF[z^{-1}] = \sfF \oplus \sfP$, we can write 
\[
\nabla = 
\begin{pmatrix} 
A & 0 \\ 
C & B 
\end{pmatrix} 
\]
with $A \in \Hom_{\C}(\sfF,  \Omega^1_{\cM} \otimes \sfF)$, 
$B \in \Hom_\C(\sfP,\Omega^1_{\cM} \otimes \sfP )$, and
$C \in \Hom_{\cO_\cM}(\sfF, \Omega^1_{\cM} \otimes \sfP)$,
because $\sfP$ is parallel. 
Here $A = (\id\otimes \Pi) \circ \nabla$. 
The flatness of $\nabla$ implies that 
$A$ and $B$ are flat connections. 
\end{proof} 

\begin{lemma} 
\label{lem:poleorder-Nabla} 
The connection $\Nabla \colon \bOmegao^1 \to 
\bOmegao^1 \otimes \bOmegao^1$ associated to a 
pseudo-opposite module $\sfP$ 
raises the pole order along the discriminant $P=0$ 
(see equation~\ref{eq:discriminant}) by at most one. 
\end{lemma} 
\begin{proof} 
The connection $\Nabla$ arises from
the connection $\tnabla^\vee \circ \Pi^* \colon \pr^* \sfF^\vee 
\to \bOmega^1 \otimes \pr^* \sfF^\vee$  
via the isomorphism $\KS^* \colon \pr^* \sfF^\vee |_{\LLo}
\cong \bOmegao^1$. Both the connection $\tnabla^\vee \circ \Pi^*$ 
and $\KS^*$ are regular along $P=0$, but the inverse $\KS^{*-1}$ 
has a pole of order one along $P=0$. 
The conclusion follows. 
(See also the formula in Example~\ref{exa:Nabla-explicit}.) 
\end{proof} 

\begin{example} 
\label{exa:Nabla-explicit} 
Assume that $\sfP$ is an opposite module and that we have 
a trivialization $\sfF \cong \C^{N+1} \otimes \cO[\![z]\!]$ 
such that $\sfP$ is identified in this 
trivialization as $\C^{N+1} \otimes z^{-1}\cO[z^{-1}]$. 
(The trivialization here is the \emph{flat trivialization}  
associated to $\sfP$ in Proposition~\ref{prop:flat_trivialization}.)
In this case, $\cC(t,z)$ in the presentation 
\eqref{eq:presentation_nabla} of $\nabla$ 
is independent of $z$. 
We write $\cC= \cC(t) = \cC(t,z)$ below. 
Let $\{t^i,x_n^i\}$ be the associated algebraic local
co-ordinate system on $\LL$. 
The flat connection 
$\Nabla \colon \bOmegao^1 \to \bOmegao^1 \otimes \bOmegao^1$ is 
given in these co-ordinates as:
\begin{align*}
\Nabla d t^h &= - [K(x_1)^{-1} \cC_i e_j]^h
(d t^i \otimes d x_1^j + d x_1^j\otimes d t^i) \\
& \quad + [K(x_1)^{-1} (\cC_i \cC_j x_2 - (\partial_i \cC_j ) x_1)]^h
d t^i \otimes d t^j, \\
\Nabla d x_n^h 
& = -[K(x_{n+1}) K(x_1)^{-1} \cC_i e_j ]^h 
(d t^i \otimes d x_1^j+ d x_1^j \otimes dt^i) \\ 
& \quad + \left 
[K(x_{n+1})K(x_1)^{-1} 
(\cC_i \cC_j x_2 - (\partial_i \cC_j) x_1) 
-(\cC_i \cC_j x_{n+2} - (\partial_i \cC_j) x_{n+1}) 
\right ]^h
d t^i \otimes d t^j \\ 
&  \quad
+ [\cC_i e_j]^h (d t^i \otimes d x_{n+1}^j + 
d x_{n+1}^j \otimes d t^i), \quad n\ge 1
\end{align*} 
where $K(x_n)\in \End(\C^{N+1}) \otimes \bcO$ 
is defined by $K(x_n) e_i := \cC_i(t) x_n$ 
and we used the Einstein summation convention for the repeated 
indices $i,j,h$. 
\end{example} 

\begin{remark} 
When $\sfP$ is an opposite module, we have two different 
flat structures on the total space $\LLo$. 
Recall that the tangent bundle $\bThetao$ is identified with 
$\pr^*\sfF$ via the Kodaira--Spencer map 
and the flat connection $\Nabla^{\sfP}$ is induced 
from the flat connection $\Pi_{\sfP} \circ \tnabla$ on 
$\pr^*\sfF$. 
Another flat structure on $\LLo$ 
is given by the flat trivialization 
$\sfF \cong (\sfF \cap z\sfP)[\![z]\!] \cong 
(z\sfP/\sfP)[\![z]\!]$ that we discussed in Proposition~\ref{prop:flat_trivialization}.  
This arises from the restriction of the flat connection 
$\Pi_{\sfP} \circ \tnabla$ to the flat-subbundle 
$\pr^*(\sfF \cap z\sfP)$ and its $z$-linear extension. 
Note that $\Nabla^{\sfP}$ is not $z$-linear (under the 
identification $\bThetao \cong \pr^*\sfF)$ 
whereas the latter flat structure is $z$-linear. 
\end{remark} 

\subsection{Flat Co-ordinates and Genus-Zero Potential}
\label{subsec:flatstronL}

We construct a flat co-ordinate system for $\Nabla = \Nabla^\sfP$ 
for a parallel pseudo-opposite module $\sfP$ 
and see that the Yukawa coupling is the third derivative of 
a certain function, called the genus-zero potential.  
A flat co-ordinate system and the genus-zero potential  
may only be defined in the formal\footnote{Or, rather than formal neighbourhood, in an $L^2$- or nuclear neighbourhood: 
see Remarks~\ref{rem:L2-neighbourhood} and~\ref{rem:nuclear-neighbourhood}.} neighbourhood 
$\hLLo$ of the fiber $\LLo_t = \pr^{-1}(t) \cap \LLo$ 
at $t\in \cM$.   

Let $s_0,\dots,s_N$ be a regular system of parameters 
in the local ring $\cO_{\cM,t}$ at $t\in \cM$. 
The formal neighbourhood $\hcM$ of $t$ is then given by:
\[
\hcM = \Spf \C[\![s_0,\dots,s_N]\!]
\]
Take a local trivialization 
$\sfF \cong \bigoplus_{i=0}^N \cO[\![z]\!]e_i$ 
in a neighbourhood of $t$  
and let $\{s^i,x_n^i\}$ be the corresponding algebraic local co-ordinate system 
on $\LL$ as in \S\ref{subsec:totalspace}. 
The formal neighbourhood $\hLL$ 
of $\LL_t=\pr^{-1}(t)$ in $\LL$ (respectively the formal neighbourhood $\hLLo$ of $\LLo_t$ in $\LLo$)
is then given by  
\begin{align*} 
\hLL & = \Spf \C\left[
\{x_n^i\}_{n\ge 1, 0\le i\le N} \right]
[\![s^0,\dots,s^N]\!] \\ 
\hLLo &  = \Spf \C
\left[\{x_n^i\}_{n\ge 1, 0\le i\le N}, 
P(t,x_1)^{-1}\right][\![s^0,\dots,s^N]\!] 
\end{align*} 
where $P(t,x_1)$ is the discriminant \eqref{eq:discriminant}. 

Let $\sfP$ be a parallel pseudo-opposite module over 
the formal neighbourhood $\hcM$ of $t$ 
(see Lemma~\ref{lem:existence_opposite}). 
The above local trivialization of $\sfF$ induces 
a trivialization $\sfF|_{\hcM} \cong 
\bigoplus_{i=0}^N \C[\![z]\!][\![s_0,\dots,s_N]\!]e_i$ 
on the formal neighbourhood $\hcM$. 
We can solve for a unique flat section 
$f_i(s)\in \widehat{\sfF[z^{-1}]}: = \projlim_n \sfF[z^{-1}]/\fm_t^n 
\sfF[z^{-1}] 
\cong \bigoplus_{i=0}^N \C(\!(z)\!)[\![s_0,\dots,s_N]\!] e_i$ 
such that: 
\begin{align*}
  \nabla f_i(s)=0 && f_i(0) = e_i
\end{align*}
This defines a parallel transportation map: 
\begin{align*}
  \PT \colon 
  \widehat{\sfF[z^{-1}]} 
  \overset{\cong}{\longrightarrow}  
  \sfF_t[z^{-1}][\![s_0,\dots,s_N]\!]
  &&
  f_i \mapsto e_i 
\end{align*}
which is an isomorphism of $\C(\!(z)\!)[\![s_0,\dots,s_N]\!]$-modules. 
Since the symplectic form $\Omega_t$ identifies $\sfP_t$ with 
$\sfF_t^\vee$, there exist unique elements 
$\xi_m^j \in \sfP_t$, $m\ge 0$, $0\le j\le N$,
such that:
\[
\Omega(\xi^j_m, e_iz^n) = \delta_i^j \delta_{n,m}
\]
Then we have a Darboux basis 
$\{e_i z^n, \xi_m^j\}_{0 \leq n,m < \infty, 0\le i,j\le N}$ of $\sfF_t[z^{-1}]$.  
A general element of $\sfF_t[z^{-1}]$ can be written as  
a linear combination 
\[
\sum_{n=0}^\infty \sum_{i=0}^N q_n^i e_i z^n 
+ \sum_{m=0}^\infty \sum_{j=0}^N p_{m,j} \xi_m^j 
\]
and the coefficients $\{q_n^i, p_{m,j}\}$ form 
a Darboux co-ordinate system on $\sfF_t[z^{-1}]$. 
Pulling back the Darboux co-ordinates via 
\[
\emb \colon 
(z\sfF)|_{\hcM} \hookrightarrow \widehat{\sfF[z^{-1}]} 
\overset{\PT}{\longrightarrow} \sfF_t[z^{-1}][\![s_0,\dots,s_N]\!]
\]
we get regular functions $q_n^i, p_{m,j}$ 
on the total space $\hLL$ of $(z\sfF)|_{\hcM}$: 
\begin{align*}
  q_n^i := \emb^*(q_n^i) && p_{m,j} := \emb^* (p_{m,j})
\end{align*}

\begin{definition} 
\label{def:flatcoordinate_on_LL} 
We call $\{q_n^i\}_{n\ge 0, 0\le i\le N}$ the \emph{flat 
co-ordinate system} on the formal neighbourhood 
$\hLLo$ of $\LL_t$. 
It depends only on the choice of a trivialization 
$\sfF_t \cong \C^{N+1}[\![z]\!]$ 
at the point $t$ and on the isotropic subspace 
$\sfP_t\subset \sfF_t[z^{-1}]$ which is complementary 
to $\sfF_t$. 
One can view this flat co-ordinate system 
as a  ``projection" to the tangent space:  
\[
\bq = \sum_{n=0}^{\infty} \sum_{i=0}^N q_n^i e_i z^n 
\colon \hLL \longrightarrow \sfF_t 
\]
such that it is the identity on $\LL_t = z\sfF_t$ 
and its derivative at any point $\bx$ in $\LL_t$ 
\[
D\bq\colon \bTheta_\bx \longrightarrow \sfF_t 
\]
coincides with the Kodaira--Spencer map 
(this will be verified in equation~\ref{eq:KSinv_dq+dp} below). 
\end{definition} 

The flatness of the co-ordinates $q_n^i$ will be shown momentarily. 
We write $e_i = \sum_{j=0}^N f_j(s) M^j_i(s,z)$ 
with $M^j_i \in \C(\!(z)\!)[\![s^0,\dots,s^N]\!]$.  
Let $M(s,z)$ be the $(N+1) \times (N+1)$ matrix with matrix elements  
$M_i^j(s,z)$. 
By definition, $M(s,z)$ is a matrix representation of the parallel transportation 
map $\PT$, i.e.~$\PT(e_i) = \sum_{j=0}^N M^j_i(s,z) e_j$. 
By the definition of the functions $q_n^i$, $p_{m,j}$ on 
$\hLL$, we have 
\begin{equation} 
\label{eq:formal_flat_coordinates} 
\bq + \bp=  M(s,z) \bx 
\end{equation} 
where:
\begin{align*}
  \bq = \sum_{n=0}^\infty \sum_{i=0}^N q_n^i e_i z^n &&
  \bp = \sum_{n=0}^\infty \sum_{i=0}^N p_{n,i} \xi_n^i &&
  \bx = \sum_{n=1}^\infty \sum_{i=0}^N x_n^i e_i z^n &&
\end{align*}
Here $M(s,z)$ acts on the column vector $\bx$ in the 
basis $e_0,\dots,e_N$. 
Let $\nabla = d - z^{-1} \cC(s,z)$ be the 
presentation of the connection in the 
trivialization given by the frame $e_0,\dots,e_N$. 
The matrix $M(s,z)$ is a solution to the differential equation 
\begin{equation} 
\label{eq:inverse_fundsol} 
d M(s,z) = -z^{-1} M(s,z) \cC(s,z) 
\end{equation} 
with the initial condition $M(0,z) = I$; i.e.~$M$ is an inverse fundamental solution, 
cf.~\eqref{eq:inversefundamentalsolution}. 
Therefore $M(s,z)= I-z^{-1}\sum_i \cC_i(0,z) s^i+\hot$, 
where $\hot$ means terms of order $2$ or more in $s^0,\dots,s^N$. 
Thus we have, by \eqref{eq:formal_flat_coordinates}:
\begin{align}
\label{eq:flat-alg}
\begin{split} 
  q_0 &= - \sum_i s^i\cC_i(0,0)x_1 + \hot \\
  q_n & =x_n - \sum_i s^i [\cC_i(0,z) \bx]_{n+1} + \hot  \quad n\ge 1 
\end{split} 
\end{align} 
where $q_n = \sum_{i=0}^N q_n^i e_i$ and $[\cdots]_n$ 
denotes the coefficient of $z^n$. The lowest order 
term of the first equation gives an 
invertible change of variables between 
$\{q_0^i\}_{i=0}^N$ and $\{s^i\}_{i=0}^N$ 
when the matrix formed by the column vectors 
$\{\cC_i(0,0)x_1\}_{i=0}^N$ is invertible, i.e.~when
$P(t,x_1)$ is invertible. 
Therefore $\{q_n^i\ : \text{$n\ge 0$, $0\le i\le N$}\}$ gives a co-ordinate system on $\hLLo$, in the sense that:
\[
\C\left[\{x_n^i\}_{n\ge 1, 0\le i\le N}, 
P(t,x_1)^{-1}\right][\![s^0,\dots,s^N]\!] 
= 
\C\left[\{q_n^i\}_{n\ge 1, 0\le i\le N}, 
P(t,q_1)^{-1}\right][\![q_0^0,\dots,q_0^N]\!]
\]
We elaborate on this in Lemma~\ref{lem:sx-poleorder} below. 
\begin{remark} 
\label{rem:difference-normalization}
Notice a small difference between $M$  
in Gromov--Witten theory (see equation~\ref{eq:inversefundamentalsolution})
and $M$ in the above construction 
(equation~\ref{eq:formal_flat_coordinates}). 
In the construction above, $M$ is normalized 
so that it is the identity at the base point. 
In Gromov--Witten theory, however, 
it is normalized by the asymptotic behaviour 
$M \sim e^{-\delta/z}$ 
at the large radius limit (see equation~\ref{eq:M-divisoreq}). 
The Gromov--Witten case will be discussed 
in Example~\ref{exa:Amodel-genuszero}. 
\end{remark} 

\begin{lemma} 
\label{lem:sx-poleorder}
When we invert the co-ordinate change \eqref{eq:flat-alg} 
and express $s^i, x_n^i$, $n\ge 1$ as functions of $q_n^i$, $n\ge 0$, 
we find  
\begin{align*} 
s^i & \in P_t  \C[q_1,q_2, P_t q_3, P_t^2 q_4,\dots]
[\![P_t^{-2} q_0]\!] \\ 
x_n^i & \in 
\delta_{n,1} q_1^i + P_t^{2-n} \C[q_1,q_2,P_t q_3, P_t^2 q_4, \dots ] 
[\![P_t^{-2} q_0]\!] \quad n\ge 1
\end{align*} 
where $P_t = P(t,q_1)$. Moreover, we have: 
\[
\sum_{i=0}^N s^i \cC_i(0,0) q_1 \in 
P_t^2 \C^{N+1}[q_1,q_2,P_t q_3,P_t^2q_4,\dots]
[\![P_t^{-2} q_0]\!]
\]
\end{lemma} 
\begin{proof} 
Because $M(s,z)$ is a solution to the differential equation 
\eqref{eq:inverse_fundsol} with $M(0,z)=\id$, we can 
expand it in the form: 
\begin{equation} 
\label{eq:M-expansion}
M(s,z) = \id + \sum_{n>0} \sum_{I =(i_1,\dots,i_n)} 
\sum_{m\ge 0} s^I 
M_{I,m} z^{-n+m}  
\end{equation} 
with $s^I = s^{i_1} \cdots s^{i_n}$. 
Let $\Pi_t \colon \sfF_t[z^{-1}] \to \sfF_t$ denote the 
projection along $\sfP_t$. 
We set $\Pi_t( v z^{-a}) = \sum_{u=0}^\infty \pi^{-a}_u(v) z^u$ 
for $a>0$, where $\pi^{-a}_u \in \End(\C^{N+1})$.  
From \eqref{eq:formal_flat_coordinates}, we have
\begin{align*} 
q_u  = x_u  &+  \sum_{\substack{l>0, n>0, m\ge 0 \\ 
-n+m+l =u}} \sum_I s^I 
M_{I,m} x_l  +  \sum_{\substack{l>0,n>0,m\ge 0 \\ -n+m+l<0}} 
\sum_I s^I  
\pi^{-n+m+l}_u M_{I, m} x_l   
\end{align*} 
for $u\ge 0$. 
Here $I=(i_1,\dots,i_n)$, and we set $x_0=0$. 
Setting 
\begin{align*}
  q_n^i  = P_t^{2-n} \hq_n^i \quad (n\neq 1),  && 
  s^i = P_t \hs^i, && 
  x_n^i = \delta_{n,1} q_1^i + P_t^{2-n} \hx_n^i \quad (n\ge 1),  
\end{align*}
we can rewrite this in the following form: 
\begin{equation}
  \label{eq:hq0} 
  \begin{aligned}
    \hq_0 & = P_t^{-1} \sum_{i=0}^N \hs^i M_{i,0} q_1 
    + \sum_{i=0}^N \hs^i M_{i,0} \hx_1 + 
    \sum_{\substack{l>0, m\ge 0 \\ n=m+l \ge 2}} 
    \sum_I P_t^m \hs^{I} M_{I,m} 
    (\delta_{l,1} P_t^{-1} q_1 + \hx_l ) \\ 
    & \qquad + 
    \sum_{\substack{l>0,n>0,m\ge 0 \\ m+l <n}} \sum_I
    P_t^{n-l} \hs^{I} 
    \pi^{-n+m+l}_0M_{I,m} ( \delta_{l,1} P_t^{-1} q_1 + \hx_l)  \\
    0 & = \hx_1 + 
    \sum_{\substack{l>0,n>0, m\ge 0 \\m+l = n+1}} \sum_I 
    P_t^m \hs^{I} M_{I,m} 
    (\delta_{l,1} P_t^{-1} q_1 + \hx_l )  \\ 
    & \qquad + \sum_{\substack{l>0, n>0, m\ge 0\\ m+l<n}} \sum_I
    P_t^{n-l+1} \hs^{I} \pi^{-n+m+l}_1 M_{I,m} 
    (\delta_{l,1} P_t^{-1} q_1 + \hx_l) \\ 
    \hq_u & = 
    \hx_u + \sum_{\substack{l>0,n>0,m\ge 0 \\ 
        m+l = n + u}} \sum_I
    P_t^m \hs^{I} M_{I,m} 
    (\delta_{l,1}P_t^{-1} q_1 + \hx_l) \\ 
    & \qquad + \sum_{\substack{l>0,n>0,m\ge 0 \\ 
        m+l<n}} \sum_I P_t^{u+(n-l)} \hs^I \pi^{-n+m+l}_uM_{I,m}
    (\delta_{l,1} P_t^{-1} q_1 + \hx_l) \qquad  (u \ge 2)
  \end{aligned}
\end{equation}
where again $I=(i_1,\dots,i_n)$.
Note that the powers of $P(t,q_1)$ appearing 
on the right-hand side are non-negative except for the 
leading term $P_t^{-1} \sum_{i=0}^N \hs^i M_{i,0} q_1$ in $\hq_0$. 
From these equations, we can solve for $\hs^i$,~$\hx_n^i$ 
as functions of $\hq_n^i$, $n\neq 1$, and $q_1^i$. 
To do this, we need to invert the leading term operator:
\[
\hs \mapsto P(t,q_1)^{-1} \sum_{i=0}^N \hs^i M_{i,0} q_1 
= - P(t,q_1)^{-1} \sum_{i=0}^N \hs^i \cC_i(0,0) q_1
\]
Because $P(t,q_1) = (-1)^{N+1}\det\big(\cC_0(0,0)q_1, \dots, \cC_N(0,0)q_1\big)$, 
the inverse operator is polynomial in $q_1^0,\dots, q_1^N$. 
(The inverse is the transpose of the cofactor matrix 
of $-(\cC_0(0,0)q_1,\dots,\cC_N(0,0)q_1)$.)  
Therefore we have:
\begin{equation} 
\label{eq:hshx}
\hs^i, \hx_n^i  \in \C[q_1,\hq_2,\hq_3,\dots][\![\hq_0]\!]
\end{equation} 
The first statement in the Lemma follows by substituting
$\hq_n = P(t,q_1)^{n-2}q_n$, $n\neq 1$. 
Equations \eqref{eq:hq0} and \eqref{eq:hshx} 
in turn show that $P_t^{-1}\sum_{i=0}^N \hs^i M_{i,0} q_1$ 
lies in $\C[q_1,\hq_2,\hq_3,\dots][\![\hq_0]\!]$. 
The second statement follows. 
\end{proof}

\begin{proposition}[flatness]   
$\Nabla^{\sfP} dq_n^i =0$. 
\end{proposition} 
\begin{proof} 
We regard $\bq+\bp$ as an $\sfF_t[z^{-1}]$-valued  
function on $\hLL$. 
By \eqref{eq:formal_flat_coordinates} and 
\eqref{eq:inverse_fundsol}, we have:
\begin{equation} 
\label{eq:dq+dp} 
d \bq + d \bp = (dM(s,z)) \bx + 
M(s,z) d\bx = 
M(s,z) (-z^{-1} \cC(s,z) \bx + d \bx) 
= M(s,z) \tnabla \bx
\end{equation} 
This is an equality in 
$\sfF_t[z^{-1}]\hotimes \bOmega^1 = 
\projlim_n \sfF_t[z^{-1}]\otimes (\bOmega^1/\fm_t^n\bOmega^1)$; 
$\tnabla \bx$ is a section of $\pr^*\sfF \hotimes \bOmega^1$
and $M(s,z)$ acts on the $\pr^*\sfF$ factor 
(via the trivialization). 
By \eqref{eq:KS_coord}, we have:
\begin{equation} 
\label{eq:KSinv_dq+dp} 
\KS^{*-1} (d\bq+ d\bp) =  
\sum_{n=0}^{\infty} \sum_{i=0}^N M(s,z)e_iz^n \otimes \varphi_n^i
\end{equation} 
This is an equality in 
$\sfF_t[z^{-1}]\hotimes \pr^*\sfF^\vee$. 
For the map $\Pi^* \colon \pr^*\sfF^\vee\to 
\pr^*(z^{-1} \sfF)^\vee$, we have:
\[
\Pi^* \varphi_n^i = \varphi_n^i + 
\sum_{j=0}^N 
\left[\Pi e_j z^{-1}\right]_n^i \varphi_{-1}^j
\]
Hence for the map $\tnabla^\vee \colon 
\pr^*(z^{-1}\sfF)^\vee \to \bOmega^1\otimes \pr^* \sfF^\vee$
we have, from \eqref{eq:tnablavee}:
\begin{align*} 
\tnabla^\vee \Pi^* \varphi_n^i 
& = \sum_{l=0}^\infty \sum_{j=0}^N 
\left[\cC(s,z) e_j z^{l} \right]_{n+1}^i \varphi_l^j 
+ \sum_{l=0}^{\infty} \sum_{j=0}^N \sum_{h=0}^N 
\left[\Pi e_j z^{-1}\right]_n^i 
\left[\cC(s,z) e_h z^l\right]_0^j \varphi_l^h \\ 
& = \sum_{l=0}^\infty \sum_{j=0}^N 
\left[z^{-1}\cC(s,z) e_j z^{l} \right]_{n}^i \varphi_l^j 
+ \sum_{h=0}^N 
\left[\Pi \cC(s,0) e_hz^{-1}\right]_n^i \varphi_0^h
\end{align*} 
Therefore, from \eqref{eq:KSinv_dq+dp} and 
\eqref{eq:inverse_fundsol}:
\begin{align*} 
& \tnabla^\vee \Pi^* \KS^{*-1} (d \bq + d\bp) \\
& = \sum_{n=0}^\infty \sum_{i=0}^N M(s,z) e_iz^n \otimes 
\left(\sum_{l=0}^\infty \sum_{j=0}^N 
\left[z^{-1}\cC(s,z) e_j z^{l} \right]_{n}^i \varphi_l^j 
+ \sum_{h=0}^N 
\left[\Pi \cC(s,0) e_hz^{-1}\right]_n^i \varphi_0^h\right) \\ 
& \qquad - \sum_{n=0}^\infty \sum_{i=0}^N 
M(s,z) z^{-1} \cC(s,z) e_i z^n \otimes \varphi_n^i \\
& = \sum_{h=0}^N M(s,z) \Pi (\cC(s,0) e_h z^{-1}) \otimes \varphi_0^h 
- \sum_{i=0}^N M(s,z) \cC(s,0) e_iz^{-1} \otimes \varphi_0^i \\ 
& = - \sum_{i=0}^N M(s,z) 
\left[ \cC(s,0) e_i z^{-1}\right]_{\sfP} \otimes \varphi_0^i
\end{align*} 
Here $[\cC(s,0) e_h z^{-1}]_\sfP$ denotes the $\sfP$-component 
of the section $\cC(s,0) e_h z^{-1}$ of 
$\Omega^1_{\cM} \otimes \sfF[z^{-1}]$ under 
the decomposition $\sfF[z^{-1}] = \sfF \oplus \sfP$. 
Applying $\id \otimes \KS^*$ to the above equality and 
using \eqref{eq:KS_coord}, we obtain 
\begin{equation} 
\label{eq:Nabla_dq+dp} 
\Nabla( d\bq+ d \bp) 
= \sum_{i=0}^N \sum_{j=0}^N 
\left( M(s,z) \left[z^{-1}\cC_i(s,0) \cC_j(s,0) x_1\right]_{\sfP}\right) 
ds^i\otimes ds^j 
\end{equation} 
where $\Nabla = \Nabla^\sfP$ is the connection on $\bOmegao^1$ 
associated to $\sfP$. 
Since $\sfP$ is parallel and $M(s,z)$ represents 
the parallel transportation map to the fiber $\sfF_t[z^{-1}]$, 
the right-hand side is a $\sfP_t$-valued quadratic differential 
on $\hLLo$. 
\end{proof} 

\begin{lemma} 
\label{lem:tensorT} 
The tensor $\bT := \Omega_t(d\bp\otimes d\bq) = 
\sum_{n=0}^\infty \sum_{i=0}^N  
d p_{n,i} \otimes d q_n^i$ on $\hLL$ is symmetric.  
In particular, $(\partial p_{n,i}/\partial q_m^j)$ 
is symmetric in $(n,i)$ and $(m,j)$. 
\end{lemma} 
\begin{proof} 
Since $\pr^*\sfF\subset \pr^*\sfF[z^{-1}]$ 
is isotropic with respect to $\Omega$, 
we have  
\[
\Omega\left(\tnabla \bx\otimes \tnabla \bx 
\right) =0 
\]
where we regard $\tnabla \bx$ 
as a section of $\pr^*\sfF \hotimes \bOmega^1$ 
and $\Omega$ contracts the $\pr^*\sfF$ component. 
Since the parallel transportation map $M(s,z)$ to 
the fiber $\sfF_t$ preserves the symplectic form, 
by \eqref{eq:dq+dp} we have\footnote
{Note that this is not a trivial equality.} 
\[
\Omega_t\left( (d\bq+d\bp) \otimes (d\bq+d\bp)\right)  = 0
\]
where $\Omega_t$ denotes the symplectic form on $\sfF_t[z^{-1}]$. 
This implies that
\[
\Omega_t( d\bq \otimes d\bp) + \Omega_t(d \bp \otimes d \bq) = 0
\]
which completes the proof. 
\end{proof} 

\begin{definition} 
\label{def:genuszeropot}
The \emph{genus-zero potential} is a function on $\hLL$ defined by:
\begin{equation} 
\label{eq:genuszeropot}
C^{(0)} :=  \frac{1}{2} \sum_{n=0}^\infty\sum_{i=0}^N 
p_{n,i} q_n^i
\end{equation} 
This depends on a choice of parallel pseudo-opposite module $\sfP$ 
over the formal neighbourhood $\hcM$. 
\end{definition} 
\begin{lemma} \[p_{n,i} = \frac{\partial C^{(0)}}{\partial q_n^i}\]
\end{lemma} 
\begin{proof} 
By Lemma~\ref{lem:tensorT}, we have:
\[
\parfrac{C^{(0)}}{q_l^j} 
= \frac{1}{2} p_{l,j} + \frac{1}{2} 
\sum_{n=0}^\infty \sum_{i=0}^N 
\parfrac{p_{n,i}}{q_l^j} q_n^i 
= \frac{1}{2} p_{l,j} 
+ \frac{1}{2} \sum_{n=0}^\infty\sum_{i=0}^N  
q_n^i \parfrac{p_{l,j}}{q_n^i}
= p_{l,j}
\]
Here we used the fact that 
the function $p_{l,j}$ is homogeneous 
of degree one with respect to the 
dilation vector field 
$\sum_{n=0}^\infty\sum_{i=0}^{N} q_n^i 
(\partial/\partial  q_n^i) 
= \sum_{n=1}^\infty \sum_{i=0}^N x_n^i 
(\partial/\partial x_n^i)$. 
\end{proof} 
\begin{proposition}[Potentiality] 
\label{prop:potentiality}
The Yukawa coupling $\bY$ is the third covariant 
derivative of $C^{(0)}$, i.e.\ $\Nabla^3 C^{(0)} = \Nabla \bT = \bY$. 
Here $\Nabla = \Nabla^{\sfP}$ is the flat connection 
associated to the parallel pseudo-opposite module $\sfP$. 
\end{proposition} 
\begin{proof}
Using $\Nabla d q_n^i=0$, we have:
\[
\Nabla^2 C^{(0)} = \Nabla 
\left(\sum_{n=0}^\infty \sum_{i=0}^N p_{n,i} d q_n^i\right) 
= \sum_{n=0}^\infty \sum_{i=0}^N 
dp_{n,i} \otimes dq_n^i = \bT
\]
Using $\Nabla d\bq =0$, we have:
\[
\Nabla \bT =  \Nabla \Omega_t(d \bp \otimes d\bq) 
= \Omega_t( (\Nabla d \bp) \otimes d \bq) 
= \Omega_t(\Nabla (d\bq +d\bp) \otimes (d \bq + d\bp)) 
\]
Using \eqref{eq:Nabla_dq+dp}, \eqref{eq:dq+dp} 
and the fact that $M(s,z)$ preserves the 
symplectic form, we have:
\begin{align*} 
\Nabla \bT & = \sum_{i=0}^N \sum_{j=0}^N \Omega\left(
[z^{-1}\cC_i(s,0) \cC_j(s,0) x_1]_{\sfP} ds^i\otimes ds^j
\otimes  
(d \bx - z^{-1} \cC(s,z) \bx)\right) \\
& = \Omega\left(z^{-1}\cC_i(s,0) \cC_j(s,0) x_1 ds^i\otimes ds^j
\otimes  
(d \bx - z^{-1} \cC(s,z) \bx)\right) 
\end{align*} 
This equals $\bY$. 
\end{proof} 

\begin{lemma}[genus-zero pole structure] 
\label{lem:genuszeropole}
The genus-zero potential $C^{(0)}$ is an element of 
$P_t^5 \C [q_1,q_2,P_t q_3, P_t^2 q_4,\dots][\![P_t^{-2} q_0]\!]$ 
where $P_t=P(t,q_1)$. 
\end{lemma} 
\begin{proof} 
Set $\cS:= \C [q_1,q_2,P_t q_3, P_t^2 q_4,\dots][\![P_t^{-2} q_0]\!]$. 
Note that we have $C^{(0)}|_{q_0=0} = C^{(0)}|_{s=0} = 0$. 
Thus it suffices to show that $p_{0,i} = \partial C^{(0)}/\partial q_0^i 
\in P_t^3 \cS$. 
We set $\Omega_t(e_j z^{-n}, e_i) = h_{n;ij} \in \C$. 
Using \eqref{eq:formal_flat_coordinates} and 
the expansion \eqref{eq:M-expansion} of $M(s,z)$, 
we have:
\begin{align*} 
p_{0,i} & = \Omega_t(\bq+ \bp,e_i) 
= \Omega(M(s,z) \bx, e_i) \\
& = \sum_{n>0, m\ge 0, l\ge 1} 
\sum_{I=(i_1,\dots,i_n)} 
s^I \Omega_t( M_{I,m} z^{-n+m+l} x_l, e_i) \\
& = \sum_{\substack{n>0, m\ge 0, l\ge 1\\ n>m+l}}  
\sum_{I=(i_1,\dots,i_n)}   \sum_{j=0}^N 
s^I h_{n-m-l; ij} [M_{I,m} x_l ]^j
\end{align*} 
By Lemma~\ref{lem:sx-poleorder}, 
we have that $s^i \in P_t\cS$ and $x_l \in \delta_{l,1} q_1+ P_t^{2-l} \cS$. 
From this we find that all the terms on the right-hand side belong to 
$P_t^3 \cS$ except perhaps for the following one, which arises from $(n,m,l)=(2,0,1)$:
\[
\sum_{i_1=0}^N \sum_{i_2=0}^N\sum_{j=0}^N 
s^{i_1}s^{i_2} h_{1,ij}[M_{i_1i_2,0} q_1]^j
\]
But the differential equation \eqref{eq:inverse_fundsol} 
for $M(s,z)$ shows that $M_{i_1i_2,0}= \cC_{i_1}(0,0)\cC_{i_2}(0,0)$. 
The second part of Lemma~\ref{lem:sx-poleorder} 
now shows that the above sum lies in $P_t^3 \cS$ as well. 
\end{proof} 

\begin{remark} 
The genus-zero potential $C^{(0)}$ may only be defined on 
the formal neighborhood $\hLLo$ whereas the Yukawa 
coupling $\Nabla^3 C^{(0)} =\bY$ is globally defined. 
The data $C^{(0)}$ and $\Nabla$ 
depend on the choice of a parallel pseudo-opposite module $\sfP$ 
whereas $\bY=\Nabla^3 C^{(0)}$ does not. 
\end{remark} 
\begin{remark}
The genus-zero potential is homogeneous of degree 2 
with respect to the dilation vector field.  
\end{remark}

\begin{remark}[$L^2$-neighbourhood] 
\label{rem:L2-neighbourhood} 
Let $U\subset \cM$ be an open set  
with co-ordinates $s^0,\dots,s^N$ centred 
at a point in $U$ 
and let $\sfP$ be an opposite module on $U$. 
Then $\sfP$ defines a flat trivialization 
$\sfF|_U \cong (z\sfP/\sfP) \otimes \C[\![z]\!]$ 
(Proposition~\ref{prop:flat_trivialization}). 
Suppose that we can trivialize 
$z\sfP/\sfP$ by a $\nabla^0$-flat frame 
over $U$. 
This defines a trivialization $\sfF|_U \cong 
\C^{N+1} \otimes \cO_U[\![z]\!]$. 
Using the local co-ordinate system 
$\{s^i,x_n^i\}$ associated to this trivialization, 
we can define the \emph{$L^2$-subspace} 
$L^2(\LL)$ of $\LL$ as:
\[
L^2(\LL) = \left\{ (s,\bx)\in \LL|_U 
\,\Big|\, s\in U, \ \sum_{n=1}^\infty 
\sum_{i=0}^N |x_n^i|^2 < \infty\right\}
\] 
This has the structure of a complex Hilbert manifold.  
In this case, $\bp$ is a strictly negative power series 
in $z$ with respect to the trivialization 
(since it belongs to $\sfP$). 
Because the inverse fundamental solution 
$M(s,z)$ in \eqref{eq:inverse_fundsol} is holomorphic 
over $U\times \C^\times$, 
$\bq$ and $\bp$ given by \eqref{eq:formal_flat_coordinates} 
belong to $L^2(S^1,\C^{N+1})$ 
when $(s,\bx) \in L^2(\LL)$. 
The genus-zero potential $C^{(0)}$ defined in 
\eqref{eq:genuszeropot} therefore converges to a holomorphic 
function on $L^2(\LL)$. 
Moreover, the inverse function theorem for 
Hilbert spaces implies that the map 
$(s,\bx) \mapsto \bq$ defines a local isomorphism 
between $L^2(\LLo)$ and $\C^{N+1}\otimes L^2(S^1,\C)$. 
This means that $\{q_n^i\}$ is a co-ordinate system 
on an $L^2$-neighbourhood of each point in $L^2(\LLo)$.  
\end{remark} 

\begin{remark}[Nuclear neighbourhood]  
\label{rem:nuclear-neighbourhood} 
Following~\cite[\S 8.4]{CI:convergence}, 
we define the space $\C\{\!\{z,z^{-1}\}\!\}$ 
of formal Laurent series in $z$ to be 
\[
\C\{\!\{z,z^{-1}\}\!\}  
= \big\{ \ba \in \C[\![z,z^{-1}]\!] \,:\, \|\ba\|_n <\infty \ 
\text{ for all } n\gg 0\big\} 
\]
where $\|\cdot\|_n$, $n=0,1,2,\dots$ is a family of 
Hilbert norms defined by 
\[
\|\ba \|_n = \left( 
\sum_{l \in \Z} \frac{|a_l|^2}{|\Gamma(\frac{1}{2}+l)|^2} 
e^{2nl} \right)^{1/2} \qquad 
\text{for} \quad \ba =  \sum_{l\in\Z} a_l z^l. 
\]
We set:
\begin{align*}
  \C\{\!\{z\}\!\} = \C\{\!\{z,z^{-1}\}\!\} \cap \C[\![z]\!]
  &&
  \C\{\!\{z^{-1}\}\!\} = \C\{\!\{z,z^{-1}\}\!\} \cap \C[\![z^{-1}]\!]
\end{align*}
Then $\C\{\!\{z\}\!\}$ is a nuclear Frechet space\footnote
{This space is Laplace-dual to the space of entire functions on $\C$;
see~\cite[Remark 8.6]{CI:convergence}.}  
whose topology is given by the countable norms\footnote{All of the norms $\|\cdot\|_n$ are well-defined on 
$\C\{\!\{z\}\!\}$.} $\|\cdot\|_n$; 
$\C\{\!\{z^{-1}\}\!\}$ is the inductive limit 
of the Hilbert space completions of $\C[z^{-1}]$ 
with respect to $\|\cdot\|_n$ and is a nuclear (DF) space. 
We also know that $\C\{\!\{z,z^{-1}\}\!\}$ is a topological 
ring~\cite[Lemma 8.5]{CI:convergence}. 
Let us consider the same situation as in the previous 
Remark~\ref{rem:L2-neighbourhood}. 
We introduce a \emph{nuclear subspace} of $\LL$ 
which is an infinite-dimensional complex manifold 
modelled on $\C\{\!\{z\}\!\}$: 
\[
\cN(\LL): = \left\{ (s,\bx) \in \LL|_U \;\Big|\; s\in U, \ 
\sup_{0\le i\le N,\, l \ge 0} 
\left( e^{n l} |x_l^i|/l! \right)< \infty,  
\  \text{for all } n \ge 0 \right\}
\]
This contains $L^2(\LL)$ as a 
proper subspace. 
The genus-zero potential $C^{(0)}$ in this section 
defines an analytic function on this nuclear subspace. 
This follows from the method of~\cite{CI:convergence}, as follows.  
Because now $\cC(s,z)$ is independent of $z$, 
the inverse fundamental solution $M(s,z)$ satisfying 
\eqref{eq:inverse_fundsol} and the initial 
condition $M(0,z) = \id$ can be written as 
$M(s,z) = \id + \sum_{n=1}^\infty M_n(s)z^{-n}$ 
with 
\[
M_n (s) = \int_{0\le s_1 \le \cdots\le s_n \le s} 
(-\cC(s_1)) \cdots (-\cC(s_n))  
\]
where $s_1,\dots,s_n$ are on the line segment 
$[0,s]\subset U$. 
Therefore, after shrinking $U$ if necessary, 
we obtain the estimate 
\[
\|M_n(s)\| \le C^n \frac{1}{n!}, \quad s \in U 
\]
for some $C>0$. 
Using the results in~\cite[\S 8.4]{CI:convergence}, 
one finds easily that for $(s,\bx)\in \cN(\LL)$, 
$(\bq,\bp)$ defined by \eqref{eq:formal_flat_coordinates} 
belongs to 
$\C^{N+1}\otimes \C\{\!\{z,z^{-1}\}\!\}$. 
Thus $C^{(0)} = \frac{1}{2} \Omega(\bp,\bq)$ converges 
to a holomorphic function on $\cN(\LL)$, since 
$\C\{\!\{z,z^{-1}\}\!\}$ is a ring.  
Moreover, one can use the Nash--Moser inverse function theorem 
to show that map $(s,\bx) \mapsto \bq$ defines a 
local isomorphism between $\cN(\LLo)$ and $\C^{N+1}\otimes 
\C\{\!\{z\}\!\}$ 
by the same method as~\cite[\S 8.5]{CI:convergence}, i.e.\ 
$\{q_n^i\}$ gives a co-ordinate system 
on a nuclear neighbourhood of each point in $\cN(\LLo)$. 
\end{remark} 

\begin{remark} 
\label{rem:L2subspace_from_TP}
When the cTP structure $(\sfF,\nabla,(\cdot,\cdot)_\sfF)$ 
is the completion of a TP structure $(\cF,\nabla,(\cdot,\cdot)_\cF)$ 
(see Remark~\ref{rem:completion-of-TP}), the total 
space $\LL$ has standard $L^2$- and nuclear subspaces 
induced from the TP structure $\cF$. 
\end{remark} 

\begin{example}[Genus-zero Gromov--Witten potential~\cite{CI:convergence}]  
\label{exa:Amodel-genuszero} 
Recall from \S\ref{subsec:dilatonshift} 
that the genus-zero descendant Gromov--Witten 
potential $\cF^0_X$ of $X$ can be viewed as a function 
on $\cH_+$ via the Dilaton shift. 
Here we explain that the construction in this section 
starting from the A-model TEP structure of $X$ 
(Example~\ref{ex:AmodelTP}) 
gives rise to the genus-zero 
descendant Gromov--Witten potential $\cF^0_X$ 
under an identification of certain flat co-ordinates on $\LLo$ with 
the linear co-ordinates  $\{q_n^i\}$ 
on $\cH_+$ in \S\ref{subsec:Givental-symplecticvs}.  

As in Example~\ref{ex:AmodelTP}, we assume 
that the non-descendant genus-zero potential  $F^0_X$ 
(\S\ref{sec:convergence}) is convergent 
and defines an analytic function over an open subset 
$\cM_{\rm A}\subset H_X\otimes \C$ 
(after the specialization $Q_1 = \cdots = Q_r = 1$); 
then we have the A-model cTP structure 
$(\sfF, \nabla, (\cdot,\cdot)_{\sfF})$ 
over $\cM_{\rm A}$.  
We use the standard opposite module $\sfP_{\rm std}$ 
described in Example~\ref{ex:Amodel-opposite}. 
The associated standard trivialization of the A-model cTP structure $\sfF$ 
(given by the basis in equation~\ref{eq:basisproperties}) 
together with the linear co-ordinates $\{t^i\}$ on $H_X$  
gives an algebraic local co-ordinate system $\{t^i, x_n^i\}$ 
on the total space $\LL$ of $\sfF$. 
The standard trivialization also defines subspaces 
$L^2(\LL) \subset \cN(\LL)\subset \LL$ 
as in Remarks~\ref{rem:L2-neighbourhood},~\ref{rem:nuclear-neighbourhood}. 
Let $M(t,z)$ be the inverse fundamental solution  
\eqref{eq:inversefundamentalsolution} in Gromov--Witten theory.  
This is analytic on $\cM_{\rm A} \times \C^\times$ 
after specialization of Novikov variables $Q_1 = \cdots = Q_r = 1$. 
The flat co-ordinate system $\{q_n^i \}$ on $\cN(\LL)$ 
is given by the formula (cf.~equation~\ref{eq:formal_flat_coordinates})
\begin{equation} 
\label{eq:flatcoord-GW}
\bq + \bp = M(t,z) \bx \Bigr|_{Q_1=\cdots =Q_r=1} 
\end{equation} 
where:
\begin{align*}
  \bq = \sum_{n=0}^\infty q_n^i \phi_i z^n 
  &&
  \bp = \sum_{n=0}^\infty p_{n,i}\phi^i (-z)^{-n-1}
  &&
  \bx = \sum_{n=1}^\infty \sum_{i=0}^N 
  x_n^i \phi_i z^n
\end{align*}
By~\cite[Lemmas~8.5,~8.8]{CI:convergence}, 
we know that $(\bq,\bp)$ here belongs to a 
\emph{nuclear version $\cH^{\rm NF}$ of the Givental space} for $X$~\cite[Definition 8.7]{CI:convergence}:
\begin{align} 
\label{eq:nuclear-Giventalsp} 
\cH^{\rm NF} := H_X \otimes \C\{\!\{z,z^{-1}\}\!\} = \cH^{\rm NF}_+ \oplus \cH^{\rm NF}_-  
&& \text{where} &&
\begin{aligned}[t]
\cH^{\rm NF}_+ &:= H_X\otimes \C\{\!\{z\}\!\} \\
\cH^{\rm NF}_- &:= H_X\otimes z^{-1}\C\{\!\{z^{-1}\}\!\}  
\end{aligned}
\end{align}
whenever $(t,\bx)\in \cN(\LL)$. 
Then the map $(t,\bx) \mapsto \bq$ defines 
a local isomorphism between $\cN(\LLo)$ and 
$\cH_+^{\rm NF}$~\cite[\S 8.5]{CI:convergence}. 
The genus-zero potential is defined by
\begin{equation} 
\label{eq:genuszeropot-GW} 
C^{(0)} = \frac{1}{2} 
\sum_{i=0}^N \sum_{n=0}^\infty p_{n,i} q_n^i
\end{equation} 
 (cf.~equation~\ref{eq:genuszeropot}).
This is a holomorphic function on $\cN(\LL)$. 
In this setting, we have: 
\begin{itemize} 
\item The genus-zero descendant potential $\cF^0_X$ 
is NF-convergent~\cite[Theorem 7.8]{CI:convergence}, that is, 
the power series \eqref{eq:genus_g_descendantpot} 
converges absolutely and uniformly on a polydisc of the form 
$|t_l^i| < \epsilon (l!) /C^l$, $|Q_i| < \epsilon$ 
for some $\epsilon>0$ and $C>0$. 

\item As $\cF^0_X$ is NF-convergent, 
the specialization $\cF^0_{X,\rm an}$ of  
$\cF^0_X$ to $Q_1= \cdots =Q_r=1$ makes sense 
as a holomorphic function on a domain $U \subset 
\cH_+^{\rm NF}$~\cite[\S 8.1]{CI:convergence} 
via the Dilaton shift from \S\ref{subsec:dilatonshift}
(see Definition~\ref{def:F^g_Xan} below). 

\item When $t$ is sufficiently close to the large radius limit 
\eqref{eq:LRLnbhd} and $\bx\in z\cH_+^{\rm NF}$ 
is sufficiently close to $-z \unit$, the flat co-ordinate 
$\bq = [M(t,z) \bx]_+|_{Q_1= \cdots =Q_r=1}$ 
of the point $(t,\bx) \in \cN(\LL)$ 
belongs to $U$. Then we have 
$C^{(0)} = \cF^0_{X,\rm an}(\bq)$~\cite[Theorem 8.12]{CI:convergence}. 
\end{itemize} 
Although the normalization 
for the inverse fundamental solution $M(t,z)$ 
in Gromov--Witten theory 
is different from the one that we used in the general construction 
(see Remark~\ref{rem:difference-normalization}), 
the same argument as in this section (\S\ref{subsec:flatstronL}) 
proves that the co-ordinates $q_n^\alpha$ on $\cN(\LLo)$ 
defined by \eqref{eq:flatcoord-GW} 
are flat with respect to $\Nabla^{\sfP_{\rm std}}$, 
and that the third derivative of $C^{(0)}$ in \eqref{eq:genuszeropot-GW}
with respect to $\Nabla^{\sfP_{\rm std}}$ 
coincides with the Yukawa coupling over $\cN(\LLo)$. 
In particular we have: 
\begin{equation} 
\label{eq:genuszerojet-flat} 
\Nabla^{n-3} \bY 
= \sum_{l_1=0}^\infty \cdots \sum_{l_n=0}^\infty 
\sum_{i_1=0}^N \cdots \sum_{i_n=0}^N
\parfrac{^{n} \cF^{(0)}_{X, \rm an}}
{q_{l_1}^{i_1} \cdots \partial q_{l_{n}}^{i_n}} 
dq^{i_1}_{l_1} \otimes \cdots \otimes dq^{i_n}_{l_n}   
\end{equation} 
with $\Nabla = \Nabla^{\sfP_{\rm std}}$. 
\end{example} 

\subsection{Propagator}
Given two pseudo-opposite modules 
$\sfP_1$, $\sfP_2$ for a cTP structure $\sfF$, 
we now define a bivector field on the space $\LLo$, called the propagator $\Delta$.
Let $\Pi_i \colon \sfF[z^{-1}] \to \sfF$, $i\in \{1,2\}$,  
be the projection along $\sfP_i$ 
given by the decomposition 
$\sfF[z^{-1}] = \sfP_i \oplus \sfF$.
\begin{definition} 
\label{def:propagator}
The \emph{propagator} 
$\Delta=\Delta(\sfP_1,\sfP_2)$ 
associated to pseudo-opposite modules 
$\sfP_1$, $\sfP_2$ 
is the section of 
$\sHom_{\bcO}(\bOmegao^1\otimes \bOmegao^1,\bcO)$  
defined by 
\begin{align*}
  \Delta(\omega_1, \omega_2) := 
  \Omega^\vee(\Pi_1^* \KS^{*-1}\omega_1, 
  \Pi_2^* \KS^{*-1} \omega_2), 
  && \omega_1,\omega_2 \in \bOmega^1.  
\end{align*}
Here
$\KS^*\colon \pr^*\sfF^\vee \to 
\bOmega^1$ is the dual Kodaira--Spencer map 
(Definition~\ref{def:KS}) 
and $\Omega^\vee\colon \sfF[z^{-1}]^\vee 
\otimes \sfF[z^{-1}]^\vee \to \cO_\cM$ is
the dual symplectic form (\ref{eq:dualsymp}). 
One can identify $\Delta$ with the push-forward of 
the Poisson bivector on $\sfF[z^{-1}]$ along $\Pi_1 \otimes \Pi_2$. 
\end{definition}  

\begin{proposition} 
\label{pro:prop-elementary} 
Let $\Delta= \Delta(\sfP_1,\sfP_2)$ be the propagator 
associated to pseudo-opposite modules $\sfP_1$, $\sfP_2$. 
\begin{enumerate}
\item The propagator $\Delta$ is symmetric:
$\Delta(\omega_1, \omega_2) = \Delta(\omega_2,\omega_1)$.   
\item If $\sfP_1, \sfP_2$ are parallel, 
then:
\[
d \Delta(\omega_1,\omega_2) = \Delta(\Nabla^{\sfP_1} \omega_1, \omega_2) 
+ \Delta(\omega_1, \Nabla^{\sfP_2} \omega_2)
\]
\end{enumerate}
(See Proposition~\ref{prop:propagator-curved} below for the non-parallel case.) 
\end{proposition} 
\begin{proof} 
Write $\varphi_i := \KS^{*-1} \omega_i \in \pr^*\sfF^\vee$ 
for $i=1,2$. 
Because  $\Image \Pi_i^* = \sfP_i^\perp$ and 
$\Image(\Pi_1^* - \Pi_2^*)\subset \sfF^\perp$, 
these subspaces are isotropic with respect to $\Omega^\vee$. 
Hence we have: 
\begin{align*}
0 & = \Omega^\vee((\Pi_1^* - \Pi_2^*) \varphi_1, 
(\Pi_1^* - \Pi_2^*) \varphi_2)  
= - \Omega^\vee(\Pi_1^* \varphi_1, \Pi_2^* \varphi_2) 
- \Omega^\vee(\Pi_2^* \varphi_1, \Pi_1^* \varphi_2) \\ 
& = -\Omega^\vee(\Pi_1^* \varphi_1, \Pi_2^* \varphi_2) 
+ \Omega^\vee(\Pi_1^* \varphi_2, \Pi_2^* \varphi_1)
\end{align*} 
This shows that $\Delta$ is symmetric.
For Part (2), we have:
\begin{align}
\label{eq:dDelta}
\begin{split}  
d \Delta(\omega_1,\omega_2)&= 
d \Omega^\vee( \Pi_1^* \varphi_1, \Pi_2^* \varphi_2) \\ 
& = \Omega^\vee( \tnabla^\vee \Pi_1^* \varphi_1, \Pi_2^* \varphi_2) 
+ \Omega^\vee( \Pi_1^* \varphi_1, \tnabla^\vee \Pi_2^* \varphi_2)  \\ 
& = \Omega^\vee( \tnabla^\vee \Pi_1^* \varphi_1, 
(\Pi_2^* - \Pi_1^*) \varphi_2) 
+ \Omega^\vee( ( \Pi_1^* - \Pi_2^* )\varphi_1, 
\tnabla^\vee \Pi_2^* \varphi_2) \\ 
& \quad + \Omega^\vee(\tnabla^\vee \Pi_1^* \varphi_1, \Pi_1^*\varphi_2) 
+ \Omega^\vee(\Pi_2^* \varphi_1, \tnabla^\vee \Pi_2^* \varphi_2)  
\end{split} 
\end{align} 
Note that $\Image \Pi_i^* = \sfP_i^\perp$ is preserved by $\tnabla^\vee$ 
because $\sfP_i$ is parallel. Therefore the two terms in the last line 
vanish. Because  both 
$\Pi_1^* (\tnabla^\vee \Pi_1^* \varphi_1|_{\sfF}) - 
\tnabla^\vee \Pi_1^* \varphi_1$ and $(\Pi_1^*-\Pi_2^*) \varphi_2$ 
lie in $\sfF^\perp$, we have:
\begin{align*} 
\Omega^\vee( \tnabla^\vee \Pi_1^* \varphi_1, 
(\Pi_2^* - \Pi_1^*) \varphi_2) & = 
\Omega^\vee(\Pi_1^* (\tnabla^\vee \Pi_1^* \varphi_1|_{\sfF}), 
(\Pi_2^* - \Pi_1^*)\varphi_2 ) \\
& = \Omega^\vee( \Pi_1^* \KS^{*-1} 
\Nabla^{\sfP_1} \omega_1, (\Pi_2^* - \Pi_1^*) \varphi_2)  \\ 
& = \Omega^\vee( \Pi_1^* \KS^{*-1} 
\Nabla^{\sfP_1} \omega_1, \Pi_2^*  \varphi_2)  
= \Delta(\Nabla^{\sfP_1} \omega_1, \omega_2)
\end{align*} 
Similarly we have 
$\Omega^\vee( ( \Pi_1^* - \Pi_2^* )\varphi_1, \tnabla^\vee \Pi_2^* \varphi_2) 
= \Delta(\omega_1, \Nabla^{\sfP_2} \omega_2)$. 
The conclusion follows.   
\end{proof} 

We introduce tensor notation. 
Let $\{\sx^\mu\}$ denote an arbitrary local 
co-ordinate system on $\LL$ (or on the formal neighbourhood 
$\hLL$ of $\LL_t$). 
For example, this could be an algebraic local co-ordinate system 
$\{t^i, x_n^i\}$ (\S\ref{subsec:totalspace}) associated to 
a local trivialization of $\sfF$, or a flat co-ordinate system 
(\S\ref{subsec:flatstronL}) on $\hLL$ associated 
to a parallel pseudo-opposite module. 
In this co-ordinate system, 
we write the Yukawa coupling and the propagator as 
\begin{align*} 
\bY = C^{(0)}_{\mu\nu\rho} d\sx^\mu \otimes 
d\sx^\nu \otimes d \sx^\rho,  \quad 
\Delta = \Delta^{\mu \nu} \partial_\mu \otimes \partial_\nu, 
\quad 
\text{where } \partial_\mu = \parfrac{}{\sx^\mu}, \ 
\partial_\nu = \parfrac{}{\sx^\nu}.  
\end{align*} 
Here we adopt Einstein's summation convention for repeated indices.  
The Christoffel symbol of the connection 
$\Nabla = \Nabla^\sfP$ on $\LLo$ (for a pseudo-opposite $\sfP$)  
is defined by 
\begin{align} 
\label{eq:Christoffel}
\Nabla_\nu d \sx^\mu = - \Gamma^\mu_{\nu\rho}  d \sx^\rho && 
\Nabla_\nu \partial_\rho = \Gamma^\mu_{\nu\rho}  \partial_\mu  
\end{align} 
where $\Nabla_\nu = \Nabla_{\partial/\partial \sx^\nu}$. 
Note that  $\Gamma^\mu_{\nu\rho} 
= \Gamma^\mu_{\rho\nu}$ because 
$\Nabla$ is torsion free; 
also $\Delta^{\mu\nu} = \Delta^{\nu\mu}$ 
by the previous Proposition. 
The propagator has the following key properties.
\begin{proposition}
\label{prop:difference_conn} 
Let $\sfP_i$ be pseudo-opposite modules 
and ${\Gamma^{(i)}}^\mu_{\nu\rho}$ denote the Christoffel 
symbols of $\Nabla^{\sfP_i}$, $i=1,2$. 
Let $\Delta = \Delta(\sfP_1,\sfP_2)$ be the associated propagator. Then:
\begin{enumerate}
\item $(\Nabla^{\sfP_2} - \Nabla^{\sfP_1}) \omega = 
  \iota(\iota_\omega \Delta) \bY$ for $\omega\in\bOmegao^1$. 
  In tensor notation:
  \[
  (\Nabla^{\sfP_2}_\mu - \Nabla^{\sfP_1}_\mu )d\sx^\nu 
  = ({\Gamma^{(1)}}^\nu_{\mu\rho} - 
  {\Gamma^{(2)}}^\nu_{\mu\rho}) d\sx^\rho
  = \Delta^{\nu \sigma} C^{(0)}_{\sigma \mu\rho} d\sx^\rho
  \]
\item[(2)] If $\sfP_1,\sfP_2$ are parallel, we have 
  $(\Nabla^{\sfP_1}\Delta)(\omega_1\otimes \omega_2) 
  = \iota(\iota_{\omega_1}\Delta\otimes 
  \iota_{\omega_2}\Delta)\bY$ for $\omega_1,\omega_2 \in \bOmegao^1$, that is:
  \[
  \Nabla^{\sfP_1}_\mu \Delta^{\nu\rho} (:= 
  \partial_\mu \Delta^{\nu\rho} +  
  {\Gamma^{(1)}}^\nu_{\mu\sigma} \Delta^{\sigma\rho} 
  + {\Gamma^{(1)}}^\rho_{\mu\sigma} \Delta^{\nu\sigma})  
  = \Delta^{\nu\sigma} C^{(0)}_{\sigma \mu \tau} \Delta^{\tau\rho}
  \]
\end{enumerate}
(See Proposition~\ref{prop:propagator-curved} below for the non-parallel case.) 
\end{proposition} 

\begin{proof}
Set $\varphi = \KS^{*-1}\omega \in \pr^*\sfF^\vee$. 
Note that $\iota_\omega \Delta$ is a section of $\bThetao$. 
For $\beta \in \pr^* \sfF[z^{-1}]$, we have:
\begin{align*} 
\Omega(\KS (\iota_\omega \Delta), \beta) 
& = \Pair{\KS(\iota_\omega \Delta)}{- \iota_\beta\Omega}  
= \Pair{\KS(\iota_\omega \Delta)}{-(\iota_\beta \Omega)|_{\pr^*\sfF}} \\
& = - \Pair{\iota_\omega \Delta}
{\KS^* ((\iota_\beta \Omega)|_{\pr^*\sfF})} \\ 
& = - \Delta(\omega,\KS^* ((\iota_\beta \Omega)|_{\pr^*\sfF})) 
 \\
& = - \Omega^\vee(\Pi_1^* \varphi, 
\Pi_2^*((\iota_\beta \Omega)|_{\pr^*\sfF})) 
&&  \!\!\!\!\!\!\!\!\!\!\!\!\!\!\!\!\!\!\!\!\!\!\!\!\!\!\!\!\!
\text{(by the definition of $\Delta$)} 
\\
& = - \Omega^\vee((\Pi_1^* - \Pi_2^*) \varphi, 
\Pi_2^* ((\iota_\beta \Omega)|_{\pr^*\sfF}))) 
&&  \!\!\!\!\!\!\!\!\!\!\!\!\!\!\!\!\!\!\!\!\!\!\!\!\!\!\!\!\! 
\text{(since $\Image \Pi_2^*=(\pr^*\sfP_2)^\perp$ is isotropic)} 
\\
& = - \Omega^\vee((\Pi_1^* -\Pi_2^*)\varphi, 
\iota_\beta\Omega)
\end{align*} 
The last line follows from the fact that 
both $(\Pi_1^*-\Pi_2^*)\varphi$ and 
$\iota_\beta\Omega- \Pi_2^* ((\iota_\beta \Omega)|_{\pr^*\sfF}))$ 
lie in the isotropic subspace $(\pr^*\sfF)^\perp$. 
Thus:
\begin{equation} 
\label{eq:iomegaDelta} 
\Omega(\KS (\iota_\omega \Delta), \beta) 
= \Pair{(\Pi_2^* -\Pi_1^*)\varphi}{\beta}. 
\end{equation} 
For $X, Y\in \bThetao$ and the tautological section $\bx$ 
of $\pr^*\sfF$, we have:
\begin{align*} 
\Pair{(\Nabla^{\sfP_2}-\Nabla^{\sfP_1}) \omega}{X\otimes Y}
&= \Pair{(\id \otimes \KS^*) \tnabla^\vee 
(\Pi_2^* -\Pi_1^*) \varphi}{X\otimes Y} \\ 
&= \Pair{\tnabla^\vee(\Pi_2^*-\Pi_1^*) \varphi}
{X\otimes \tnabla_Y \bx} \\
& = \Pair{\tnabla^\vee_X(\Pi_2^* - \Pi_1^*) \varphi}{\tnabla_{Y}\bx}\\
&=X\Pair{(\Pi_2^*-\Pi_1^*)\varphi}{\tnabla_Y \bx} 
- \Pair{(\Pi_2^*-\Pi_1^*)\varphi}{\tnabla_X\tnabla_Y \bx} 
\end{align*}
Because $(\Pi_2^*-\Pi_1^*)\varphi$ vanishes on $\pr^*\sfF$, 
the first term vanishes. 
By \eqref{eq:iomegaDelta}, we now have:
\[
\Pair{(\Nabla^{\sfP_2}-\Nabla^{\sfP_1}) \omega}{X\otimes Y}
 = - \Omega(\KS(\iota_\omega\Delta), \tnabla_X \tnabla_Y \bx)
 = \bY(\iota_\omega\Delta, X, Y)
\]
This proves Part (1). 
For Part (2), using Proposition~\ref{pro:prop-elementary}(2), 
we have:
\begin{align*} 
d\Delta(\omega_1,\omega_2) - \Delta(\Nabla^{\sfP_1} \omega_1,\omega_2) 
- \Delta(\omega_1,\Nabla^{\sfP_1} \omega_2) 
& = \Delta(\omega_1, (\Nabla^{\sfP_2} -\Nabla^{\sfP_1}) \omega_2) 
\end{align*}
This equals $\iota(\iota_{\omega_1} \Delta \otimes 
\iota_{\omega_2}\Delta) \bY$ 
by Part (1).  
\end{proof} 
\begin{proposition} 
\label{prop:Deltasum} 
Let $\sfP_1,\sfP_2,\sfP_3$ be pseudo-opposite modules 
and let 
$\Delta_{ij}=\Delta(\sfP_i, \sfP_j)$ denote the corresponding propagators. 
We have:
\[
\Delta_{13} = \Delta_{12} + \Delta_{23} 
\]
In particular, $\Delta(\sfP_1,\sfP_2) = - \Delta(\sfP_2,\sfP_1)$. 
\end{proposition} 
\begin{proof} 
Putting $\varphi_i = \KS^{*-1}\omega_i
\in \pr^*\sfF^\vee$, we have:
\begin{align*}
\Delta_{13}(\omega_1,\omega_2) 
& = \Omega^\vee((\Pi_1^*-\Pi_3^*)\varphi_1,\Pi_3^*\varphi_2) \\
&= \Omega^\vee((\Pi_1^*-\Pi_2^*)\varphi_1,\Pi_3^*\varphi_2) 
+ \Omega^\vee((\Pi_2^*-\Pi_3^*)\varphi_1,\Pi_3^*\varphi_2) \\
&= \Omega^\vee((\Pi_1^*-\Pi_2^*)\varphi_1, \Pi_2^*\varphi_2) 
+ \Omega^\vee(\Pi_2^*\varphi_1,\Pi_3^*\varphi_2) \\ 
&= \Delta_{12}(\omega_1,\omega_2) + \Delta_{23}(\omega_1,\omega_2)
\end{align*} 
We used the fact that $\Image \Pi_i^* = \sfP_i^\perp$ is isotropic 
and that $\Image (\Pi_i^* - \Pi_j^*)$ is contained in the 
isotropic subspace $\sfF^\perp$. 
The last statement follows from the case $\sfP_1=\sfP_3$. 
\end{proof} 

\subsubsection{Givental's Propagator} 
\label{subsubsec:Giventalpropagator}
Suppose that we have two opposite modules $\sfP_1,\sfP_2$ over $U$
and that we have the corresponding 
two trivializations 
\[
\Phi_i \colon 
\C^{N+1} \otimes \cO_U[\![z]\!] \to \sfF|_U, \quad i=1,2 
\]
such that:
\begin{itemize} 
\item $\sfP_i = \Phi_i(\C^{N+1} \otimes z^{-1}\cO_U[z^{-1}] )$, $i=1,2$.  
\item The values $g_{ij}= (\Phi_1(e_i), \Phi_1 (e_j))_{\sfF}$ and 
$\tg_{ij} = (\Phi_2(e_i), \Phi_2(e_j))_{\sfF}$ are constant.
\end{itemize} 
Here $e_0,\dots,e_N$ are the standard basis of $\C^{N+1}$. 
Such a trivialization arises from the flat trivialization 
(see Proposition~\ref{prop:flat_trivialization}) 
associated to $\sfP_i$ and a $\nabla^0$-flat frame of $z\sfP_i/\sfP_i$. 
Let $R(z) =\Phi_2^{-1} \circ \Phi_1 = R_0 + R_1 z + R_2 z^2 + 
\cdots \in GL(N+1,\cO[\![z]\!])$ 
denote the gauge transformation between the two trivializations: 
\[
R(z) \colon 
\C^{N+1} \otimes \cO_U[\![z]\!] \overset{\Phi_1}{\longrightarrow} 
\sfF \overset{\Phi_2^{-1}}{\longrightarrow} 
\C^{N+1} \otimes \cO_U[\![z]\!]
\]
Let $g$,~$\tg\colon \C^{N+1} \otimes \C^{N+1} \to \C$ denote 
the pairings with the Gram matrices $(g_{ij})$, $(\tg_{ij})$. 
Then the gauge transformation intertwines these pairings:  
\begin{equation} 
\label{eq:gaugetr-unitarity} 
\tg(R(-z) v, R(z) w) = g(v, w), \quad v,w \in \C^{N+1}
\end{equation} 
\begin{definition}[\cite{Givental:quantization}] 
\label{def:Giventalprop}
\emph{Givental's propagator} is a collection of 
elements $V^{(n,j),(m,i)} \in \cO_U$, 
$0 \leq n,m < \infty$, $0\le i,j\le N$ defined by the formula 
\begin{equation} 
\label{eq:Giventalprop}
\sum_{n=0}^\infty \sum_{m=0}^\infty (-1)^{n+m} V^{(n,j), (m,i)} w^n z^m 
= g\left(e^j, \frac{R(w)^\dagger R(z) -\id}{z+w} e^i\right)  
\end{equation} 
where $R(w)^\dagger=R(-w)^{-1}$ denotes the adjoint of $R(w)$ 
with respect to $g$ and $\tg$ (see equation~\ref{eq:gaugetr-unitarity}) 
and $e^i = \sum_j g^{ij} e_j$ 
with $(g^{ij})$ the matrix inverse to 
$(g_{ij})$.  
\end{definition} 
Let $\varphi_m^i$ 
be the frame of $\pr^*\sfF[z^{-1}]^\vee$ defined by 
the trivialization $\Phi_1$ 
(cf.~equation~\ref{eq:frame_varphi}): 
\begin{align*} 
\varphi_m^i 
\colon \pr^*\sfF[z^{-1}] \to \bcO
&&
\varphi_m^i(s) = [\Phi_1^{-1} s]_m^i
\end{align*} 
where $\Phi_1 \colon \C^{N+1}\otimes \bcO(\!(z)\!) 
\cong \pr^*\sfF[z^{-1}]$  
and we followed Notation~\ref{nota:[]_^}. 
\begin{lemma} 
\label{lem:V}
$V^{(n,j),(m,i)} = 
- \left[R(z)^{-1} [R(z) (-z)^{-n-1} e^j  ]_+
\right]^i_m  
= \Omega^\vee(\Pi_1^* \varphi_n^j, \Pi_2^* \varphi_m^i)$, 
where $[\cdots]_+$ denotes the non-negative part as a $z$-series. 
\end{lemma} 
\begin{proof} 
The first equality follows from the calculation: 
\begin{align*} 
\text{equation \eqref{eq:Giventalprop}}
& = g\left( \frac{R(z)^\dagger R(w) - \id}{z+w} e^j, e^i\right) 
= g\left( R(-z)^{-1} 
\left ( \frac{R(w) - R(-z)}{z+w} e^j \right),  e^i\right) \\ 
& = - g\left( R(-z)^{-1} \left [
R(-z) \frac{e^j}{z+w} \right ]_+, e^i\right) 
\quad \text{when $|w|<|z|$} 
\\ 
& = -\sum_{n=0}^\infty 
(-1)^{n} 
g\left(R(-z)^{-1} \left[R(-z) z^{-n-1}e^j\right]_+, 
e^i\right) w^n \\ 
& = - \sum_{n=0}^\infty \sum_{m=0}^\infty (-1)^{n+m} 
\left[R(z)^{-1} \left[R(z) (-z)^{-n-1} e^j \right]_+
\right]^i_m  w^n z^m   
\end{align*} 
In the second line, we expanded $e^j/(z+w)$ in power series in $z^{-1}$ 
(i.e.~around $z=\infty$). 
Under the trivialization $\Phi_1$, 
the projection $\Pi_2$ can be presented as 
\[
\Pi_2(e_h z^n) = R(z)^{-1} [R(z) e_h z^n]_+, \quad n\in \Z.  
\]
Therefore, for $m\ge 0$:
\begin{align*} 
\Pi_2^* \varphi_m^i  
&= \sum_{n\in \Z} \sum_h \left[
R(z)^{-1} \left[R(z) e_h z^n\right]_+\right]_m^i \; \varphi_n^h \\
&= \varphi_m^i  + \sum_{n=0}^\infty \sum_h 
\left[ R(z)^{-1} \left[R(z) e_h z^{-n-1}\right]_+ \right]_m^i 
\; \varphi_{-n-1}^h
\end{align*} 
Consequently, under the isomorphism 
$\sfF[z^{-1}]\cong \sfF[z^{-1}]^\vee$, 
$v\mapsto \iota_v \Omega$, 
the section $\Pi_2^* \varphi_m^i\in 
\sfF[z^{-1}]^\vee$ corresponds to:
\[
v_m^i =  e^i (-z)^{-m-1} + \sum_{n=0}^\infty \sum_h 
\left[ R(z)^{-1} [R(z) e_h z^{-n-1}]_+ \right]_m^i 
\; (-z)^n e^h
\]
Hence we have 
$\Omega^\vee(\Pi_1^* \varphi_n^j, \Pi_2^* \varphi_m^i) 
= \pair{\Pi_1^* \varphi_n^j}{v_m^i} 
= \pair{\varphi_n^j}{[v_m^i]_+} = V^{(n,j),(m,i)}$. 
\end{proof} 

\begin{proposition} 
\label{prop:Giventalprop=prop} 
For $t\in \cM$, let $\{q_n^i\}_{n\ge 0, 0\le i\le N}$ 
be the flat co-ordinate system 
(Definition~\ref{def:flatcoordinate_on_LL}) 
on the formal neighbourhood $\hLLo$ of $\LLo_t$ associated 
to the trivialization $\Phi_1$ and 
the opposite module $\sfP_1$. 
The propagator $\Delta = \Delta(\sfP_1,\sfP_2)$ 
restricted to the fiber $\LLo_t$ 
can be written in terms of the flat co-ordinates as 
\[
\Delta\Bigr|_{\LLo_t} = \sum_{n=0}^\infty \sum_{m=0}^\infty \sum_{i=0}^N \sum_{j=0}^N
V^{(n,j),(m,i)} \parfrac{}{q_n^j} \otimes \parfrac{}{q_m^i}   
\]
where $V^{(n,j),(m,i)}$ is Givental's propagator in 
Definition~\ref{def:Giventalprop}. 
\end{proposition} 
\begin{proof}
Restricting \eqref{eq:KSinv_dq+dp} to the fiber $\LL_t$ 
(i.e.\  $s=0$), we have:
\begin{equation} 
\label{eq:flatcoord-atthecentre}
\KS^{*-1}d q_n^j = \varphi_n^j  \quad 
\text{over $\LL_t$}
\end{equation} 
Hence $\Delta(dq_n^j, dq_m^i)|_{\LL_t} = 
\Omega^\vee(\Pi_1^* \varphi_n^j, \Pi_2^* \varphi_m^i) 
= V^{(n,j),(m,i)}$ by Lemma~\ref{lem:V}.  
\end{proof} 

\begin{remark} 
\label{rem:propagator-explicit}
In terms of the algebraic co-ordinates $(t^i, x_n^i)_{n\ge 1, 0\le i\le N}$ 
on $\LL$ associated to the trivialization $\Phi_1$, 
the propagator $\Delta = \Delta(\sfP_1,\sfP_2)$ 
can be written as 
\begin{align*} 
\Delta(d t^a \otimes d t^b) 
& = [K(x_1)^{-1} e_i]^a [K(x_1)^{-1} e_j]^b 
V^{(0,i),(0,j)}\\
\Delta(d t^a \otimes d x_n^b) 
&= -[K(x_1)^{-1}e_i]^a V^{(0,i),(n,b)} 
+ [K(x_1)^{-1} e_i]^a [K(x_{n+1}) K(x_1)^{-1}e_i]^b 
V^{(0,i),(0,j)} \\
\Delta( d x^a_m \otimes d x^b_n) 
&= V^{(m,a),(n,b)} 
- [K(x_{m+1})K(x_1)^{-1}e_i]^a V^{(0,i),(n,b)} \\
& \qquad 
- [K(x_{n+1})K(x_1)^{-1}e_j]^b V^{(m,a),(0,j)} \\ 
& \qquad + [K(x_{m+1})K(x_1)^{-1}e_i]^a 
[K(x_{n+1})K(x_1)^{-1} e_j]^b 
V^{(0,i),(0,j)} 
\end{align*} 
where $K(x_n)$ is as in Example~\ref{exa:Nabla-explicit}.  
\end{remark}

\subsubsection{Difference One-Form}
\begin{definition}
For two pseudo-opposite modules $\sfP$ and $\sfQ$, 
we define a one-form on $\LLo$ by 
\begin{equation} 
\label{eq:difference1-form} 
\omega_{\sfP\sfQ} = \frac{1}{2} 
\sum_{\mu,\nu,\rho} 
C^{(0)}_{\mu\nu\rho} \Delta^{\nu\rho}(\sfP,\sfQ) 
d\sx^\mu 
= \frac{1}{2} 
\sum_{0\le i,j,h \le N} C^{(0)}_{i j h} 
\Delta^{j h}(\sfP,\sfQ) dt^i  
\end{equation} 
where in the second expression the indices $i,j,h$ are 
labels of the $t$-variables 
of an algebraic local co-ordinate system 
$\{t^i,x_n^i\}_{n\ge 1, 0\le i\le N}$. 
We call $\omega_{\sfP\sfQ}$ the \emph{difference one-form}, because it appears as the difference of genus-one one-point functions 
\eqref{eq:Feynman-genusone}. 
By Proposition~\ref{prop:Deltasum}, we have 
$\omega_{\sfP\sfQ} + \omega_{\sfQ \sfR} 
= \omega_{\sfP\sfR}$ for any three pseudo-opposite 
modules $\sfP,\sfQ,\sfR$. 
\end{definition} 

\begin{lemma} 
\label{lem:difference1-form} 
The difference one-form $\omega_{\sfP\sfQ}$ 
is pulled-back from the base $\cM$;  
we have:
\[
\omega_{\sfP\sfQ} = \sum_{i=0}^N 
\frac{1}{2} 
\Tr_{\sfF_0}\big((\Pi_{\sfP} - \Pi_{\sfQ})\nabla_i\big) dt^i
\]
\end{lemma} 
\begin{proof} 
One can easily check that the operator 
$(\Pi_{\sfP} - \Pi_{\sfQ}) \nabla_i$ 
defines an $\cO_{\cM}$-linear endomorphism 
of $\sfF_0=\sfF/z\sfF$. 
By Proposition~\ref{prop:difference_conn}(1), we have:
\[
(\Nabla^{\sfQ}_i - \Nabla^{\sfP}_i) 
(\partial/\partial t^j)
= -C^{(0)}_{ijh} \Delta^{h \mu}(\sfP,\sfQ) 
(\partial/\partial \sx^\mu)
\]
Here $\sx^\mu$ can be either $t^i$ or $x_n^i$. 
The minus sign here is because we are working with connections
on the tangent bundle $\bTheta$. 
On the other hand, by the definition of $\Nabla^{\sfP}$ and 
$\Nabla^{\sfQ}$, we have 
$\Nabla^{\sfQ}_i - \Nabla^{\sfP}_i = 
\KS^{-1}(\Pi_{\sfQ} - \Pi_{\sfP}) \tnabla_i \KS$. 
Hence $\Nabla^{\sfQ}_i - \Nabla^{\sfP}_i$ 
induces a map  $\bTheta/\KS^{-1}( \pr^*(z\sfF)) 
\to \bTheta/\KS^{-1} (\pr^*(z\sfF))$  
which is conjugate to 
$(\Pi_{\sfQ} - \Pi_{\sfP})\tnabla_i \in \End(\sfF_0)$. 
Because $\{\partial/\partial t^i\}$ is a basis of 
$\bTheta/\KS^{-1}(\pr^*(z\sfF)) \cong \sfF_0$, 
the conclusion follows. 
\end{proof} 

\subsection{Grading and Filtration}
Recall that we introduced a grading and an increasing 
filtration on $\bcO$ and $\bOmega^1$ in \S\ref{subsec:totalspace}. 
The grading and the filtration 
on $\pr^*\sfF[z^{-1}]^\vee$ are defined as follows. 
For a pull-back $\pr^*\varphi 
\in \pr^{-1}\sfF[z^{-1}]^\vee$ of 
$\varphi \in \sfF[z^{-1}]^\vee$, 
we set $\deg (\pr^* \varphi) = 0$.
The grading on $\pr^*\sfF[z^{-1}]^\vee$ 
is determined by this and the grading on $\bcO$. 
To define the filtration, recall that $\sfF[z^{-1}]^\vee$ is 
the projective limit of the sequence: 
\[
\cdots \surj (z^{-2} \sfF)^\vee \surj (z^{-1}\sfF)^\vee 
\surj \sfF^\vee \surj (z \sfF)^\vee  \surj  \cdots
\]
Let $\sfF[z^{-1}]^\vee_n\subset 
\sfF[z^{-1}]^\vee$ be the kernel of 
$\sfF[z^{-1}]^\vee \surj (z^{n+2}\sfF)^\vee$. 
This defines an increasing filtration of 
$\sfF[z^{-1}]^\vee$ by subsheaves. 
The filtration on $\pr^*\sfF[z^{-1}]^\vee$ is 
induced from this and the filtration on $\bcO$: 
\[
(\pr^*\sfF[z^{-1}]^\vee)_n = \sum_{i+j \le n} 
\bcO_i \cdot \pr^{-1} (\sfF[z^{-1}]^\vee_j) 
\] 
Note that the filtration on $\bcO$ is bounded from below: 
$\{0\} = \bcO_{-1} \subset \bcO_0 \subset \bcO_1 \subset \cdots$ 
whereas the filtration on $\sfF[z^{-1}]^\vee$ is 
unbounded in both directions. 
The grading and the filtration on $\pr^*\sfF^\vee$ 
are induced from those on $\pr^*\sfF[z^{-1}]^\vee$ 
by the surjection $\pr^*\sfF[z^{-1}]^\vee 
\surj \pr^*\sfF^\vee$. 

Take a local trivialization 
$\sfF|_U \cong \C^{N+1} \otimes \cO_U[\![z]\!]$. 
This defines a local frame  
$\{e_i z^n\}_{n\in \Z, 0\le i\le N}$ 
for $\pr^* \sfF[z^{-1}]$ 
and the dual local frame 
$\{\varphi_n^i\}_{n\in \Z, 0\le i\le N}$ 
for $\pr^* \sfF[z^{-1}]^\vee$ 
(see equation~\ref{eq:frame_varphi}). 
The image of $\varphi^i_n$ under $\pr^*\sfF[z^{-1}]^\vee
\surj \pr^*\sfF^\vee$ is denoted by the same symbol. 
Note that we have 
\[
\deg \varphi_n^i = 0, \quad 
\filt\varphi_n^i = n -1. 
\]
Here as before $\filt(y)$ denotes the least number $m$ 
such that $y$ belongs to the $m$th filter. 

\begin{lemma} \ 
\label{lem:grading-filtration}
\begin{enumerate}
\item  The dual Kodaira--Spencer map $\KS^*\colon \pr^*\sfF^\vee \to \bOmega^1$ 
(respectively its inverse $\KS^{*-1}$) raises (respectively lowers) 
the degree by one and preserves the filtration. 
\item The connection $\tnabla^\vee \colon 
\pr^*\sfF[z^{-1}]^\vee \to \pr^* \sfF[z^{-1}]^\vee 
\otimes \bOmega^1$ preserves both the grading and 
the filtration. 
\item Let $\Pi \colon \pr^*\sfF[z^{-1}] \to \pr^* \sfF$ 
denote the projection along a pseudo-opposite module $\sfP$. 
The dual map $\Pi^* \colon \pr^*\sfF^\vee \to \pr^*\sfF[z^{-1}]^\vee$ 
preserves the grading and the filtration. 
\item The pairing $\Omega^\vee \colon \pr^*\sfF[z^{-1}]^\vee
\otimes \pr^*\sfF[z^{-1}]^\vee \to \bcO$ preserves the grading 
and raises the filtration by three. 
\end{enumerate}
\end{lemma} 
\begin{proof} 
Part (1) follows easily from \eqref{eq:KS_coord}.
Notice that $d \colon \bcO \to \bOmega^1$ 
preserves the degree and the filtration. 
Part (2) follows from this and \eqref{eq:tnablavee}. 
Part (3) is obvious from the definition. 
For Part (4), notice that $\Omega^\vee(\varphi_m^i, \varphi_n^j)$ 
is in $\cO_U$ and of degree zero. 
Hence $\Omega^\vee$ preserves the grading. 
Also, for $f,g\in \bcO$:
\[
\filt(\Omega^\vee(f \varphi_m^i, g \varphi_n^j)) 
\le \filt(f) + \filt(g) + (m+n+1) 
= \filt(f \varphi_m^i) + \filt(g \varphi_n^j) 
+ 3
\]
The first inequality follows from the fact 
that $\Omega^\vee(\varphi_m^i,\varphi_n^j)$ vanishes  
unless $m+n+1\ge 0$.  
\end{proof} 
\begin{proposition}
\label{prop:grading-filtration-Nabla} 
The connection $\Nabla \colon \bOmegao^1 \to 
\bOmegao^1 \otimes \bOmegao^1$ associated to a 
pseudo-opposite module $\sfP$ preserves the grading 
and the filtration. 
\end{proposition} 
\begin{proof}
This follows from $\Nabla \omega 
= \KS^* 
((\tnabla^\vee \Pi_\sfP^* \KS^{*-1}\omega)|_{\pr^*\sfF})$ 
and Lemma~\ref{lem:grading-filtration}. 
\end{proof} 
\begin{proposition} 
\label{prop:grading-filtration-prop}
Let $\sfP_1,\sfP_2$ be pseudo-opposite modules.  
The propagator 
\[
\Delta(\sfP_1,\sfP_2) \colon \bOmegao^1 \otimes \bOmegao^1 \to \bcO
\]
 lowers the 
degree by two and raises the filtration by two, that is,
$\deg \Delta = -2$, $\filt\Delta \le 2$. 
\end{proposition} 
\begin{proof}
Recall that the propagator $\Delta=\Delta(\sfP_1,\sfP_2)$ is defined 
by:
\[
\Delta(\omega_1,\omega_2) = 
\Omega^\vee( \Pi_{1}^* \KS^{*-1}\omega_1, 
\Pi_{2}^* \KS^{*-1} \omega_2)  
\]
By Lemma~\ref{lem:grading-filtration}, it follows that $\Delta$ lowers the 
degree by two and raises the filtration by three. 
One can improve the estimate on the filtration. 
Note that:
\[
\Delta(\omega_1,\omega_2) = 
\Omega^\vee( \Pi_{1}^* \KS^{*-1}\omega_1, 
(\Pi_{2}^*-\Pi_{1}^*) \KS^{*-1} \omega_2)
\]
We claim that $\Omega^\vee(\cdot, (\Pi_2^* -\Pi_1^*)\cdot) 
\colon \pr^* \sfF[z^{-1}]^\vee \otimes 
\pr^* \sfF^\vee \to \bcO$ 
raises the filtration only by two.  
Let $l\ge 0$. Because $(\Pi_2^*-\Pi_1^*)\varphi_l^j$ 
vanishes on $\pr^*\sfF$, one can write 
$(\Pi_2^* - \Pi_1^*)\varphi_l^j 
= \sum_{m\le -1} c_{m,a}(t) \varphi_m^a$. 
Thus $\Omega^\vee(\varphi_n^i, 
(\Pi_2^*-\Pi_1^*) \varphi_l^j) \in \cO_U$ 
can be non-zero only if $n\ge 0$,
so in particular only if $n+l \ge 0$. 
Therefore, for $f,g \in \bcO$:
\[
\filt\left( 
\Omega^\vee(f \varphi_n^i, 
(\Pi_2^*-\Pi_1^*) (g \varphi_l^j) ) \right) 
\le \filt(f) + \filt(g) + n+ l 
= \filt(f \varphi_n^i) + \filt(g \varphi_l^j)+2
\] 
The conclusion follows. 
\end{proof} 

\subsection{Local Fock Space} 
\label{subsec:Fockspace} 
We work with a local co-ordinate system $\{t^i, x_n^i\}$ 
on $\LL$ associated to a local trivialization of $\sfF$ 
(\S\ref{subsec:totalspace}). 
This local co-ordinate system is denoted also by $\{\sx^\mu\}$. 
We use the notation and summation convention 
which appeared in and above Proposition~\ref{prop:difference_conn}. 
For any $n$-tensor $C_{\mu_1,\dots,\mu_n} d\sx^{\mu_1} 
\otimes \cdots \otimes d\sx^{\mu_n} \in 
(\bOmegao^1)^{\otimes n}$, we write 
\[
\Nabla( C_{\mu_1,\dots,\mu_n} d\sx^{\mu_1}\otimes \cdots \otimes 
d\sx^{\mu_n}) 
= (\Nabla_\nu C_{\mu_1,\dots,\mu_n}) d\sx^{\nu} 
\otimes d \sx^{\mu_1} \otimes 
\cdots \otimes d \sx^{\mu_n} 
\] 
with 
\begin{equation} 
\label{eq:covariant-derivative}
\Nabla_\nu C_{\mu_1,\dots,\mu_n} := 
\partial_\nu C_{\mu_1,\dots,\mu_n}  
- \sum_{i=1}^n C_{\mu_1,\dots,\underset{i}{\rho},\dots,\mu_n} 
\Gamma^{\rho}_{\mu_i \nu} 
\end{equation} 
where $\Gamma^\rho_{\mu_i \nu}$ is the Christoffel symbol 
\eqref{eq:Christoffel} of $\Nabla$. 
Similarly, for a ``contravariant" tensor 
$C^{\mu_1,\dots,\mu_n} \partial_{\mu_1} 
\otimes \cdots \otimes \partial_{\mu_n} \in \bThetao^{\otimes n}$, 
we write 
$
\Nabla_\nu C^{\mu_1,\dots,\mu_n} = 
\partial_\nu C^{\mu_1,\dots,\mu_n}  + 
\sum_{i=1}^n C^{\mu_1,\dots,\overset{i}{\rho},\dots,\mu_n} 
\Gamma^{\mu_i}_{\nu\rho} 
$. 
\begin{definition}[local Fock space] 
\label{def:localFock}
Let $\sfP$ be a parallel pseudo-opposite module 
over an open set $U\subset \cM$. 
Let $\Nabla$ be the associated flat connection 
on the total space $\LLo$. 
Let $\{t^i,x_n^i\}_{n\ge 1, 0\le i\le N}$ be an algebraic
local co-ordinate system on $\pr^{-1}(U)$ and let
$P=P(t,x_1)$ denote the discriminant \eqref{eq:discriminant}. 
The \emph{Fock space $\Fock(U; \sfP)$ over $U$ 
associated with $\sfP$}    
consists of collections 
\[
\wave = 
\big\{\Nabla^n C^{(g)}\in (\bOmega^1)^{\otimes n}
(\pr^{-1}(U)^\circ) : \text{$g\ge 0$, $n\ge 0$, $2g-2+n> 0$}\big\}
\] 
of completely symmetric tensors  such that 
the following conditions are satisfied: 

\begin{description} 
\item[(Yukawa)]
$\Nabla^3 C^{(0)}$ is the Yukawa coupling $\bY$ 
(see \S\ref{subsec:Yukawa}); 

\item[(Jetness)] $\Nabla (\Nabla^n C^{(g)}) = \Nabla^{n+1} C^{(g)}$;

\item[(Grading \& Filtration)] 
$\Nabla^n C^{(g)} \in \big((\bOmega^1)^{\otimes n}
(\pr^{-1}(U)^\circ)\big)^{2-2g}_{3g-3}$; 

\item[(Pole)]
$P \Nabla C^{(1)}$ extends to a regular one-form on 
$\pr^{-1}(U)$, and
for $g\ge 2$:
\[
C^{(g)} \in P^{-(5g-5)} \cO(U)
[x_1,x_2,P x_3, P^2 x_4,\dots, P^{3g-4} x_{3g-2}]
\]
\end{description} 

Using local co-ordinates $\{\sx^\mu\}=\{t^i,x_n^i\}$, 
we write 
\[
\Nabla^n C^{(g)} = C^{(g)}_{\mu_1,\dots,\mu_n} 
d \sx^{\mu_1} \otimes  \cdots \otimes d \sx^{\mu_n} 
\]
where once again we use Einstein's summation convention 
for the indices $\mu_i$.  
We call $\Nabla^n C^{(g)}$ or $C^{(g)}_{\mu_1,\dots,\mu_n}$ 
the \emph{genus-$g$, $n$-point correlation functions of $\wave$}. 
\end{definition} 
\begin{remark} \ 
\label{rem:Fockspace}
\begin{enumerate}
\item We do not define $\Nabla^n C^{(g)}$ in the unstable range
$(g,n)= (0,0)$, $(0,1)$, $(0,2)$, $(1,0)$.  
The genus-zero data are given by the cubic tensor $\bY$, and
the genus-one data are given by a one-form $\Nabla C^{(1)}$.  

\item 
The fact that $C^{(1)}_{\mu\nu}=\Nabla_\mu C^{(1)}_\nu$ 
is symmetric implies that 
$\Nabla C^{(1)} =  C^{(1)}_\nu d\sx^\nu$ 
is a closed one-form. 
By (Grading \& Filtration) and (Pole), 
one can write it in local co-ordinates as 
\begin{equation} 
\label{eq:genusone-onepoint}
\Nabla C^{(1)} = \frac{1}{P(t,x_1)} 
\left( \sum_{i=0}^N F_i(t,x_1) dt^i 
+ \sum_{i=0}^N G_i(t,x_1) dx_1^i \right) 
\end{equation} 
for some homogeneous polynomials 
$F_i, G_i \in \cO(U)[x_1^0,\dots,x_1^N]$ 
of degree $N+1$ and $N$ respectively. 
The condition (Grading \& Filtration) does not prevent $F_i$ 
from containing $x_2$, but the closedness of $\Nabla C^{(1)}$ 
implies that $F_i$ does not depend on $x_2$.  
The primitive $C^{(1)} = \int \Nabla C^{(1)}$ 
is a multi-valued function defined up to a constant. 
The symmetry of $\Nabla^nC^{(g)}$ 
is automatic for $g\ge 2$ because $\Nabla$ is flat; 
the symmetry of $\Nabla^n C^{(0)} = \Nabla^{n-3} \bY$, 
$n\ge 3$ 
follows from the existence of $C^{(0)}$ in the 
formal neighbourhood $\hLL$ (Proposition~\ref{prop:potentiality}). 

\item 
Because $\Nabla^n C^{(g)} \in ((\bOmega^1)^{\otimes n})_{3g-3}$, 
we have 
\[
C^{(g)}_{\mu_1,\dots,\mu_n} 
\in \bcO_{3 g -3 - |\mu_1|-\cdots - |\mu_n|+n}
\]
where we set 
\[
|\mu| =
\begin{cases} 
n & \text{if } \sx^\mu=x_n^i \\ 
0 & \text{if } \sx^\mu = t^i
\end{cases} 
\] 
so that $d \sx^\mu \in (\bOmega^1)_{|\mu|-1}$. 
In particular, the following \emph{$(3g-2)$-jet condition} holds: 
\begin{equation} 
\label{eq:jetcondition}
C^{(g)}_{\mu_1,\dots,\mu_n} = 0 \quad \text{if
$|\mu_1|+\dots+|\mu_n| > 3 g -3 +n$} 
\end{equation} 
For $t\in U$, 
let $\{q_n^i\}_{n\ge 0,0\le i\le N}$ be the flat co-ordinate 
system (Definition~\ref{def:flatcoordinate_on_LL}) 
on the formal neighbourhood $\hLLo$ of $\LLo_t$ 
associated to $\sfP$. 
Then we have (see equation~\ref{eq:flat-alg}):
\[
\parfrac{}{q_n^i} \Bigg|_{\LLo_t}= 
\begin{cases} 
\text{a linear combination of $\parfrac{}{t^i}\Big|_{\LLo_t}$ and 
$\parfrac{}{x_m^i}\Big|_{\LLo_t}$, $m\ge 1$} 
& \text{if }n=0 \\ 
\parfrac{}{x_n^i}\Big|_{\LLo_t} & \text{otherwise} 
\end{cases} 
\]
Therefore the $(3g-2)$-jet condition implies the following 
\emph{tameness}: 
\begin{equation} 
\label{eq:tameness-Fockelement}
\parfrac{^nC^{(g)}}{ q_{l_1}^{i_1} \cdots \partial q_{l_n}^{i_n}} 
\Biggr |_{q_0=0} = 0 \quad 
\text{if $l_1 + \cdots + l_n > 3g-3 +n$}  
\end{equation} 
(cf.~\cite{Givental:nKdV}). Notice that we need the restriction to $q_0=0$ here. 

\item 
The discriminant depends on the choice 
of local co-ordinates $\{t^i\}$ on $U$, and making a different choice 
changes it as $P \to f(t) P$ for some function 
$f(t)$ on $U$. 
Note however that the condition (Pole) 
does not depend on the choice of co-ordinates. 

\item  Recall that $\Nabla$ preserves the grading 
and the filtration (Proposition~\ref{prop:grading-filtration-Nabla}). 
Therefore (Grading \& Filtration) for genus $g\ge 1$ is equivalent 
to the condition that $\Nabla C^{(1)} \in \bcO(\pr^{-1}(U)^\circ)^0_0$ 
and $C^{(g)}\in \bcO(\pr^{-1}(U)^\circ)^{2-2g}_{3g-3}$ 
for $g\ge 2$. 
Note that (Grading \& Filtration) 
at genus zero follows from 
$\bY \in ((\bOmega^1)^{\otimes 3})^2_{-3}$. 

\item 
The condition (Pole) for $g\ge 2$ is equivalent to 
the fact that $C^{(g)}$ has the following expansion: 
\[
C^{(g)} = \sum_{n=0}^\infty \sum_{\substack
{L =(l_1,\dots,l_n) \\ l_a \ge 2 \text{  for all $a$}}} 
\sum_{I = (i_1,\dots, i_n)} 
\frac{1}{n!} 
\frac{f_{g,L,I}(t,x_1)}{P(t,x_1)^{5g-5+2n-(l_1+\cdots + l_n)}} 
x_{l_1}^{i_1} \cdots x_{l_n}^{i_n} 
\]
for some polynomials $f_{g,L,I}(t,x_1)$ in $x_1$. 
This can be further rephrased as
\[
P^{5g-5+2n - (l_1+\cdots +l_n) } 
\parfrac{^nC^{(g)}}{x_{l_1}^{i_1} \cdots \partial x_{l_n}^{i_n}} 
\quad 
\text{extends regularly to $\pr^{-1}(U)$ for $l_1,\dots,l_n \ge 1$.}  
\]
Since $C^{(g)} \in \bcO(\pr^{-1}(U)^\circ)_{3g-3}$, 
the exponent $5g-5+2n-(l_1+\cdots+l_n) 
\ge 2g-2+n$ here is always positive 
unless the derivative is zero. 
Note that $\partial/\partial x_1^i$ 
raises the pole order by at most one, and that $\partial/\partial x_2^i$ 
does not raise the pole order.  
\end{enumerate}
\end{remark} 

\subsection{Pole Order Along The Discriminant}

We next study the pole order of 
tensors $\Nabla^n C^{(g)}$ in 
algebraic local co-ordinates and also in 
flat co-ordinates on $\hLLo$. 
\begin{lemma} 
\label{lem:4,5-tensor} 
Let $\Nabla = \Nabla^{\sfP}$ be the connection 
associated to a pseudo-opposite module $\sfP$ 
over $U$. 
\begin{enumerate}
\item The $4$-tensor $\Nabla \bY$ is completely symmetric\footnote
{This is obvious if $\sfP$ is parallel, but we do not assume here 
that $\sfP$ is parallel.} 
and is regular on $\pr^{-1}(U)$. 

\item  We have 
\[
\Nabla \bY = (\id^{\otimes 3} \otimes \KS^*) \Xi 
\]
for some regular section $\Xi$ of 
$(\bOmega^1)^{\otimes 3} \otimes \pr^* \sfF^\vee$ 
over $\pr^{-1}(U)$ 
(cf.\ Lemma~\ref{lem:Yukawa-KS}). 

\item For $n \geq 5$, the $n$-tensor $P^{n-5}\Nabla^{n-3} \bY$ is regular on $\pr^{-1}(U)$. 
\end{enumerate}
\end{lemma} 
\begin{proof} 
Using \eqref{eq:Yukawa-KS}, we calculate: 
(writing $\cC_i = \cC_i(t,0)$) 
\begin{align}
\label{eq:four-derivatives} 
\begin{split}  
\Nabla \bY &=  \sum (d(\cC_i e_f, e_g)_{\sfF_0}) \otimes 
dt^i \otimes \KS^*(\varphi_0^f) 
\otimes \KS^* (\varphi_0^g) \\
& + \sum (\cC_i e_f, e_g)_{\sfF_0} d\sx^\nu \otimes \Nabla_\nu dt^i 
\otimes \KS^*(\varphi_0^f) 
\otimes \KS^*(\varphi_0^g) 
\\ 
& + \sum(\cC_i e_f, e_g)_{\sfF_0} d\sx^\nu \otimes dt^i 
\otimes \Nabla_\nu \KS^*(\varphi_0^f) 
\otimes \KS^*(\varphi_0^g) \\
& + \sum(\cC_i e_f, e_g)_{\sfF_0} d\sx^\nu \otimes dt^i 
\otimes \KS^*(\varphi_0^f) 
\otimes \Nabla_\nu \KS^*(\varphi_0^g)
\end{split} 
\end{align} 
The first term is regular and is in the image of 
$\id^{\otimes 3} \otimes \KS^*$. 
So are the third and the fourth terms, because 
$\Nabla \KS^*(\varphi_0^i) 
= \KS^*( (\tnabla^\vee\Pi^* \varphi_0^i)|_{\pr^*\sfF})$. 
Using $\KS^*(\varphi_0^f) = -\sum_j [\cC_j x_1]^f dt^j$ 
(see equation~\ref{eq:KS_coord}), we can rewrite the second term as:
\begin{align*} 
\sum(\cC_i \cC_j x_1, e_g)_{\sfF_0} (\Nabla dt^i) \otimes dt^j 
\otimes \KS^*(\varphi_0^g)  
= \sum (\cC_j e_f, e_g)_{\sfF_0}  [\cC_i x_1]^f (\Nabla dt^i) 
\otimes dt^j \otimes \KS^*(\varphi_0^g) 
\end{align*} 
Using $\KS^*(\varphi_0^f) = -\sum_i [\cC_i x_1]^f dt^j$ again, 
we have:
\[
\sum_i  [\cC_i x_1]^f (\Nabla dt^i) = 
- \Nabla(\KS^*(\varphi_0^f)) - \sum_i (d[\cC_i x_1]^f) \otimes dt^i 
\]
The right-hand side is regular on $\pr^{-1}(U)$ 
for the same reason as before. 
Thus the second term of \eqref{eq:four-derivatives} 
is regular and is in the image of $\id^{\otimes 3} \otimes \KS^*$. 
This establishes the regularity of $\Nabla \bY$ and Part (2). 
Next we show that $\Nabla \bY$ is symmetric. 
Take an opposite module $\sfQ$ in the formal neighbourhood 
of $t\in \cM$ (see Lemma~\ref{lem:existence_opposite}). 
We have by Proposition~\ref{prop:difference_conn}(1) 
\begin{align*} 
\Nabla^{\sfP}_\mu C^{(0)}_{\nu \rho \sigma} & = 
\partial_\mu C^{(0)}_{\nu\rho \sigma} - C^{(0)}_{\tau \rho \sigma} 
{\Gamma^{\sfP}}^{\tau}_{\mu\nu} - 
C^{(0)}_{\nu \tau \sigma} 
{\Gamma^{\sfP}}^{\tau}_{\mu \rho} - 
C^{(0)}_{\nu \rho \tau} 
{\Gamma^{\sfP}}^{\tau}_{\mu \sigma} \\ 
& = 
\Nabla^{\sfQ}_\mu C^{(0)}_{\nu\rho \sigma} 
- C^{(0)}_{\tau \rho \sigma} 
\Delta^{\tau\kappa} C^{(0)}_{\kappa \mu\nu} 
- C^{(0)}_{\nu \tau \sigma} 
\Delta^{\tau\kappa} C^{(0)}_{\kappa \mu \rho} 
- C^{(0)}_{\nu\rho \tau} 
\Delta^{\tau \kappa} C^{(0)}_{\kappa \mu \sigma} 
\end{align*}
where ${\Gamma^{\sfP}}^{\tau}_{\mu\nu}$ denotes 
the Christoffel symbol (see equation~\ref{eq:Christoffel}) of 
$\Nabla^{\sfP}$ and $\Delta = \Delta(\sfP,\sfQ)$ 
is the propagator. 
Because the propagator is symmetric 
(Proposition~\ref{pro:prop-elementary}) 
and the tensor $\Nabla^{\sfQ}_{\mu} C^{(0)}_{\nu\rho\sigma}$ 
associated to a parallel $\sfQ$ is symmetric, 
we find that $\Nabla^{\sfP} \bY$ is also symmetric. 
Finally we prove Part (3). 
We can write 
\[
\Nabla \bY = \sum_a  \kappa_a \otimes \lambda_a\otimes 
\mu_a \otimes \nu_a 
\]
with $\kappa_a, \lambda_a, \mu_a, \nu_a \in \bOmega^1$ regular
and $\nu_a$ the image under $\KS^*$ of a regular section 
of $\pr^*\sfF^\vee$. 
Then we have:
\begin{equation} 
\label{eq:5-tensor} 
\Nabla^2 \bY = \sum_a \left[ 
(\Nabla \kappa_a) \otimes \lambda_a 
\otimes \mu_a \otimes \nu_a) + \cdots  + 
(\kappa_a \otimes \lambda_a \otimes \mu_a \otimes \Nabla \nu_a)
\right] 
\end{equation} 
Note that $\Nabla \nu_a$ is regular, by the definition 
of $\Nabla$. We claim that the difference 
\[
\sum_a (\Nabla \kappa_a)\otimes \lambda_a \otimes \mu_a \otimes \nu_a 
- 
\sum_a (\Nabla \nu_a) \otimes \lambda_a \otimes \mu_a \otimes \kappa_a
\]
is regular.  
Since $\Nabla \bY$ is symmetric, the image of 
\[
\sum_a  \kappa_a \otimes_\C \lambda_a \otimes_\C \mu_a \otimes_\C \nu_a 
- \sum_a \nu_a \otimes_\C \lambda_a \otimes_\C \mu_a \otimes_\C \kappa_a 
\in \bOmega^1 \otimes_\C \bOmega^1\otimes_\C \bOmega^1 
\otimes_\C \bOmega^1  
\]
in $\bOmega^1\otimes_{\bcO} \bOmega^1\otimes_{\bcO} 
\otimes \bOmega^1  \otimes_{\bcO} \bOmega^1$ 
is zero. Therefore it is generated by elements of the form 
\begin{align*} 
& f \phi_1 \otimes_\C \phi_2 \otimes_\C \phi_3 \otimes_\C \phi_4 
- \phi_1 \otimes_\C f \phi_2 \otimes_\C \phi_3 \otimes_\C \phi_4, \\
& \phi_1 \otimes_\C f\phi_2 \otimes_\C \phi_3 \otimes_\C \phi_4 
- \phi_1 \otimes_\C \phi_2 \otimes_\C f\phi_3 \otimes_\C \phi_4, \\
& \phi_1 \otimes_\C \phi_2 \otimes_\C f \phi_3 \otimes_\C \phi_4 
- \phi_1 \otimes_\C \phi_2 \otimes_\C \phi_3 \otimes_\C f \phi_4 
\end{align*} 
with $f\in \bcO$ regular and regular sections 
$\phi_i \in \bOmega^1$. 
If one applies $\Nabla \otimes \id^{\otimes 3}$ 
to any of these generators and maps it to 
$\bOmega^1\otimes_{\bcO} \bOmega^1\otimes_{\bcO} 
\bOmega^1 \otimes_{\bcO} \bOmega^1 \otimes_{\bcO} \bOmega^1$, 
we always get a regular 5-tensor. 
This proves the claim. 
Because $\Nabla \nu_a$ is regular, 
it follows that every term in \eqref{eq:5-tensor} is regular. 
Part (3) is proved. 
Part (4) follows from Part (1), Part (3), and Lemma~\ref{lem:poleorder-Nabla}. 
\end{proof} 

\begin{proposition} 
For $\wave =\{\Nabla^n C^{(g)}\}_{2g-2+n>0}
\in \Fock(U;\sfP)$,  
$P^{\max(5g-5+n,0)} \Nabla^n C^{(g)}$ 
extends to a regular $n$-tensor 
on $\pr^{-1}(U)$. 
\end{proposition} 
\begin{proof} 
At genus zero, this was shown in the previous lemma. 
At higher genera, it follows from Lemma~\ref{lem:poleorder-Nabla} 
and the fact that $C^{(g)}$ has poles of order $5g-5$ 
along $P(t,x_1)=0$.  
\end{proof} 

\begin{proposition}[Pole structure in flat co-ordinates]  
\label{prop:pole-flat} 
Let $\sfP$ be a parallel pseudo-opposite module 
over an open set $U\subset \cM$ 
and let $t\in U$. 
Let $\{q_n^i\}_{n\ge 0, 0\le i\le N}$ 
be a flat co-ordinate system on the formal neighbourhood 
$\hLLo$ of $\LLo_t$ 
associated to $\sfP$ 
(see Definition~\ref{def:flatcoordinate_on_LL}).  
For any element $\wave =\{\Nabla^n C^{(g)}\}
\in \Fock(U;\sfP)$, we have 
\begin{equation} 
\label{eq:pole-formalization} 
\parfrac{^n C^{(g)}}{q_{l_1}^{i_1} \cdots \partial q_{l_n}^{i_n}} 
\in P_t^{-(5g-5 +2n -(l_1+\cdots + l_n))} 
\C[q_1, q_2, P_t q_3, P_t^2 q_4,\dots][\![P_t^{-2} q_0]\!]  
\end{equation} 
whenever $2g-2+n>0$, where $P_t=P(t,q_1)$. 
\end{proposition} 
\begin{proof} 
At genus zero, this follows from Lemma~\ref{lem:genuszeropole}. 
For $g\ge 2$, it suffices to show that 
$C^{(g)}$ lies in $P_t^{-(5g-5)} \cS$ where 
$\cS:= \C[q_1, q_2, P_t q_3, P_t^2 q_4, \dots][\![P_t^{-2} q_0]\!]$. 
Let $s=(s^0,\dots,s^N)$ be a local co-ordinate on $\cM$ centred 
at $t\in \cM$ as in \S\ref{subsec:flatstronL},
and write $t+s$ for the corresponding point 
in a neighbourhood of $t$. 
The condition (Pole) implies that
\[
C^{(g)} \in P^{-(5g-5)} \cO(U)[x_1,x_2, P x_3, P^2 x_4,
\dots, P^{3g-4} x_{3g-2}]
\] 
where $P=P(t+s,x_1)$. 
By Lemma~\ref{lem:sx-poleorder}, we have 
$x_1^i \in q_1^i + P_t \cS$, $x_n^i\in P_t^{2-n} \cS$, $n\ge 2$,
and $s^i \in P_t \cS$. 
Thus we have 
\[
P/P_t = P(t+s, x_1) /P(t,q_1) \in \cS
\]
We also have $P_t/P \in \cS$, because $(P/P_t)|_{q_0=0} = 1$ 
so that $P/P_t$ is invertible in $\cS$. 
Therefore:
\[
P^{-(5g-5)} \in P_t^{-(5g-5)} \cS, \quad 
P^{n-2} x_n \in \cS \quad (n \ge 2)
\]
Hence $C^{(g)} \in P_t^{-(5g-5)}\cS$ for $g\ge 2$. 
For $g=1$, it suffices to show \eqref{eq:pole-formalization} 
for $n=1$. 
Using the expression \eqref{eq:genusone-onepoint}, we have:
\[
\parfrac{C^{(1)}}{q_l^j} 
= \frac{1}{P} 
\sum_{i=0}^N \left(F_i(t+s, x_1) \parfrac{s^i}{q_l^j} 
+ G_i(t+s, x_1) \parfrac{x_1^i}{q_l^j} \right) 
\] 
Because $P^{-1} \in P_t^{-1} \cS$, $F_i(t+s,x_1), G_i(t+s,x_1) \in \cS$  
and $\partial s^i/\partial q_l^i, 
\partial x_1^i /\partial q_l^j \in P_t^{l-1} \cS$, 
the right-hand side here belongs to $P_t^{l-2} \cS$. 
\end{proof} 

\begin{remark} 
The previous proposition implies, and is implied by  
\[
\parfrac{^n C^{(g)}}{q_{l_1}^{i_1} \cdots \partial q_{l_n}^{i_n}} 
\Bigg|_{q_0=0} 
= \frac{f_{g,I,L}(q_1,q_2,\dots)}{P(t,q_1)^{5g-5 +2n -(l_1+\cdots + l_n)}}  
\]
for some polynomial $f_{g,I,L}(q_1,q_2,\dots)$. 
By tameness \eqref{eq:tameness-Fockelement}, 
the exponent $5g-5+2n - (l_1+ \cdots + l_n)$ is positive 
unless the derivative vanishes. 
\end{remark} 

\subsection{Transformation Rule and Fock Sheaf} 
\label{subsec:transformation}
As outlined in \S\ref{subsec:motivation-Focksheaf}, 
we now define a Fock sheaf by gluing local Fock spaces 
using a transformation rule. 

Let $\{\sy^\mu\}$ denote the fiber co-ordinates of the tangent 
bundle $\bTheta$ dual to $\{\sx^\mu\}$. A general point 
on the total space of $\bTheta$ can be written as:
\[
\sum_{\mu} \sy^\mu \left( \parfrac{}{\sx^\mu}\right)_{\sx}  \in \bTheta_{\sx}
\]
Much as we did for $\LL$, one can give a co-ordinate-free definition of the total space of the 
tangent bundle $\bTheta$, as 
a certain ringed space.
To avoid excessive formalism, however, we will work in terms of local co-ordinates 
$\{\sx^\mu, \sy^\mu\}$ on $\bTheta$, regarding any polynomial 
(or formal power series) expression 
in $\sx^\mu$,~$\sy^\mu$ as a regular (or formal) function 
on the total space of $\bTheta$.

\begin{definition}[jet potential] 
\label{def:jetpotential}
An element $\wave= \{\Nabla^n C^{(g)}\}_{g,n}$
of $\Fock(U;\sfP)$ is encoded by the following formal function 
$\cW$ on the total space of $\bTheta|_{\pr^{-1}(U)^\circ}$: 
\begin{align*} 
\cW(\sx,\sy) 
& = \sum_{g=0}^\infty \hbar^{g-1}\cW^g(\sx,\sy) \\
\intertext{where:}
\cW_g(\sx,\sy) & = \sum_{n=\max(3-2g,0)}^\infty 
\frac{1}{n!} 
C^{(g)}_{\mu_1,\dots,\mu_n}(\sx) \sy^{\mu_1}\cdots \sy^{\mu_n} \notag
\end{align*} 
We call $\cW^g$ the \emph{genus-$g$ jet potential}  
and $\exp(\cW)$ the \emph{total jet potential} 
associated to $\wave$.  
\end{definition} 

\begin{remark} \ 
  \begin{enumerate}
  \item For a fixed $\sx \in \LLo$, $\cW^g(\sx,\sy)$ should be viewed as a packaging of the multilinear tensors $\{\Nabla^n C^{(g)}\}_n$ on the tangent space $\bTheta_{\sx}$.  This can be identified with the Taylor expansion at $\sx$ of the potential $C^{(g)}$ \emph{in flat co-ordinates}.  Namely the linear co-ordinates $\sy^\mu$ on $\bTheta_{\sx}$ play the role of flat co-ordinates on $\LL$ centred at $\sx$ such that $d \sy^\mu|_{\sx} = d \sx^\mu|_{\sx}$.  Notice however that the constant term at genus one and the quadratic term at genus zero are ignored.

  \item Let $y_n^i$ be the fiber co-ordinate dual to $\partial/\partial x_n^i$ for $n\ge 1$ and let $y_0^i$ be the fiber co-ordinate dual to $\partial/\partial t^i$.  Then the jet potential $\cW(\sx,\sy)$ is an element of:
    \[ 
    \hbar^{-1} \cO(U)\left[ \{x_n^i\}_{n\ge 1, 0\le i\le N}, P(t,x_1)^{-1}, \{y_n^i\}_{n\ge 2, 0\le i\le N} \right] [\![y_0^0,\dots,y_0^N, y_1^0,\dots,y_1^N]\!]  [\![ \hbar]\!]
    \] 
    Moreover, the $\hbar^{-1}$-coefficient of $\cW$ (i.e.\ $\cW^0$) vanishes along $y_0^0=\cdots = y_0^N=0$ and thus $\exp(\cW(\sx,\sy))$ is a well-defined formal Laurent series in $\hbar$ (infinite in both directions).
  \end{enumerate}
\end{remark}

To describe the transformation rule (or Feynman rule), we use the following terminology 
for graphs. A graph $\Gamma$ is given by four finite 
sets  $V(\Gamma)$, $E(\Gamma)$, $L(\Gamma)$, $F(\Gamma)$ 
called (the sets of) \emph{vertices}, \emph{edges}, \emph{legs} and 
\emph{flags} respectively, 
together with incidence maps 
\begin{align*}
  \pi_V \colon F(\Gamma)  \to V(\Gamma) &&
  \pi_E \colon F(\Gamma) \to E(\Gamma) \sqcup L(\Gamma) 
\end{align*}
such that $|\pi_E^{-1}(e)|=2$ for each $e\in E(\Gamma)$ 
and $|\pi_E^{-1}(l)|=1$ for each $l \in L(\Gamma)$. 
We assign to an edge $e$ a closed interval $I_e \cong [0,1]$, 
to a leg $l$ a half-open interval $H_l \cong [0,1)$,
and to a vertex $v$ a point $p_v$, 
and fix identifications $\pi_E^{-1}(e) \cong \partial I_e$, 
$\pi_E^{-1}(l) \cong \partial H_l$. 
By identifying $I_e$,~$H_l$,~$p_v$ via the map 
$\pi_V \colon F(\Gamma) \cong \bigsqcup \partial I_e \sqcup 
\bigsqcup \partial H_l 
\to V(\Gamma) \cong \{p_v\}$, we get a topological 
realization $|\Gamma|$ of the graph $\Gamma$.  
We say that $\Gamma$ is connected if $|\Gamma|$ is connected; 
we also write $\chi(\Gamma)=\chi(|\Gamma|)=|V(\Gamma)|-|E(\Gamma)|$ 
for the topological Euler characterisitic of $|\Gamma|$. 

\begin{definition}[transformation rule] 
\label{def:transformation} 
Let $\sfP_1$, $\sfP_2$ be parallel pseudo-opposite modules 
over $U \subset \cM$ 
and let $\Delta = \Delta(\sfP_1,\sfP_2)$ be the propagator. 
The \emph{transformation rule}  $T(\sfP_1, \sfP_2)\colon 
\Fock(U;\sfP_1) \to \Fock(U;\sfP_2)$ 
is a map which assigns, to the jet potential $\exp(\cW)$ 
for an element of $\Fock(U;\sfP_1)$, 
the jet potential $\exp(\hcW)$ for an element of $\Fock(U;\sfP_2)$ 
given by:  
\begin{equation} 
\label{eq:transformationrule-jet} 
\exp(\hcW(\sx,\sy)) =  
\exp\left(\frac{\hbar}{2}\Delta^{\mu\nu} \partial_{\sy^\mu}
\partial_{\sy^\nu}\right) 
\exp(\cW(\sx,\sy))
\end{equation} 
Let $\wave = \{C^{(g)}_{\mu_1,\dots,\mu_n}\}_{g,n}$, 
$\widehat\wave=\{\hC^{(g)}_{\mu_1,\dots,\mu_n} \}_{g,n}$ 
be the correlation functions corresponding respectively to 
$\cW$, $\hcW$. 
The above formula is equivalent to the following 
\emph{Feynman rule}: 
\[
\hC^{(g)}_{\mu_1,\dots,\mu_n} 
= \sum_\Gamma \frac{1}{|\Aut(\Gamma)|} 
\Cont_{\Gamma}(\wave, \Delta)_{\mu_1,\dots,\mu_n}
\]
Here the summation is over all connected decorated graphs 
$\Gamma$ such that:
\begin{itemize}
\item 
To each vertex $v\in V(\Gamma)$ is 
assigned a non-negative integer $g_v\ge 0$, called the genus; 
  
\item 
$\Gamma$ has $n$ labelled legs: an isomorphism 
$L(\Gamma) \cong \{1,2,\dots,n\}$ is given; 

\item $\Gamma$ is stable, i.e.~$2 g_v - 2 + n_v>0$ for every vertex $v$.  
Here $n_v = |\pi_V^{-1}(v)|$ denotes the number of edges or 
legs incident to $v$; 

\item $g = \sum_{v} g_v + 1- \chi(\Gamma)$. 
\end{itemize}
We put the index $\mu_i$ at the $i$th leg, the correlation function $\Nabla^{n_v} C^{(g_v)}$ 
on the vertex $v$, and the propagator $\Delta$ 
on every edge. 
Then 
$\Cont_\Gamma(\wave, \Delta
)_{\mu_1,\dots,\mu_n}$ 
is defined to be the contraction of all these tensors 
with the indices $\mu_1,\dots,\mu_n$ 
on the legs fixed. 
Here $\Aut(\Gamma)$ denotes the automorphism group 
of the decorated graph $\Gamma$. 
\end{definition} 

\begin{example} \ 
  \begin{enumerate}
  \item  The Feynman rule for genus-zero 
    three point correlation functions is trivial: 
    \begin{equation} 
      \label{eq:Yukawa-Feynman}
      \hC^{(0)}_{\mu \nu \rho} 
      = C^{(0)}_{\mu \nu \rho}
    \end{equation} 
    since there is only one genus-zero stable graph 
    with three legs. 
    This is compatible with the fact that the 
    Yukawa coupling was defined independently 
    of the choice of pseudo-opposite module.

  \item The Feynman rule for genus-one one-point functions  
    is given by:
    \begin{equation} 
      \label{eq:Feynman-genusone} 
      \hC^{(1)}_\mu = C^{(1)}_\mu + 
      \frac{1}{2} C^{(0)}_{\mu \nu \rho} \Delta^{\nu\rho} 
      = C^{(1)}_\mu + \left( \omega_{\sfP_1\sfP_2} \right)_\mu 
    \end{equation} 
    Here $\omega_{\sfP_1\sfP_2}$ is the difference one-form 
    defined in \eqref{eq:difference1-form} 
    and comes from the graph $\mu \boldsymbol{-\!\bigcirc}$. 
    
  \item The Feynman rule for genus-two potentials is given by:
    \begin{align*} 
      \hC^{(2)} = & C^{(2)} + \frac{1}{2} C^{(1)}_{\mu\nu} \Delta^{\mu\nu} 
      + \frac{1}{2} C^{(1)}_\mu \Delta^{\mu\nu} C^{(1)}_\nu 
      + \frac{1}{2} C^{(1)}_\mu \Delta^{\mu\nu} C^{(0)}_{\nu \rho \sigma} 
      \Delta^{\rho \sigma} \\
      & + \frac{1}{8} C^{(0)}_{\mu\nu\rho \sigma} \Delta^{\mu\nu} 
      \Delta^{\rho \sigma} 
      + \frac{1}{8} \Delta^{\mu\nu} 
      C^{(0)}_{\mu\nu\rho} \Delta^{\rho \sigma} 
      C^{(0)}_{\sigma \tau \omega} \Delta^{\tau\omega} 
      + \frac{1}{12} 
      C^{(0)}_{\mu\nu\rho} 
      \Delta^{\mu\mu'} \Delta^{\nu\nu'} \Delta^{\rho\rho'} 
      C^{(0)}_{\mu' \nu' \rho'}. 
    \end{align*} 
    cf.~\cite[Figure~1]{ABK}.
  \end{enumerate}
\end{example} 

\begin{remark} 
We can use the flat co-ordinate system $\{q_n^i\}_{n\ge 0,0\le i\le N}$ 
associated to the parallel pseudo-opposite $\sfP_1$ (see Definition~\ref{def:flatcoordinate_on_LL}) to expand the 
$n$-point correlation functions as follows: 
\[
\Nabla^n C^{(g)} = \sum_{\substack{
l_1,\dots,l_n\ge 0 \\ 
0\le i_1,\dots,i_n\le N}}  
\parfrac{^n C^{(g)}}{q_{l_1}^{i_1} 
\cdots \partial q_{l_n}^{i_n}} 
dq_{l_1}^{i_1} \otimes \cdots \otimes dq_{l_n}^{i_n}
\]
Written in flat co-ordinates, the transformation rule above 
matches with the action of Givental's quantized operators on tame potentials. 
This will be explained in \S\ref{subsec:Givental-Global} below. 
\end{remark}

We show in Lemmas~\ref{lem:jetness}--\ref{lem:pole} below 
that the transformation rule is well-defined, i.e.~that
$\widehat\wave = \{\hC^{(g)}_{\mu_1,\dots,\mu_n}\}$ 
in Definition~\ref{def:transformation} 
satisfies the conditions in Definition~\ref{def:localFock}. 
Observe first that the tensor $\hC^{(g)}_{\mu_1,\dots,\mu_n}$ 
defined by the Feynman rule is automatically completely symmetric. 
We already saw in \eqref{eq:Yukawa-Feynman} that $\hC^{(0)}_{\mu\nu\rho}$ 
is the Yukawa coupling. 
Let $\Nabla$, $\hNabla$ denote the 
flat connections on $\LLo$ associated with $\sfP_1$, $\sfP_2$ 
respectively. 

\begin{lemma}[Jetness] 
\label{lem:jetness} 
$\hNabla({\hNabla}^n \hC^{(g)}) = {\hNabla}^{n+1} \hC^{(g)}$, 
i.e.\ 
$\hNabla_\nu \hC^{(g)}_{\mu_1,\dots,\mu_n} 
= \hC^{(g)}_{\nu,\mu_1,\dots,\mu_n}$ 
(see equation~\ref{eq:covariant-derivative} for this notation). 
\end{lemma} 
\begin{proof}
We have  
\begin{equation} 
\label{eq:NablaprimehC}
\hNabla_\nu\hC^{(g)}_{\mu_1,\dots,\mu_n} 
=\Nabla_\nu \hC^{(g)}_{\mu_1,\dots,\mu_n} 
+ (\hNabla - \Nabla)_\nu \hC^{(g)}_{\mu_1,\dots,\mu_n} 
\end{equation} 
By the Feynman rule for $\hC^{(g)}_{\mu_1,\dots,\mu_n}$, 
we can write the first term as 
\[
\Nabla_\nu \hC^{(g)}_{\mu_1,\dots,\mu_n} 
= C_{\rm vert} + C_{\rm prop} 
\]
where $C_{\rm vert}$ and $C_{\rm prop}$ arise 
from the vertex and the propagator 
differentiations respectively. 
The term $C_{\rm vert}$ is the sum over stable graphs 
with one extra leg $\nu$ attached to a vertex $v$; 
note that the vertex $v$ 
satisfies $2 g_v -2 + n_v>1$.  
The term $C_{\rm prop}$ is the sum over stable graphs 
with the differentiated propagator $\Nabla_\nu \Delta$ 
on one of the edges. 
By Proposition~\ref{prop:difference_conn}(2), 
we can replace the edge $\Nabla_\nu \Delta$ 
by the genus-zero trivalent vertex with the leg $\nu$: 
\[
\begin{xy} 
{\ar@{-}(20,5) *=0 {\bullet}; (20,-5) *=0 {\bullet} 
_{\Nabla_\nu\Delta}}, 
{\ar@{->}(30,0); (40,0)}, 
{\ar@{-}(50,5) *=0 {\bullet}; (60,0) *=0 {\bullet}
^(.4){\Delta}}, 
{\ar@{-}(50,-5) *=0{\bullet} ; (60,0)*=0{\bullet}
_(.4){\Delta}}, 
{\ar@{-}(60,0) ; (70, 0) *{\nu} _<{g=0}} 
\end{xy} 
\]
By Proposition~\ref{prop:difference_conn}(1), the second term 
of \eqref{eq:NablaprimehC} is:
\[
(\hNabla - \Nabla)_\nu \hC^{(g)}_{\mu_1,\dots,\mu_n} 
= \sum_{i=1}^n \sum_\Gamma \frac{1}{\Aut(\Gamma)}
\Cont_\Gamma(\wave, \Delta)_{
\mu_1,\dots,\underset{i}{\sigma},\dots,\mu_n} 
\Delta^{\sigma \rho} C^{(0)}_{\rho \mu_i \nu}
\]
Namely, we add to the leg $\mu_i$ the genus-zero 
trivalent vertex: 
\begin{equation} 
\label{eq:leg-modification} 
\begin{xy} 
{\ar@{-}(0,0)*=0{\bullet}; (20,0)*{\mu_i}}, 
{\ar@{->}(30,0); (40,0)}, 
{\ar@{-}(50,0)*=0{\bullet};
(60,0)*=0{\bullet}_>{g=0}^{\Delta}}, 
{\ar@{-}(60,0);(70,4)*{\mu_i}},
{\ar@{-}(60,0);(70,-4)*{\nu}} 
\end{xy} 
\end{equation} 
On the other hand, by the Feynman rule, 
$\hC^{(g)}_{\nu, \mu_1,\dots,\mu_n}$ 
can be written as a summation over genus-$g$ stable graphs 
with legs $\nu,\mu_1,\dots,\mu_n$; 
if $v$ denotes the vertex incident to the leg $\nu$, 
we have the following three cases:  
\begin{itemize}
\item $n_v + 2 g_v -2 >1$
\item $g_v=0$, $n_v=3$ and $v$ has only one leg $\nu$
\item $g_v=0$, $n_v=3$ and $v$ has two legs $\nu,\mu_i$ 
\end{itemize}
These three cases correspond to $C_{\rm vert}$, 
$C_{\rm prop}$ and $(\hNabla - \Nabla)_\nu
\hC^{(g)}_{\mu_1,\dots,\mu_n}$ 
respectively. 
\end{proof} 

\begin{lemma}[Grading \& Filtration] 
\label{lem:GrFil} 
$\hNabla^n\hC^{(g)}\in 
((\bOmega^1)^{\otimes n})^{2-2g}_{3g-3}$. 
\end{lemma} 
\begin{proof} 
Let $\Gamma$ be a decorated graph contributing to the Feynman 
rule of $\hC^{(g)}_{\mu_1,\dots,\mu_n}$. 
We estimate the grading and the filtration 
of the contribution from $\Gamma$. 
Recall that $\deg \Delta = -2$, $\filt \Delta \le 2$ 
by Proposition~\ref{prop:grading-filtration-prop} 
and $\deg \Nabla = 0$, $\filt \Nabla \le 0$ by 
Proposition~\ref{prop:grading-filtration-Nabla}. 
The degree can be calculated as:
\[
\sum_{v\in V(\Gamma)} (2-2g_v) + 
\sum_{e\in E(\Gamma)} (-2) = 2 - 2 g
\]
The filtration is estimated as:
\[
\sum_{v\in V(\Gamma)} (3 g_v -3 ) + \sum_{e\in E(\Gamma)} 2 = 
3 g - 3 - |E(\Gamma)| \le 3 g -3
\]
The conclusion follows.   
\end{proof} 

\begin{lemma}[Pole] 
\label{lem:pole} 
Let $P=P(t,x_1)$ be the discriminant \eqref{eq:discriminant}. 
Then $P\hNabla \hC^{(1)}$ extends to a regular 1-form on 
$\pr^{-1}(U)$ and $\hC^{(g)}$ belongs to 
$P^{-(5g-5)}\cO(U)[x_1,x_2,P x_3, P^2 x_4,\dots]$ 
for $g\ge 2$.  
\end{lemma} 
\begin{proof} 
By the Feynman rule \eqref{eq:Feynman-genusone} at genus one, 
$\hNabla \hC^{(1)}$ differs from 
$\Nabla C^{(1)}$ by a regular one-form $\omega_{\sfP_1\sfP_2}$ 
on $U$ (see equation~\ref{eq:difference1-form}). 
Thus $P \hNabla \hC^{(1)}$ is regular. 

For $g\ge 2$, we apply the Feynman rule to correlation functions 
\emph{written in flat co-ordinates}. 
Take a point $t\in U$ and flat co-ordinates 
$\{q_n^i\}_{n\ge 0, 0\le i\le N}$ 
on the formal neighbourhood $\hLLo$ of $\hLLo_t$ 
associated to $\sfP_1$. 
Take a graph $\Gamma$ (without legs) 
which contributes to the Feynman rule for $\hC^{(g)}$.  
By Proposition~\ref{prop:pole-flat}, 
the vertex term for $v\in V(\Gamma)$ 
\[
\parfrac{^n C^{(g)}}{q_{l_1}^{i_1} \cdots \partial 
q_{l_{n}}^{i_{n}}} 
\Biggr |_{q_0=0} \qquad 
g=g_v, \ n=n_v 
\]
belongs to $P_t^{-(5g_v-5+2n_v-(l_1+\cdots+l_m))} 
\C[q_1,q_2,P_tq_3, P_t^2 q_4,\dots]$ 
with $P_t = P(t,q_1)$. 
The propagator 
\[
\Delta(dq_n^i , dq_m^j ) \bigr|_{q_0=0}
= \Omega^\vee(\Pi_1^* \varphi_n^i, \Pi_2^*\varphi_m^j) 
\]
is constant (see equation~\ref{eq:flatcoord-atthecentre}) along 
the fiber $\LL_t=\{q_0=0\}$. 
Using the formulae: 
\begin{align*}
  \sum_{v\in V(\Gamma)} (g_v- 1) = g-1 - |E(\Gamma)| &&
  \sum_{v\in V(\Gamma)} n_v = 2 |E(\Gamma)| 
\end{align*}
we bound the pole order 
of $\Cont_\Gamma(\wave, \Delta)|_{q_0=0}$ along $P_t=0$ 
from above as:
\[
\sum_{v\in V(\Gamma)} (5g_v-5 + 2n_v) 
= 5g-5 - |E(\Gamma)| 
\le 5g-5
\]
Thus:
\[
\hC^{(g)}\bigr|_{\LLo_t} \in P_t^{-(5g-5)} 
\C[q_1,q_2,P_t q_3,P_t^2 q_4,\dots]
\]
Since $q_n^i|_{\LL_t} = x_n^i$ ($n\ge 1$) and 
this holds for every point $t$, 
the conclusion follows. 
\end{proof} 

\begin{proposition}[cocycle condition] 
\label{prop:cocycle} 
Let $\sfP_1,\sfP_2, \sfP_2$ 
be parallel pseudo-opposite modules over an open set $U$. 
The transformation rules $T_{ij} = T(\sfP_i, \sfP_j)
\colon \Fock(U;\sfP_i) \to \Fock(U;\sfP_j)$ 
satisfy the cocycle condition:  
\[
T_{13} = T_{23} \circ T_{12}. 
\] 
\end{proposition} 
\begin{proof}
This is immediate from the definition 
\eqref{eq:transformationrule-jet} 
and Proposition~\ref{prop:Deltasum}. 
\end{proof} 

We define the Fock sheaf over $\cM$ under the following assumption. 
(The Fock sheaf without this assumption will be considered in 
the next section \S\ref{subsec:curved}.)  
\begin{assumption}[Covering Assumption]  
\label{assump:covering} 
There exists an open covering $\{U_\alpha\}_{\alpha \in A}$ 
of $\cM$ such that we can find a parallel 
pseudo-opposite module $\sfP_\alpha$ over $U_\alpha$ 
for each $\alpha\in A$.   
\end{assumption} 

\begin{definition}[Fock sheaf] 
\label{def:Focksheaf} 
The \emph{Fock sheaf} is a sheaf of sets over $\cM$ 
which is obtained by gluing the local Fock 
spaces $\Fock(U_\alpha;\sfP_\alpha)$ over $U_\alpha$ 
by the transformation rule: 
\[
T_{\alpha\beta} = T(\sfP_\alpha,\sfP_\beta) \colon 
\Fock(U_{\alpha\beta}; \sfP_\alpha) \longrightarrow 
\Fock(U_{\alpha\beta}; \sfP_\beta) 
\]
where $U_{\alpha\beta} = U_\alpha \cap U_\beta$. 
More precisely, we define the set $\Fock(U)$ for an open set $U$ 
as the equalizer of the sequence: 
\[
\begin{CD} 
\displaystyle 
\prod_{\alpha: U\cap U_\alpha \neq \varnothing} 
\Fock(U \cap U_\alpha; \sfP_\alpha) 
\quad 
\begin{xy} 
{(0,-1) \ar (10,-1)_{\pi_2}}, 
{(0,1) \ar (10,1)^{\pi_1}}
\end{xy} 
\quad 
\displaystyle \prod_{(\alpha,\beta): U\cap U_{\alpha\beta} \neq 
\varnothing} \Fock(U \cap U_{\alpha\beta} ; \sfP_\alpha) 
\end{CD} 
\]
where $\pi_1(\{u_\alpha\}_\alpha) = 
\{u_\alpha|_{U\cap U_{\alpha\beta}} \}_{\alpha, \beta}$ 
and $\pi_2(\{u_\alpha\}_\alpha ) 
= \{ T_{\beta\alpha} (u_\beta|_{
U\cap U_{\alpha\beta}}) \}_{\alpha, \beta}$, that is:
\[
\Fock(U) = 
\left \{ \{\wave_\alpha \in \Fock(U\cap U_\alpha;\sfP_\alpha) \}_{\alpha\in A} 
\, \Big | \, T_{\alpha\beta} \wave_\alpha|_{U\cap U_\alpha \cap U_\beta} 
= \wave_\beta|_{U \cap U_\alpha \cap U_\beta} \right\}
\] 
\end{definition} 

\begin{remark} 
Note that $\Fock(U)$ is not a $\C$-vector space 
but is just a \emph{set}. 
This is because the transformation rule is not $\C$-linear. 
A natural $\C$-linear structure should be considered on 
the space of \emph{exponentiated} potentials 
$\exp(C^{(1)} + \hbar C^{(2)}+ \hbar^2 C^{(3)}+ \cdots)$.  
In fact, we can construct a Fock sheaf of $\C$-vector spaces 
by choosing certain ``orientation data'' and regard 
these exponentiated potentials as sections of the sheaf. 
We hope to discuss this issue elsewhere. 
\end{remark} 

\subsection{Anomaly Equation For Curved Polarizations} 
\label{subsec:curved} 
In this section we introduce a Fock space for possibly 
curved pseudo-opposite modules.  Correlation functions associated with a curved pseudo-opposite module satisfy, instead of the jetness condition, a certain anomaly equation.  As we explain in \S\ref{sec:HAE} below, when the curved pseudo-opposite module is the so-called complex conjugate polarization, the anomaly equation becomes the celebrated \emph{holomorphic anomaly equation} of Bershadsky--Cecotti--Ooguri--Vafa~\cite{BCOV:HA,BCOV:KS}.

Recall that a pseudo-opposite module for a cTP structure 
is said to be \emph{curved} if it is not preserved by $\nabla$ 
(Definition~\ref{def:opposite}). 
For a curved pseudo-opposite module $\sfQ$, 
$(\Nabla^{\sfQ})^n C^{(g)}$ is not necessarily symmetric,
because $\Nabla^{\sfQ}$ is not flat. 
The completely symmetric correlation functions 
associated to a curved pseudo-opposite module $\sfQ$ are defined in a 
different way, as follows. 
Suppose that we are given an element 
$\{C_{\sfP;\mu_1,\dots,\mu_n}^{(g)}\}\in \Fock(U; \sfP)$ 
for a parallel pseudo-opposite module $\sfP$. 
For each $t\in \cM$, there is a unique parallel pseudo-opposite 
module $\tsfQ(t)$ in the formal neighbourhood of $t$ 
such that $\tsfQ(t)_t = \sfQ_t$ (this is the parallel translation of $\sfQ_t$:
see the proof of Lemma~\ref{lem:existence_opposite}). 
From the transformation rule, we obtain correlation functions 
\[
\{C^{(g)}_{\tsfQ(t);\mu_1,\dots,\mu_n}\} = 
T(\sfP, \tsfQ(t)) \left( \{C_{\sfP; \mu_1,\dots,\mu_n}^{(g)}\}  
\right) 
\]
over the formal neighbourhood of $\LLo_t$.  
Restricting these to the fiber $\LLo_t$ and varying the point $t$,  
we obtain the correlation functions $C^{(g)}_{\sfQ;\mu_1,\dots,\mu_n}$ 
associated to $\sfQ$ 
such that 
\[
C^{(g)}_{\sfQ;\mu_1,\dots,\mu_n}\Bigr|_{\LLo_t} 
= C^{(g)}_{\tsfQ(t);\mu_1,\dots,\mu_n}\Bigr|_{\LLo_t} 
\]
for every $t$. 
Because the propagator $\Delta(\sfP,\tsfQ(t))$ 
coincides with $\Delta(\sfP,\sfQ)$ along the fiber $\LLo_t$, 
the new correlation functions $C^{(g)}_{\sfQ;\mu_1,\dots,\mu_n}$ 
can be described using \emph{the same Feynman rule}
as before. 
 
\begin{definition}
\label{def:under-curvedopposite} 
Let $\wave_\sfP= \{C^{(g)}_{\sfP;\mu_1,\dots,\mu_m}\}_{g,m}$ 
be an element of the local Fock space $\Fock(U;\sfP)$ 
associated to a parallel pseudo-opposite module $\sfP$ over $U$,
and let $\exp(\cW_{\sfP}(\sx,\sy))$ denote the corresponding jet potential 
(Definition~\ref{def:jetpotential}). 
Let $\sfQ$ be a (not necessarily parallel) pseudo-opposite 
module over $U$. 
We define a set 
$\wave_\sfQ= \{C^{(g)}_{\sfQ;\mu_1\dots,\mu_n}: \text{$g\ge 0$, $n\ge 0$, $2g-2+n>0$}\}$ 
of completely symmetric tensors 
by the same Feynman rule as appears in Definition~\ref{def:transformation}:
\[
C^{(g)}_{\sfQ;\mu_1,\dots,\mu_n} 
= \sum_{\Gamma} \frac{1}{|\Aut(\Gamma)|} 
\Cont_{\Gamma}(\wave_\sfP, \Delta(\sfP,\sfQ))
_{\mu_1,\dots,\mu_n}
\]
We write 
\[
\wave_\sfQ = T(\sfP,\sfQ) \wave_\sfP 
\]
and call $\wave_\sfQ$ the \emph{correlation functions 
under $\sfQ$ corresponding to $\wave_\sfP$}. 
The \emph{jet potential} associated to $\sfQ$ 
\begin{equation} 
\label{eq:jetpotential-curved}
\cW_{\sfQ} (\sx,\sy) = \sum_{g,n \ge 0, 2g-2+n>0} 
\frac{\hbar^{g-1}}{n!} 
C^{(g)}_{\sfQ; \mu_1,\dots,\mu_n}(\sx) \sy^{\mu_1} 
\cdots \sy^{\mu_n} 
\end{equation} 
is related to the jet potential $\cW_\sfP$ associated to $\wave_\sfP$ 
by the same formula \eqref{eq:transformationrule-jet}
as before: 
\begin{equation*} 
\exp(\cW_{\sfQ}(\sx,\sy)) = \exp\left( 
\frac{\hbar}{2} \Delta^{\mu\nu}(\sfP,\sfQ) 
\partial_{\sy^\mu} \partial_{\sy^\nu} 
\right) \exp(\cW_{\sfP}(\sx,\sy))
\end{equation*} 
\end{definition} 

\begin{proposition} 
The correlation functions under a curved pseudo-opposite module $\sfQ$ 
corresponding to a Fock space element 
satisfy the conditions (Yukawa), (Grading \& Filtration) 
and (Pole) but not necessarily the condition 
(Jetness) in Definition~\ref{def:localFock}. 
\end{proposition} 
\begin{proof} 
The proofs of Lemmas~\ref{lem:GrFil},~\ref{lem:pole}  
work for curved pseudo-opposite modules too.  
\end{proof} 

We will shortly (in Theorem~\ref{thm:anomaly} below) describe an anomaly equation that gives a substitute for the jetness condition for correlation functions under a curved pseudo-opposite module. 
The simplest case of this anomaly equation is the 
\emph{curvature condition} for the genus-one one-point function: 
the one-form $C^{(1)}_{\sfQ;\mu} d\sx^\mu$ is not necessarily 
closed, but its derivative $d (C^{(1)}_{\sfQ;\mu} d\sx^\mu)$ 
equals a certain two-form $\vartheta_\sfQ$ associated to $\sfQ$. 

\begin{lemma} 
\label{lem:curvature2form} 
Let $\sfP$ be a parallel pseudo-opposite module 
and $\sfQ$ be a (not necessarily parallel) pseudo-opposite module. 
Let $\omega_{\sfP\sfQ}$ denote the difference one-form 
\eqref{eq:difference1-form} between $\sfP$ and $\sfQ$. 
The two-form $\vartheta_\sfQ = d \omega_{\sfP \sfQ}$ does not 
depend on the choice of a parallel $\sfP$ 
and vanishes if $\sfQ$ is parallel. 
\end{lemma} 

\begin{proof}
When both $\sfP$ and $\sfQ$ are parallel, $\omega_{\sfP\sfQ}$ 
arises as the difference \eqref{eq:Feynman-genusone} 
of the genus-one 1-forms $C^{(1)}_\mu d\sx^\mu$, 
which are closed. 
More directly, by Proposition~\ref{prop:difference_conn}(2), 
we have 
\[
2 \left( \Nabla^{\sfP} \omega_{\sfP\sfQ} \right)_{\mu\nu}= 
(\Nabla^{\sfP}_\mu C^{(0)}_{\nu\rho\tau}) 
\Delta^{\rho\tau} 
+ C^{(0)}_{\nu\rho\tau} (\Nabla^{\sfP}_\mu \Delta^{\rho\tau})
= C^{(0)}_{\mu\nu\rho\tau} \Delta^{\rho\tau}
+ C^{(0)}_{\nu\rho\tau} \Delta^{\rho\sigma} 
C^{(0)}_{\sigma\mu\lambda} \Delta^{\lambda\tau}  
\] 
where $\Delta = \Delta(\sfP,\sfQ)$. 
This 2-tensor is symmetric with respect to $\mu$ and $\nu$; 
thus  $\omega_{\sfP\sfQ}$ is closed. 
Because $\omega_{\sfP\sfQ} - \omega_{\sfP' \sfQ} = 
\omega_{\sfP\sfP'}$, it follows that $d\omega_{\sfP\sfQ}$ 
does not depend on the choice of parallel $\sfP$. 
\end{proof} 

\begin{definition}
\label{def:curvature2form} 
The two-form $\vartheta_\sfQ := 
d\omega_{\sfP\sfQ} \in \pr^*\Omega_\cM^2$ in the 
above Lemma is called the \emph{curvature two-form} of $\sfQ$. 
This is the pull-back of a two-form on $\cM$. 
We will give an explicit and intrinsic formula 
in \eqref{eq:curvature-Lambda} below. 
\end{definition} 

\begin{proposition}[curvature condition] 
\label{prop:curvature_condition} 
For a genus-one one point function $C^{(1)}_{\sfQ;\mu} d\sx^\mu$ 
under $\sfQ$ corresponding to a Fock space element 
in $\Fock(U;\sfP)$ with $\sfP$ a parallel pseudo-opposite module, 
we have:
\[
d(C^{(1)}_{\sfQ;\mu} d\sx^\mu ) = \vartheta_\sfQ
\]
\end{proposition} 
\begin{proof} 
This follows from the Feynman rule at genus zero: 
$C^{(1)}_{\sfQ;\mu} d\sx^\mu = 
C^{(1)}_{\sfP;\mu} d\sx^\mu + \omega_{\sfP\sfQ}$ 
(see equation~\ref{eq:Feynman-genusone}) and 
the definition of $\vartheta_\sfQ$. 
\end{proof} 

Let $\sfP$ be a parallel pseudo-opposite module and let $\sfQ$ be 
a possibly curved pseudo-opposite module. 
An element of $\Fock(U;\sfP)$ 
induces correlation functions under $\sfQ$. 
Conversely, an element of $\Fock(U;\sfP)$ 
can be uniquely reconstructed from a genus-one one-point function 
and higher-genus zero-point functions under $\sfQ$. 

\begin{proposition} 
\label{prop:Fockspace-curved} 
Let $\sfQ$ be a pseudo-opposite module over $U$. 
Assume that we have a one-form 
$C^{(1)}_{\sfQ;\mu}d\sx^\mu \in \bOmega^1$ and 
functions 
$C^{(g)}_{\sfQ}\in \bcO$ 
for $g\ge 2$ over $\pr^{-1}(U)^\circ$ 
satisfying the following conditions: 
\begin{itemize} 
\item \emph{(Grading \& Filtration)}  
$C^{(1)}_{\sfQ;\mu}d\sx^\mu \in (\bOmega^1)^0_0$; 
$C^{(g)}_{\sfQ}\in \bcO^{2-2g}_{3g-3}$;  

\item \emph{(Curvature)} $d (C^{(1)}_{\sfQ;\mu} d \sx^\mu) = \vartheta_\sfQ$, 
where  $\vartheta_\sfQ$ is the curvature two-form 
in Definition~\ref{def:curvature2form}; 

\item \emph{(Pole)} $P ( C^{(1)}_{\sfQ;\mu} d \sx^\mu) $ 
extends to a regular 1-form on $\pr^{-1}(U)$; 
for $g\ge 2$, $C^{(g)}_{\sfQ} \in P^{-(5g-5)} 
\cO(U)[x_1,x_2,P x_3, P^2 x_4, \dots,P^{3g-4}x_{3g-2}]$;
\end{itemize} 
where $P = P(t,x_1)$ is the discriminant \eqref{eq:discriminant}. 
For a parallel pseudo-opposite module $\sfP$ over $U$, there exists a 
unique Fock space element 
$\wave_\sfP= \{C^{(g)}_{\sfP;\mu_1,\dots,\mu_n}\}_{g,n} \in 
\Fock(U;\sfP)$ such that 
\begin{itemize} 
\item $C^{(1)}_{\sfQ;\mu}$ is  
the genus-one one-point correlation function under $\sfQ$ 
corresponding to $\wave_\sfP$; 

\item For $g \ge 2$, $C^{(g)}_{\sfQ}$ is 
the genus-$g$ zero-point correlation function under $\sfQ$ 
corresponding to $\wave_\sfP$. 
\end{itemize} 
The formula 
\[
\wave_\sfQ = T(\sfP,\sfQ) \wave_\sfP 
\]
reconstructs the multi-point correlation functions 
$\wave_{\sfQ} = \{ C^{(g)}_{\sfQ;\mu_1,\dots,\mu_n}\}$ under $\sfQ$ 
that satisfy the conditions (Yukawa), (Grading \& Filtration) 
and (Pole) in Definition~\ref{def:localFock}. 
The multi-point correlation functions $C^{(g)}_{\sfQ;\mu_1,\dots,\mu_n}$ 
are independent of the choice of $\sfP$. 
\end{proposition} 
\begin{proof} 
We solve for the Fock space element 
$\wave_{\sfP}= \{C^{(h)}_{\sfP;\nu_1,\dots,\nu_m}\}_{h,m}$ 
satisfying the Feynman rule (see equation~\ref{eq:Feynman-genusone}) 
\begin{align*} 
C^{(1)}_{\sfQ;\mu}d\sx^\mu & = 
C^{(1)}_{\sfP;\mu} d\sx^\mu
+ \omega_{\sfP\sfQ} \\ 
C^{(g)}_{\sfQ} & = \sum_{\Gamma} \frac{1}{|\Aut(\Gamma)|} 
\Cont_\Gamma(\wave_\sfP, 
\Delta(\sfP, \sfQ)) 
\end{align*} 
inductively on the genus and the number of insertions. 
Imposing the jetness: 
\begin{align*} 
(\Nabla^\sfP)^{n-1}  C^{(1)}_{\sfP;\mu} d\sx^{\mu} 
&= C^{(1)}_{\sfP;\mu_1,\dots,\mu_n} 
d\sx^{\mu_1} \otimes \cdots \otimes d\sx^{\mu_n} 
\\ 
(\Nabla^\sfP)^n C^{(g)}_{\sfP} 
& =  C^{(g)}_{\sfP;\mu_1,\dots,\mu_n} 
d\sx^{\mu_1}  \otimes \cdots \otimes d\sx^{\mu_n} \quad 
(g\ge 2) 
\end{align*} 
and the condition 
$C^{(0)}_{\sfP;\mu_1,\dots,\mu_n} 
=(\Nabla^{\sfP})^{n-3}\bY$, 
we can uniquely determine the symmetric tensors  
$C^{(g)}_{\sfP;\mu_1,\dots,\mu_n}$.  
The genus-one tensors 
$C^{(1)}_{\sfP;\mu_1,\dots,\mu_n}$ 
become completely symmetric by the curvature condition 
$d(C^{(1)}_{\sfQ;\mu} d\sx^\mu) = \vartheta_{\sfQ}
= d \omega_{\sfP\sfQ}$. 
It remains to check that the reconstructed 
correlation functions $C^{(g)}_{\sfP;\mu_1,\dots,\mu_n}$ 
satisfy the conditions (Grading \& Filtration) and (Pole) 
in Definition~\ref{def:localFock}. 
At genus one, $C^{(1)}_{\sfP;\mu} d\sx^\mu$ 
satisfies (Grading \& Filtration) and (Pole) 
because so does $C^{(1)}_{\sfQ;\mu} d\sx^\mu$. 
Note that the conditions (Grading \& Filtration) are
stable under $\Nabla^{\sfP}$ 
by Proposition~\ref{prop:grading-filtration-Nabla}.  
Suppose that (Grading \& Filtration) and 
(Pole) are satisfied up to genus $g-1$. 
We can write $C^{(g)}_{\sfP}$ as the 
sum of $C^{(g)}_{\sfQ}$ and the Feynman graph 
contributions from lower genus $n$-point functions 
$C^{(h)}_{\sfP;\mu_1,\dots,\mu_n}$. 
Therefore the argument of Lemmas~\ref{lem:GrFil} and~\ref{lem:pole} applies here too. 

Finally we check that $\wave_\sfQ$ is independent of the choice 
of parallel $\sfP$. Suppose we have two parallel pseudo-opposite 
modules $\sfP_1,\sfP_2$. The above procedure gives 
two Fock space elements $\wave_{\sfP_1} \in \Fock(U;\sfP_1)$, 
$\wave_{\sfP_2} \in \Fock(U;\sfP_2)$. 
Then $T(\sfP_1,\sfQ) \wave_{\sfP_1}$ and $T(\sfP_2,\sfQ)\wave_{\sfP_2}$ 
have the same genus-one one point functions and 
higher-genus zero-point functions. On the other hand we can write 
$T(\sfP_2,\sfQ) \wave_{\sfP_2} = T(\sfP_1,\sfQ) T(\sfP_2,\sfP_1) 
\wave_{\sfP_2}$ by the cocycle condition for the transformation rule 
(Proposition~\ref{prop:cocycle}). The above reconstruction procedure 
implies that $\wave_{\sfP_1} = T(\sfP_2,\sfP_1) \wave_{\sfP_2}$. 
Therefore $T(\sfP_1,\sfQ) \wave_{\sfP_1} = T(\sfP_2,\sfQ)\wave_{\sfP_2}$. 
\end{proof} 

\begin{remark} 
Since an opposite module exists in the formal neighbourhood of every 
point $t\in \cM$ (Lemma~\ref{lem:existence_opposite}), the reconstruction of 
multi-point correlation functions satisfying (Curvature), 
(Pole), and (Grading \& Filtration) from the data $\{C^{(1)}_{\sfQ;\mu}, 
C^{(2)}_{\sfQ},C^{(3)}_{\sfQ}, \dots\}$ is always possible, even without 
Assumption~\ref{assump:covering}. 
\end{remark} 

In view of the above Proposition, we make the following definition 
for the local Fock space with respect to a possibly curved 
opposite module (cf.~Definition~\ref{def:localFock}). 
This definition does not rely on Assumption~\ref{assump:covering}. 

\begin{definition}[local Fock space and transformation rule: general case] 
\label{def:localFock_transformation_general}
Let $\sfQ$ be a (not necessarily parallel) pseudo-opposite module 
over an open set $U\subset \cM$. 
The \emph{local Fock space} $\Fock(U;\sfQ)$ consists of 
collections 
$\{C^{(1)}_{\sfQ;\mu}, C^{(2)}_{\sfQ}, C^{(3)}_{\sfQ}, \dots\}$ 
satisfying the conditions (Curvature), (Grading \& Filtration),
and (Pole) in Proposition~\ref{prop:Fockspace-curved}, 
where 
\[
C^{(1)}_{\sfQ;\mu}d\sx^\mu \in \bOmega^1(\pr^{-1}(U)^\circ )  
\quad 
\text{and} 
\quad   
C^{(g)}_{\sfQ}\in \bcO (\pr^{-1}(U)^\circ ) \quad g \ge 2.  
\] 
Proposition~\ref{prop:Fockspace-curved} allows us to reconstruct 
\emph{multi-point correlation functions}  
$C^{(g)}_{\sfQ;\mu_1,\dots,\mu_n}$  
from the data $\{C^{(1)}_{\sfQ;\mu},C^{(2)}_\sfQ, 
C^{(3)}_\sfQ,\dots\}$ 
and they define the associated jet potential 
$\cW_\sfQ(\sx,\sy)$ as in \eqref{eq:jetpotential-curved}. 
For two pseudo-opposite modules $\sfQ_1$, $\sfQ_2$ over 
$U$, the \emph{transformation rule} $T(\sfQ_1,\sfQ_2) \colon 
\Fock(U;\sfQ_1) \to \Fock(U;\sfQ_2)$ is defined 
in terms of  jet potentials (reconstructed thus)
in the same way as Definition~\ref{def:transformation}: 
\[
\exp(\cW_{\sfQ_2}(\sx,\sy)) = \exp\left( 
\frac{\hbar}{2} \Delta^{\mu\nu}(\sfQ_1,\sfQ_2)
\partial_{\sy^\mu} \partial_{\sy^\nu} 
\right) \exp(\cW_{\sfQ_1}(\sx,\sy))
\]
This can be also described by the Feynman rule in 
Definition~\ref{def:transformation}. 
\end{definition}

\begin{remark} 
For parallel $\sfQ$, the above definition reduces 
to the original definitions of local Fock spaces and the 
transformation rule. 
Multi-point correlation functions under a parallel $\sfQ$ 
can be obtained by the covariant derivative $\Nabla^{\sfQ}$ 
from zero-point correlation functions. 
The transformation rule for general pseudo-opposite 
modules satisfies the cocycle condition: 
the same proof as Proposition~\ref{prop:cocycle} works. 
\end{remark} 

Let $\sfQ$ be a possibly curved pseudo-opposite module 
over $U$. 
The reconstruction of multi-point correlation functions 
$C^{(g)}_{\sfQ;\mu_1,\dots,\mu_n}$ from 
the data $\{C^{(1)}_{\sfQ;\mu}, C^{(2)}_{\sfQ}, C^{(3)}_{\sfQ}, \dots\}$ 
in Proposition~\ref{prop:Fockspace-curved} was implicit. 
We can describe it in a more explicit and intrinsic way,
without reference to a parallel pseudo-opposite module.  
For this purpose, we introduce the following ``background torsion''
$\Lambda_\sfQ$ associated to $\sfQ$. 

\begin{definition} 
\label{def:background torsion} 
Let $\sfQ$ be a pseudo-opposite module.  
The (background) \emph{torsion} of $\sfQ$ is 
an operator $\Lambda_{\sfQ} \colon 
\bOmegao^1 \times \bOmegao^1 \to \bOmegao^1$ 
defined by 
\[
\Lambda_{\sfQ} (\omega_1, \omega_2) = 
\Omega^\vee(\tnabla^\vee \Pi^* \varphi_1, 
\Pi^* \varphi_2) 
\]
where $\varphi_i := (\KS^*)^{-1} \omega_i$, $i\in\{1,2\}$, 
and $\Pi\colon \pr^*\sfF[z^{-1}] \to \pr^*\sfF$ is the projection 
along $\sfQ$. 
Recall that $\tnabla^\vee\colon \pr^*\sfF[z^{-1}]^\vee 
\to \bOmega^1 \hotimes \pr^*\sfF[z^{-1}]^\vee$ 
is the connection dual to $\tnabla$ (see \S\ref{subsec:conn-LL}). 
\end{definition} 

\begin{lemma}\ 
\label{lem:background torsion}
\begin{enumerate}
\item The operator $\Lambda_\sfQ(\omega_1,\omega_2)$ 
is symmetric, $\bcO$-bilinear, and takes values in 
$\pr^*\Omega_{\cM}^1$. 
\item  A pseudo-opposite module $\sfQ$ is parallel if and only if 
$\Lambda_\sfQ =0$. 
\end{enumerate}
\end{lemma} 
\begin{proof} 
Write $\varphi_i := (\KS^*)^{-1}\omega_i$. 
Because $\Omega^\vee(\Pi^* \varphi_1, \Pi^* \varphi_2)=0$, 
we have:
\[
0 = d \Omega^\vee(\Pi^* \varphi_1, \Pi^* \varphi_2) 
= \Omega^\vee(\tnabla^\vee \Pi^* \varphi_1, \Pi^* \varphi_2) 
+ \Omega^\vee(\Pi^* \varphi_1, \tnabla^\vee \Pi^*\varphi_2) 
\]
The right-hand side equals $\Lambda_\sfQ(\omega_1,\omega_2) 
- \Lambda_{\sfQ}(\omega_2,\omega_1)$. 
By definition $\Lambda_\sfQ$ is $\bcO$-linear in $\omega_2$. 
Thus it is also $\bcO$-linear in $\omega_1$. 
Note that for a local co-ordinate system $\{t^i,x_n^i\}_{n\ge 1,0\le i\le N}$ 
on $\LL$, 
we have $\tnabla^\vee_{n,i} 
(\pr^*\sfQ)^\perp \subset (\pr^*\sfQ)^\perp$ 
for $n\ge 1$, where we write $\tnabla^\vee_{n,i} := 
\tnabla^{\vee}_{\partial /\partial x_n^i}$.  
This is because $\pr^*\sfQ$ is ``constant'' along the fiber 
of $\pr \colon \LL \to \cM$. 
Thus  
\[ 
\pair{\Lambda_\sfQ(\omega_1,\omega_2)}{\partial_{n,i}} =  
\Omega^\vee(\tnabla^\vee_{n,i} \Pi^* \varphi_1, \Pi^* \varphi_2) 
=0 \qquad n\ge 1 
\]
since $\Pi^*\varphi_1 \in (\pr^*\sfQ)^\perp 
\subset \pr^*\sfF[z^{-1}]^\vee$.  
This proves Part (1). 
Note that $\sfQ$ is parallel if and only if 
$\nabla^\vee$ preserves 
$\sfQ^\perp = \Pi^*\sfF^\vee$. 
This happens if and only if 
$\nabla^\vee (\Pi^*\sfF^\vee)$ 
is perpendicular to 
$\Pi^*\sfF^\vee$ with respect to $\Omega^\vee$,
since $\Pi^* \sfF^\vee$ is maximally isotropic. 
Part (2) follows. 
\end{proof} 

We use the following co-ordinate expression: 
\[
\Lambda_\sfQ (d\sx^\mu, d\sx^\nu) = 
{\Lambda_{\sfQ}}_{\rho}^{\mu\nu} d\sx^\rho 
= {\Lambda_{\sfQ}}_{i}^{\mu\nu} dt^i 
\]
where $\{\sx^\mu\}=\{t^i,x_n^i\}$ 
is a local co-ordinate system on $\LL$ and 
$\{t^i\}_{i=0}^N$ is a local co-ordinate system on $\cM$. 
We need to generalize Propositions~\ref{pro:prop-elementary} and~\ref{prop:difference_conn} to curved pseudo-opposite modules.  

\begin{proposition}[cf.\ Propositions~\ref{pro:prop-elementary},~\ref{prop:difference_conn}]
\label{prop:propagator-curved} 
Let $\sfQ_1$, $\sfQ_2$ be possibly curved pseudo-opposite modules 
and $\Delta = \Delta(\sfQ_1,\sfQ_2)$ be the propagator. 
\begin{enumerate}
\item We have:
\[d \Delta(\omega_1,\omega_2) = \Delta(\Nabla^{\sfQ_1} \omega_1, 
\omega_2) + \Delta(\omega_1, \Nabla^{\sfQ_2} \omega_2) 
+ \Lambda_{\sfQ_1}(\omega_1,\omega_2) - 
\Lambda_{\sfQ_2}(\omega_1, \omega_2)
\]
\item We have 
\[
\Nabla^{\sfQ_1}_\mu \Delta^{\nu\rho} 
(:= 
\partial_\mu \Delta^{\nu\rho} +  
{\Gamma^{(1)}}^\nu_{\mu\sigma} \Delta^{\sigma\rho} 
+ {\Gamma^{(1)}}^\rho_{\mu\sigma} \Delta^{\nu\sigma})  
= 
{\Lambda_{\sfQ_1}}_{\mu}^{\nu\rho} 
- {\Lambda_{\sfQ_2}}_{\mu}^{\nu\rho}
+ \Delta^{\nu\sigma} C^{(0)}_{\sigma \mu \tau} \Delta^{\tau\rho}
\]
where ${\Gamma^{(1)}}^\nu_{\mu\rho}$ 
are Christoffel coefficients of $\Nabla^{\sfQ_1}$ (see 
equation~\ref{eq:Christoffel}). 
\end{enumerate}
\end{proposition} 
\begin{proof} 
Part (1) is essentially shown in the proof of Proposition~\ref{pro:prop-elementary}. 
In fact, this formula appears in \eqref{eq:dDelta}. 
The last two terms of \eqref{eq:dDelta}, which vanish there, 
correspond to $\Lambda_{\sfQ_1}(\omega_1,\omega_2) - 
\Lambda_{\sfQ_2}(\omega_1,\omega_2)$.  

Part (2) is also similar to the proof of Proposition~\ref{prop:difference_conn}(2). 
Using Part (1), we have 
\begin{align*} 
(\Nabla^{\sfQ_1} \Delta) (\omega_1,\omega_2) 
& = d \Delta(\omega_1,\omega_2) 
- \Delta (\Nabla^{\sfQ_1}\omega_1,\omega_2) 
- \Delta( \omega_1, \Nabla^{\sfQ_1} \omega_2) \\ 
& = \Lambda_{\sfQ_1}(\omega_1,\omega_2) 
- \Lambda_{\sfQ_2}(\omega_1,\omega_2) 
+ \Delta(\omega_1, (\Nabla^{\sfQ_2} -\Nabla^{\sfQ_1}) \omega_2).  
\end{align*}
The conclusion follows from Proposition~\ref{prop:difference_conn}(1).  
\end{proof} 

Let $\sfQ$ be a possibly curved pseudo-opposite module 
and $\sfP$ be a parallel pseudo-opposite module, both defined over $U$. 
Let $\wave_{\sfQ} =\{C^{(g)}_{\sfQ;\mu_1,\dots,\mu_n}\}$ 
be the correlation functions under $\sfQ$ corresponding 
to an element $\wave_{\sfP} =
\{C^{(g)}_{\sfP;\mu_1,\dots,\mu_n}\} \in \Fock(U;\sfP)$ 
(see  Definition~\ref{def:under-curvedopposite}). 
By differentiating the Feynman rule expressing 
$C^{(g)}_{\sfQ;\mu_1,\dots,\mu_n}$ in terms of 
$C^{(h)}_{\sfP; \nu_1,\dots,\nu_m}$ and 
$\Delta(\sfP, \sfQ)$, 
we obtain the following \emph{anomaly equation}. 
\begin{theorem}[anomaly equation] 
\label{thm:anomaly} 
Multi-point correlation functions under a possibly 
curved pseudo-opposite module $\sfQ$ satisfy the 
following anomaly equation (see equation~\ref{eq:covariant-derivative} for the notation
for covariant derivatives): 
\begin{equation}
\label{eq:anomalyequation} 
C^{(g)}_{\sfQ;\mu_1 \dots \mu_n} 
= \Nabla_{\mu_1}^{\sfQ} C^{(g)}_{\sfQ;\mu_2\dots \mu_n} 
+ \frac{1}{2} 
\sum_{\substack{\{2,\dots,n\} = I \sqcup J \\ 
k+ l = g}} 
C^{(k)}_{\sfQ; \mu_{I}, \alpha} 
{\Lambda_{\sfQ}}_{\mu_1}^{\alpha\beta}
C^{(l)}_{\sfQ; \mu_{J}, \beta}  
+ \frac{1}{2} 
C^{(g-1)}_{\sfQ;\mu_2\dots\mu_n\alpha\beta} 
{\Lambda_{\sfQ}}^{\alpha\beta}_{\mu_1}
\end{equation}  
where $\mu_I$ stands for $\mu_{i_1},\dots,\mu_{i_p}$ 
if $I = \{i_1,\dots,i_p\}$ and similarly for $\mu_J$.  
\end{theorem} 
\begin{proof} 
The argument is very similar to the proof of (Jetnesss) in Lemma~\ref{lem:jetness}. We have
\[
\Nabla^{\sfQ}_\nu C^{(g)}_{\sfQ;\mu_1,\dots,\mu_n} 
= \Nabla^{\sfP}_\nu C^{(g)}_{\sfQ;\mu_1,\dots,\mu_n} 
+ 
(\Nabla^{\sfQ}- \Nabla^{\sfP})_\nu C^{(g)}_{\sfQ;\mu_1,\dots,\mu_n}  
\]
 (cf.~equation~~\ref{eq:NablaprimehC}).
The second term here corresponds to the modification of legs 
depicted in \eqref{eq:leg-modification},
by Proposition~\ref{prop:difference_conn}(1),  
and the first term is the sum of the vertex derivative 
$C_{\rm vert}$ and the propagator derivative $C_{\rm prop}$. 
The vertex derivative $C_{\rm vert}$ is the same as in 
Lemma~\ref{lem:jetness}, 
but the propagator derivative $C_{\rm prop}$ 
has extra contributions from $-\Lambda_{\sfQ}$ 
because  
$\Nabla^{\sfP}_\mu \Delta(\sfP,\sfQ)^{\nu\rho} 
= - {\Lambda_{\sfQ}}_{\mu}^{\nu\rho} + 
\Delta(\sfP,\sfQ)^{\nu\sigma} C^{(0)}_{\sigma \mu \tau}
\Delta(\sfP,\sfQ)^{\tau\rho}$
by Proposition~\ref{prop:propagator-curved}(2). 
Hence the difference 
from the computation in Lemma~\ref{lem:jetness} arises 
from the insertion of $-\Lambda_{\sfQ}$ at internal edges. 
The second and third term on the right-hand side of \eqref{eq:anomalyequation}
correspond respectively to the cases where (i) the chosen edge 
separates the graph or (ii) the chosen edge does not 
separate the graph. The factor $1/2$ comes from  
automorphisms exchanging the two branches 
of the edge. 
\end{proof} 

\begin{remark} 
\label{rem:Fockspace-curved} 
The anomaly equation gives a substitute for (Jetness) 
for correlation functions under a curved pseudo-opposite 
module. 
Note that the parallel pseudo-opposite module $\sfP$ does not appear 
explicitly in the anomaly equation. 
Therefore we can define the local Fock space for $\sfQ$ 
as the set of symmetric tensors $\{C_{\sfQ;\mu_1,\dots,\mu_n}^{(g)} 
: \text{$g\ge 0$, $n\ge 0$, $2g-2+n>0$}\}$ satisfying the conditions 
(Yukawa), (Grading \& Filtration), (Pole) in Definition~\ref{def:localFock} 
and the anomaly equation \eqref{eq:anomalyequation}.  The condition (Curvature) is contained in the anomaly equation: 
see Remark~\ref{rem:curvature}.
\end{remark}

\begin{example} 
The anomaly equation allows us to calculate 
$C^{(g)}_{\sfQ;\mu_1\dots\mu_n}$ 
iteratively in terms of $C^{(h)}_{\sfQ}$, $h\le g$, 
$C^{(1)}_{\sfQ;\nu} d\sx^\nu$, 
$C^{(0)}_{\tau \rho \sigma}$ and 
their $\Nabla^{\sfQ}$-derivatives. 
We also need $\Lambda_\sfQ$ and its derivatives 
in the iteration process. 
For example: 
\allowdisplaybreaks
\begin{align*} 
C^{(0)}_{1234} & 
= \Nabla_1 C^{(0)}_{123} \\ 
C^{(0)}_{12345} & 
= \Nabla_1 \Nabla_2 C^{(0)}_{345} 
+\left[ C^{(0)}_{23\alpha} \Lambda^{\alpha\beta}_1 
C^{(0)}_{45\beta} 
+ (2\leftrightarrow 4) + (2 \leftrightarrow 5)\right] \\ 
C^{(0)}_{123456} & 
=  
\Nabla_1\Nabla_2 \Nabla_3 C^{(0)}_{456}  \\ 
& + \frac{1}{2} 
\sum_{\{3,4,5,6\}=I \sqcup J} 
(\Nabla_1 C^{(0)}_{I \alpha}) 
\Lambda^{\alpha\beta}_2 
C^{(0)}_{J \beta} 
+ C^{(0)}_{I \alpha}  
(\Nabla_1 \Lambda^{\alpha\beta}_2)  
C^{(0)}_{J \beta} 
+ C^{(0)}_{I \alpha}
\Lambda^{\alpha\beta}_2 
(\Nabla_1C^{(0)}_{J \beta})  \\
& + 
\sum_{\{2,3,4,5,6\} = I \sqcup J, 
\, |I| =3,\, |J| =2} 
C^{(0)}_{I\alpha} \Lambda^{\alpha\beta}_1 
C^{(0)}_{J\beta} \\
C^{(1)}_{12} &
= \Nabla_1 C^{(1)}_{2} 
+ \frac{1}{2} C^{(0)}_{2\alpha\beta} \Lambda^{\alpha\beta}_1 \\ 
C^{(1)}_{123} & 
= \Nabla_1\Nabla_2 C^{(1)}_3 
+ \frac{1}{2} (\Nabla_1 C^{(0)}_{3\alpha\beta}) 
\Lambda_2^{\alpha\beta} 
+ \frac{1}{2} 
C^{(0)}_{3\alpha\beta} (\Nabla_1 \Lambda_2^{\alpha\beta}) 
+ C^{(1)}_{\alpha} \Lambda^{\alpha\beta}_1 C^{(0)}_{23\beta} 
 + \frac{1}{2}(\Nabla_2 C^{(0)}_{3\alpha\beta}) \Lambda^{\alpha\beta}_1 
 \\ 
C^{(2)}_1 & 
= \Nabla_1 C^{(2)} 
+ \frac{1}{2} 
C^{(1)}_\alpha \Lambda^{\alpha\beta}_1 C^{(1)}_\beta
\end{align*}
where we omit the super/subscript $\sfQ$ 
on $\Nabla$, $\Lambda$ and $C^{(g)}_{\mu_1,\dots,\mu_n}$. 
(Here we used the numbers $1,2,3,4,5,6$ in place of 
small Greek letters; 
$\Nabla_1\Nabla_2 C^{(0)}_{345}$ denotes the $d\sx^1\otimes 
d\sx^2 \otimes d\sx^3\otimes d\sx^4 \otimes d\sx^5$-component 
of $\Nabla^2 \bY$.) 
It is not obvious that these formulas give symmetric tensors 
$C^{(g)}_{\sfQ;\mu_1,\dots,\mu_n}$, but this is ensured 
by general theory. 
\end{example}

We now calculate the curvature of the connection $\Nabla^\sfQ$ 
and relate it to the curvature two-form $\vartheta_\sfQ$ 
(Definition~\ref{def:curvature2form}) 
which appears in the condition (Curvature) in Proposition~\ref{prop:Fockspace-curved}. 

\begin{proposition}[curvature] 
\label{prop:curvature} 
Let $\sfQ$ be a pseudo-opposite module 
and let $\Lambda = \Lambda_{\sfQ}$ be the torsion of $\sfQ$. 
Let $(\Nabla^{\sfQ})^2$ denote the curvature of 
$\Nabla^{\sfQ}$ on $\bOmegao^1$, which is an 
$\End(\bOmegao^1)$-valued $2$-form on $\LLo$. 
(Note that it is minus the transpose of the curvature on the tangent bundle.)
We also use a local co-ordinate system $\{t^i\}_{i=0}^N$ on $\cM$ 
which has Roman letters as indices. 
\begin{enumerate}
\item The curvature of $\Nabla^{\sfQ}$ is given by: 
\[
(\Nabla^{\sfQ})^2 d\sx^\nu= 
C^{(0)}_{\mu_1 \rho \tau} \Lambda^{\tau\nu}_{\mu_2} (d\sx^{\mu_1} 
\wedge d\sx^{\mu_2} )\otimes d\sx^\rho
= C^{(0)}_{i j k} \Lambda^{k \nu}_l (dt^i \wedge dt^l) \otimes dt^j
\]
\item The curvature two-form $\vartheta_\sfQ$ is 
half of the trace of $(\Nabla^{\sfQ})^2$: 
\begin{align} 
\label{eq:curvature-Lambda}
\begin{split}  
\vartheta_\sfQ & = \frac{1}{2} \Tr((\Nabla^{\sfQ})^2) 
= \frac{1}{2} 
C^{(0)}_{ijk} \Lambda^{jk}_l dt^i \wedge dt^l \\
& = -\frac{1}{4} 
\sum_{a=0}^N \sum_{b=0}^N 
\Tr_{\sfF_0}\left(\Pi_\sfQ \nabla_a \Pi_{\sfQ} 
\nabla_b - \Pi_{\sfQ} \nabla_b \Pi_{\sfQ} \nabla_a\right)  
dt^a \wedge dt^b.  
\end{split}
\end{align}  
The trace here makes sense since 
$(\Nabla^{\sfQ})^2 d\sx^\nu$ above lies in the finite rank 
subbundle $\pr^*\Omega^2_{\cM} \otimes \pr^*\Omega^1_{\cM}$. 
\end{enumerate}
\end{proposition} 

\begin{proof} 
By Proposition~\ref{prop:difference_conn}(1), 
for a reference parallel pseudo-opposite module $\sfP$, we have 
\[
\Nabla^\sfQ_\mu d\sx^\nu = 
\Nabla^\sfP_\mu d\sx^\nu + 
C^{(0)}_{\mu\rho\sigma} \Delta^{\sigma\nu}  d\sx^\rho 
\]
where $\Delta =\Delta(\sfP,\sfQ)$. 
Because $\Nabla^\sfP$ is flat, one can calculate 
the curvature of $\Nabla^\sfQ$ by regarding the tensor 
$C^{(0)}_{\mu\rho\sigma} \Delta^{\sigma\nu}$ 
as the Christoffel symbol: 
\begin{align*} 
&[\Nabla^{\sfQ}_{\mu_1}, \Nabla^{\sfQ}_{\mu_2}] d\sx^\nu \\ 
&= \left[ 
\Nabla^\sfP_{\mu_1} 
(C^{(0)}_{\mu_2\rho\sigma} \Delta^{\sigma\nu}) 
- \Nabla^\sfP_{\mu_2} 
(C^{(0)}_{\mu_1\rho\sigma} \Delta^{\sigma\nu}) 
+ C^{(0)}_{\mu_1\rho \sigma} \Delta^{\sigma\tau} 
C^{(0)}_{\mu_2 \tau\sigma'} \Delta^{\sigma' \nu} 
- C^{(0)}_{\mu_2\rho \sigma} \Delta^{\sigma\tau} 
C^{(0)}_{\mu_1 \tau\sigma'} \Delta^{\sigma' \nu} 
\right]  d\sx^\rho
\end{align*} 
This formula can be easily shown when $\sx^\mu$ are flat co-ordinates 
with respect to $\Nabla^\sfP$. 
Then observe that the right-hand side is tensorial 
with respect to $\mu_1,\mu_2,\rho,\nu$. 
By Proposition~\ref{prop:propagator-curved}, we have 
$\Nabla^{\sfP}_\mu \Delta^{\nu\rho} = -\Lambda_\mu^{\nu\rho} 
+ \Delta^{\nu\tau} C^{(0)}_{\tau \mu \sigma} \Delta^{\sigma\rho}$. 
Using this we arrive at:
\[
[\Nabla^\sfQ_{\mu_1}, \Nabla^{\sfQ}_{\mu_2}] d\sx^\nu 
= (C^{(0)}_{\mu_1\rho\tau} \Lambda^{\tau \nu}_{\mu_2} 
- C^{(0)}_{\mu_2 \rho \tau} \Lambda^{\tau\nu}_{\mu_1}) d\sx^\rho
\]
This proves Part (1). 
Because $\Nabla^{\sfP}$ is torsion-free, 
$\vartheta_\sfQ = d \omega_{\sfP\sfQ}$ is the 
anti-symmetrization of $\Nabla^{\sfP} \omega_{\sfP\sfQ}
\in (\bOmegao^1)^{\otimes 2}$, i.e.\ (see equation~\ref{eq:difference1-form}):  
\[
d\omega_{\sfP\sfQ} = 
\frac{1}{2} 
\Nabla^{\sfP}_\sigma (C^{(0)}_{\mu\nu\rho} \Delta^{\nu\rho}) 
d\sx^\sigma \wedge d\sx^\mu
\]
The first line of \eqref{eq:curvature-Lambda} follows from 
this and $\Nabla^{\sfP}_\mu \Delta^{\nu\rho} = -\Lambda_\mu^{\nu\rho} 
+ \Delta^{\nu\tau} C^{(0)}_{\tau \mu \sigma} \Delta^{\sigma\rho}$. 
To see the second line of \eqref{eq:curvature-Lambda}, note that 
the trace of $(\Nabla^\sfQ)^2$ on $\bOmegao^1$ is minus  
the trace of $(\Nabla^\sfQ)^2$ on $\bThetao$, and therefore 
is minus the trace of the curvature of the connection 
$\Pi_\sfQ \tnabla$ on $\pr^*\sfF$. 
On the other hand, the operator $\Pi_\sfQ \nabla_a \Pi_{\sfQ} 
\nabla_b - \Pi_{\sfQ} \nabla_b \Pi_{\sfQ} \nabla_a$ vanishes 
on $z\sfF$ (since $\nabla_a \nabla_b =\nabla_b \nabla_a$) 
and defines an $\cO_\cM$-linear endomorphism of 
$\sfF_0 = \sfF/z\sfF$; this means that the trace of 
the curvature operator $\Pi_\sfQ \nabla_a \Pi_\sfQ \nabla_b - \Pi_\sfQ \nabla_b 
\Pi_\sfQ \nabla_a$ on $\sfF$ is well-defined and coincides with 
the trace of the induced operator on $\sfF_0$. 
\end{proof} 

\begin{remark} 
\label{rem:curvature}
Part (2) of Proposition~\ref{prop:curvature} says, heuristically, that 
one can think of $\vartheta_\sfQ$ as the curvature of a 
line bundle ``$\det(\bOmegao^1)^{1/2}$".  
The anomaly equation in Theorem~\ref{thm:anomaly} at genus one gives:
\[
C^{(1)}_{\mu\nu} = \Nabla^{\sfQ}_\mu C^{(1)}_\nu 
+ \frac{1}{2} C^{(0)}_{\alpha\beta \nu} 
{\Lambda_{\sfQ}}_{\mu}^{\alpha\beta}
\]
The fact that $C^{(1)}_{\mu\nu}$ is symmetric now 
implies the curvature condition in Proposition~\ref{prop:curvature_condition}. 
In fact,  we have 
\[
d( C^{(1)}_\nu d\sx^\nu) 
= 
(\Nabla^{\sfQ}_{\mu} C^{(1)}_\nu) d\sx^\mu \wedge d\sx^\nu 
= \frac{1}{2} 
C^{(0)}_{\alpha\beta\mu} \Lambda^{\alpha\beta}_{\nu} 
d\sx^\mu \wedge d\sx^\nu = \vartheta_\sfQ 
\]
by Proposition~\ref{prop:curvature}(2). Therefore the curvature 
condition is a special case of the anomaly equation. 
\end{remark}

\subsection{Logarithmic Case} 
\label{subsec:logarithmic}
We have hitherto studied the case where 
the connection $\nabla$ of the underlying cTP structure 
is smooth. 
In this section we allow logarithmic singularities for the 
connection---in other words, we consider $\log$-cTP structures rather than cTP-structures---
 and generalize the construction of a Fock sheaf 
to this case.
This extra generality is important in applications to mirror symmetry: 
genus-zero Gromov--Witten theory (or quantum cohomology) 
naturally defines a $\log$-cTEP structure near the large radius limit point. 
Almost all the discussions in this section are 
parallel to the previous ones. 

\subsubsection{$\log$-cTP and $\log$-cTEP Structures}

We introduce the notions of $\log$-cTP and $\log$-cTEP structure 
(cf.~Definition~\ref{def:cTP}). 
As before, we write $\cM$ for the base complex manifold and 
$\hA^1 = \Spf \C[\![z]\!]$ for the formal 
neighbourhood of the origin in $\C$. 
Let $(-)\colon \cM \times \hA^1 \to \cM \times \hA^1$ 
denote the map sending $(t, z)$ to $(t,-z)$. 
For a normal crossing divisor $D \subset \cM$, we write 
$\Omega^1_\cM(\log D)$ for the sheaf of 
one-forms on $\cM$ with logarithmic poles along $D$. 
This is a locally free sheaf; its dual, the logarithmic tangent sheaf, 
is denoted by $\Theta_\cM(\log D)$. 

\begin{definition}[cf.~Definition~\ref{def:cTP}]
\label{def:logDcTEP} 
Let $D$ be a normal crossing divisor in $\cM$. 
\begin{enumerate}
\item  
A \emph{$\log$-cTP structure} $(\sfF, \nabla, (\cdot,\cdot)_{\sfF})$
with base $(\cM,D)$ consists of a locally free $\cO_\cM[\![z]\!]$-module
$\sfF$ of rank $N+1$ and a meromorphic flat connection:
\[
\nabla \colon \sfF \to \Omega_{\cM}^1(\log D) 
\otimes_{\cO_{\cM}} z^{-1} \sfF
\]
together with a non-degenerate pairing:
\[
(\cdot,\cdot)_{\sfF} \colon 
(-)^* \sfF \otimes_{\cO_{\cM}[\![z]\!]} \sfF \to \cO_{\cM}[\![z]\!] 
\]
which satisfy the properties listed in Definition~\ref{def:cTP}(1). 

\item  A \emph{$\log$-cTEP structure} with base $(\cM,D)$ is a $\log$-cTP 
structure with base $(\cM,D)$ such that the connection $\nabla$ 
is extended in the $z$-direction with a pole of order 2 along $z=0$. 
More precisely, it is a $\log$-cTP structure $(\sfF, \nabla, (\cdot,\cdot)_\sfF)$ 
equipped with an $\cO_{\cM}$-module map 
$\nabla_{z\partial_z} \colon \sfF \to z^{-1}\sfF$ satisfying 
the properties listed in Definition~\ref{def:cTP}(2). 
Combining the $\cM$-direction and the $z$-direction, we have a 
meromorphic flat connection:
\[
\nabla \colon \sfF \to 
\left( \Omega_{\cM}^1(\log D) \oplus \cO_{\cM}z^{-1}dz \right)  
\otimes_{\cO_{\cM}} z^{-1} \sfF 
\]
\end{enumerate}
We sometimes refer to $D$ as the \emph{singularity divisor}. 
\end{definition} 

\begin{remark} 
A closely related notion of $\log$-trTLEP structure has been introduced 
by Reichelt~\cite{Reichelt:reconstruction}. 
\end{remark} 

\begin{remark} 
$\log$-cTP and $\log$-cTEP structures should be 
viewed as sheaves on $\cM \times \hA^1$. 
The letter ``c" for $\log$-cTP and $\log$-cTEP means the completion 
with respect to the $z$-adic topology. 
One can similarly define the corresponding analytic structures 
over $\cM \times \C$: these are $\log$-TP or $\log$-TEP structures 
(cf.~\S\ref{subsec:TP_TEP}). 
\end{remark}

\begin{example}
\label{ex:log_cTEP_A}
A key example is the \emph{A-model $\log$-cTEP structure} 
given by the quantum cohomology of a smooth projective variety $X$. 
Roughly speaking, this is obtained from the A-model cTEP structure 
(Example~\ref{ex:AmodelTP}, Remark~\ref{rem:completion-of-TP}) 
by taking the quotient of the base by $H^2(X;2\pi\iu \Z)$ and 
partially compactifying it by adding a normal crossing divisor. 
We use the notation from \S\ref{sec:nota}. 
Let $H^2(X;2\pi\iu \Z)$ act on the vector space 
$H_X \otimes \C$ by translation. 
By the Divisor Equation, the (extended) Dubrovin connection 
(see \S\ref{subsec:Dubrovin_conn}) is invariant under 
this action and descends to the quotient space $H_X \otimes \C/H^2(X;2\pi\iu \Z)$. 
The quotient space is partially compactified to $\C^{N+1}$ via the map: 
\begin{align*}
& H_X \otimes \C/H^2(X;2\pi\iu \Z) \hookrightarrow \C^{N+1}  \\ 
& \left[t = \textstyle \sum_{i=0}^N t^i \phi_i \right] 
\longmapsto  
(t^0,q_1,\dots,q_r, t^{r+1},\dots,t^N) 
\end{align*} 
where $q_i = e^{t^i}$ for $1\le i\le r$. 
The complement of the open embedding is the normal crossing 
divisor $q_1 q_2 \cdots q_r =0$. 
The partial compactification here depends on the choice of 
a nef basis $\phi_1,\dots,\phi_r$ of $H^2(X;\Z)$. 
Suppose that $F^0_X$ is convergent in the sense of 
\S\ref{sec:convergence}. Then the A-model $\log$-cTEP structure 
is defined over the base $(\cMbar_{\rm A}, D)$: 
\begin{align*}
  \cMbar_{\rm A} &= \{ (t^0,q_1,\dots,q_r, t^{r+1},\dots,t^N)\in \C^{N+1} : 
  |t^i|<\epsilon, |q_i|<\epsilon \} \\
  D & = \{q_1q_2 \cdots q_r = 0\} 
\end{align*}
with $\epsilon>0$ sufficiently small, 
by the following data (cf.~equation~\ref{eq:AmodelTP}):
\begin{itemize} 
      \item $\sfF = H_X \otimes_\Q \cO_{\cMbar_{\rm A}}[\![z]\!]$;
      \item $\nabla = d - \frac{1}{z} 
( (\phi_0 *) dt^0 + \sum_{i=1}^r (\phi_i*) \frac{dq_i}{q_i} + 
\sum_{j=r+1}^N (\phi_j *) dt^j ) 
        +(\frac{1}{z^2} (E *) + \frac{1}{z} \mu) dz $;  
      \item $((-)^*\alpha, \beta)_{\sfF} = \int_X \alpha(-z) \cup
        \beta(z)$
\end{itemize}
\end{example}

\begin{recall} 
The following objects associated to a $\log$-cTP structure 
$(\sfF,\nabla,(\cdot,\cdot)_\sfF)$ 
are defined exactly as in the non-logarithmic case. 
We do not repeat their definitions. 
\begin{itemize} 
\item the dual sheaves $(z^n \sfF)^\vee$, $\sfF[z^{-1}]^\vee$: see
\eqref{eq:dualmodules}; 
\item the symplectic pairing $\Omega\colon \sfF[z^{-1}]\otimes_{\cO_{\cM}} 
\sfF[z^{-1}]\to \cO_\cM$: see \eqref{eq:symplecticform}; 
\item the dual symplectic pairing 
$\Omega^\vee \colon \sfF[z^{-1}]^\vee \otimes_{\cO_{\cM}}\sfF[z^{-1}]^\vee
\to \cO_{\cM}$: see \eqref{eq:dualsymp}; 
\item the dual flat connection 
$\nabla^\vee \colon 
(z^{-1} \sfF)^\vee \to 
\Omega^1_{\cM}(\log D) \otimes _{\cO_{\cM}} \sfF^\vee$: see 
\eqref{eq:dual_flat_connection}; 
\item the dual frame 
$x_n^i \colon \sfF[z^{-1}]|_U \to \cO_U$, $n\in \Z$, $0\le i\le N$ 
associated to a trivialization $\sfF|_U \cong \C^{N+1}\otimes \cO_U[\![z]\!]$ 
over $U$: see \eqref{eq:frame_x}. 
\end{itemize} 
\end{recall} 

\subsubsection{The Total Space of a $\log$-cTP Structure} 
\label{subsubsec:totalspace_logcTP} 
Let $(\sfF,\nabla,(\cdot,\cdot)_{\sfF})$ 
a $\log$-cTP structure with base $(\cM,D)$. 
\begin{definition}[cf.~Definition~\ref{def:totalspace}] 
The \emph{total space} $\LL$ of a $\log$-cTP structure 
$(\sfF,\nabla,(\cdot,\cdot)_{\sfF})$ is the total space of the infinite-dimensional vector bundle associated to $z\sfF$. 
As a set, $\LL$ consists of pairs $(t,\bx)$ such that 
$t\in\cM$ and $\bx \in z \sfF_t$. 
We write $\pr \colon \LL \to \cM$ for the natural projection. 
We equip $\LL$ with the structure of a ringed space 
as in Definition~\ref{def:totalspace}; 
we denote by $\bcO$ the structure sheaf of $\LL$. 
\end{definition}

An algebraic local co-ordinate system on $\LL$ is given 
similarly to the non-logarithmic case; for the sake of 
exposition we shall always use local co-ordinates of the 
following type, which are compatible with logarithmic 
singularities. 

\begin{definition} 
\label{def:alg_coord_log}
Let $U\subset \cM$ be a co-ordinate neighbourhood 
with co-ordinates $\{t^0,q_1,\dots,q_r,t^{r+1},\dots, t^M\}$ 
such that $D \cap U$ is given by $\{q_1 q_2 \cdots q_r = 0\}$.  
Choose a trivialization of $\sfF|_U \cong \C^{N+1} \otimes \cO_U[\![z]\!]$ 
over $U$ and define the corresponding 
dual frame $x_n^i \in \sfF[z^{-1}]^\vee$. 
We call the set 
\[
\{ t^0, q_1, \dots, q_r, t^{r+1},\dots, t^M \} \cup  
\{ x_n^i : \text{$0 \le i \le N$, $n\ge 1$} \}
\]
an \emph{algebraic local co-ordinate system} on $\LL$. 
We also write $q_i = e^{t^i}$ for $1\le i\le r$,
so that we have: 
\[
\text{$\frac{dq_i}{q_i} = dt^i$ and $q_i \parfrac{}{q_i} = \parfrac{}{t^i}$} 
\qquad \qquad (1\le i\le r)
\]
Abusing notation, we write $f(t) = f(t^0,t^1,\dots,t^r,t^{r+1},\dots, t^M)$ 
to denote a function $f \colon U \to \C$ on $U$, where 
the identification $t^i = \log q_i$ with $1\le i\le r$ is understood. 
\end{definition}

Using algebraic local co-ordinates on $\LL|_U$, one has (as before): 
\[
\bcO(\pr^{-1}(U)) = 
\cO(U)\left[x_n^i : \text{$n\ge 1$, $0\le i\le N$}\right]
\]
The ring $\bcO(\pr^{-1}(U))$ is equipped with a grading and filtration 
as in Definition~\ref{def:totalspace}. 

\subsubsection{Miniversality} 
Let $(\sfF,\nabla, (\cdot,\cdot)_\sfF)$ be a $\log$-cTP structure 
with base $(\cM,D)$. 
Here and hereafter we restrict to the case where $M=N$, that is, 
$\dim \cM = \rank \sfF$. 
Choose a trivialization $\sfF|_U \cong \C^{N+1} \otimes \cO_U[\![z]\!]$ 
over $U$. We can write the connection $\nabla$ in the form 
\begin{equation}
\label{eq:definition_of_C}
\nabla s = d s - \frac{1}{z} \cC(t,z) s 
\end{equation}
with 
\begin{equation} 
\label{eq:logarithmic_C}
\cC(t,z) = \sum_{i=0}^M \cC_i(t,z) dt^i = 
\cC_0(t,z) dt^0 + \sum_{i=1}^r \cC_i(t,z) \frac{dq_i}{q_i} 
+ \sum_{j=r+1}^N \cC_j(t,z) dt^j 
\end{equation} 
where $s \in \C^{N+1} \otimes \cO_U[\![z]\!]  \cong \sfF|_U$ and
$\cC_i(t,z) \in \End(\C^{N+1})\otimes \cO_U[\![z]\!]$. 
The residual part $\cC(t,0) = (-z\nabla)|_{z=0}$ determines a section of
$\End(\sfF_0) \otimes \Omega_U^1(\log D)$, independently of
the choice of trivialization.

\begin{example}
In the case of the A-model $\log$-cTEP structure (Example~\ref{ex:log_cTEP_A}), we have $\cC(t,z) = (\phi_0*)dt^0 
+ \sum_{i=1}^r (\phi_i*) \frac{dq_i}{q_i} + 
\sum_{i=r+1}^N (\phi_i*) dt^i$.
\end{example}

\begin{definition}[cf.~Definition~\ref{def:possibilityofdilaton}] 
  \label{def:miniversal_log-cTP} 
For a $\log$-cTP structure $(\sfF,\nabla, (\cdot,\cdot)_{\sfF})$, 
we define:
  \begin{align*}
    \sfF_{0,t}^\circ & := 
    \{ x_1 \in \sfF_{0,t} :
    \text{$\Theta_\cM(\log D)_t \to \sfF_{0,t}$, $\ v\mapsto \iota_v \cC(t,0) x_1$
      is an isomorphism}\}  \\  
    \LLo & := \{ (t,\bx) \in \LL : \text{$t\in \cM$,  
    $\bx \in z \sfF_t$,  
    $(\bx/z)|_{z=0} \in \sfF_{0,t}^\circ$} \} \\
    \sfF_0^\circ & := \bigcup_{t\in \cM} \sfF_{0,t}^\circ
  \end{align*}
These are Zariski open subsets of, respectively, $\sfF_{0,t}$, $\LL$, and
$\sfF_0$.  If for every point $t\in \cM$, $\sfF_{0,t}^\circ$ 
is non-empty, then we say that $(\sfF,\nabla, (\cdot,\cdot)_{\sfF})$ 
is \emph{miniversal}.
A miniversal $\log$-cTP structure necessarily satisfies $\dim \cM =\rank \sfF$. 
\end{definition}

Henceforth all $\log$-cTP structures are assumed to be miniversal 
unless otherwise stated.  
Choose a trivialization of $\sfF|_U$ and present the connection $\nabla$ 
in terms of the trivialization as in \eqref{eq:definition_of_C}, \eqref{eq:logarithmic_C}.  
The \emph{discriminant} in the logarithmic situation is defined to be 
\begin{equation} 
\label{eq:log_discriminant}
P(t,x_1) := (-1)^{N+1}\det(\cC_0(t,0) x_1, \cC_1(t,0) x_1, \dots, \cC_N(t,0) x_1) 
\end{equation} 
(cf.~equation~\ref{eq:discriminant}).  This is a polynomial in $x_1$ of degree $N+1$ and belongs to 
$\cO(U)[x_1^0,\dots,x_1^N]$. 
The set $\LLo$ is the complement of the zero-locus of $P(t,x_1)$.  
More invariantly, $P(t,x_1) dt^0\wedge \cdots \wedge dt^N$ 
should be thought of as a section of the line bundle 
$\pr^*(\det(\sfF_0) \otimes \Omega^{N+1}_\cM(\log D))$ 
over $\LL$, and $\LLo$ is the complement of the zero-locus. 
The ring of regular functions over 
$\pr^{-1}(U)^\circ := \pr^{-1}(U) \cap \LLo$ is:
\[
\bcO(\pr^{-1}(U)^\circ)  
= \cO(U)[\{x_n^i\}_{n\ge 1,0\le i\le N}, 
P(t,x_1)^{-1}]
\]
As before, the grading and filtration on $\bcO(\pr^{-1}(U))$ 
descends to $\bcO(\pr^{-1}(U)^\circ)$. 

\subsubsection{Logarithmic One-Forms and Vector Fields on $\LL$}
We need to consider the sheaves of \emph{logarithmic} 
one-forms and vector fields on the total space $\LL$. 
In terms of algebraic local co-ordinates $\{t^i, q_j = e^{t^j},x_n^k\}$, 
they are defined by: 
\begin{align*} 
\bOmega^1(\log D) & = 
 \bigoplus_{j=0}^N \bcO dt^j 
 \oplus 
\bigoplus_{n =1}^\infty 
\bigoplus_{i=0}^N  \bcO d x_n^i \\
\bTheta(\log D) &= \sHom(\bOmega^1(\log D),\bcO) = 
\prod_{j=0}^N \bcO \partial_j \times
\prod_{n=1}^\infty
\prod_{i=0}^N \bcO \partial_{n,i}
\end{align*} 
where we set $\partial_j = \partial/\partial t^j$ and 
$\partial_{n,i} = \partial/\partial x_n^i$. 
Recall that $dt^i = dq_i/q_i$ and $\partial_i = q_i (\partial/\partial q_i)$ 
for $i=1,\dots,r$. 
The grading and filtration on $\bOmega^1(\log D)$ is given by 
\eqref{eq:grading_filtration_coord}. 

\subsubsection{The Yukawa Coupling and the Kodaira--Spencer Map} 
\label{subsubsec:Yukawa_KS_log}
The Yukawa coupling and Kodaira--Spencer map can also be
adapted to the logarithmic setting. 
Let $\{t^i, q_j = e^{t^j}, x_n^k\}$ be an algebraic local 
co-ordinate system on the total space $\LL$ and write the connection 
endomorphism as $\cC(t,z) = \sum_{i=0}^N \cC_i(t,z) dt^i$ 
(see equation~\ref{eq:logarithmic_C}). 

\begin{definition}[cf.~Definition~\ref{def:Yukawa}] 
The \emph{Yukawa coupling} is a cubic tensor:
  \[
  \bY = \sum_{i=0}^N \sum_{j=0}^N \sum_{k=0}^N C^{(0)}_{ijk} dt^i \otimes dt^j \otimes dt^k \in
  \big((\bOmega^1(\log D))^{\otimes 3}\big)^2_{-3}
  \]
where 
\[
    C_{ijk}^{(0)}(t,\bx) = \big( \cC_i(t,0) x_1,
    \cC_j(t,0) \cC_k(t,0) x_1 \big)_{\sfF_0}
\] 
with $x_1 = (\bx/z)|_{z=0}$. 
Recall again that $d t^i = dq_i/q_i$ for $i=1,\dots,r$. 
\end{definition}

The pulled-back sheaves $\pr^*(z^n \sfF)$, $\pr^* \sfF[z^{-1}]$, 
$\pr^*(z^n\sfF)^\vee$, $\pr^*\sfF[z^{-1}]^\vee$ on $\LL$ 
are defined as in \eqref{eq:pullbacks}. 
The connection $\nabla$ induces a flat connection $\tnabla:=\pr^*\nabla$ 
on $\pr^* \sfF[z^{-1}]$ 
(cf.~equation~\ref{eq:tnabla} and \S\ref{subsec:conn-LL}): 
\begin{equation*} 
\tnabla \colon \pr^*\sfF[z^{-1}] \to \bOmega^1(\log D)
\hotimes \pr^*(\sfF[z^{-1}])  
\end{equation*} 
such that $\tnabla \pr^*(z^n\sfF) \subset 
\bOmega^1(\log D) \hotimes \pr^*(z^{n-1}\sfF)$. 
The dual connection  
\[
\tnabla^\vee \colon \pr^* \sfF[z^{-1}]^\vee \to 
\bOmega^1(\log D) \hotimes \pr^* \sfF[z^{-1}]^\vee
\]
is defined by $\pair{\tnabla^\vee \varphi}{s} 
:= d\pair{\varphi}{s} - \pair{\varphi}{\tnabla s}$. 
The explicit presentation \eqref{eq:tnablavee} of $\tnabla^\vee$ 
holds also in the logarithmic case; we also have 
a commutative diagram similar to \eqref{eq:tnablavee-CD}. 

\begin{definition}[cf.~Definition~\ref{def:KS}] 
Define the \emph{tautological section} $\bx$ of $\pr^* (z\sfF)$ by 
  \[
  \bx(t,\bx) = \bx   
  \] 
where $(t,\bx)$ denotes the point $\bx \in z \sfF_t$ on $\LL$. 
The \emph{Kodaira--Spencer map} 
$\KS \colon \bTheta(\log D) \to \pr^*\sfF$ 
and the \emph{dual Kodaira--Spencer map} 
$\KS^* \colon \pr^*\sfF^\vee \to \bOmega(\log D)$ 
are defined by:
\begin{align*}
  \KS(v) = \tnabla_v \bx &&
  \KS^*(\varphi) = \varphi(\tnabla \bx)
\end{align*}
The maps $\KS$ and $\KS^*$ are isomorphisms over $\LLo\subset \LL$.
\end{definition}

\begin{notation}
As before we denote by $\bThetao(\log D)$ 
the restriction of $\bTheta(\log D)$ to $\LLo \subset \LL$, and 
denote by $\bOmegao^1(\log D)$ the restriction of 
$\bOmega^1(\log D)$ to $\LLo \subset \LL$. 
\end{notation}

\subsubsection{Opposite Modules and Logarithmic Frobenius Manifolds}

We extend the notion of (pseudo-) opposite modules 
to the setting of $\log$-cTP structures.

\begin{definition}[cf.~Definition~\ref{def:opposite}]
\label{def:opposite_log-cTP} 
A \emph{pseudo-opposite module} $\sfP$ for a $\log$-cTP structure 
$(\sfF,\nabla, (\cdot,\cdot)_{\sfF})$ is an $\cO_{\cM}$-submodule 
$\sfP$ of $\sfF[z^{-1}]$ satisfying the conditions 
\begin{description}
\item[(Opp1)] (opposedness) 
$\sfF[z^{-1}] = \sfF \oplus \sfP$ and 
\item[(Opp2)] (isotropy) 
$\Omega(\sfP,\sfP) = 0$. 
\end{description}
A pseudo-opposite module $\sfP$ is said to be \emph{parallel} 
if it satisfies 
\begin{description}
\item[(Opp3)] $\nabla \sfP \subset \Omega^1_{\cM}(\log D) \otimes \sfP$. 
\end{description} 
If $\sfP$ satisfies (Opp1--Opp3) and
\begin{description}
\item[(Opp4)] ($z^{-1}$-linearity) 
$z^{-1}\sfP \subset \sfP$
\end{description} 
then it is called an \emph{opposite module}.  

Suppose that $(\sfF,\nabla,(\cdot,\cdot)_\sfF)$ is a $\log$-cTEP structure. 
An opposite module $\sfP$ for the underlying $\log$-cTP structure is said to be 
\emph{homogeneous} if it satisfies  
\begin{description} 
\item[(Opp5)] (homogeneity) 
$\nabla_{z\partial_z} \sfP \subset \sfP$. 
\end{description} 
\end{definition} 

\begin{example}[cf.~Example~\ref{ex:Amodel-opposite}]
\label{exa:log_Amodel_opposite}
The A-model $\log$-cTEP structure (Example~\ref{ex:log_cTEP_A}) 
is equipped with a 
standard homogeneous opposite module $\sfP_{\rm std} = 
H_X \otimes_\Q z^{-1} \cO_{\cMbar_{\rm A}}[z^{-1}]$. 
\end{example} 

An opposite module always exists in a formal neighbourhood of a point 
outside the singularity divisor $D$ by virtue of Lemma~\ref{lem:existence_opposite}. 
However it is not clear if Lemma~\ref{lem:existence_opposite} 
can be generalized to a point on the divisor $D$. 
In practice, in a geometric example such as the A-model $\log$-cTEP 
structure, one can often find an opposite module that extends across $D$.  

Similarly to \S\ref{subsec:opp-Frob}, by choosing 
an opposite module $\sfP$ and a primitive section 
$\zeta$ for a miniversal $\log$-cTP (or $\log$-cTEP) structure, 
one can equip the base with 
a \emph{logarithmic} Frobenius manifold 
structure (with or without Euler vector field) 
in the sense of  Reichelt~\cite{Reichelt:reconstruction}. 
The argument is completely parallel to Proposition~\ref{prop:flat_trivialization} and Remark~\ref{rem:opposite_conformalcase} 
and we give only the statement.  

\begin{proposition}[{cf.~Proposition~\ref{prop:flat_trivialization}, 
Remark~\ref{rem:opposite_conformalcase},\cite[Propositions~1.10,~1.11]{Reichelt:reconstruction}}]
Consider a $\log$-cTP structure $(\sfF,\nabla,(\cdot,\cdot)_\sfF)$ 
with base $(\cM,D)$. 
Let $\sfP$ be an opposite module for $(\sfF,\nabla,(\cdot,\cdot)_\sfF)$ 
over $U$. Then:
\begin{itemize}
\item[(i)] The natural maps $\sfF_0 = \sfF/z\sfF \leftarrow 
\sfF \cap z \sfP \rightarrow z \sfP/\sfP$ 
are isomorphisms of $\cO_U$-modules. 
\item[(ii)] We have $\sfF = (\sfF \cap z \sfP )\otimes \C[\![z]\!] 
\cong (z\sfP/\sfP) \otimes \C[\![z]\!]$, which we call a \emph{flat trivialization}. 
Note that $z\sfP/\sfP$ is a locally free coherent $\cO_U$-module 
with a logarithmic flat connection, and let $\nabla^0 \colon z\sfP/\sfP 
\to \Omega^1_U(\log D) \otimes_{\cO_U} (z\sfP/\sfP) $ denote the 
flat connection induced by $\nabla$. 
\item[(iii)] Under the flat trivialization, the connection 
$\nabla$ takes the form 
\[
\nabla = \nabla^0 - \frac{1}{z} \cC(t) 
\]
where $\cC(t) \in \Omega^1_U(\log D) \otimes_{\cO_U}
\End(z\sfP/\sfP)$ is independent of $z$.
\item[(iv)] Under the flat trivialization, the pairing 
$(\cdot,\cdot)_{\sfF}$ induces and can be recovered from 
a $z$-independent symmetric pairing 
\[
(\cdot,\cdot)_{z\sfP/\sfP} \colon (z\sfP/\sfP) 
\otimes (z\sfP/\sfP) \to \cO_U 
\]
which is flat with respect to $\nabla^0$. 
\item[(v)] Assume that there exists a section $\zeta$ of 
$\sfF$ over $U$ which is flat with respect to $\nabla^0$ 
in the flat trivialization and whose image 
under $\sfF \to \sfF_0 = \sfF/z\sfF$ lies in $\sfF_0^\circ$. 
(This assumption implies the miniversality of 
$(\sfF,\nabla,(\cdot,\cdot)_\sfF)$.) 
We call such a section $\zeta$ a \emph{primitive section} 
associated to $\sfP$. 
Then the base $U$ carries the structure of a \emph{logarithmic Frobenius manifold without Euler vector field}.
It consists of 
\begin{itemize} 
\item A flat symmetric $\cO_U$-bilinear 
metric $g: \Theta_U(\log D) \otimes_{\cO_U} \Theta_U(\log D) \to \cO_U$, defined by:
\[
g(v_1,v_2) = 
(z\nabla_{v_1} \zeta, z\nabla_{v_2} \zeta)_{\sfF}
\]

\item A commutative and associative product  
$* \colon \Theta_U(\log D) \otimes_{\cO_U} \Theta_U(\log D)
\to \Theta_U(\log D)$, defined by:
\[
z\nabla_{v_1} z\nabla_{v_2} \zeta = -z\nabla_{v_1*v_2} \zeta
\]

\item A flat identity vector field $e\in \Theta_U(\log D)$ for the product $*$, 
defined by:
\[
- z \nabla_e \zeta= \zeta
\]
\end{itemize} 
such that the connection $\nabla^\lambda_v = 
\nabla^{\rm LC}_v - \lambda (v *)$ on the logarithmic 
tangent sheaf $\Theta_U(\log D)$ is a flat pencil of 
connections with parameter $\lambda$. 
Here $\nabla^{\rm LC}$ denotes the Levi-Civita  
connection for the metric $g$. 
\item[(vi)] Suppose now that $(\sfF,\nabla,(\cdot,\cdot)_\sfF)$ is 
a miniversal $\log$-cTEP structure with base $(\cM,D)$. 
Miniversality implies that there exists a unique logarithmic vector 
field $E \in \Theta_{\cM}(\log D)$ such that:
\[
(\nabla_{z\partial_z} + \nabla_E ) \sfF \subset \sfF
\]
This is called the \emph{Euler vector field}. 
Suppose that we have a homogeneous opposite module $\sfP$ 
for $(\sfF,\nabla,(\cdot,\cdot)_\sfF)$ over $U$. 
This defines a flat trivialization of $\sfF$ as above. 
Suppose also that there exists a section $\zeta$ of 
$\sfF$ over $U$ such that $\zeta$ is flat respect to $\nabla^0$ 
in the flat trivialization, satisfies the homogeneity condition
\[
(\nabla_{z\partial_z} + \nabla_{E}) \zeta = - \frac{\hc}{2} \zeta
\]
for some $\hc \in \C$, and is such that the image of $\zeta$ 
under $\sfF \to \sfF_0 = \sfF/z\sfF$ lies in $\sfF_0^\circ$.  
Then $U$ carries the structure of a logarithmic Frobenius manifold. 
It is given by the structures $(g,*,e)$ from Part {\rm (v)} and the 
Euler vector field $E$, which satisfy the additional properties 
listed in \eqref{eq:Euler_compatibility}. 
\end{itemize}
\end{proposition} 

\begin{example} 
The A-model $\log$-cTEP structure (Example~\ref{ex:log_cTEP_A}) 
equipped with the standard homogeneous opposite module $\sfP_{\rm std}$ 
(Example~\ref{exa:log_Amodel_opposite}) yields the standard logarithmic 
Frobenius manifold structure on the base $\cMbar_{\rm A}$. 
\end{example} 

\subsubsection{Flat Connection on the Total Space} 
A pseudo-opposite module $\sfP$ determines flat connections on the
logarithmic tangent sheaf and logarithmic cotangent sheaf of $\LLo$, as follows.  

\begin{definition}[cf.~Definition~\ref{def:Nabla}] 
  \label{def:Nabla_log}
  Let $\sfP$ be a pseudo-opposite module for a $\log$-cTP structure
  $(\sfF,\nabla, (\cdot,\cdot)_{\sfF})$, and let $\Pi \colon
  \sfF[z^{-1}] \to \sfF$ be the projection along $\sfP$.  
The composition of the maps:
  \begin{align*}
    & \xymatrix{
      \pr^* \sfF \ar[rr]^-{\tnabla} && 
      \bOmega^1(\log D) \hotimes \pr^*(z^{-1} \sfF) 
      \ar[rr]^-{\id \otimes \Pi} &&
      \bOmega^1(\log D) \hotimes \pr^* \sfF
    } \\
    & \xymatrix{
      \pr^* \sfF^\vee \ar[rr]^-{\Pi^\vee} && 
      \pr^* (z^{-1} \sfF)^\vee \ar[rr]^-{\tnabla^\vee} &&
      \bOmega^1(\log D) \otimes \pr^* \sfF^\vee 
    }
  \end{align*}
  (restricted to $\LLo$) with the Kodaira--Spencer isomorphisms $\KS
  \colon \bThetao(\log D) \to \pr^*\sfF$, $\KS^* \colon \pr^*\sfF^\vee
  \to \bOmegao^1(\log D)$ induces connections:
  \begin{equation}
    \label{eq:Nabla_log}
    \begin{aligned}
      &\Nabla \colon \bThetao(\log D) \to 
      \bOmegao^1(\log D) \hotimes \bThetao(\log D) \\
      &\Nabla \colon \bOmegao^1(\log D) \to 
      \bOmegao^1(\log D) \otimes \bOmegao^1(\log D) 
    \end{aligned}
  \end{equation}
  where $\bOmegao^1(\log D) \hotimes \bThetao(\log D) :=
  \projlim_n\big(\bOmegao^1(\log D) \otimes \big( \bThetao(\log D) /
  \KS^{-1}(\pr^* (z^n \sfF)) \big) \big)$. 
  The connection on $\bOmegao^1(\log D)$ also induces the connection 
  on logarithmic $n$-tensors: 
  \[
 \Nabla \colon \bOmegao^1(\log D)^{\otimes n} 
 \to \bOmegao^1(\log D) \otimes \bOmegao^1(\log D)^{\otimes n}
 \]
\end{definition}

The connections in \eqref{eq:Nabla_log} are dual to each other. 
The argument of Proposition~\ref{prop:Nabla-torsionfree} 
shows the following: 

\begin{proposition} 
The flat connection $\Nabla$ on $\bThetao(\log D)$ associated 
to a pseudo-opposite module $\sfP$ is torsion-free. 
It is flat if $\sfP$ is parallel. 
\end{proposition} 

In the non-logarithmic case, given a parallel pseudo-opposite module, 
we constructed in \S\ref{subsec:flatstronL} 
the genus-zero potential and a flat co-ordinate system 
on the formal neighbourhood $\hLLo_t$ of $\LLo_t$ in $\LLo$.
The construction there does not work if $t$ is on the singularity 
divisor $D$, but works if $t$ is away from $D$. 

\subsubsection{Propagators} 
In the logarithmic case, propagators are defined as \emph{logarithmic} bivector fields. 
\begin{definition}[cf.~Definition~\ref{def:propagator}]
  \label{def:propagator_log}
  Let $\sfP_1$,~$\sfP_2$ be pseudo-opposite modules for the
  $\log$-cTP structure $(\sfF,\nabla, (\cdot,\cdot)_{\sfF})$.  
Let $\Pi_i \colon \sfF[z^{-1}] \to \sfF$, $i \in \{1,2\}$, be the 
projection along $\sfP_i$ defined by the decomposition $\sfF[z^{-1}]
  = \sfP_i \oplus \sfF$.  The \emph{propagator}
  $\Delta=\Delta(\sfP_1,\sfP_2) \in
  \sHom_{\bcO}(\bOmegao^1(\log D) \otimes
  \bOmegao^1(\log D),\bcO)$
  is defined by:
  \begin{align*}
    \Delta(\omega_1, \omega_2) = 
    \Omega^\vee\big(\Pi_1^* \KS^{*-1}\omega_1, 
    \Pi_2^* \KS^{*-1} \omega_2\big)
    && \omega_1,\omega_2 \in \bOmegao^1(\log D)
  \end{align*}
  The logarithmic bivector field $\Delta$ is identified, via the
  Kodaira--Spencer isomorphism $\KS^*$, with the push-forward along
  $\Pi_1 \times \Pi_2$ of the Poisson bivector field on $\sfF[z^{-1}]$
  defined by $\Omega^\vee$.
\end{definition}  

The propagator in the logarithmic case satisfies the same properties 
as in the non-logarithmic case. The proofs are completely parallel 
and are omitted. 

\begin{proposition}[cf.~Propositions~\ref{pro:prop-elementary},~\ref{prop:difference_conn}] 
\label{prop:propagator_log}
Let $\sfP_1,\sfP_2$ be pseudo-opposite modules for 
the $\log$-cTP structure $(\sfF,\nabla,(\cdot,\cdot)_{\sfF})$ 
and let $\Delta = \Delta(\sfP_1,\sfP_2)$ be the propagator. 
Then:
\begin{enumerate}
\item $\Delta$ is symmetric, i.e.~$\Delta(\omega_1,\omega_2) 
= \Delta(\omega_2,\omega_1)$; 
\item $(\Nabla^{\sfP_2} - \Nabla^{\sfP_1}) \omega 
= \iota(\iota_\omega) \bY$ for $\omega \in \bOmegao^1(\log D)$; 
\item if $\sfP_1,\sfP_2$ are parallel, we have 
$(\Nabla^{\sfP_1} \Delta)(\omega_1\otimes \omega_2) 
= \iota(\iota_{\omega_1} \Delta \otimes \iota_{\omega_2} \Delta) \bY$ 
for $\omega_1,\omega_2 \in \bOmegao^1(\log D)$. 
\end{enumerate}
\end{proposition} 
\begin{proposition}[cf.~Proposition~\ref{prop:Deltasum}]
\label{prop:Deltasum_log} 
Let $\sfP_1,\sfP_2,\sfP_3$ be pseudo-opposite modules 
and let $\Delta_{ij} = \Delta(\sfP_i,\sfP_j)$, $i,j\in \{1,2,3\}$,
be the propagators. Then 
$\Delta_{13} = \Delta_{12} + \Delta_{23}$. 
\end{proposition}

\subsubsection{Local Fock Space}
\label{sec:localFock_log} 
Let $\{t^i,q_j = e^{t^j},x_n^i\}$ denote an algebraic local co-ordinate 
system on $\LL$ as in Definition~\ref{def:alg_coord_log}.
We write the co-ordinates $\{t^0,\log q_1,\dots,\log q_r, t^{r+1},
\dots,t^N, x_n^i\}$ as $\{\sx_\mu\}$ and use similar 
tensor notation as before, e.g.~writing 
the Yukawa coupling and propagator as 
\begin{align*} 
\bY = C^{(0)}_{\mu \nu \rho}
 d\sx^\mu \otimes d\sx^\nu \otimes d\sx^\rho &&
 \Delta = \Delta^{\mu \nu} \partial_\mu \otimes \partial_\nu
\end{align*} 
where $\partial_\nu = \partial/\partial \sx^\nu$.  
\begin{definition}[cf.~Definition~\ref{def:localFock}] 
\label{def:localFock_log}
Consider a miniversal $\log$-cTP structure $(\sfF,\nabla,(\cdot,\cdot)_{\sfF})$ 
with base $(\cM,D)$. 
Let $\sfP$ be a parallel pseudo-opposite module over 
an open set $U \subset \cM$ and let $\Nabla =\Nabla^{\sfP}$ 
be the associated flat connection on $\LLo$. 
Let $P=P(t,x_1)$ denote the discriminant \eqref{eq:log_discriminant}. 
The \emph{local Fock space} $\Fock(U;\sfP)$ consists of collections:
\[
\wave = 
\big\{ \Nabla^n C^{(g)} \in
\big(\bOmega^1(\log D)\big)^{\otimes n}
\big(\pr^{-1}(U)^\circ \big): 
\text{$g \geq 0$, $n \geq 0$, $2g-2+n>0$} 
\big \}
\]
of completely symmetric logarithmic $n$-tensors on $\pr^{-1}(U)^\circ$
such that the following conditions hold:
\begin{description}
\item[(Yukawa)] $\Nabla^3 C^{(0)}$ is the Yukawa coupling $\bY$ 
in \S\ref{subsubsec:Yukawa_KS_log}; 
\item[(Jetness)] $\Nabla(\Nabla^n C^{(g)}) = \Nabla^{n+1} C^{(g)}$;
\item[(Grading \& Filtration)] $\Nabla^n C^{(g)} \in
  \big(\big(\bOmega^1(\log D)\big)^{\otimes n} \big(\pr^{-1}(U)^\circ
  \big)\big)^{2-2g}_{3g-3}$;
\item[(Pole)] $P \Nabla C^{(1)}$ extends to a regular $1$-form on
  $\pr^{-1}(U)$, and for $g \geq 2$:
  \[
  C^{(g)} \in 
  P^{-(5g-5)} \cO(U)[x_1,x_2,Px_3,\ldots,P^{3g-4} x_{3g-2}]
  \]
\end{description}
In local co-ordinates $\{\sx^\mu\}$, we write 
\[
\Nabla^n C^{(g)} = 
C^{(g)}_{\mu_1\cdots\mu_n} d\sx^{\mu_1} \otimes \cdots \otimes
d\sx^{\mu_n}
\]
and refer to $\Nabla^n C^{(g)}$ or $C^{(g)}_{\mu_1\cdots\mu_n}$ as
\emph{$n$-point correlation functions}.
\end{definition} 

\subsubsection{Transformation Rule} 
As before we encode elements of the local Fock space $\Fock(U;\sfP)$ 
by \emph{jet potentials} on the total space of the logarithmic tangent bundle 
$\bTheta(\log D)|_{\pr^{-1}(U)^\circ}$. 
The transformation rule in the logarithmic case is then described in 
terms of jet potentials. 

Let $\{\sy^\mu\}$ denote the fiber co-ordinates of the logarithmic tangent
bundle $\bTheta(\log D)$ dual to $\{\partial/\partial \sx^\mu\}$, 
so that $(\sx,\sy)$ denotes a point in the total space of $\bTheta(\log
D)|_{\pr^{-1}(U)^\circ}$.

\begin{definition}[cf.~Definition~\ref{def:jetpotential}] 
\label{def:jetpotential_log}
Given an element $\wave = \{\Nabla^n C^{(g)}\}_{g,n}$ of
$\Fock(U;\sfP)$, we set:
\begin{align*}
  \cW(\sx,\sy) 
  & = \sum_{g=0}^\infty \hbar^{g-1} \cW^g(\sx,\sy)
  \intertext{where:}
  \cW^g(\sx,\sy) 
  & = \sum_{n=\max(3-2g,0)}^\infty  
  \frac{1}{n!} 
  C^{(g)}_{\mu_1,\dots,\mu_n}(\sx) 
  \sy^{\mu_1} \cdots \sy^{\mu_n}
\end{align*}
We call $\cW^g$ the \emph{genus-$g$ jet potential} and $\exp(\cW)$ the
\emph{total jet potential} associated to $\wave$.
\end{definition} 

\begin{definition}[cf.~Definition~\ref{def:transformation}]
\label{def:transformation_log} 
Let $\sfP_1,\sfP_2$ be parallel pseudo-opposite modules 
for the $\log$-cTP structure $(\sfF,\nabla,(\cdot,\cdot)_{\sfF})$. 
Let $\Delta$ denote the propagator $\Delta(\sfP_1,\sfP_2)$.  
The transformation rule $T(\sfP_1, \sfP_2)\colon 
\Fock(U;\sfP_1) \to \Fock(U;\sfP_2)$ is a map which assigns, to the jet
potential $\exp(\cW)$ for an element of $\Fock(U;\sfP_1)$, the jet
potential $\exp(\hcW)$ for an element of $\Fock(U;\sfP_2)$ given by:
\begin{equation} 
\label{eq:transformationrule-jet_log} 
\exp\big(\hcW(\sx,\sy)\big) =  
\exp\left(\frac{\hbar}{2}\Delta^{\mu\nu} \partial_{\sy^\mu}
\partial_{\sy^\nu}\right) 
\exp\big(\cW(\sx,\sy)\big)
\end{equation} 
The transformation rule can be also expressed via a Feynman rule. 
In the notation of Definition~\ref{def:transformation}, we have  
\[
\hC^{(g)}_{\mu_1,\dots,\mu_n} 
= \sum_\Gamma \frac{1}{|\Aut(\Gamma)|} 
\Cont_{\Gamma}(\{C^{(h)}_{\nu_1,\dots,\nu_m}\}, 
\Delta)_{\mu_1,\dots,\mu_n} 
\]
where $\{C^{(g)}_{\mu_1,\dots,\mu_n}\}$ are the correlation 
functions associated to $\cW$ and 
$\{\hC^{(g)}_{\mu_1,\dots,\mu_n}\}$ are the correlation 
functions associated to $\hcW$.  
\end{definition} 

\begin{proposition}[cf.~Lemmas~\ref{lem:jetness}--\ref{lem:pole}] 
The transformation rule in Definition~\ref{def:transformation_log} 
is well-defined, i.e.~it preserves the conditions (Yukawa), (Jetness), 
(Grading \& Filtration), and (Pole) in the definition of the local Fock space 
$\Fock(U;\sfP_i)$.
\end{proposition}
\begin{proof} 
We argue as in \S\ref{subsec:transformation} 
using the co-ordinate system 
\[
\{\sx^\mu\}=
\{t^0,\log q_1,\dots,\log q_r,t^{r+1},\dots,t^N, x_n^i\}
\] 
associated to the algebraic local co-ordinate system 
$\{t^0,q_1,\dots,q_r,t^{r+1},\dots,t^N, x_n^i\}$ 
in Definition~\ref{def:alg_coord_log}. 
The Yukawa coupling does not change: $\hC^{(0)}_{\mu\nu\rho} = 
C^{(0)}_{\mu\nu\rho}$ under the transformation rule 
(see equation~\ref{eq:Yukawa-Feynman}) and the condition (Yukawa) holds. 
The condition (Jetness) for $\hC^{(g)}_{\mu_1,\dots,\mu_n}$ 
follows from the same argument as in Lemma~\ref{lem:jetness}, 
using Proposition~\ref{prop:propagator_log} 
instead of Proposition~\ref{prop:difference_conn}. 
The analogues of Propositions~\ref{prop:grading-filtration-Nabla},~\ref{prop:grading-filtration-prop} hold in the logarithmic case, 
and the condition (Grading \& Filtration) for 
$\{\hC^{(g)}_{\mu_1,\dots,\mu_n}\}$ 
follows from them and the argument in Lemma~\ref{lem:GrFil}. 
Regarding the condition (Pole), we can repeat the argument  
of Lemma~\ref{lem:pole} to show that $\hC^{(g)}$ for $g\ge 2$ belongs 
to $P^{-(5g-5)} \cO(U\setminus D)[x_1,x_2,P x_3, P^2 x_4,\dots,
P^{3g-4}x_{3g-2}]$. 
(The argument there only applies to $t\in U\setminus D$, as a flat 
co-ordinate system exists only at such $t$.) 
On the other hand, $\hC^{(g)}$ belongs to $\bcO(\pr^{-1}(U)^\circ)$ 
by the Feynman rule. The condition (Pole) now follows from
 Hartogs' extension theorem. 
\end{proof} 

The transformation rule satisfies the cocycle condition 
by virtue of Proposition~\ref{prop:Deltasum_log}.  

\begin{proposition}[cf.~Proposition~\ref{prop:cocycle}]
  The transformation rule \eqref{eq:transformationrule-jet_log} satisfies
  the cocycle condition: if $\sfP_1$,~$\sfP_2$,~$\sfP_3$ are parallel 
pseudo-opposite modules for $\sfF$ over $U$ and 
$T_{ij} = T(\sfP_i,\sfP_j)$ is the
transformation rule from $\Fock(U;\sfP_i)$ to $\Fock(U;\sfP_j)$ then
$T_{13} = T_{23} \circ T_{12}$.
\end{proposition}

\subsubsection{Fock Sheaf} We now define the Fock sheaf in the logarithmic case.

\begin{assumption}[cf.~Assumption~\ref{assump:covering}]
  \label{assumption:covering_log}
  There is an open covering $\{U_\alpha\}_{\alpha \in A}$ 
  of $\cM$ such that for
  each $\alpha \in A$ there exists a parallel pseudo-opposite module 
  $\sfP_\alpha$ for $\sfF$ over $U_\alpha$.
\end{assumption}

\begin{definition}[cf.~Definition~\ref{def:Focksheaf}]  
Suppose that Assumption~\ref{assumption:covering_log} holds. 
We define the \emph{Fock sheaf} to be the sheaf of sets on $\cM$ 
obtained by gluing the local Fock spaces $\Fock(U_\alpha;\sfP_\alpha)$, 
$\alpha \in A$, using the transformation rule
\begin{align*}
  T(\sfP_\alpha,\sfP_\beta) : \Fock(U_{\alpha \beta}; \sfP_\alpha) 
\to \Fock(U_{\alpha\beta}; \sfP_\beta)
  && \alpha, \beta \in A
\end{align*}
over $U_{\alpha\beta} = U_\alpha \cap U_\beta$.
\end{definition} 

\begin{remark}
Note that the Fock sheaf in logarithmic case 
is a sheaf over all of $\cM$, not just over $\cM \setminus D$.
\end{remark}

\subsubsection{Correlation Functions Under Curved Opposite Modules} 
The discussion in \S\ref{subsec:curved} can be easily adapted to 
the logarithmic setting. 
The \emph{difference one-form} $\omega_{\sfP\sfQ}$ 
\eqref{eq:difference1-form} for pseudo-opposite modules $\sfP$, $\sfQ$ 
is now the pull-back of a logarithmic form in $\Omega^1_{\cM}(\log D)$. 
The \emph{curvature two-form} $\vartheta_\sfQ = d \omega_{\sfP\sfQ}$ 
(where $\sfP$ is a parallel pseudo-opposite module) 
in Definition~\ref{def:curvature2form} is  
the pull-back of a logarithmic form in $\Omega^2_{\cM}(\log D)$. 
We now give a definition of the local Fock space and the transformation 
rule for a general pseudo-opposite module in the logarithmic 
setting, leaving the necessary details to the reader. 

\begin{definition}[cf.~Definition~\ref{def:localFock_transformation_general}, 
Proposition~\ref{prop:Fockspace-curved}] 
Consider a $\log$-cTP structure $(\sfF,\nabla,(\cdot,\cdot)_{\sfF})$. 
Let $\sfQ$ be a (not necessarily parallel) pseudo-opposite module 
over $U$. 
The local Fock space $\Fock(U;\sfQ)$ consists of collections 
$\{C_{\sfQ,\mu}^{(1)} d\sx^\mu, C_{\sfQ}^{(1)}, C_{\sfQ}^{(2)}, 
C^{(3)}_{\sfQ}, \dots \}$ 
\begin{align*} 
C_{\sfQ,\mu}^{(1)} d\sx^\mu & \in \bOmega^1(\log D)(\pr^{-1}(U)^\circ) \\ 
C_\sfQ^{(g)} & \in \bcO(\pr^{-1}(U)^\circ)  \qquad \text{with } g\ge 2 
\end{align*} 
such that the following conditions hold: 
\begin{description} 
\item[(Grading \& Filtration)]
$C^{(1)}_{\sfQ;\mu}d\sx^\mu \in (\bOmega^1(\log D))^0_0$; 
$C^{(g)}_{\sfQ}\in \bcO^{2-2g}_{3g-3}$;  

\item[(Curvature)]
$d (C^{(1)}_{\sfQ;\mu} d \sx^\mu) = \vartheta_\sfQ$; 

\item[(Pole)] $P ( C^{(1)}_{\sfQ;\mu} d \sx^\mu) $ 
extends to a regular 1-form on $\pr^{-1}(U)$,
and for $g\ge 2$:
\[
C^{(g)}_{\sfQ} \in P^{-(5g-5)} 
\cO(U)[x_1,x_2,P x_3, P^2 x_4, \dots,P^{3g-4}x_{3g-2}]
\]
where $P = P(t,x_1)$ is the discriminant \eqref{eq:log_discriminant}. 
\end{description} 
Following the procedure in Proposition~\ref{prop:Fockspace-curved} 
in the logarithmic context, 
we can reconstruct multi-point correlation functions 
$\{C_{\sfQ;\mu_1,\dots,\mu_n}^{(g)}\}$ from 
the element $\{C_{\sfQ,\mu}^{(1)} d\sx^\mu, C_{\sfQ}^{(1)}, C_{\sfQ}^{(2)}, 
C^{(3)}_{\sfQ}, \dots \}$ in $\Fock(U;\sfQ)$; 
these multi-point functions 
again satisfy the conditions (Yukawa), (Grading\& Filtration) and (Pole) in Definition~\ref{def:localFock_log}.
(They do not necessarily satisfy (Jetness).) 
The transformation rule $T(\sfQ_1,\sfQ_2) \colon 
\Fock(U;\sfQ_1) \to \Fock(U;\sfQ_2)$ for two pseudo-opposite 
modules $\sfQ_1$, $\sfQ_2$ is defined in terms of these multi-point 
correlation functions and the Feynman rule as in Definition~\ref{def:localFock_log}: 
\[
C_{\sfQ_2;\mu_1,\dots,\mu_n}^{(g)} 
= \sum_{\Gamma} \frac{1}{|\Aut(\Gamma)|} 
\Cont_\Gamma\left(\{C_{\sfQ_1;\nu_1,\dots,\nu_m}^{(h)}\}; 
\Delta(\sfQ_1,\sfQ_2)\right)_{\mu_1,\dots,\mu_n}  
\]
or equivalently, in terms of the corresponding jet potentials 
as in \eqref{eq:transformationrule-jet_log}. 
\end{definition} 

\subsubsection{Anomaly Equation} 
Finally we remark on the anomaly equation in the logarithmic setting. 
The background torsion $\Lambda_\sfQ$ 
(Definition~\ref{def:background torsion}) is 
defined as an operator:
\[
\Lambda_{\sfQ} \colon \bOmegao^1(\log D) \otimes 
\bOmegao^1(\log D) \to \pr^*\Omega^1_{\cM}(\log D)
\]
This vanishes if and only if $\sfQ$ is parallel 
and satisfies the same properties as in Proposition~\ref{prop:propagator-curved}. 
The multi-point correlation functions 
$C^{(g)}_{\sfQ;\mu_1,\dots, \mu_n}$ under a pseudo-opposite 
module $\sfQ$ satisfy the same anomaly equation as before: 
\[
C^{(g)}_{\sfQ;\mu_1 \dots \mu_n} 
= \Nabla_{\mu_1}^{\sfQ} C^{(g)}_{\sfQ;\mu_2\dots \mu_n} 
+ \frac{1}{2} 
\sum_{\substack{\{2,\dots,n\} = S_1 \sqcup S_2 \\ 
k+ l = g}} 
C^{(k)}_{\sfQ; S_1, \alpha} 
{\Lambda_{\sfQ}}_{\mu_1}^{\alpha\beta}
C^{(l)}_{\sfQ; S_2, \beta}  
+ \frac{1}{2} 
C^{(g-1)}_{\sfQ;\mu_2\dots\mu_n\alpha\beta} 
{\Lambda_{\sfQ}}^{\alpha\beta}_{\mu_1}
\]
The curvature formulae in Proposition~\ref{prop:curvature} also 
hold in the logarithmic setting: here the curvature of $\Nabla^{\sfQ}$ 
is an $\End(\bOmegao^1(\log D))$-valued logarithmic 2-form 
on $\LLo$.

\section{Global Quantization and Givental Quantization}

\label{sec:Givental} 

In this section we explain the relationship 
between Givental quantization~\cite{Givental:quantization} 
and the global quantization constructed in \S\ref{sec:global_theory}.  
Givental defined the quantized operator $\hUU$ for a linear symplectic 
transformation $\U\in Sp(\cH)$ by specifying a certain 
normal ordering of quadratic Hamiltonians. 
When $\U$ is given by an upper-triangular loop group element 
$R = R(z) \in LGL_{+}(H_X^\C)$, 
Givental showed that $\widehat{R}$ acts on certain ancestor 
potentials satisfying the tameness condition. 
In \S\S \ref{sec:ancestor_Fock}-\ref{subsec:Givental-Global}, 
we will see that Givental's operator $\widehat{R}$
on ancestor Fock spaces (see Definition~\ref{def:quantizedoperator}) 
arises from our transformation rule (Definition~\ref{def:transformation}) 
in the formal neighbourhood of a point of $\LLo$. 
In \S\ref{subsec:global_L2}, we adapt the global quantization formalism 
in \S \ref{sec:global_theory} to the $L^2$-setting and explain 
that an $L^2$-version of the transformation rule matches with 
Givental's quantized operators for general (not necessarily 
upper or lower triangular) symplectic transformations. 

\subsection{Ancestor Fock Space} 
\label{sec:ancestor_Fock}

Let $K$ be a field containing $\Q$ and 
let $V$ be a finite dimensional $K$-vector space 
equipped with a symmetric non-degenerate 
pairing:
\[
\pair{\cdot}{\cdot}_V \colon V\otimes_K V \to K
\]
Recall Givental's tameness condition~\eqref{eq:tameness-Fockelement}.  We now introduce a Fock space for ``ancestor potentials" 
as the set of certain formal power series on $V[\![z]\!]$ 
which satisfy tameness.
Let $(q_0,q_1,q_2,\dots)$ be a sequence 
of variables in $V$ and denote a general element 
of $V[\![z]\!]$ by:
\[
\bq = \sum_{n=0}^\infty q_n z^n 
\]
Choosing a basis $\{e_i\}_{i=0}^N$ 
of $V$, we write 
$q_n = \sum_{i=0}^N q_n^i e_i$. 
For $\bD \in zV[\![z]\!]$, we introduce 
the co-ordinate system $\by = \sum_{n=0}^{\infty} y_n z^n$ 
on $V[\![z]\!]$ shifted by $\bD$:  
\[
\by = \bq + \bD
\]
Writing $\bD = \sum_{n=1}^\infty D_n z^n = 
\sum_{n=1}^\infty \sum_{i=0}^N 
D_n^i z^n e_i$, $y_n=\sum_{i=0}^N y_n^i e_i$, 
this gives: 
\[
y_n^i =\begin{cases}  
q_0^i  & n=0 \\
q_n^i + D_n^i & n\ge 1
\end{cases} 
\] 
In other words, $\by$ is an affine co-ordinate system 
on $V[\![z]\!]$ centred at $\bq(z) = -\bD$. 
This shift of co-ordinates is called the \emph{Dilaton shift} 
(cf.\ \S\ref{subsec:dilatonshift}). 
The following notions of ancestor Fock space 
and rationality for ancestor potentials were originally worked out
in a joint project with Hsian-Hua Tseng; see also~\cite{CI:convergence}.

\begin{definition}[Ancestor Fock Space] 
Let $V$ and $\bD$ be as above. 
The \emph{ancestor Fock space} $\Fockan(V,\bD)$ 
consists of formal power series 
\[
\cA = \exp\left(\sum_{g=0}^{\infty} \hbar^{g-1} \cF^g
\right) 
\] 
with $\cF^g \in K[\{y_n^i \}_{n\ge 2, 0\le i \le N}]
[\![y_0^0,\dots,y_0^N, y_1^0,\dots,y_1^N]\!]$ 
such that:
\begin{align}
\label{eq:tameness-potential} 
\begin{split} 
& \left. \cF^0 \right|_{\by=0} = 
\left.
\parfrac{\cF^0}{y_l^i} \right |_{\by=0}= 
\left. \parfrac{^2 \cF^0}{y_{l_1}^{i_1} \partial y_{l_2}^{i_2}} 
\right|_{\by=0}=0\\ 
& \left. \cF^1 \right|_{\by=0} = 0 \\ 
& \left. 
\parfrac{^n\cF^g}{y_{l_1}^{i_1} \cdots 
\partial y_{l_n}^{i_n}} 
\right |_{\by=0}
 = 0 \quad 
\text{ if $l_1+\cdots +l_n > 3g-3 + n$}
\end{split} 
\end{align} 
An element $\cA$ of $\Fockan(V,\bD)$ should be considered 
as a function on the formal neighourhood of 
$\bq(z) = -\bD \in zV[\![z]\!]$. 
We call $\cF^g$ the \emph{genus-$g$ potential} of $\cA$. 
Condition \eqref{eq:tameness-potential} is referred to as  \emph{tameness} 
of the genus-$g$ potential; cf.~the corresponding conditions \eqref{eq:jetcondition}, 
\eqref{eq:tameness-Fockelement} in the discussion of global quantization. 
When comparing with \eqref{eq:tameness-Fockelement}, 
note that the third line of \eqref{eq:tameness-potential} automatically implies: 
\[
\left.  
\parfrac{^n\cF^g}{y_{l_1}^{i_1} \cdots 
\partial y_{l_n}^{i_n}} 
\right |_{y_0=0}
 = 0 \quad 
\text{ if $l_1+\cdots +l_n > 3g-3 + n$}
\]
\end{definition} 

\begin{definition}[Rationality]
An element $\cA$ of $\Fockan(V,\bD)$ is said to be 
\emph{rational} if there exists a polynomial $P\in K[V^\vee]$ 
on $V$ with $P(-D_1)=1$ such that, whenever $(g,n) \ne(1,0)$:
\begin{equation} 
\label{eq:rationality} 
\left. 
\parfrac{^n\cF^g}{y_{l_1}^{i_1} \cdots 
\partial y_{l_n}^{i_n}} 
\right |_{\by=y_1 z = (q_1 + D_1) z}
=  
\frac{f_{g,L, I}(q_1)}{P(q_1)^{5g-5+2n - (l_1+\cdots+l_n)}}  
\end{equation} 
for some polynomials $f_{g,L,I} \in K[V^\vee]$ 
where $L = \{l_1,\dots,l_n\}$ and $I=\{i_1,\dots,i_n\}$.
By tameness \eqref{eq:tameness-potential}, 
$5g-5+2n-(l_1+\cdots+ l_n) = 
3g-3+n-(l_1+\cdots+l_n)+ 2g-2+n$ is  
positive unless the derivative vanishes or 
$(g,n)=(1,0)$. 
We call $P$ the \emph{discriminant} of $\cA$. 
We denote by $\Fockanrat(V,\bD,P)$ the set of rational 
elements in $\Fockan(V,\bD)$ with discriminant $P$. 
\end{definition} 

\begin{remark} 
\label{rem:tameness-rationality}
A potential satisfying tameness \eqref{eq:tameness-potential} 
and rationality \eqref{eq:rationality} can be expanded in 
the following form: 
\begin{align*}
\cF^{g}= 
\delta_{g,1} c^{(1)}(q_1) + 
\sum_{n:2g-2+n > 0}  
  \frac{1}{n!}
  \sum_{\substack{ 
      L : L = (l_1,\ldots,l_n) \\
      \text{$l_j \neq 1$ for all $j$} \\ 
      l_1 + \cdots + l_n \leq 3 g -3 + n}} 
 \sum_{I= (i_1,\dots, i_n)}
 c^{(g)}_{L,I}(q_1) \,
  q_{l_1}^{i_1}\cdots q_{l_n}^{i_n}
\end{align*}
with 
\[
\parfrac{c^{(1)}(q_1)}{q_1^i} 
= 
\frac{f_{1,1,i}(q_1)}{P(q_1)}, 
\qquad 
c^{(g)}_{L,I} (q_1) 
 = 
\frac{f_{g,L,I}(q_1)}{P(q_1)^{5g-5+2n
-(l_1+\cdots+l_n)}}   
\]
for some polynomials 
$f_{1,1,i}, f_{g,L,I}(q_1) \in R[V^\vee]$. 
The genus-one term $c^{(1)}(q_1)$ is in general not 
a rational function (see Example~\ref{exa:Apt} below). 
Given tameness, we can rephrase the rationality condition as follows:  
\[
\parfrac{^n \cF^g}{q_{l_1}^{i_1} \cdots \partial q_{l_n}^{i_n}} 
\in  P(q_1)^{-(5g-5+2n-(l_1+\cdots+l_n))} 
K[q_1, q_2, P(q_1) q_3, P(q_1)^2 q_4, \dots ]
[\![ P(q_1)^{-2} q_0]\!] 
\] 
for $2g-2+n>0$ (cf.~equation~\ref{eq:pole-formalization}).  
\end{remark} 

\begin{example}
\label{exa:Apt}
The ancestor Gromov--Witten potential $\cA_{{\rm pt}, t}$ 
of a point \eqref{eq:ancestorpotential} does not depend on 
$t\in H_{\rm pt}\cong \Q$ 
and coincides with the descendant potential 
$\cZ_{\rm pt}$ in \eqref{eq:totaldescendantpot}. 
This is called the \emph{Witten--Kontsevich tau-function} and
denoted by $\tau(\bq)$.  
It defines an element of $\Fockanrat(H_{\rm pt},1,-q_1)$ 
via the Dilaton shift (\S\ref{subsec:dilatonshift}): 
\[
q_n = y_n - \delta_{n,1}
\]
In fact, applying the Dilaton Equation, we find that:
\begin{equation} 
\label{eq:Fpt}
\cF^g_{\rm pt} 
= -  \frac{1}{24} \log (-q_1) \delta_{g,1}
+ 
\sum_{n:2g-2+n>0} \frac{1}{n!} 
\sum_{\substack{l_1,\dots,l_n \ge 0 \\ 
l_j \neq 1 \text{ for all $j$} \\ 
l_1+ \cdots + l_n = 3g-3+n}}
\frac{\langle \psi_1^{l_1},\dots, \psi_n^{l_n}\rangle_{g,n}^{\rm pt}
}{(-q_1)^{2g-2+n}} 
q_{l_1} \cdots q_{l_n}  
\end{equation} 
Hence we can take $P(q_1) = -q_1$. 
Note that $l_1+ \cdots + l_n =3g-3+n$ implies that: 
\[
2g-2+n = 5g-5+2n- (l_1+\cdots+l_n)
\]
\end{example}

\begin{definition}[Shift isomorphism] \ 
\label{def:shiftofdilatonshift}
\begin{enumerate}
\item  For $\bxi\in z^2 V[\![z]\!]$, 
the shift of co-ordinates $\tby = \by + \bxi$ 
preserves tameness \eqref{eq:tameness-potential} 
and defines a canonical isomorphism 
\[
T_{\bxi} \colon \Fockan(V,\bD) \cong \Fockan(V, \bD + \bxi) 
\quad 
\text{for $\bxi\in z^2 V[\![z]\!]$.}  
\] 
Thus $\Fockan(V,\bD)$ essentially
depends only on the leading term $z D_1$ of $\bD$. 

\item Let $P\in K[V^\vee]$, $\bD = \sum_{n\ge 1} D_n z^n 
\in zV[\![z]\!]$, 
$\bxi  = \sum_{n\ge 1} \xi_n z^n 
\in z V[\![z]\!]$ be such that 
$P(-D_1) =1$ and $P(-D_1 - \xi_1) \neq 0$. 
A truncated Taylor expansion with respect to the shifted 
co-ordinate $\tby = \by+\bxi$ 
defines an isomorphism: 
\[
T_{\bxi} \colon \Fockanrat(V,\bD, P) \cong 
\Fockanrat(V, \bD+\bxi, P/P(-D_1-\xi_1))
\]
This is given by $T_{\bxi} \cA = \exp( \sum_{g=0}^\infty 
\hbar^{g-1} T_{\bxi} \cF^g)$ with  
\[
T_{\bxi}\cF^g = 
\sum_{n: 2g-2+n>0} \sum_{\substack{L=(l_1,\dots,l_n) \\ 
l_j \neq 1 \text{ for all $j$}
}} 
\sum_{I= (i_1,\dots,i_n)} 
\frac{1}{n!} 
\parfrac{^n \cF^g}{y_{l_1}^{i_1} \cdots \partial y_{l_n}^{i_n}}
\Biggr|_{\by= (\ty_1- \xi_1) z} 
(\ty_{l_1}^{i_1}- \xi_{l_1}^{i_2}) \cdots (\ty_{l_n}^{i_n} - \xi_{l_n}^{i_n}) 
\]
where we set $\xi_0^i=0$, $\xi_n = \sum_{i=0}^N \xi_n^i$ 
for $n\ge 1$. 
It is easy to check that this shift preserves 
tameness \eqref{eq:tameness-potential} 
and rationality \eqref{eq:rationality}. 
Note that the Taylor expansion of $T_{\bxi}\cF^1$ 
is truncated so that it is zero at the shifted origin $\tby=0$. 
\end{enumerate}
\end{definition} 

Let $(V,\pair{\cdot}{\cdot}_V)$, 
$(W,\pair{\cdot}{\cdot}_W)$ be $K$-vector spaces 
with perfect pairings. 
A $K[\![z]\!]$-module isomorphism 
$R \colon V[\![z]\!]  \to W[\![z]\!]$ is said to be
\emph{unitary} if it satisfies
\[
\pair{R(-z) v_1}{R(z) v_2}_W = \pair{v_1}{v_2}_V 
\]
for all $v_1, v_2 \in V$. 
In this slightly more abstract setting, 
Givental's propagator from \S\ref{subsubsec:Giventalpropagator}  
can be described as follows. 

\begin{definition}[\!\cite{Givental:quantization}]
Let $R:V[\![z]\!] \to W[\![z]\!]$ be a unitary isomorphism.  
\emph{Givental's propagator} associated to $R$ 
is a bivector field $\Delta$ on
$V[\![z]\!]$ defined by
\[
  \Delta = \sum_{n=0}^\infty \sum_{m=0}^\infty 
  \sum_{i = 0}^N \sum_{j = 0}^N
  \Delta^{(n, i), (m,j)} \parfrac{}{q_n^i} \otimes 
  \parfrac{}{q_m^j}
  \]
  with: 
  \[
  \sum_{n=0}^\infty \sum_{m=0}^\infty 
  \sum_{i = 0}^N \sum_{j = 0}^N
  \Delta^{(n,i),(m,j)} (-1)^{n+m}
  w^n z^m= \Pair{e^i}{ 
    \frac{R(w)^\dagger R(z) - \id}{z+w} e^j}_V
\]
where $\{e^i\}$ is a basis of $V$ dual to $\{e_i\}$ 
with respect to $\pair{\cdot}{\cdot}_V$,
and $R(w)^\dagger$ denotes the adjoint of 
$R(w)$ with respect to $\pair{\cdot}{\cdot}_V$ 
and $\pair{\cdot}{\cdot}_W$. 
(Unitarity implies that $R(w)^\dagger = R(-w)^{-1}$.) 
\end{definition}

\begin{definition}[Givental~\cite{Givental:quantization}] 
\label{def:quantizedoperator}
For a unitary isomorphism $R \colon V[\![z]\!] \to W[\![z]\!]$, 
the quantized operator 
\[
\hR: \Fockan(V, \bD) \to \Fockan(W, R \bD)
\]
is defined as follows.  
For a given element $\cA \in \Fockan(V,\bD)$, we set
\[
\tcA = \exp\left(\frac{\hbar}{2} \Delta \right) \cA 
\in \Fockan(V,\bD) 
\]
where $\Delta$ is Givental's propagator associated to $R$,
and then push $\tcA$ forward along the identification $R(z) \colon
V[\![z]\!] \cong W[\![z]\!]$, so that:
\[
(\hR \cA) ( \bq ) := \tcA(R^{-1} \bq )
\]
\end{definition} 

\begin{theorem}[{Givental~\cite{Givental:nKdV}, 
Coates--Iritani~\cite{CI:convergence}}] 
\label{thm:quantization-welldefined} 
The quantized operator $\hR$ is well-defined, i.e.~it preserves the tameness condition \eqref{eq:tameness-potential}. 
Moreover, $\hR$ preserves rationality 
and induces an operator  
\[
\hR \colon \Fockanrat(V,\bD,P) \longrightarrow 
\Fockanrat(W, R\bD, P\circ R_0^{-1}) 
\]
where $R = R_0 + R_1 z + R_2 z^2 + R_3 z^3 + \cdots$ 
with $R_n \in \End_K(V,W)$. 
\end{theorem} 

\begin{remark} 
\label{rem:quantizedoperator-shiftisom}
When combined with the shift isomorphism in 
Definition~\ref{def:shiftofdilatonshift}, 
the quantized operator gives a map 
\[
T_{\bD'- R \bD} \circ \hR \colon \Fockan(V, \bD) \longrightarrow \Fockan(W, \bD') 
\]
for $\bD\in zV[\![z]\!]$, $\bD'\in zW[\![z]\!]$ such that 
$\bD' - R \bD \in z^2 W[\![z]\!]$. 
On the subspace of rational elements, we have a map: 
\[
T_{\bD'-R\bD} \circ 
\hR \colon 
\Fockanrat(V,\bD, P) \longrightarrow 
\Fockanrat(V, \bD', P\circ R_0^{-1}/P(-R_0^{-1}D'_1)) 
\]
when $\bD' = \sum_{n=1}^\infty D_n' z^n \in z W[\![z]\!]$  
satisfies $P(- R_0^{-1} D'_1) \neq 0$.  
\end{remark} 

\subsection{Global Quantization is Compatible with Givental Quantization} 
\label{subsec:Givental-Global}
We now show that Givental's quantized operator 
on ancestor Fock spaces 
(Definition~\ref{def:quantizedoperator}) 
arises from our transformation rule (Definition~\ref{def:transformation}) 
in the formal neighbourhood of a point of $\LLo$.   
Suppose that we are given a 
miniversal\footnote{See Assumption~\ref{assump:miniversal} for miniversality.}
cTP structure $(\sfF,\nabla,(\cdot,\cdot)_{\sfF})$ 
over $\cM$ as in Definition~\ref{def:cTP}. 
A \emph{unitary frame} at $t\in \cM$ is a $\C[\![z]\!]$-linear 
isomorphism 
\[
\Phi \colon V[\![z]\!] \cong \sfF_t 
\]
with a $\C$-vector space $V$ 
such that 
\[
\pair{v_1}{v_2}_V := (\Phi(v_1), \Phi(v_2))_{\sfF} 
\]
is independent of $z$ for any $v_1, v_2 \in V$. 
A unitary frame $\Phi$ admits a unique extension to an isomorphism 
$V(\!(z)\!) \cong \sfF_t[z^{-1}]$ 
of $\C(\!(z)\!)$-modules, which we also denote by $\Phi$.  
The following lemma is obvious from the proof of 
Lemma~\ref{lem:existence_opposite}. 

\begin{lemma} 
\label{lem:unitaryframe-opposite}
Let $V$ be a vector space over $\C$ of dimension $(N+1) = \rank \sfF$. 
A unitary frame $\Phi \colon V[\![z]\!] \cong \sfF_t$ 
at $t\in \cM$ defines a unique opposite module $\sfP$ 
over the formal neighbourhood of $t$ such that 
$\sfP_t = \Phi(z^{-1} V[z^{-1}])$. 
Conversely, any opposite module over the formal neighbourhood 
of $t$ determines a gauge equivalence 
class of unitary frame. 
\end{lemma} 

\begin{definition}[Formalization map]
\label{def:formalizationmap} 
Let $\sfP$ be an opposite module over an open set $U$. 
By the preceding lemma, $\sfP$ associates to a point $t\in U$ 
a unitary frame $\Phi \colon V[\![z]\!] \cong \sfF_t$ 
such that $\Phi(z^{-1} V[z^{-1}]) = \sfP_{t}$, where $V$ is a 
$\C$-vector space. 
Let $e_0,\dots,e_N$ be a basis of $V$. 
Recall from Definition~\ref{def:flatcoordinate_on_LL}
that the trivialization $\Phi^{-1} \colon \sfF_t \cong V[\![z]\!] = 
\bigoplus_{i=0}^N \C[\![z]\!] e_i$ and the opposite module $\sfP$ 
define a flat co-ordinate system $\{q_n^i\}_{n\ge 0, 0\le i\le N}$ 
on the formal neighbourhood $\hLLo$ of $\LLo_t$. 
Write: 
\[
\bq = \sum_{n=0}^\infty \sum_{i=0}^N q_n^i e_i z^n 
\colon \hLLo \longrightarrow  V[\![z]\!]
\] 
Take a point $\sx \in \LLo_t$ and let   
$-\bD = \bq|_{\sx} \in z V[\![z]\!]$ 
be the co-ordinate of $\sx$. 
The \emph{formalization map} 
$\For_\sx \colon \Fock(U;\sfP) \to \Fockan(V,\bD)$ 
is defined by the Taylor expansion: 
\[
\For_\sx ( \wave ) = \exp\left( 
\sum_{g = 0} ^\infty 
\sum_{n: 2g-2 + n >0} 
\sum_{l_1,\dots,l_n \ge 0}
\sum_{0\le i_1,\dots,i_n \le N} 
\frac{\hbar^{g-1}}{n!} 
\parfrac{^nC^{(g)}}{q_{l_1}^{i_1} \cdots \partial q_{l_n}^{i_n}}  
(\sx) 
y_{l_1}^{i_1} \cdots y_{l_n}^{i_n} 
\right) 
\]
where $\wave = \{ \Nabla^n C^{(g)}\} \in \Fock(U;\sfP)$ 
and $\by = \sum_{n=0}^\infty \sum_{i=0}^N y_n^i e_i z^n 
= \bD +\bq$.  
\end{definition} 
\begin{remark} \ 
  \begin{enumerate}
  \item The formalization $\For_\sx(\wave)$ is nothing but the jet potential $\exp(\cW(\sx,\sy))$ (Definition~\ref{def:jetpotential}) at the point $\sx$.  A small difference here is that $\For_{\sx}(\wave)$ is written in a specific co-ordinate system $\{y_n^i\}$ on $T_{\sx} \LLo$, induced by the flat co-ordinate system $\{q_n^i\}$ associated to a trivialization of $\sfF_{\pr(\sx)}$, whereas the jet potential is defined abstractly without a specific choice of co-ordinates.

  \item Because $C^{(0)}$, $\Nabla C^{(0)}$, $\Nabla^2 C^{(0)}$, $C^{(1)}$ are not defined, the Taylor series $\For_\sx(\wave)$ is truncated at genus zero and one.
  \end{enumerate}
\end{remark}

\begin{lemma} 
\label{lem:formalization-shiftisom}
The image of the formalization map $\For_\sx$ lies in 
the subspace $\Fockanrat(V, \bD, P_{t,D_1})$ of rational elements 
with discriminant 
\[
P_{t,D_1}(q_1) = P(t, q_1)/P(t,-D_1) \qquad  q_1\in V
\]
where $P(t,q_1)$
is the discriminant \eqref{eq:discriminant} on the total space $\LL$ 
written in terms of the unitary frame $\Phi$ which we used to define $\For_\sx$. 
Moreover we have the commutative diagram: 
\[
\xymatrix{
  \Fock(U;\sfP) \ar[rr]^-{\For_\sx} \ar@{=}[d] && \Fockanrat(V,\bD, P_{t,D_1}) \ar[d]^{T_{\bD'-\bD}}\\  
  \Fock(U;\sfP) \ar[rr]^-{\For_{\sx'}} && \Fockanrat(V,\bD', P_{t,D_1'}) 
}
\]
where $\bD' $ is an element of $zV[\![z]\!]$ 
such that $\sx ' = \Phi(-\bD') \in \LLo_t$ 
and the right vertical arrow is the shift isomorphism 
defined in Definition~\ref{def:shiftofdilatonshift}. 
\end{lemma} 
\begin{proof} 
The tameness of the formalization was established in 
\eqref{eq:tameness-Fockelement} and rationality 
was established in Proposition~\ref{prop:pole-flat}. 
The commutativity of the diagram is obvious from 
the definition. 
\end{proof} 
 
\begin{theorem} 
\label{thm:transformationrule-Giventalquantization} 
The transformation rule for the Fock sheaf is compatible with 
Givental's quantized operator $\hR$ in the following sense. 
Let $\sfP$, $\sfP'$ be two 
opposite modules over $U$ and 
let $\Phi\colon V[\![z]\!] \cong \sfF_t$, $\Phi' \colon V'[\![z]\!] 
\cong \sfF_t$ be the corresponding unitary frames at 
$t\in U$ via Lemma~\ref{lem:unitaryframe-opposite}. 
Let $R$ denote the unitary isomorphism 
\[
R := \Phi'^{-1} \circ \Phi \colon V[\![z]\!] 
\overset{\cong}{\longrightarrow} V'[\![z]\!] 
\] 
and let $\bD\in zV[\![z]\!]$, $\bD' \in zV'[\![z]\!]$ be 
such that $\sx = \Phi(-\bD)\in \LLo_t$ and 
$\sx' = \Phi'(-\bD') \in \LLo_t$. 
Let $P(t,q_1)$, $q_1\in V$, $P'(t,q'_1)$, $q_1'\in V'$ 
be the discriminants \eqref{eq:discriminant} 
written in terms of the trivializations $\Phi$ and $\Phi'$ respectively. 
Then we have $P'(t,q_1') = P(t,R_0^{-1} q_1')$
for $R_0 = R|_{z=0}$. 
Set:
\begin{align*}
  P_{t,D_1}(q_1) = P(t,q_1)/P(t,-D_1) &&
  P'_{t,D_1'}(q_1') = P'(t,q_1')/P'(t,-D'_1)
\end{align*}
Then there is a commutative diagram: 
\begin{align*} 
\xymatrix{
\Fock(U;\sfP) \ar[rr]^{T(\sfP,\sfP')} \ar[d]^{\For_\sx} 
&
& \Fock(U;\sfP') \ar[d]^{\For_{\sx'}} \\ 
\Fockanrat(V,\bD,P_{t,D_1}) \ar[rr]^{T_{\bD'-R\bD}\circ \hR}
& & 
\Fockanrat(V',\bD',P'_{t,D_1'}) }
\end{align*} 
\end{theorem} 
\begin{proof} 
By definition, the formalization map $\For_{\sx}$ assigns 
to a Fock space element the jet potential at $\sx$ viewed as a 
function on $V[\![z]\!]$, where $V[\![z]\!]$ is identified with 
$\bTheta_\sx$ via $d \bq = \Phi^{-1} \circ \KS \colon 
\bTheta_\sx \cong V[\![z]\!]$. 
On the other hand, we showed in 
Proposition~\ref{prop:Giventalprop=prop} that 
Givental's propagator coincides with the propagator 
for global quantization written in the frame 
$V[\![z]\!] \cong \bTheta_\sx$. 
The statement follows immediately from this, 
the definitions of $T(\sfP,\sfP')$ and $\hR$, 
and Lemma~\ref{lem:formalization-shiftisom}.  
\end{proof} 

\subsection{Global Quantization in the $L^2$-Setting}
\label{subsec:global_L2} 

We now describe global quantization  
in the $L^2$-setting and explain its relation 
to Givental's quantization. 
In particular we describe the quantization $\hUU$ 
of a symplectic transformation $\U \in Sp(\cH)$ 
which is not necessarily lower or upper triangular. 
One may notice a similarity between the $L^2$-formalism 
in this section and the Segal--Wilson Grassmannian \cite{Segal--Wilson}; 
whereas the general theory in \S \ref{sec:global_theory} 
is closer in spirit to the Sato Grassmannian \cite{Sato--Sato}. 
The $L^2$-formalism here also follows closely 
the heuristic argument in \S \ref{sec:globalquantization:motivation}. 
Since the discussion is analogous to \S \ref{sec:global_theory}, 
we will omit most of the details. 

In this section, we fix a miniversal TP structure 
$(\cF = \cO(F),\nabla,(\cdot,\cdot)_{\cF})$ with base $\cM$. 
We write $(\sfF,\nabla,(\cdot,\cdot)_{\sfF})$ for the 
corresponding cTP structure. 
Consider the space 
\[
\cH_t = L^2(\{t\} \times S^1, F) 
\]
of $L^2$-sections over $\{t\} \times S^1$.  
This has a non-degenerate symplectic form 
\[
\Omega_t(u,v) = \frac{1}{2\pi\iu} \int_{S^1} \big(u(-z),v(z)\big)_\cF \, dz  
\]
and contains the Lagrangian subspace 
\[
\F_t := \big\{ s(z) \in \cH_t : \text{$s$ is the boundary value of 
a holomorphic section over $\{t\}\times \D$}\big\}
\]
where $\D = \{z\in \C: |z|<1\}$ is the unit open disc. 
The pair $(\cH_t,\Omega_t)$ is an analogue of Givental's 
symplectic space (\S \ref{subsec:Givental-symplecticvs}) 
and $\F_t$ corresponds to a tangent space to the 
Givental cone (\S \ref{subsec:Lagrangian_TP}).  We fix a separable complex Hilbert space $\cH$ equipped 
with an orthonormal basis\footnote{The $L^2$-metric does 
not play a role.} 
$\{e^\alpha, f_\alpha : \alpha\in \Z_{\ge 0}\}$ 
and a symplectic form:
\begin{align*}
  \Omega(e^\alpha,f_\beta) = \delta_{\alpha\beta} &&
  \Omega(e^\alpha,e^\beta) = \Omega(f_\alpha,f_\beta) = 0
\end{align*}
We call $\{e^\alpha,f_\alpha\}$ the Darboux basis of $\cH$. 
We write $\{p_\alpha,q^\alpha : \alpha\in \Z_{\ge 0}\}$ for 
the dual linear co-ordinates on $\cH$, so that we have 
$\Omega = \sum_{\alpha} dp_\alpha \wedge dq^\alpha$. 
We have the standard decomposition $\cH = \cH_+ \oplus \cH_-$,  
where $\cH_+$ is spanned by $f_\alpha$ 
and $\cH_-$ is spanned by $e^\alpha$. We write 
\begin{align*}
  \bp = \sum_{\alpha=0}^\infty p_\alpha f^\alpha \in \cH_- && 
  \bq = \sum_{\alpha=0}^\infty q^\alpha e_\alpha \in \cH_+ 
\end{align*}
for variables in $\cH_\pm$. 

\begin{definition}[cf.~unitary frame in \S \ref{subsec:Givental-Global}] 
\label{def:Darboux_frame} 
A \emph{Darboux frame} of the TP structure $(\cF,\nabla,(\cdot,\cdot)_\cF)$ 
at $t\in \cM$ is an isomorphism  
\[
\Phi_t \colon \cH \to \cH_t 
\] 
of topological vector spaces such that:
\begin{enumerate} 
\item $\Phi_t$ intertwines the symplectic forms $\Omega$ and $\Omega_t$; 
\item the projection $\Phi_t^{-1}(\F_t) \to \cH_+$ along $\cH_-$ 
is an isomorphism.
\end{enumerate} 
\end{definition} 
\noindent Suppose that a Darboux frame $\Phi_t$ at $t$ is given. 
When $t'$ is close to $t$, 
parallel translation by $\nabla$ defines a symplectic isomorphism 
\[
P_{tt'} \colon \cH_t \cong \cH_{t'} 
\]
and thus the Darboux frame $\Phi_t$ induces a 
frame $\Phi_{t'} = P_{tt'} \circ \Phi_t \colon \cH \cong \cH_{t'}$ 
that respects the symplectic forms. We note that condition (2) 
remains true for $\Phi_{t'}$ whenever $t'$ is sufficiently close to $t$. 
Therefore a Darboux frame at any point extends to 
its small neighbourhood by parallel translation. 

\begin{example} 
\label{exa:Darbouxframe_by_triv} 
Suppose that we have a trivialization 
$\phi\colon \C^{N+1} \otimes\cO_{\C^\times} \cong 
\cF|_{\{t\}\times \C^\times}$ 
such that:
\begin{itemize} 
\item $(\phi(e_i)(-z),\phi(e_j)(z))_{\cF} = \delta_{ij}$; 
\item  letting $\cF^{(\infty)}$ be the extension of $\cF|_{\{t\}\times \C}$ across 
$z=\infty$ such that the sections $\{\phi(e_i) : 0\le i\le N\}$ 
extend to $z=\infty$ and form a basis there, we have that $\cF^{(\infty)}$ 
is trivial as a holomorphic vector bundle over $\Proj^1$.  
\end{itemize} 
This induces a Darboux frame, by identifying $\cH$ with 
the space 
$L^2(S^1,\C^{N+1})$ equipped with the Darboux basis 
$\{\text{$e_i (-z)^{-n-1}$, $e_i z^n$} : \text{$n \ge 0$, $0\le i\le N$}\}$. 
The subspace $\cH_+$ corresponds to the space of non-negative Fourier series 
$\sum_{n\ge 0} a_n z^n$ and the subspace $\cH_-$ corresponds to the space of 
strictly negative Fourier series $\sum_{n<0} a_n z^n$. 
Condition (2) follows from the triviality of $\cF^{(\infty)}$. 
\end{example} 

\begin{example} 
\label{exa:Darbouxframe_fundsol}
This is a special case of Example \ref{exa:Darbouxframe_by_triv}. 
Suppose that the genus-zero Gromov--Witten potential $F^0_X$ 
is convergent. 
Then the fundamental solution $L(t,z)$ 
(equation \ref{eq:fundamentalsolution}) with $Q=1$ 
defines a Darboux frame of the A-model TP structure 
(Example \ref{ex:AmodelTP}), 
by identifying $\cH$ with 
Givental's symplectic vector space (\S \ref{subsec:Givental-symplecticvs}) 
for $X$. 
\end{example} 

\begin{example} 
\label{exa:compatible_with_L2}
We say that a parallel pseudo-opposite module 
$\sfP$ for $(\sfF,\nabla,(\cdot,\cdot)_\sfF)$ 
is \emph{compatible with the $L^2$-structure} if:
\begin{itemize} 
\item every element of $\sfP_t \subset \sfF_t[z^{-1}]$ extends to 
a holomorphic section of $F|_{\{t\}\times \D^*}$ over the 
unit punctured disc $\D^* = \{z\in \C : 0<|z|<1\}$ 
and has an $L^2$-boundary value along $S^1$. 
Thus $\sfP_t$ is a subspace of $\cH_t = L^2(\{t\} \times S^1, F)$; 
\item the $L^2$-closure $\Proj_t$ of $\sfP_t$ is complementary to 
$\F_t$, i.e.~$\cH_t= \Proj_t \oplus \F_t$. 
\end{itemize}
Then we can find a Darboux frame $\Phi_t$ such that 
$\Phi_t(\cH_-) = \Proj_t$. When this holds, we say that the 
\emph{Darboux frame $\Phi_t$ is compatible with $\sfP$}. 
Given a Darboux frame, one may not be able to find a 
parallel pseudo-opposite module compatible with the Darboux 
frame. 
Darboux frames from Example \ref{exa:Darbouxframe_by_triv} 
are compatible with the corresponding opposite modules. 
\end{example} 

Let $\Phi$ be a Darboux frame extended by parallel translation to a simply-connected 
open set $U\subset \cM$. 
We consider the map from the $L^2$-subspace $L^2(\LLo)|_U$ 
(see Remarks \ref{rem:L2-neighbourhood}, \ref{rem:L2subspace_from_TP}) 
into $\cH$: 
\[
\iota \colon L^2(\LLo)|_U \to \cH \qquad 
(t,\bx) \mapsto \Phi_t^{-1} \bx
\]
Miniversality implies that the differential $d\iota$ is 
injective and that $d\iota(T_{(t,\bx)}L^2(\LLo)) =\Phi_t^{-1}\F_t$. 
Therefore $\iota$ is a Lagrangian immersion. 
The image $\cL=\iota(L^2(\LLo)|_U)$ is preserved by multiplication by 
$\C^\times$ and we call it the \emph{Givental cone} 
associated to the Darboux frame $\Phi$. 
The projection $\cL \to \cH_+$ along $\cH_-$ 
is a local isomorphism (by the inverse function theorem for 
Hilbert manifolds) and therefore $\cL$ 
can be locally written as the graph 
\[
\cL = \left\{(\bp,\bq) \in \cH: p_\alpha = \parfrac{C^{(0)}}{q^\alpha} \right\}
\]
of the differential of a holomorphic function\footnote
{For holomorphic functions in infinite dimensions, we refer the reader 
to~\cite{Chae}.}
$C^{(0)} \colon \cH_+ \to \C$. 
The function $C^{(0)}$ is defined up to a constant; we can 
fix the constant ambiguity by requiring that $C^{(0)}$ is homogeneous 
of degree two with respect to the dilation of co-ordinates 
$\bq$. Thus we have:
\[
C^{(0)} = \frac{1}{2} 
\sum_{\alpha=0}^\infty q^{\alpha}\parfrac{C^{(0)}}{q^\alpha} 
= \frac{1}{2} \Omega(\bp,\bq) \Bigr |_{\cL} 
\]
We call $C^{(0)}$ the \emph{genus-zero potential associated 
to $\Phi$}. 
This is an $L^2$-version of the genus-zero potential in 
\S \ref{subsec:flatstronL}
(see also Remark \ref{rem:L2-neighbourhood}). 
The third derivative 
\[
C^{(0)}_{\alpha\beta\gamma} = \parfrac{C^{(0)}}{
q^\alpha \partial q^\beta \partial q^\gamma} 
\]
coincides with the Yukawa coupling on $L^2(\LLo)$, via
the projection $L^2(\LLo) \looparrowright \cH \to \cH_+$. 
(Here $\looparrowright$ means an immersion.) 

\begin{definition} 
\label{def:Darboux_close} 
Let $\Phi_1$, $\Phi_2$ be Darboux frames of the TP structure 
$(\cF,\nabla,(\cdot,\cdot)_\cF)$ at $t$. We say that 
$\Phi_1$ and $\Phi_2$ are \emph{close} if the map 
\[
\Pi_+ \Phi_2^{-1}\Phi_1 \colon 
\cH_- \xrightarrow{\Phi_1} \cH_t \xrightarrow{\Phi_2^{-1}} \cH 
\xrightarrow{\Pi_+} \cH_+ 
\]
is of trace class. Here $\Pi_+$ denotes the projection 
along $\cH_-$. Being close is an equivalence relation. 
\end{definition} 

Given two Darboux frames $\Phi_1,\Phi_2$, we have a 
symplectic transformation $\U$ such that $\Phi_1 = \Phi_2 \U$. 
We write $\U$ in the block matrix form: 
\begin{equation} 
\label{eq:U_block} 
\U = \begin{pmatrix} 
A & B \\ 
C & D 
\end{pmatrix} 
\end{equation} 
where $A \in \Hom(\cH_-,\cH_-)$, $B\in \Hom(\cH_+,\cH_-)$, 
$C \in \Hom(\cH_-,\cH_+)$, $D \in \Hom(\cH_+,\cH_+)$. 
The frame $\Phi_1$ is close to $\Phi_2$ 
if and only if $C$ is of trace class. 
Using the basis $\{e^\alpha,f_\alpha\}$, 
we regard $A,B,C,D$ as infinite matrices, 
writing $A e^\beta = {A_\alpha}^\beta e^\alpha$, 
$B f_\beta = B_{\alpha\beta} e^\alpha$, 
$C e^\beta = C^{\alpha\beta} f_\alpha$, 
$D f_\beta = {D^\alpha}_\beta f_\alpha$. 
The symplectic property of $\U$ implies that: 
\[
\U^{-1} 
= \begin{pmatrix} 
D^\tr & -B^\tr \\ 
-C^\tr & A^\tr 
\end{pmatrix}  
\]
where $\tr$ stands for the transpose. 
In particular, we see that $\Phi_1$ is close to $\Phi_2$ if and only if 
$\Phi_2$ is close to $\Phi_1$. 
\begin{example} 
All Darboux frames arising from the method of 
Example \ref{exa:Darbouxframe_by_triv} 
are close to each other. In fact, the symplectic transformation $\U$ 
relating two Darboux frames in Example \ref{exa:Darbouxframe_by_triv} 
is given by the multiplication by a loop group element 
$\gamma(z) \in C^\infty(S^1,GL_{N+1}(\C))$, which is 
the gauge transformation between the two trivializations. 
In this case, the operator $C\in \Hom(\cH_-,\cH_+)$ 
is given by $f(z) \mapsto [\gamma(z) f(z) ]_+$ with $f(z)\in \cH_-$.  
It is easy to see that this defines a linear operator 
of trace class (see e.g.~\cite[Proposition 2.3]{Segal--Wilson}). 
If moreover $\gamma(z)$ is a Laurent polynomial loop, we can see 
that $C$ is a finite rank operator. (This is the typical 
situation when $\U$ arises from the monodromy of a TEP structure.)  
\end{example}

Let $\cL_i$, $i\in \{1,2\}$, be the Givental cones associated 
to the Darboux frame $\Phi_i$, $i\in \{1,2\}$.  
The symplectic transformation $\U$ maps $\cL_1$ isomorphically 
onto $\cL_2$: $\U \cL_1 = \cL_2$. 
By identifying the two Givental cones via $\U$, 
we will mainly work with $\cL_1$. 
For a point $\sx\in \cL_1$, we have 
$\cH = T_\sx \cL_1 \oplus \cH_- 
= T_\sx \cL_1 \oplus \U^{-1}\cH_-$. Thus 
the symplectic form $\Omega$ defines two isomorphisms  
\begin{align}
\label{eq:sharps}
\begin{split}   
\sharp_1 \colon & \cH_- \cong (T_\sx \cL_1)'  
\quad \qquad v \mapsto \iota_v \Omega = \Omega(v, \cdot)  \\ 
\sharp_2 \colon & \U^{-1}\cH_- \cong (T_\sx \cL_1)' 
\quad \; v \mapsto \iota_v \Omega = \Omega(v,\cdot) 
\end{split} 
\end{align} 
where $(T_\sx \cL_1)'$ means the topological dual of $T_\sx \cL_1$. 
We define the propagator in the $L^2$-setting as follows. 

\begin{definition}[cf.~Definition \ref{def:propagator}]
\label{def:propagator_L2} 
The propagator $\Delta = \Delta(\Phi_1,\Phi_2)$ associated to the two Darboux 
frames $\Phi_1$, $\Phi_2$ is the bivector field 
$\Delta$ on $\cL_1$ defined by 
\[
\Delta(v_1,v_2) = \Omega\big(\sharp_1^{-1}(v_1), \sharp_2^{-1}(v_2)\big)
\]
with $v_1,v_2 \in (T_\sx \cL_1)'$. 
\end{definition} 

\noindent The propagator is symmetric.

The projection $\cL_1 \to \cH_+$ along $\cH_-$ 
defines a local co-ordinate system $(q^0,q^1,q^2,q^3,\dots)$ 
on $\cL_1$. We will find a co-ordinate expression for the propagator. 
We write $\Delta^{\alpha\beta} =  \Delta(dq^\alpha,dq^\beta)$.
Let $C^{(0)}$ denote the genus-zero potential 
associated with $\Phi_1$. Define $\tau$ to be the matrix with coefficients: 
\[
\tau_{\alpha\beta} = \parfrac{^2C^{(0)}}{q^\alpha \partial q^\beta}
\]
This defines a bounded bilinear form $\cH_+\times \cH_+ \to \C$; 
it can be also viewed as  
a bounded linear operator $\cH_+ \to \cH_-$. 

\begin{lemma} 
\label{lem:propagator_traceclass} 
The operator $C \tau + D \colon \cH_+ \to \cH_+$ is an 
isomorphism and 
the propagator is given by:
\[
\Delta^{\alpha\beta} = -\left[(C \tau + D)^{-1} C\right]^{\alpha\beta}
\]
In particular, if $\Phi_1$ and $\Phi_2$ are close, 
the propagator $\Delta^{\alpha\beta}$ is of trace class 
as a linear operator $(T_\sx \cL_1)' \to T_\sx \cL_1$. 
\end{lemma} 
\begin{proof} 
The projection $\cL_2 \to \cH_+$ along $\cH_-$ introduces 
co-ordinates $(q^0,q^1,q^2,\dots)$ on $\cL_2$. 
The tangent map $\cH_+ \cong T_\sx \cL_1 \to 
T_{\U(\sx)} \cL_2 \cong \cH_+$ of $\U$ is given 
in these co-ordinates as: 
\[ 
\bq \mapsto 
\begin{pmatrix} 
\tau \bq \\
\bq
\end{pmatrix} 
\xmapsto{\phantom{X}\U\phantom{X}} 
\begin{pmatrix}
(A \tau + B) \bq \\
(C\tau + D) \bq
\end{pmatrix}
\mapsto 
(C\tau + D)\bq.  
\] 
Thus $C\tau + D$ is a linear isomorphism. 
These co-ordinates on $\cL_1$,~$\cL_2$ identify the 
cotangent spaces $(T_\sx \cL_1)'$,~$(T_{\U(\sx)} \cL_2)'$ 
with $\cH_-$. 
Using these co-ordinatizations and the above identification $T_\sx \cL_1 \cong 
T_{\U(\sx)}\cL_2$, we can view the propagator 
as the bilinear form 
$(T_\sx \cL_1)' \times (T_{\U(\sx)} \cL_2)' \to \C$ given by:
\[
\bp_1 \times \bp_2 \longmapsto \Omega(\bp_1,\U^{-1} \bp_2) 
= - \bp_1 \cdot (C^\tr \bp_2)
\]
Since the covector $\bp_2\in (T_{\U(\sx)}\cL_2)'$ 
corresponds to the covector $(C\tau + D)^{\tr} \bp_2 \in (T_\sx\cL_1)'$, 
the conclusion follows.  
\end{proof} 

We give a definition of the local Fock space in the $L^2$-setting. 
The definition here is very simple. 
\begin{definition}[cf.~Definition \ref{def:localFock}]
\label{def:Fockspace_L2} 
Let $\Phi$ be a Darboux frame and let $\cL$ be the 
Givental cone associated to $\Phi$. 
For an open subset $\cU$ of $\cL$, 
the local Fock space $\Fock_{L^2}(\cU,\Phi)$ consists of tuples 
\[
\{dC^{(1)}, C^{(2)}, C^{(3)}, \dots\}
\] 
where $dC^{(1)}$ is a holomorphic closed one-form on $\cU$ and 
$C^{(g)}$, $g\ge 2$, are holomorphic functions on $\cU$. 
We call $C^{(g)}$ the genus-$g$ potential. 
\end{definition} 

\begin{remark} 
\label{rem:Fockspace_comparison} 
Suppose that a Darboux frame $\Phi$ is compatible with 
a parallel pseudo-opposite module $\sfP$. When $\cU 
\subset \cL$ is the image of an open subset of 
$L^2(\LLo)|_U$, there is a natural restriction map 
$\Fock(U;\sfP) \to \Fock_{L^2}(\cU;\Phi)$. 
\end{remark} 

\begin{remark} 
The $n$-fold derivative of the genus-$g$ potential defines an $n$-tensor: 
\[
C^{(g)}_{\alpha_1\dots\alpha_n} = \frac{\partial^n C^{(g)}}
{\partial q^{\alpha_1} \cdots \partial q^{\alpha_n}} 
\]
At each point $\sx \in \cL$, 
this defines a bounded multi-linear form on $T_\sx \cL$. 
\end{remark} 

We now describe the transformation rule in the $L^2$-setting. 
Let $\Phi_1$,~$\Phi_2$ be Darboux frames which are close to 
each other in the sense of Definition \ref{def:Darboux_close}. 
Let $\cL_i$ be the Givental cone associated to 
$\Phi_i$ for $i=1,2$ and let $\Delta = \Delta(\Phi_1,\Phi_2)$ 
be the propagator. Let $\U = \Phi_2^{-1}\Phi_1$  be the 
symplectic transformation. 
As usual we introduce co-ordinates on $\cL_1$ 
by the projection $\cL_1 \to \cH_+$ along $\cH_-$ 
and we regard the genus-zero potential $C^{(0)}$ associated to 
$\Phi_1$ as a function on $\cL_1$. 
We have another genus-zero potential $\hC^{(0)} \colon \cH_+ \to \C$ 
associated to the Darboux frame $\Phi_2$. Via the identification 
$\U \colon \cL_1 \xrightarrow{\cong} \cL_2$ followed by the projection 
$\cL_2 \to \cH_+$, we also regard $\hC^{(0)}$ 
as a function on $\cL_1$. Although the functions $C^{(0)}$, 
$\hC^{(0)}$ do not match, the third derivatives match:  
\[
C^{(0)}_{\alpha\beta\gamma} = \hC^{(0)}_{\alpha\beta\gamma}   
\]
as they are the Yukawa coupling. 

\begin{definition}[transformation rule in the $L^2$-setting; cf.~Definition \ref{def:transformation}] 
\label{def:transformation_L2} 
Let $\Phi_1,\Phi_2$ be Darboux frames which are close to each other.  
We use notation as above. 
Let $\cU\subset \cL_1$ be an open subset. 
For an element $\{d C^{(1)},C^{(2)},C^{(3)},\dots\}$ 
of $\Fock_{L^2}(\cU;\Phi_1)$, we define a tuple 
\[
\left\{\hC^{(g)}_{\alpha_1,\dots,\alpha_n}: g\ge 0, n\ge 0, 2g-2+n>0
\right\}
\]
of holomorphic tensors on $\cU$ 
by the same Feynman rule as in Definition \ref{def:transformation}: 
\[
\hC^{(g)}_{\alpha_1,\dots,\alpha_n} = 
\sum_{\Gamma} \frac{1}{|\Aut(\Gamma)|} 
\Cont_\Gamma( \{C^{(h)}_{\beta_1,\dots,\beta_m}\}, 
\Delta)_{\alpha_1,\dots,\alpha_n} 
\]
where $\Gamma$ ranges over all decorated stable graphs 
with legs $\alpha_1,\dots,\alpha_n$ as in 
Definition \ref{def:transformation}.  
We can check, by a similar argument to the previous case, 
that the new correlators satisfy the jetness condition
\[
\parfrac{\hC^{(g)}_{\alpha_1\dots\alpha_n}}{q^{\beta}} 
= \hC^{(g)}_{\beta \alpha_1 \dots \alpha_n} 
\]
and therefore they are determined by 
the tuple $\{d \hC^{(1)}, \hC^{(2)},\hC^{(3)},\dots\}$. 
We can regard $\hC^{(g)}_{\alpha_1 \dots \alpha_n}$ 
as a tensor on $\U(\cU) \subset \cL_2$ via the identification 
$\U \colon \cL_1 \cong \cL_2$. 
Therefore we obtain a \emph{transformation rule} 
\[
\hUU \colon \Fock_{L^2}(\cU;\Phi_1) \to 
\Fock_{L^2}(\U(\cU);\Phi_2) 
\]
sending $\{dC^{(1)},C^{(2)},C^{(3)},\dots\}$ 
to $\{d\hC^{(1)},\hC^{(2)},\hC^{(3)},\dots\}$. 
Integrating $dC^{(1)}$ and $d \hC^{(1)}$ locally to holomorphic functions  
$C^{(1)}$ and $\hC^{(1)}$, we consider the total potentials:
\begin{align*} 
\cZ &= \exp\left(
\frac{1}{\hbar} C^{(0)} + C^{(1)} + C^{(2)} \hbar + C^{(3)} \hbar^2 
\cdots \right)  \\ 
\hcZ &= \exp \left(\frac{1}{\hbar} 
\hC^{(0)} + \hC^{(1)} + \hC^{(2)} \hbar + \hC^{(3)} \hbar^2 
\cdots \right)
\end{align*} 
With this notation, we write 
\[
\hcZ \propto \hUU \cZ.
\] 
where $\propto$ indicates that 
we have a constant ambiguity at genus one. 
\end{definition} 

\begin{remark} 
In the above definition, it is important that $\Phi_1$ and $\Phi_2$ 
are close to each other in the sense of Definition \ref{def:Darboux_close}. 
The closeness implies that $\Delta$ is of trace class by Lemma 
\ref{lem:propagator_traceclass}, 
and thus ensures that the contraction 
$\Cont(\Gamma)_{\alpha_1\dots\alpha_n} = 
\Cont_\Gamma(\{C^{(h)}_{\beta_1\dots\beta_m}\},
\Delta)_{\alpha_1\dots\alpha_n}$ over a graph $\Gamma$ 
defines a bounded multi-linear form on $T_\sx \cL_1$. 
We can prove this by induction on the number of 
edges: by removing one edge from $\Gamma$ we can write 
\[
\Cont(\Gamma)_{\alpha_1\dots\alpha_n} 
= \begin{cases} 
\Cont(\Gamma_1)_{\alpha_{i_1}\dots\alpha_{i_k} \beta_1} 
\Delta^{\beta_1\beta_2} 
\Cont(\Gamma_2)_{\alpha_{j_1}\dots \alpha_{j_l} \beta_2} 
& \text{separating case;} \\ 
\Cont(\Gamma')_{\alpha_1\dots\alpha_n\beta_1\beta_2} \Delta^{\beta_1
\beta_2} & \text{non-separating case} \\
\end{cases} 
\]
where $\{i_1,\dots,i_k\} \sqcup \{j_1,\dots,j_l\} = \{1,\dots,n\}$. 
In the former case, the well-definedness follows from the fact that 
$\Delta$ is a bounded bilinear form and the induction hypothesis; 
in the latter case it follows from the fact that $\Delta$ is 
of trace class and the induction hypothesis. 
\end{remark} 

\begin{remark} 
The jetness of the new correlation functions 
$\hC^{(g)}_{\alpha_1\dots\alpha_n}$ follows from 
the formula: 
\[
\partial_\alpha \Delta= (C\tau + D)^{-1} C (\partial_\alpha \tau)  
(C\tau + D)^{-1} C = \Delta (\partial_\alpha\tau) \Delta.  
\] 
This is an analogue of Proposition \ref{prop:difference_conn}(2). 
\end{remark}

The following proposition is obvious from the definition. 

\begin{proposition} 
Let $\Phi_1$,~$\Phi_2$ be Darboux frames and 
$\sfP_1$,~$\sfP_2$ be parallel pseudo-opposite modules for 
the cTP structure 
$(\sfF,\nabla,(\cdot,\cdot)_\sfF)$ over $U$. 
Suppose that $\Phi_i$ is compatible with $\sfP_i$ 
for $i=1,2$. 
Then the transformation rule $\hUU$ above coincides with the 
transformation rule $T(\sfP_1,\sfP_2)$ 
from Definition \ref{def:transformation} under the identification  
$L^2(\LLo)|_U \cong \cL_1$. 
\end{proposition} 

\begin{remark} 
  \label{rem:very_detailed_remark}
Here we describe the relationship to Givental's quantization \cite{Givental:quantization} 
of a symplectic transformation $\U \in Sp(\cH)$. 
In Givental's formalism, we regard the total potential 
$\cZ$ (respectively $\hcZ$) as a function 
on $\cH_+$ via the projection $\cL_1 \to \cH_+$ along $\cH_-$ 
(respectively via the projection $\cL_2 \to \cH_+$ along $\cH_-$). 
We assume that the component $C\in \Hom(\cH_-,\cH_+)$ 
of $\U$ (see equation~\ref{eq:U_block}) is of trace class as before. 
There are two cases\footnote{Unfortunately these terminologies 
are opposite to the shape of the matrix $\U$.}: 
\begin{description} 
\item[lower triangular] $\U$ preserves $\cH_-$, i.e.~$C=0$. 
\item[upper triangular]  $\U$ preserves $\cH_+$, i.e.~$B=0$. 
\end{description} 
We describe Givental's quantized operator $\hUU$ in these two cases. 
More generally we decompose $\U$ into the product $\U_+ \U_-$ 
of a lower-triangular transformation $\U_-$ and an upper-triangular 
transformation $\U_+$ and define $\hUU = \hUU_+ \hUU_-$. 

In the lower-triangular case, $\hUU$ acts on the higher-genus 
potentials $C^{(1)}, C^{(2)}, C^{(3)},\dots$ 
by the change of variables $\bq \to D^{-1}\bq$ 
and on the genus-zero potential $C^{(0)}$ by the same change 
of variables followed by the shift by a quadratic function. 
We define (see \cite[Proposition 5.3]{Givental:quantization} and 
Remark \ref{rem:lower_triangular}):
\[
(\hUU \cZ)(\bq) = 
e^{\frac{1}{2\hbar}\Omega(B D^{-1}\bq, \bq)} \cZ(D^{-1}\bq)
\]
This coincides with our transformation rule in Definition \ref{def:transformation_L2} 
as in this case we have $\Delta =0$ and the transformation rule 
is essentially a co-ordinate change. 

In the upper-triangular case, the quantized operator $\hUU$ is 
more complicated. 
The symplectic condition for $\U$ now reads: 
\begin{align*}
  A = (D^\tr)^{-1} &&
  A^\tr C = C^\tr A
\end{align*}
Givental's propagator $V$ is defined by the formula 
(cf.~\S \ref{subsubsec:Giventalpropagator}):
\[
V^{\alpha\beta} = - (A^\tr C)^{\alpha\beta} = - (D^{-1} C)^{\alpha \beta}
\]
This is a symmetric tensor of trace class. 
Givental's propagator $V$ arises from the definition 
of $\Delta$ by replacing 
$(T_\sx \cL)'$ in the isomorphisms \eqref{eq:sharps} 
with $(\cH_+)'$, namely, if we write 
$\flat_1 \colon \cH_- \cong (\cH_+)'$, 
$\flat_2 \colon \U^{-1}\cH_- \cong (\cH_+)'$ 
for the isomorphisms given by the symplectic form, we have:
\[
V^{\alpha\beta} = \Omega\left(\flat_1^{-1} dq^\alpha, 
\flat_2^{-1} dq^\beta\right)
\]
Givental's quantized operator $\hUU$ is given 
by the formula \cite[Proposition 7.3]{Givental:quantization}: 
\begin{equation} 
\label{eq:quantization_Givental} 
(\hUU \cZ)(\bq) = \left( \exp\left( \frac{\hbar}{2} 
V^{\alpha \beta} \partial_{q^\alpha}\partial_{q^\beta} \right) 
\cZ\right)(D^{-1}\bq) 
\end{equation} 
We show that the right-hand side is well-defined if $\U$ is close to the identity 
(in the operator norm), 
and gives the same result as the transformation rule from 
Definition \ref{def:transformation_L2}.  
Suppose that the total potential $\cZ$ is defined 
in a neighbourhood of $\bq_1\in \cH_+$ which 
is the projection of $\sx = (\bp_1,\bq_1) \in \cL_1$ 
to $\cH_+$. Here $\bp_1 = d C^{(0)}(\bq_1)$. 
Let 
\[
\bq_2 = [\U \sx]_+ = C \bp_1 + D \bq_1 
\] 
denote the projection of the point $\U\sx \in \cL_2$ to $\cH_+$.  
We show that the total potential $\hcZ = \hUU \cZ$ is well-defined 
in a neighbourhood of $\bq_2$ (when $\U$ is close to the identity): 
we shall evaluate the right-hand side of 
\eqref{eq:quantization_Givental} at $\bq = \bq_2$. 
We also write 
\begin{equation}
\label{eq:q2q1}
\bq_2' = D^{-1} \bq_2 = D^{-1} C \bp_1 + \bq_1 
= - V \bp_1 + \bq_1 
\end{equation} 
and assume that $\cZ$ is analytically continued to $\bq_2'$. 
This is possible if $V\bp_1$ is small, so if 
$\U$ is close to the identity. 
The formula \eqref{eq:quantization_Givental} can be written 
as a similar Feynman rule: 
\[
\hC^{(g)}(\bq_2) = \sum_{\Gamma} \frac{1}{|\Aut(\Gamma)|}
\Cont_\Gamma\left(\{C^{(h)}_{\alpha_1\dots\alpha_n}(\bq_2')\}, V\right) 
\]
where $\Gamma$ now ranges over connected decorated graphs without legs 
which are not necessarily stable: 
we allow genus-zero vertices of $\Gamma$ to have 
one or two incident edges. There are infinitely many such graphs, 
and the convergence of the above sum is nontrivial. 
We consider the following process of collapsing graphs 
and reduce the above sum to a sum over stable graphs. 
Let  $\Gamma$ be a possibly unstable decorated graph 
without legs. 
We collapse every subtree of $\Gamma$ consisting of  
genus-zero vertices to its root vertex (see Figure \ref{fig:collapsing_trees}). 
\begin{figure}[htbp]
\centering 
\includegraphics[bb=44 579 486 787]{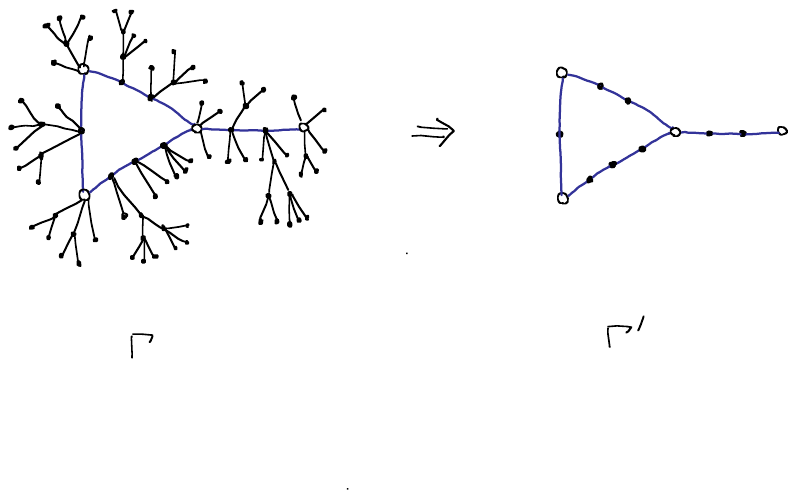} 
\caption{Collapsing subtrees: black vertices are of genus zero and 
white vertices in $\Gamma'$ are of genus $\ge 1$ 
or have more than two edges.} 
\label{fig:collapsing_trees}
\end{figure}
Let $\Gamma'$ be the graph obtained from $\Gamma$ by 
this tree collapsing. 
The graph $\Gamma'$ can be still unstable, as it can 
contain genus-zero two-valent vertices. 
This happens if $\Gamma'$ is an affine $A_n$ graph 
as in Figure \ref{fig:Atilde} or if $\Gamma'$ contains 
$A_n$ subgraphs as in Figure \ref{fig:An}. 
If $\Gamma'$ is not an affine $A_n$ graph, we collapse 
every $A_n$ subgraph of $\Gamma'$ to an edge 
to obtain a stable graph $\Gamma''$. 

\begin{figure}[htbp]
\centering
\includegraphics[bb=109 630 495 712]{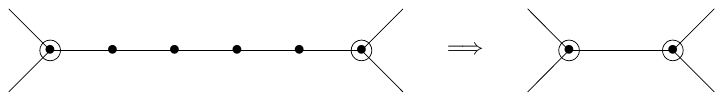}
\caption{$A_n$ subgraph: 
the encircled vertices are either of higher genus $g\ge 1$ 
or have more than two edges; uncircled vertices are of 
genus zero.}
\label{fig:An} 
\end{figure}
\begin{figure}[htbp]
\centering
\includegraphics[bb=121 617 483 721]{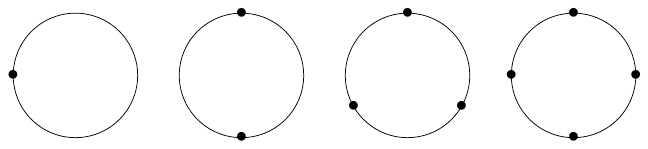}
\caption{Affine $A_n$ graphs: all vertices are 
of genus zero}
\label{fig:Atilde}
\end{figure} 

We first compute the contribution of tree graphs 
with only genus-zero vertices. We claim that 
\[
[V \bp_1]^\alpha = V^{\alpha\beta}  
\left( \begin{array}{c}
\text{the sum of contributions of trees} \\ 
\text{with one leg labelled by $\beta$} 
\end{array} \right) 
\]
Using \eqref{eq:q2q1}, we have 
\begin{align*} 
[V \bp_1]^\alpha & = \left[V dC^{(0)}(\bq_1)\right]^\alpha 
= \left[V d C^{(0)}(\bq_2' + V \bp_1)\right]^\alpha  \\
& =  \sum_{n=0}^\infty \sum_{\gamma_1,\dots,\gamma_n} 
\frac{1}{n!} 
V^{\alpha\beta} C^{(0)}_{\beta\gamma_1\dots\gamma_n}(\bq_2')  
[V\bp_1]^{\gamma_1} \cdots [V\bp_q]^{\gamma_n} 
\end{align*} 
We can use this equation\footnote{We can view $V \bp_1$ as a fixed point 
for the mapping $x\mapsto V dC^{(0)}(\bq_2'+x)$: if $V$ is 
sufficiently small, we have a unique fixed point in a neighbourhood of $x=0$ 
by the contraction mapping principle. 
The sum over trees in question is precisely 
the limit of the sequence $\{x_n\}$ defined recursively by 
$x_{n+1} = V dC^{(0)}(\bq_2'+x_n)$ together with $x_0 = 0$.} 
recursively to solve for $V \bp_1$ for a given $\bq_2'$: 
the answer can be written as the sum over tree graphs. 
The claim follows. This sum over tree graphs converges 
if $\|V\|$ is small, so if $\U$ is close to the identity. 

We fix a graph $\Gamma'$ and sum over contributions from 
all the graphs which collapse to $\Gamma'$. 
This amount to replacing each vertex term 
$C^{(g)}_{\alpha_1\dots\alpha_k}(\bq_2')$ with 
\[
C^{(g)}_{\alpha_1\dots\alpha_k}(\bq_1) = 
C^{(g)}_{\alpha_1\dots\alpha_k}(\bq_2' + V \bp_1) = 
\sum_{n=0}^\infty \frac{1}{n!} C^{(g)}_{\alpha_1\dots\alpha_k 
\beta_1\dots\beta_n}(\bq_2') [V \bp_1]^{\beta_1} \cdots [V\bp_1]^{\beta_n}  
\]
where we again used the relation \eqref{eq:q2q1}. 
The Taylor series is convergent if $V \bp_1$ is sufficiently small. 
In other words, the contribution of each $\Gamma'$ 
is given by the contraction:
\[
\frac{1}{\Aut(\Gamma')} \Cont_{\Gamma'}(
C^{(h)}_{\alpha_1\dots\alpha_n}(\bq_1),
V)
\]
We now fix a stable decorated graph $\Gamma''$ and 
sum over contributions from all $\Gamma'$ which collapse to $\Gamma''$. 
This amounts to replacing the propagator $V^{\alpha\beta}$ with 
\[
(1 - V\tau(\bq_1) )^{-1} V = 
\sum_{n=0}^\infty
\sum_{\gamma_1,\dots,\gamma_n} 
V^{\alpha \gamma_1} \tau_{\gamma_1\gamma_2}(\bq_1)
V^{\gamma_2\gamma_3} \tau_{\gamma_3\gamma_4}(\bq_1)  
V^{\gamma_4 \gamma_5} 
\cdots \tau_{\gamma_{n-1}\gamma_n}(\bq_1) V^{\gamma_n \beta}
\]
where we set $\tau_{\alpha\beta}(\bq) = C^{(0)}_{\alpha\beta}(\bq)$. 
Each summand is a contribution from an $A_n$ graph. 
On the other hand, by Lemma \ref{lem:propagator_traceclass}, 
the propagator of Definition \ref{def:propagator_L2} is given by:
\[
\Delta(\bq_1) =  - (C \tau(\bq_1) + D)^{-1} C = 
(1 - V \tau(\bq_1))^{-1} V
\]
Therefore the contribution of each stable graph $\Gamma$ 
is: 
\[
\frac{1}{|\Aut(\Gamma)|} \Cont_{\Gamma}(
C^{(h)}_{\alpha_1\dots\alpha_n}(\bq_1), \Delta(\bq_1))
\]
We have shown that Givental's quantized operator 
matches with our transformation rule except possibly at genus one. 
At genus one, we need to compute the contribution from 
affine $A_n$ graphs $\Gamma'$. 
This is: 
\[
\log \det\big(1- V\tau(\bq_1)\big) = 
\sum_{n=1}^\infty \frac{1}{n} 
\Tr\big( (V \tau(\bq_1))^n \big) 
\]
where $1/n$ is the symmetry factor of the affine $A_n$ graph. 
This sum converges if $V$ is small. Recall that an operator 
has a determinant if it differs from the identity 
by an operator of trace class.  
Therefore we have:
\[
\hC^{(1)}(\bq_2) = C^{(1)}(\bq_1) + \log \det(1 - V \tau(\bq_1))
\]
This gives an integrated form of the genus-one transformation rule. 
\end{remark}

\section{The Gromov--Witten Wave Function}
\label{sec:GW-wavefunction}

We next explain how one can regard the Gromov--Witten potential of $X$ 
as a section of the Fock sheaf associated to the genus-zero 
Gromov--Witten theory of $X$. 
For this, we need the following convergence assumption 
on Gromov--Witten potentials. 
\begin{assumption}[Convergence] \ 
\label{assump:convergence} 
\begin{enumerate}
\item The genus-zero Gromov--Witten potential 
$F_X^0$ converges in the sense of \S\ref{sec:convergence}; 
in particular its restriction to $Q_1 = \cdots =Q_r =1$ defines an 
analytic function on a region $\cM_{\rm A} \subset H_X\otimes \C$ 
of the form \eqref{eq:LRLnbhd}. 
We denote by $*$ the analytic quantum product over $\cM_{\rm A}$ 
defined by the third derivatives (see equation~\ref{eq:bigQC}) of 
$F_X^0|_{Q_1= \cdots =Q_r=1}$. 
\item Recall that, for any target space $X$, 
the genus-$g$ ancestor potential $\bar\cF^g_X$, $g=0,1,2,\dots$ 
(see equation~\ref{eq:ancestor}) can be expanded as a power series 
in $y_0$,~$y_2$,~$y_3$,~$y_4$,\ldots: 
\[
\bar\cF^g_X 
= 
\delta_{g,1} c^{(1)}(t, y_1;Q) + 
\sum_{n \colon 2g-2+n > 0} \frac{1}{n!} 
\sum_{\substack{L = (l_1,\dots, l_n)\\ 
l_1 + \cdots + l_n \le 3g-3 +n \\
l_j \neq 1 \text{ for all  $j$}}}  
\sum_{I =(i_1,\dots,i_n)} 
c^{(g)}_{L,I} (t, y_1; Q)  y_{l_1}^{i_1} \cdots y_{l_n}^{i_n}
\]
where each coefficient $c^{(g)}_{L,I}(t,y_1;Q)$ belongs to:
\[
\Q[\![t^0, e^{t^1} Q_1,\dots,e^{t^r}Q_r, t^{r+1},\dots,t^N] \!] [\![y_1^0,\dots y_1^N]\!]
\]
(This follows from the Divisor Equation.) 
In particular the restriction to $Q_1 = \cdots = Q_r=1$ makes sense. 
We assume that the restriction 
\[
c^{(g)}_{L,I}(t,y_1) 
= c^{(g)}_{L,I}(t,y_1;Q)|_{Q_1= \cdots = Q_r=1}
\]
takes the form: 
\begin{equation} 
\label{eq:rationality_for_ancestor} 
\parfrac{c^{(1)}(t,y_1)}{y_1^i}  = \frac{f_{1,1,i}(t,q_1)}{
\det(-q_1*)},  \quad 
c^{(g)}_{L,I}(t,y_1)  
= \frac{f_{g,L,I}(t,q_1)}{\det(-q_1 *)^{5g-5+2n-(i_1+\cdots+i_n)}} 
\end{equation} 
under the Dilaton shift  $q_1=y_1 - \unit$, for some polynomials  
\[
f_{1,1,i}(t,q_1), \ 
f_{g,L,I}(t, q_1)\in  
\Q[\![t^0,e^{t^1},\dots,e^{t^r},t^{r+1},\dots,t^N]\!]
[q_1^0,\dots,q_1^N]
\]
Cf.~the rationality condition appearing in 
Remark~\ref{rem:tameness-rationality}. 
\item The polynomials $f_{1,1,i}(t,q_1)$,~$f_{g,L,I}(t,q_1)$ 
in {\rm (2)} are convergent as functions in $t$ and belong to 
$\cO(\cM_{\rm A})[q_1^0,\dots,q_1^N]$. 
\end{enumerate}
\end{assumption} 
\begin{remark} 
Assumption~\ref{assump:convergence} is equivalent to 
the notion of convergence for the total ancestor potential $\cA_X$ 
introduced in~\cite[Definition 3.13]{CI:convergence}; it implies that $\cA_X$ is an element of 
$\Fockanrat(H_X, \unit z, \det(-q_1*_t))$ 
for $t\in \cM_{\rm A}$. 
We showed in~\cite[Theorem 6.5]{CI:convergence} that Assumption~\ref{assump:convergence} is satisfied when the quantum cohomology 
of $X$ is convergent and generically semisimple. 
\end{remark}
\begin{remark} 
It is not difficult to show that the rationality condition 
\eqref{eq:rationality_for_ancestor} holds at genus zero and one. 
For example, at genus one, the term $c^{(1)}(t,y_1)$ appearing in 
the assumption is given by~\cite{Dijkgraaf--Witten:meanfield}: 
\[
c^{(1)}(t,y_1) = -\frac{1}{24} \log \sdet(-q_1*_t) 
\] 
where $\sdet(-q_1*_t) = \det_{\rm ev}(-q_1*_t)
/\det_{\rm odd}(-q_1*_t)$ denotes the superdeterminant of the 
quantum product $(-q_1*_t)$ on the total cohomology ring 
$H^{\rm even}(X) \oplus H^{\rm odd}(X)$. 
Note that the determinant $\det(-q_1*)$ in the assumption 
is the one on the \emph{even} part. 
One can easily check that, when $q_1, t$ are in $H^{\rm even}(X)$, 
every irreducible factor of $\det_{\rm odd}(-q_1*_t)$ is a factor of 
$\det_{\rm even}(-q_1*_t)$, and thus the rationality condition 
\eqref{eq:rationality_for_ancestor} holds for $c^{(1)}(t,y_1)$. 
\end{remark} 

Assumption~\ref{assump:convergence} 
ensures that $\bar\cF^g_X|_{y_0=0, Q_1= \cdots=Q_r=1}$ for $g\ge 2$ 
and $d(\bar\cF^1_X|_{y_0=0, Q_1= \cdots = Q_r=1})$
depend analytically on $t\in \cM_{\rm A}$, rationally 
on $y_1$, and polynomially on $y_2,\dots,y_{3g-2}$.
Therefore the following definition makes sense. 

\begin{definition}[Gromov--Witten wave function] 
\label{def:GW-wave}
Suppose that Assumption~\ref{assump:convergence} holds. 
Then we have the A-model cTEP structure 
$(\sfF,\nabla,(\cdot,\cdot)_{\sfF})$ over $\cM_{\rm A}$ 
(see Example~\ref{ex:AmodelTP} and Remark~\ref{rem:completion-of-TP}). 
The associated Fock sheaf $\Fock_X$ over $\cM_{\rm A}$ 
is called the \emph{A-model Fock sheaf} for $X$.  
Let $\{\phi_i \}_{i=0}^N$ be a homogeneous
basis of $H_X$ as in \S\ref{subsec:bases} 
and let $\{t^i, x_n^i\}_{n\ge 1, 
0\le i \le N}$ be the algebraic co-ordinates on the total space $\LL$ 
of the A-model cTEP structure. 
Let $\sfP_{\rm std}$ denote the standard opposite module 
from Example~\ref{ex:Amodel-opposite}. 
The Gromov--Witten potentials of $X$ define 
a \emph{Gromov--Witten wave function} 
$\wave_X  = \{ \Nabla^n C^{(g)}_X \}_{g,n} 
\in \Fock_X(\cM_{\rm A}; \sfP_{\rm std})$ 
by 
\begin{align*} 
\Nabla^3 C_X^{(0)} &= \bY = 
\sum_{i=0}^N \sum_{j=0}^N \sum_{k=0}^N
dt^i \otimes dt^j  \otimes dt^k 
\int_X (\phi_i * \phi_j * x_1) \cup (\phi_k * x_1) 
\bigg|_{Q_1= \cdots = Q_r =1} 
 \\
\Nabla C_X^{(1)} & = d(F^1_X(t) + \bar{\cF}_X^1)
\Big |_{y_0 =0, Q_1 = \cdots = Q_r =1} \\ 
C_X^{(g)} &= \bar{\cF}^g_X \Big |_{y_0 = 0, Q_1 = \cdots = Q_r =1} 
\qquad (g\ge 2) 
\end{align*} 
and their covariant derivatives with respect to 
$\Nabla = \Nabla^{\sfP_{\rm std}}$.
Here we used the Dilaton shift: 
\[
y_n^i = x_n^i + \delta_n^1 \delta^i_0 \qquad n\ge 1 
\]
to identify the variables $\{t^i,y_n^i\}$ on the right-hand side 
with the co-ordinates $\{t^i, x_n^i\}$ on $\LL$,
and $F^1_X(t)$ is the \emph{non-descendant} genus-one 
Gromov--Witten potential: 
\[
F^1_X(t) = \sum_{n=0}^\infty \sum_{\substack{
d\in \NE(X) \\ (n,d)\neq (0,0)}}
\frac{Q^d}{n!} 
\corr{t,\dots,t}_{1,n,d} 
\]
with $t = \sum_{i=0}^N t^i \phi_i$. 
(Assumption~\ref{assump:convergence} implies in particular that $F^1_X(t)|_{Q_1= \cdots = Q_r=1}$ 
converges on $\cM_{\rm A}$.) 

\begin{remark} 
Supposing again that Assumption~\ref{assump:convergence} holds, we have 
the A-model $\log$-cTEP structure (Example~\ref{ex:log_cTEP_A}) 
with base $(\cMbar_{\rm A},D)$
and the associated Fock sheaf $\overline{\Fock}_X$ over 
$\cMbar_{\rm A}$. The Gromov--Witten wave function 
$\wave_X$ extends to an element of 
$\overline{\Fock}_X(\cMbar_{\rm A};\sfP_{\rm std})$,  
where $\sfP_{\rm std}$ is the standard opposite module from 
Example~\ref{exa:log_Amodel_opposite}. 
\end{remark} 

\end{definition} 

\begin{remark} 
One can check that the Gromov--Witten wave function satisfies 
the conditions (Yukawa), (Jetness), (Grading \& Filtration) and 
(Pole) in Definition~\ref{def:localFock}. (Yukawa) and (Jetness) 
are obvious. The Dilaton equation 
(see, e.g.~\cite[Theorem 8.3.1]{AGV}) 
\begin{multline*} 
\corr{\alpha_1 \bar{\psi}_1^{k_1},\dots,\alpha_m \bar{\psi}_m^{k_m}, 
\bar{\psi}_{m+1} : 
t,\dots,t}_{g,1+m+n,d}^X \\
= (2g-2+m) \corr{\alpha_1\bar{\psi}_1^{k_1},\dots,
\alpha_m \bar{\psi}_m^{k_m} : t,\dots,t }_{g,m+n,d}^X 
\end{multline*} 
shows that we have, for $g\ge 1$: 
\[
\sum_{n=0}^\infty \sum_{i=0}^N 
(\delta_{n,1} \delta_{i,0} -y_n^i)\parfrac{}{y_n^i} 
\bar\cF_X^g = (2g-2) \bar\cF_X^g + \delta_{g,1} 
\frac{1}{24} F^1_X(t) 
\]
where the last term arises from the unstable term $(g,m)=(1,0)$ 
and the fact that $\int_{\Mbar_{1,1}} \psi = \frac{1}{24}$. 
This means that the function $\bar\cF^g_X$ for $g\ge 2$, or the one-form $d\bar\cF^1_X$ at genus one,
is homogeneous of degree $2-2g$ with respect to the Dilaton-shifted 
variables $x_n^i = - \delta_{n,1} \delta_{i,0}+ y_n^i$. 
The grading condition follows. 
The filtration condition follows from the dimension 
formula $\dim \Mbar_{g,n} = 3g-3+n$ 
(see equation~\ref{eq:filtration}). 
(Pole) follows from Assumption~\ref{assump:convergence}, 
in particular from \eqref{eq:rationality_for_ancestor}. 
\end{remark}

For the rest of this section (\S\ref{sec:GW-wavefunction}) we will assume that
Assumption~\ref{assump:convergence} holds.
In our previous paper~\cite{CI:convergence}, 
we studied analytic properties of various Gromov--Witten 
potentials under this assumption.  
We need to review some of these results. 
Recall from Remark~\ref{rem:nuclear-neighbourhood} 
that we have the nuclear subspace of the 
total space $\LL$ of the A-model cTEP structure: 
\[
\cN(\LL) = \left\{ (t,\bx) \in \LL : 
\text{$t\in \cM_{\rm A}$,
$\sup_{0\le i\le N,\, l\ge 0} (e^{nl} |x_l^i|/l!) <\infty$ 
for all $n\ge 0$}\right\}
\]
As we explained in Example~\ref{exa:Amodel-genuszero}, 
there is a holomorphic mapping (see equation~\ref{eq:flatcoord-GW})
\begin{equation} 
\label{eq:flatcoordinate-in-GWtheory}
\bq = [M(t,z) \bx]_+ \Bigr|_{Q_1= \cdots =Q_r=1} 
\colon \cN(\LL) \longrightarrow 
\cH_+^{\rm NF} 
\end{equation} 
taking values in the positive part of a nuclear version of 
Givental's symplectic space \eqref{eq:nuclear-Giventalsp}.  
Here $M(t,z)$ is the inverse fundamental solution 
\eqref{eq:inversefundamentalsolution} in Gromov--Witten theory. 
This map $\bq$ is a local isomorphism between $\cN(\LLo)$ 
and $\cH_+^{\rm NF}$~\cite[\S 8.5]{CI:convergence} 
and gives a flat co-ordinate system on $\cN(\LLo)$ 
with respect to $\Nabla^{\sfP_{\rm std}}$. 
For $(t,\bx) = (t, - z\unit)$, we have 
$\bq = -z \unit + t$. 

We showed in~\cite[Theorem 7.9]{CI:convergence} 
that the total descendant potential $\cZ_X$ is NF-convergent 
under Assumption~\ref{assump:convergence}, i.e.~that the 
power series \eqref{eq:genus_g_descendantpot} defining 
each genus-$g$ descendant potential 
$\cF^g_X$ converges uniformly and absolutely on an 
infinite-dimensional polydisc of the form: 
\begin{equation} 
\label{eq:NF-convergencedomain}
\begin{cases}
  |t_l^i| < \epsilon \frac{l!}{C^l} & \text{$0\le i\le N$, $l\ge 0$} \\
  |Q_i| <\epsilon & 0 \leq i \leq N
\end{cases}
\end{equation} 
for some $\epsilon>0$ and $C>0$ independent of $g$. 
Define an open subset $U\subset \cH_+^{\rm NF}$ by  
\begin{equation} 
\label{eq:convergencedomain-U} 
U:= \bigcup_{
\substack{\delta \in H^2(X;\C) \\ 
\Re(\delta_i) < \log \epsilon 
} } \left[ 
e^{-\delta/z} 
\left(-\unit z + \left\{\bt \in \cH_+^{\rm NF} 
\, :\, |t_l^i| < \epsilon l!/C^l, \ 0\le i\le N,\ l\ge 0 \right\} 
\right) \right]_+  
\end{equation} 
where $\epsilon$, $C$ are the constants in \eqref{eq:NF-convergencedomain} 
and we write $\bt = \sum_{l=0}^\infty \sum_{i=0}^N t_l^i \phi_i z^l$.  
The Divisor Equation justifies the following definition 
(see~\cite[Lemma 8.1]{CI:convergence}): 
\begin{definition}[\!\!{\cite[Definition-Proposition 8.2]{CI:convergence}}] 
\label{def:F^g_Xan} 
Under Assumption~\ref{assump:convergence}, 
there is an analytic function $\cF^g_{X, \rm an}\colon U \to \C$ 
such that  
\begin{align*} 
\cF^g_{X,\rm an}([e^{-\delta/z} \bq]_+) & = \cF^g_X(\bq) 
\Big|_{
Q_1=e^{\delta_1},\dots,Q_r=e^{\delta_r}} \\
& + \delta_{g,0}\frac{1}{2}  \Omega(e^{-\delta/z} \bq, [e^{-\delta/z} \bq]_+) 
-\delta_{g,1}\frac{1}{24}  \int_X \delta \cup c_{D-1}(X) 
\end{align*} 
where $\delta = \sum_{i=1}^r \delta_i \phi_i \in H^2(X;\C)$ 
and $\bq = \bt - z \unit$    
are chosen so that $Q_i = e^{\delta_i}$ and $t_l^i$ satisfy 
\eqref{eq:NF-convergencedomain}. 
We call $\cF^g_{X,\rm an}$ the \emph{specialization of 
$\cF^g_X$ to $Q_1=\cdots =Q_r=1$}. 
Note that the domain $U$ contains the locus $-z \unit + t$ with $t$ in 
a neighbourhood \eqref{eq:LRLnbhd} of the large radius limit. 
\end{definition}

\begin{theorem}[analytic version of the Ancestor-Descendant relation;
cf.~\S\ref{subsec:AncDec}] 
\label{thm:ancdec-analytic} 
When $t\in \cM_{\rm A}$ is sufficiently close to the 
large radius limit \eqref{eq:LRLnbhd} and $\bx\in z\cH_+^{\rm NF}$ 
is sufficiently close to $-z \unit$, the flat co-ordinate  
$\bq = [M(t,z) \bx]_+|_{Q_1=\cdots=Q_r=1}$ 
\eqref{eq:flatcoordinate-in-GWtheory} 
of the point $(t,\bx) \in \cN(\LL)$ lies 
in the domain $U$ for $\cF^g_{X,\rm an}$. 
For $g\ge 1$ and for such $(t,\bx) \in \cN(\LL)$, 
we have:
\[
\cF^g_{X,\rm an}(\bq) = \delta_{g,1} F_X^1(t) +
\bar\cF^g_t \Big|_{y_0=0, y_1 = x_1 - \unit, 
y_l = x_l (l\ge 2),  Q_1=\cdots = Q_r= 1}  
\]
In particular, in a neighbourhood of such a point $(t,\bx) \in \cN(\LL)$, 
the Gromov--Witten wave function (Definition~\ref{def:GW-wave}) 
can be written in terms of 
flat co-ordinates \eqref{eq:flatcoordinate-in-GWtheory} as: 
\begin{align*} 
\Nabla^3 C^{(0)}_X & = \bY = \sum_{l=0}^\infty \sum_{m=0}^\infty \sum_{n=0}^\infty \sum_{i=0}^N \sum_{j=0}^N \sum_{h=0}^N 
\parfrac{^3\cF^0_{X,\rm an}(\bq)}{q_l^i \partial q_n^j \partial q_m^h}
dq_l^i \otimes dq_n^j \otimes dq_m^h \\ 
\Nabla C^{(1)}_X &  = d\cF^1_{X,\rm an}(\bq) \\ 
C^{(g)}_X & = \cF^g_{X,\rm an}(\bq) \qquad \text{for $g\ge 2$.}
\end{align*} 
\end{theorem} 
\begin{proof} 
By equation \eqref{eq:M-divisoreq},  $M(t,z)$ satisfies  
\[
e^{\delta/z} M(t, z)|_{Q_1= \cdots =Q_r= 1} 
=  M(t-\delta,z)|_{Q_1= e^{\delta_1},\dots,Q_r= e^{\delta_r}} 
\]
for $\delta = \sum_{i=1}^r \delta_i \phi_i \in H^2(X;\C)$. 
Since $\bq \in U$, we can write $\bq = [e^{-\delta/z} 
\tbq]_+$ for some $\delta\in H^2(X;\C)$ with 
$\Re(\delta_i) <\log \epsilon$ and 
$\tbq = -z\unit + \bt$ with $|t_l^i|< \epsilon l!/C^l$. 
Then: 
\[
\tbq = [e^{\delta/z} \bq]_+ 
= [e^{\delta/z} M(t,z)\bx]_+ \Bigr|_{Q_1= \dots=Q_r=1}
= [M(t-\delta,z)\bx]_+\Bigr|_{Q_1= e^{\delta_1},\dots,Q_r=e^{\delta_r}}
\]
Thus we have for $g\ge 1$: 
\begin{align*} 
& \cF^g_{X,\rm an}(\bq) 
= \cF^g_X(\tbq) \Bigr|_{Q_1=e^{\delta_1},\dots, 
Q_r = e^{\delta_r}} - \delta_{g,1} \frac{1}{24} 
\int_X \delta \cup c_{D-1}(X) 
\qquad 
\text{(by definition)} \\
& = \delta_{g,1} \left(F^1_X(t-\delta) - \frac{1}{24} \int_X \delta \cup c_{D-1}(X)
\right ) 
 + \bar\cF^g_{t-\delta} \Bigr
|_{y_0=0, y_1= x_1+\unit, x_l=y_l (l\ge 2), 
Q_1= e^{\delta_1},\dots
Q_r = e^{\delta_r}} 
\end{align*} 
by the original version of the ancestor-descendant relation 
(\S\ref{subsec:AncDec}). 
The conclusion follows from the Divisor Equation for 
$F^1_X(t)$ and $\bar\cF^g_{t}$. 
The formula for $\Nabla^3 C^{(0)}_X$ appeared in Example~\ref{exa:Amodel-genuszero}, \eqref{eq:genuszerojet-flat}.  
\end{proof} 

\subsection{The Jet-Descendant Relation}
We next give a generalization of the Ancestor-Descendant relation---called the 
Jet-Descendant relation---which justifies the name ``jet" 
for the jet potential $\cW_X$ \eqref{eq:jetpotential-GW} in Gromov--Witten 
theory. 
For a sequence $(t_0,t_1,t_2,\dots)$ of variables in $H_X$, 
we write $\bt(\psi) = \sum_{n=0}^\infty t_k \psi^k$. 
Define generalized (inverse) fundamental solutions  
(cf.~equations~\ref{eq:fundamentalsolution} and~\ref{eq:inversefundamentalsolution}) by 
\begin{align*} 
L(\bt,z) v & = v + 
\sum_{d\in \NE(X)} \sum_{n=0}^{\infty} \sum_{\epsilon=0}^N 
\frac{Q^d}{n!} 
\corr{\frac{v}{z-\psi}, \bt(\psi),\dots,\bt(\psi), \phi^\epsilon}_{0,n+2,d}^X 
\phi_\epsilon
\\
M(\bt,z) v & = v + 
\sum_{d\in \NE(X)} \sum_{n=0}^{\infty} \sum_{\epsilon=0}^N 
\frac{Q^d}{n!} 
\corr{\frac{\phi^\epsilon}{-z-\psi}, \bt(\psi),\dots,\bt(\psi), v}_{0,n+2,d}^X 
\phi_\epsilon
\end{align*} 
A result of Dijkgraaf--Witten~\cite{Dijkgraaf--Witten:meanfield} 
(see also~\cite[equation 2]{Givental:symplectic},~\cite[Proposition 4.6]{Getzler:jetspace}) 
shows that 
\begin{align} 
\label{eq:Lbt=Ltau}
L(\bt, z) = L(\tau(\bq),z) &&
M(\bt,z) = M(\tau(\bq),z) 
\end{align}
for 
\begin{equation} 
\label{eq:tau(bq)}
\tau(\bq) := \sum_{\epsilon=0}^N 
\parfrac{^2\cF_X^0}{q_0^0 \partial q_0^\epsilon}(\bq) \phi^\epsilon 
\end{equation} 
and thus that $M(\bt,z) = L(\bt, z)^{-1}$ 
(recall from \S\ref{subsec:Dubrovin_conn} that $M(t,z) = L(t,z)^{-1}$). 
Here the Dilaton shift $\bq =- z \unit + \bt$ 
is used. 
\begin{theorem}[Jet-Descendant relation]  
\label{thm:jet-descendant} 
We regard the jet potential 
$\cW_X= \sum_{g=0}^\infty \hbar^{g-1} W^g_X$ 
\eqref{eq:jetpotential-GW} 
in Gromov--Witten theory as a function of $\bq = -z \unit + \bt 
= -z \unit + \sum_{n=0}^\infty t_n z^n$ and 
$\by=\sum_{n=0}^\infty y_n z^n$. 
Introduce a new variable $\bs = \sum_{n=0}^\infty s_n z^n$, 
$s_n\in H_X$ depending on $(\bq, \by)$ as:
\begin{equation}  
\label{eq:s-y-relation} 
\bs = [M(\bt,z) \by]_+
\end{equation} 
Then we have 
\begin{equation}
  \label{eq:jet-descendant}
  \begin{aligned} 
    \cW^g_X & =  \cF^g_X(\bq + \bs)   \qquad \quad \text{for $g\ge 2$} \\ 
    \cW^1_X & = \cF^1_X(\bq + \bs) - \cF^1_X(\bq)  \\ 
    \cW^0_X & = [\cF^0_X(\bq+\bs)]_{\ge 3}
  \end{aligned}
\end{equation}
where $[\cF^0_X(\bq+\bs)]_{\ge 3}$ denotes the degree 
$\ge 3$ part with respect to $\bs$, i.e.~the Taylor expansion of 
$\cF^0_X(\bq+\bs)$ in $\bs$ with the constant, linear, and quadratic terms removed.
\end{theorem} 
\begin{proof} 
The proof is a straightforward generalization of the 
argument in~\cite{Kontsevich--Manin:relations,Coates--Givental}. 
The key point is the following comparison of $\psi$-classes. 
Let $\psi_i$ denote the cotangent line class on $X_{g,m+l,d}$ 
and $\bar\psi_i$ denote the cotangent line class 
pulled back from $\Mbar_{g,m}$ (both at the $i$th marking). 
Then the class $\psi_i-\bar\psi_i$ is Poincar\'{e} dual 
to the divisor consisting of stable maps whose 
$i$th marking is on a component contracted under 
the forgetful morphism $X_{g,m+l,d} \to \Mbar_{g,m}$, i.e. 
\[
\psi_i - \bar\psi_i = \sum_{L_1 \sqcup L_2 = \{1,\dots,l\}}
\sum_{d= d_1 + d_2} \left [X_{0,3+|L_1|,d_1} \times_X 
X_{0,m+|L_2|, d_2}\right]^{\rm vir}
\]
For any cohomology-valued polynomials $a_1(\psi,\bar\psi), \dots, 
a_m(\psi,\bar\psi)$  in two variables $\psi$ and $\bar\psi$, we write 
\begin{multline*} 
\corr{a_1(\psi,\bar\psi),\dots,a_m(\psi,\bar\psi)}_{g,m}(\bt)  \\
=\sum_{d\in \NE(X)} \sum_{l = 0}^\infty 
\frac{Q^d}{l!}
\corr{a_1(\psi_1,\bar\psi_1),\dots,a_m(\psi_m,\bar\psi_m) :
\bt(\psi_{m+1}),\dots,\bt(\psi_{m+l}) }_{g,m+l,d}^X 
\end{multline*} 
where $\bt(\psi_i) = \sum_{n=0}^\infty t_n \psi_i^n$ as before. 
Then the above relation shows that
\[
\corr{\phi_i \psi^{a+1}\bar\psi^b, \dots}_{g,m}(\bt) = 
\corr{\phi_i \psi^a \bar\psi^{b+1}, \dots}_{g,m}(\bt) 
+ \sum_{\epsilon=0}^N 
\corr{\phi_i \psi^a, \phi_\epsilon}_{0,2}(\bt) 
\corr{\phi^\epsilon \bar\psi^b, \dots }_{g,m}(\bt)  
\]
where dots denote arbitrary insertions and are the same in all places. 
Using this repeatedly, we find that: 
\[
\corr{\phi_i \psi^n,\dots}_{g,m}(\bt)= 
\corr{\phi_i \bar\psi^n,\dots}_{g,m}(\bt) + 
\sum_{k=0}^{n-1} \sum_{\epsilon=0}^N 
\corr{\phi_i \psi^{k}, \phi_\epsilon}_{0,2} (\bt) 
\corr{\phi^\epsilon \bar\psi^{n-k-1},\dots}_{g,m} (\bt)
\]
Multiplying by $s_n^i$ and summing over all $n\ge 0$ and $0\le i\le N$ yields
\[
\corr{\bs(\psi),\dots}_{g,m}(\bt) = 
\corr{\by(\bar\psi), \dots}_{g,m}(\bt) 
\]
for $\bs(\psi) = \sum_{n=0}^\infty \sum_{i=0}^N 
s_n^i \phi_i \psi^n$ and 
$\by(\bar\psi) = \sum_{n=0}^\infty \sum_{i=0}^N 
y_n^i \phi_i \bar\psi^n$, 
where $\by$ is given by:
\[
y_n^i = s_n^i + \sum_{l=0}^\infty \sum_{j=0}^N s_{l+n+1}^j 
\corr{\phi_j \psi^l, \phi^i}_{0,2}(\bt)
\]
This is equivalent to $\by = [L(\bt,z) \bs]_+$ and to \eqref{eq:s-y-relation}. 
Repeating the same argument at other markings gives:
\[
\corr{\bs(\psi),\dots,\bs(\psi)}_{g,m}(\bt) 
= \corr{\by(\bar\psi),\dots \by(\bar\psi)}_{g,m}(\bt)
\]
Note that the right-hand side makes sense only for 
$2g-2+m>0$. We have, for $g\ge 2$:
\begin{align*} 
\cF^g_X(\bq+\bs) & 
= \sum_{m=0}^\infty 
\frac{1}{m!} \corr{\bs(\psi),\dots,\bs(\psi)}_{g,m}(\bt) \\ 
& = \sum_{m=0}^\infty 
\frac{1}{m!} \corr{\by(\bar\psi),\dots,\by(\bar\psi)}_{g,m}(\bt) 
= \cW^g_X
\end{align*}
For $g=0$ or $1$, restricting the above summation to the range 
$m\ge 3$ or $m\ge 1$ respectively yields the remaining parts of \eqref{eq:jet-descendant}. 
\end{proof} 
\begin{remark} 
When restricted to $\bq = -z\unit + t$ and $y_0=0$, 
the Jet-Descendant relation above reduces 
to the Ancestor-Descendant relation from \S\ref{subsec:AncDec}. 
\end{remark} 

Next we study the analyticity of 
the Gromov--Witten jet potential $\cW_X$ \eqref{eq:jetpotential-GW} 
and discuss the specialization to $Q_1=\cdots =Q_r=1$.  
More precisely, we regard $\cW_X$ as a formal power series in $\by
= \sum_{n=0}^\infty y_n z^n$ with 
coefficients in analytic functions in $\bq = -z\unit + \bt 
= -z \unit + \sum_{n=0}^\infty t_n z^n$. 
Firstly recall that, under 
Assumption~\ref{assump:convergence}, the descendant potentials 
$\cF^g_X$, $g=0,1,2,\dots$, are NF-convergent on the region 
\eqref{eq:NF-convergencedomain}. 
This means that the function $\tau(\bq)$ introduced 
in \eqref{eq:tau(bq)} is also convergent on the same region. 
Since $\tau(\bq)|_{Q=\bt=0} = 0$, after taking a bigger $C$ 
or a smaller $\epsilon$ if necessary, 
$M(\bt,z) = M(\tau(\bq),z)$ is convergent on the region 
\eqref{eq:NF-convergencedomain}. 
Thus Theorem~\ref{thm:jet-descendant} 
implies that each Taylor coefficient of $\cW^g$ with respect to  
$\by$ converges to an analytic function on the region 
\eqref{eq:NF-convergencedomain}. 
The Divisor Equation shows that: 
\[
\cW^g_X([e^{-\delta/z}\bq]_+,\by) 
= \cW^g_X(\bq, \by)
\Bigr|_{Q_1\to e^{\delta_1}Q_1,\dots, Q_r\to e^{\delta_r} Q_r} 
\]
This justifies the following definition 
(cf.~Definition~\ref{def:F^g_Xan}).
\begin{definition} 
Let $U\subset \cH_+^{\rm NF}$ be the domain in 
\eqref{eq:convergencedomain-U}. 
Under Assumption~\ref{assump:convergence}, 
there exists a formal power series $\cW^g_{X,\rm an}(\bq,\by)$ 
in the variable $\by =\sum_{n=0}^\infty \sum_{i=0}^N 
y_n^i \phi_i z^n$ with coefficients in analytic functions 
in $\bq$ over $U$ with the following property: 
\[
\cW^g_{X,\rm an} ([e^{-\delta/z}\bq]_+, \by) = 
\cW^g_X(\bq, \by)\Bigr|_{Q_1=e^{\delta_1},\dots, Q_r= e^{\delta_r}}  
\]
where $(\bt=\bq+z \unit, Q_i = e^{\delta_i})$ lies in the convergence 
domain \eqref{eq:NF-convergencedomain} for $\cW_X^g$. 
We refer to $\cW^g_{X,\rm an}$ as the \emph{specialization 
of $\cW^g_X$ to $Q_1=\cdots =Q_r =1$}. 
\end{definition} 

The Divisor Equation for $\bar \cF^g_X$ implies that
\begin{equation} 
\label{eq:jet-ancestor} 
\cW^g_{X,\rm an}\Bigr|_{\bq= -z\unit + t}
= \bar\cF^g_t\Bigr|_{Q_1=\cdots =Q_r=1}
\end{equation} 
as a formal power series in $\by$, 
for $t$ in a neighbourhood \eqref{eq:LRLnbhd} 
of the large radius limit. 
Let $\tau_{\rm an}\colon U \to H_X\otimes \C$ be the 
holomorphic map defined by
\[
\tau_{\rm an} (\bq ) = \sum_{\epsilon=0}^N 
\parfrac{^2\cF^0_{X,\rm an}}{
q_0^0 \partial q_0^\epsilon} (\bq) \phi^\epsilon
\]
(cf.~equation~\ref{eq:tau(bq)}).
One can check directly from the definition of $\cF^0_{X, \rm an}$ that 
\begin{equation} 
\label{eq:tauan-tau} 
\tau_{\rm an}([e^{-\delta/z}\bq]_+) = \tau(\bq)\Bigr|_{Q_1= e^{\delta_1},\dots,
Q_r = e^{\delta_r}} + \delta 
\end{equation} 
when $(\bt = \bq + z \unit, Q_i = e^{\delta_i})$ satisfy 
\eqref{eq:NF-convergencedomain}. 
The following theorem follows from a routine application 
of the Divisor Equation, much as in Theorem~\ref{thm:ancdec-analytic}.  
We leave the details to the reader. 
\begin{theorem}[analytic version of Jet-Descendant relation] 
\label{thm:jetdec-analytic} 
Let $\bq\in \cH_+^{\rm NF}$ 
be in the convergence domain \eqref{eq:convergencedomain-U} 
for $\cF^g_{X,\rm an}$ and $\cW^g_{X,\rm an}$. 
Then:
\begin{align*} 
\cW^g_{X,\rm an}(\bq, \by) & = \cF^g_{X,\rm an}(\bq+\bs) 
\qquad g \ge 2 \\ 
\cW^1_{X,\rm an}(\bq, \by) & = \cF^1_{X,\rm an}(\bq + \bs) 
- \cF^1_{X,\rm an}(\bq)\\ 
\cW^0_{X,\rm an}(\bq, \by) & = 
[\cF^0_{X,\rm an}(\bq + \bs)]_{\ge 3} 
\end{align*} 
where $\bs = [M(\tau_{\rm an}(\bq),z) \by]_+|_{Q_1= \cdots = Q_r=1}$.  These are identities of
formal power series in $\by$. 
\end{theorem} 

\begin{lemma} \ 
\label{lem:tau-as-inverse} 
\begin{enumerate}
\item Let $\bx = \sum_{n=1}^\infty x_n z^n$,  
$\bq = \sum_{n=0}^\infty q_n z^n$ be variables 
in $zH_X[\![z]\!]$ and $H_X[\![z]\!]$ 
respectively. 
The formula $\bq = [M(t,z) \bx]_+$ defines 
an isomorphism over the Novikov ring $\Lambda$ 
between the formal neighbourhood of $(t, \bx) = (0,-z\unit)$ 
in $H_X \times zH_X[\![z]\!]$ 
and the formal neighbourhood of $\bq = -z \unit$ 
in $H_X[\![z]\!]$.  
The inverse map is given by 
\[
t = \tau(\bq), \quad \bx = [L(\tau(\bq),z)\bq]_+
\]
where $\tau$ is given in \eqref{eq:tau(bq)}. 
\item Let $t\in \cM_{\rm A}$ be sufficiently close to the large 
radius limit point and let $\bx\in \cH_+^{\rm NF}$ be 
sufficiently close to $-z \unit$ so that the flat co-ordinate 
$\bq = [M(t,z)\bx]_+|_{Q_1= \cdots =Q_r=1}$ 
\eqref{eq:flatcoordinate-in-GWtheory}
of the point $\sx = (t,\bx) \in \cN(\LLo)$ lies 
in the domain $U \subset \cH_+^{\rm NF}$ 
of $\cF^0_{X,\rm an}$ \eqref{eq:convergencedomain-U}. 
Then we have $t = \tau_{\rm an}(\bq)$.  
\end{enumerate}
\end{lemma} 
\begin{proof} 
(1) It was explained in~\cite[Remark 8.4]{CI:convergence} 
that $\bq = [M(t,z) \bx]_+$ defines an isomorphism 
between the formal neighbourhoods of $(t,\bx) = (0,-z\unit)$ 
and $\bq = -z\unit$. 
Since $L(t,z) = M(t,z)^{-1}$, we have:
\[
\bx = [L(t,z) \bq]_+
\]
The variable $t$ is determined implicitly by the equation 
$[L(t,z) \bq]_0 =0$, where $[\cdots]_0$ denotes the coefficient 
of $z^0$. It now suffices to show that 
$[L(\tau(\bq),z)\bq]_0=0$. 
By \eqref{eq:Lbt=Ltau}, we have 
$[L(\tau(\bq),z)\bq]_0 = [L(\bt,z)\bq]_0$ 
under the Dilaton shift $\bq = -z \unit + \bt$. 
Then: 
\begin{align*} 
[L(\bt,z)\bq]_0 &= \left[ 
\bq + \sum_{d\in \NE(X)} \sum_{n=0}^\infty 
\sum_{\epsilon=0}^N 
\corr{\frac{\bq}{z-\psi},\bt(\psi),\dots,\bt(\psi),\phi_\epsilon}_{0,2+n,d} 
\frac{Q^d}{n!}  \phi^{\epsilon} 
 \right]_0 \\
& = q_0 + \sum_{n=0}^\infty \sum_{i,\epsilon=0}^N 
q_{n+1}^i \parfrac{^2\cF^0_X}{q_n^i \partial q_0^\epsilon} 
\phi^\epsilon
\end{align*} 
The String equation (see~\cite{Givental:symplectic}) shows that 
the last expression is identically zero.   

(2) Part (2) follows from Part (1), \eqref{eq:tauan-tau} and an argument similar to 
the proof of Theorem~\ref{thm:ancdec-analytic}. 
\end{proof}

\begin{theorem} 
\label{thm:jet-GW} 
Let $\wave_X$ denote the Gromov--Witten wave function 
(Definition~\ref{def:GW-wave}) of $X$. 
Let $t\in \cM_{\rm A}$ be sufficiently close to the large 
radius limit point and let $\bx\in \cH_+^{\rm NF}$ be 
sufficiently close to $-z \unit$ so that the flat co-ordinate 
$\bq = [M(t,z)\bx]_+|_{Q_1= \cdots =Q_r=1}$ 
\eqref{eq:flatcoordinate-in-GWtheory}
of the point $\sx = (t,\bx) \in \cN(\LLo)$ lies 
in the domain $U \subset \cH_+^{\rm NF}$ 
for $\cF^g_{X,\rm an}$ and $\cW^g_{X,\rm an}$. 
Let $\For_{\sx}$ be the formalization map 
(Definition~\ref{def:formalizationmap}) at $\sx$ 
associated to the standard unitary frame 
$H_X \otimes \C[\![z]\!]  \to \sfF_t$ 
of the A-model cTEP structure. 
Then we have 
\[
\For_\sx \wave_X = \exp\left(\sum_{g=0}^\infty 
\hbar^{g-1}\cW_{X,\rm an}^g(\bq, \by) \right)  
\] 
as a formal power series in $\by$. 
In particular, by \eqref{eq:jet-ancestor}, 
\[
\For_{\sx} \wave_X = \cA_{X,t}\Bigr|_{Q_1= \cdots =Q_r=1} 
\]
when $\sx = (t,\bx) = (t,-z \unit)$. 
\end{theorem} 
\begin{proof} 
Recall that the formalization map (Definition~\ref{def:formalizationmap}) 
at $\sx=(t,\bx)$ 
is defined as a truncated Taylor expansion of the potential 
with respect to the flat co-ordinates associated to a given unitary 
frame of $\sfF_t$. The standard unitary frame at $t\in H_X$ defines 
the following flat co-ordinate system on $\LLo$ 
(see Definition~\ref{def:flatcoordinate_on_LL} and 
equation~\ref{eq:formal_flat_coordinates}): 
\[
(t+s, \bx) \mapsto \bq_t = [M(t,z)^{-1} M(t+s,z) \bx]_+ \Bigr|_{Q_1 = \cdots = Q_r =1} 
\] 
Note that the inverse fundamental solution in \eqref{eq:formal_flat_coordinates} 
is normalized by the condition that it is identity at $t$ and so we need the 
factor $M(t,z)^{-1}$ here. This is related to the flat co-ordinate $\bq$ 
in \eqref{eq:flatcoordinate-in-GWtheory} by a linear transformation 
$\bq_t = [M(t,z)^{-1}\bq]_+|_{Q_1= \cdots = Q_r= 1}$. 
Also, by Lemma~\ref{lem:tau-as-inverse}, we have 
$t = \tau_{\rm an}(\bq)$. 
The analytic version of Ancestor-Descendant relation 
(Theorem~\ref{thm:ancdec-analytic}) shows that 
$\For_\sx(\wave_X)$ is a truncated Taylor expansion of 
$\exp(\sum_{g=0}^\infty \hbar^{g-1} \cF^g_{X,\rm an}(\bq))$ 
with respect to $\bq_t$. 
Since the co-ordinate change $\bq \mapsto \bq_t$ here 
is the same as the co-ordinate change $\by \mapsto \bs$ 
of jet co-ordinates 
in the Jet-Descendant relation (Theorem~\ref{thm:jetdec-analytic}), 
the conclusion follows.
\end{proof} 

\begin{remark} 
Theorem~\ref{thm:jet-GW} shows that the jet potential \eqref{eq:jetpotential-GW} 
in Gromov--Witten theory can be identified with the jet potential 
(Definition~\ref{def:jetpotential}) associated with 
the Gromov--Witten wave function. 
\end{remark} 

\begin{corollary} 
With notation as in Theorem \ref{thm:jet-GW}, the formalization 
of the Gromov--Witten wave function at $\sx=(t,\bx)\in \LLo$ 
associated to the standard unitary frame 
$H_X\otimes \C[\![z]\!] \to \sfF_t$ is given by 
\[
\For_\sx \wave_X = \exp
\left(- \bar\cF^1_{X,t}\Bigr|_{\by =\bx + z\unit, 
Q_1=\cdots =Q_r =1}\right) 
\cA_{X,t}\Bigr|_{\by \to \by+\bx+z\unit,  
Q_1 = \cdots = Q_r = 1}.   
\]
\end{corollary} 
\begin{proof} 
Combine the latter statement of Theorem \ref{thm:jet-GW} 
with Lemma \ref{lem:formalization-shiftisom}. 
\end{proof} 

\section{The Semisimple Case}
\label{sec:semisimple_case}

In this section we use Givental's formula for the higher-genus potentials 
associated to a semisimple Frobenius manifold to define a canonical global 
section of the Fock sheaf for any tame semisimple cTEP structure.  
This global section is called the \emph{Givental wave function}.   
We use a theorem of Teleman to show that if $X$ has generically 
semisimple quantum cohomology then the Givental wave function 
for the A-model cTEP structure associated to $X$ coincides 
with the Gromov--Witten wave function for $X$. 

\subsection{Semisimple Opposite Module} 
\label{subsec:semisimpleopposite}
Recall from Definition~\ref{def:cTP} that a cTEP structure is a cTP structure such that the connection $\nabla$ is extended in
  the $z$-direction with a pole of order 2 along $z=0$. 
Let $\cU \colon \sfF_0 \to \sfF_0$ denote the residual part 
of the connection 
defined by 
\begin{align*}
  \cU \colon \sfF_{0} \to \sfF_{0} &&
  \text{$\cU [\alpha] = [z^2 \nabla_{\partial_z} \alpha]$ for $\alpha \in \sfF$.}  
\end{align*}
Flatness of the pairing implies that 
$\cU$ is self-adjoint with respect to $(\cdot,\cdot)_{\sfF_0}$. 

\begin{definition} 
\label{def:tamesemisimple}
A cTEP structure $(\sfF,\nabla,(\cdot,\cdot)_{\sfF})$ 
over $\cM$ is said to be \emph{tame semisimple} at $t\in \cM$ 
if the residual part $\cU\in \End(\sfF_{0,t})$ at $t$ 
is a semisimple endomorphism without repeated 
eigenvalues.  
\end{definition} 

The following proposition shows that any tame semisimple cTEP 
structure of rank $N+1$ is locally isomorphic to the A-model cTEP structure 
(Remark \ref{rem:completion-of-TP}) of $N+1$ points. 
This can be viewed, module the treatment on 
the pairing, as a special case 
of the classical Levelt--Turrittin theorem 
on the formal structure of irregular connections 
(see, e.g.~\cite[Chapter II, Theorem 5.7]{Sabbah:isomonodromic}). 
In fact, the existence of a pairing makes the proof easier. 
In the context of semisimple Frobenius manifolds, similar 
results has appeared in Dubrovin \cite[Lecture 4]{Dubrovin:Painleve} 
and Givental \cite[\S1.3]{Givental:semisimple}. 

\begin{proposition} 
\label{prop:semisimpletriv} 
Suppose that a cTEP structure 
$(\sfF,\nabla,(\cdot,\cdot)_{\sfF})$ is tame semisimple at $t_0\in \cM$. 
Then there exists a trivialization over a neighbourhood of $t_0$ 
\[
\Phiss \colon \C^{N+1}\otimes \cO[\![z]\!]  \cong \sfF
\]
such that $(\Phiss(e_i), \Phiss(e_j))_{\sfF} = \delta_{ij}$ and 
\[
\Phiss^* \nabla = \bigoplus_{i=0}^N 
\left( 
d -d(u_i/z)\right) 
\]
where $u_0,\dots,u_N$ are the eigenvalues of $\cU$. 
Moreover, the trivialization $\Phiss$ is unique up to 
reordering and changing the signs of the basis elements: $e_i \mapsto 
\pm e_{\sigma(i)}$, $\sigma \in \mathfrak{S}_{N+1}$. 
We call $\Phiss$ the \emph{semisimple trivialization} 
of $(\sfF,\nabla,(\cdot,\cdot)_{\sfF})$. 
\end{proposition} 
\begin{proof} 
The operator $\cU$ has distinct eigenvalues 
$u_0,\dots,u_N$ in a neighbourhood of $t_0$. 
Throughout the proof, we fix this neighbourhood and 
work over it. 
Let $\delta_i\in \sfF_0$, $i\in\{0,\dots,N\}$,
be the eigensection of $\cU$ with eigenvalue $u_i$. 
We normalize $\delta_i$ by the condition 
$(\delta_i, \delta_i)_{\sfF_0} =1$. 
Because $\cU$ is self-adjoint, it follows that
$(\delta_i, \delta_j)_{\sfF_0} = \delta_{ij}$. 
There exist lifts $\hdelta_i\in \sfF$ of $\delta_i$ 
such that $(\hdelta_i,\hdelta_j)_{\sfF}=1$. 
In the local basis $\hdelta_0,\dots,\hdelta_N$, 
we write the connection in the form  
\[
\nabla = d - \frac{1}{z} \sum_{j} C_j(z)dt^j+ (U + zV(z)) \frac{dz}{z^2} 
\]
where $\{t^j\}$ is a local co-ordinate system on $\cM$, 
$U = \diag(u_0,\dots,u_N)$, and 
$C_j(z), 
V(z)\in \End(\C^{N+1})\otimes \cO[\![z]\!]$. 
The fact that $\nabla$ preserves $(\cdot,\cdot)_{\sfF}$
implies that $V(-z) + V(z)^{\rm T}=0$. 

If we have a trivialization $\Phiss$ satisfying the 
conditions in the statement, then
$[\Phiss(e_i)]$ is an eigenvector of $\cU$ of eigenvalue $u_i$ 
and $([\Phiss(e_i)], [\Phiss(e_i)])_{\sfF_0} = 1$. 
Therefore, up to the choice of signs of $\hdelta_i$, we have 
\[
\bigl( \Phiss(e_0), \dots, \Phiss(e_N) \bigr) 
= 
\bigl(
\hdelta_0, \dots, \hdelta_N\bigr) R(z) 
\]
for some $(N+1,N+1)$-matrix $R(z)$ 
with entries in $\cO[\![z]\!]$ such that $R(0)=I$. 
It thus suffices to show that there 
exists a unique gauge transformation 
$R(z) \in GL(N+1,\cO[\![z]\!])$ such that $R(0)=I$ and:
\begin{align}
\label{eq:R-diffeq} 
R(z)^{-1} \circ \nabla \circ R(z) & = d - d(U/z) \\
\label{eq:R-unitary}
R(-z)^{\rm T} R(z) & = \id
\end{align} 
The differential equation \eqref{eq:R-diffeq} 
in the $z$-direction is: 
\[
\partial_z R + z^{-2}[U,R] + z^{-1} VR =0
\]
Writing $V(z) = V_0 + V_1 z + V_2 z^2 +\cdots$ 
and $R(z) = I + R_1 z + R_2 z^2 +\cdots$, we find:  
\begin{align}
\label{eq:recursion-R} 
\begin{split} 
& [U, R_1] + V_0 = 0 \\ 
& R_1 + [U,R_2] + V_1 + V_0 R_1 = 0 \\ 
& nR_n + [U, R_{n+1}] + V_n + V_{n-1} R_1 + \cdots + V_0 R_n =0 
\quad (n\ge 1)
\end{split}  
\end{align} 
We claim that the equations \eqref{eq:recursion-R} 
can be solved inductively and uniquely. 
The off-diagonal part of $R_1$ can be determined from the first 
equation: 
\[
(R_1)_{ij} = - \frac{(V_0)_{ij}}{u_i - u_j} \quad i\neq j. 
\]
Here the solvability is ensured by $(V_0)_{ii}=0$, 
which holds because $V_0$ is anti-symmetric. 
The second equation gives the diagonal part of $R_1$: 
\[
(R_1)_{ii} = -(V_1)_{ii} - \sum_{j : j \neq i} 
(V_0)_{ij} (R_1)_{ji}
\]
Similarly, we can solve for $R_2,R_3,\dots$ inductively. 
We check that $R$ constructed in this way satisfies 
unitarity \eqref{eq:R-unitary}. 
From the differential equation for $R$, we find:
\[
\parfrac{}{z} (R(-z)^{\rm T} R(z)) = - \frac{1}{z^2} [U, R(-z)^{\rm T} R(z)]
\]
Writing $R(-z)^{\rm T} R(z) = I + M_1 z + M_2 z^2 + \cdots$ gives:
\begin{align*} 
[U,M_1]& =0 \\ 
nM_n& = -[U, M_{n+1}] \quad (n\ge 1)
\end{align*}  
The first equation shows that $M_1$ is diagonal. 
The second equation with $n=1$ shows that $M_1$ is off-diagonal. 
Thus $M_1=0$. Hence $[U,M_2]=0$ and $M_2$ is diagonal. 
The second equation with $n=2$ shows that $M_2$ is off-diagonal 
and $M_2=0$. Repeating this, we find that $M_n=0$ for all $n\ge 1$. 

Finally we show that $R(z)$ satisfies the differential equation 
\eqref{eq:R-diffeq} in the $t$-direction.  
Note that $R(z)$ in the above construction 
depends analytically on $t\in \cM$. 
We can write 
\begin{equation} 
\label{eq:connectiontransformedbyR}
R(z)^{-1} \circ \nabla \circ R(z) = 
d - \frac{1}{z} \sum_j A_j(z) dt^j +U  \frac{dz}{z^2}
\end{equation} 
for some $A_j(z) \in \End(\C^{N+1}) \otimes \cO[\![z]\!]$.
The flatness of the connection yields: 
\[
\partial_j U - A_j(z) + z \partial_z A_j - z^{-1} [A_j(z), U]=0
\]
Writing $A_j(z)= A_{j,0}+ A_{j,1} z + A_{j,2} z^2+ \cdots$, 
we have 
\begin{align*} 
- [A_{j,0}, U] &= 0 \\ 
\partial_j U - A_{j,0} - [A_{j,1},U] &= 0 \\ 
-[A_{j,2},U] &=0 \\ 
(n-1) A_{j,n} - [A_{j,n+1}, U] &=0 \quad (n\ge 2)
\end{align*} 
The first equation shows that $A_{j,0}$ is a diagonal matrix. 
The second equation shows $A_{j,0} = \partial_j U$ 
and $[A_{j,1},U]=0$. Hence $A_{j,1}$ is diagonal.  
The third equation shows that $A_{j,2}$ is diagonal.  
The fourth equation with $n=2$ shows that $A_{j,2} =0$ and 
that $A_{j,3}$ is diagonal. 
Repeating this, we find that $A_{j,n} =0$ for all $n\ge 2$. 
It remains to show that $A_{j,1} =0$. 
We know that $\nabla$ preserves the pairing $(\cdot,\cdot)_{\sfF}$ 
and that $R(z)$ is unitary \eqref{eq:R-unitary},  
thus the connection \eqref{eq:connectiontransformedbyR} 
also preserves the diagonal pairing. 
This shows that $A_j(-z)^{\rm T} = A_j(z)$. 
In particular $A_{j,1}$ is anti-symmetric. 
As we have already seen that $A_{j,1}$ is diagonal, 
$A_{j,1} =0$. 
\end{proof} 

\begin{definition} 
\label{def:semisimple_opposite}
Let  $(\sfF,\nabla,(\cdot,\cdot)_{\sfF})$ 
be a cTEP structure which is 
tame semisimple over an open set $\cMss \subset \cM$. 
The \emph{semisimple opposite module}  
for $(\sfF,\nabla,(\cdot,\cdot)_{\sfF})$
is an opposite module $\sfPss$ over $\cMss$ 
such that, for any point $t\in \cMss$
and a semisimple trivialization $\Phiss$ 
(Proposition~\ref{prop:semisimpletriv})
over a neighbourhood of $t$, 
we have 
\[
\sfPss = \Phiss( \C^{N+1}\otimes z^{-1}\cO[z^{-1}])   
\]
in the neighbourhood. The opposite module $\sfPss$ is independent of the choice of $\Phiss$. 
\end{definition} 

\begin{remark} 
Even if the semisimple trivialization $\Phiss$ has monodromy, 
the semisimple opposite module $\sfPss$ is single-valued on 
the tame semi-simple locus $\cMss$. 
The semisimple opposite module is automatically homogeneous:
$\nabla_{z\partial_z} \sfPss \subset \sfPss$. 
\end{remark} 

When a cTEP structure is tame semisimple and miniversal, 
the semisimple opposite module defines a (rather trivial) 
Frobenius manifold structure 
on the base by Proposition~\ref{prop:flat_trivialization} and 
Remark~\ref{rem:opposite_conformalcase}. 
Miniversality implies that the eigenvalues 
$u_0,\dots,u_N$ form a local co-ordinate system. 
The following proposition follows straightforwardly 
from Proposition~\ref{prop:semisimpletriv}. 

\begin{proposition} 
\label{prop:trivial-Frobeniusmfd} 
Let  $(\sfF,\nabla,(\cdot,\cdot)_{\sfF})$ 
be a cTEP structure which is 
tame semisimple over an open set $\cMss \subset \cM$. 
The semisimple opposite module $\sfPss$ over $\cMss$ 
defines a Frobenius 
manifold structure on $\cMss$ which is isomorphic to 
the quantum cohomology Frobenius manifold 
of $(N+1)$ points.  It is given by:
\begin{itemize} 
\item the flat metric 
\[
g\left(\parfrac{}{u_i}, \parfrac{}{u_j}\right) = \delta_{ij}
\]
\item the semisimple product 
\[
\parfrac{}{u_i} * \parfrac{}{u_j} = \delta_{ij} \parfrac{}{u_i}
\] 
\item the flat identity vector field $e = \sum_{i=0}^N \partial/\partial u_i$
\item the Euler vector field $E = \sum_{i=0}^N u_i (\partial/\partial u_i)$
\end{itemize} 
where $u_0,\dots,u_N$ are the eigenvalues of the residual part $\cU$; these
are flat co-ordinates for this Frobenius manifold.  
\end{proposition} 

\subsection{A Section of the Fock Sheaf via Givental's Formula} 
\label{sec:Givental_global_section}

Givental has defined an abstract ancestor potential 
associated to any semisimple Frobenius manifold~\cite{Givental:semisimple,Givental:symplectic,Givental:quantization}. 
We will see that Givental's formula gives rise 
to a global section of the Fock sheaf associated to 
a semisimple Frobenius manifold 
(or more generally to a semisimple cTEP structure).

\subsubsection{Givental's Abstract Potential}
\begin{definition}[cTEP Structure Associated to a Frobenius Manifold] 
\label{def:cTEP-of-Frobeniusmfd}
Let $(\cM,*,e,g,E)$ be a Frobenius manifold 
(see Proposition~\ref{prop:flat_trivialization} and 
Remark~\ref{rem:opposite_conformalcase} for the notation). 
The Dubrovin connection defines a miniversal 
cTEP structure $(\sfF,\nabla,(\cdot,\cdot)_{\sfF})$ over 
$\cM$: 
\begin{align*} 
\sfF &= T \cM [\![z]\!] \\ 
\nabla & = \nabla^{\rm LC} - \frac{1}{z} \sum_{i=0}^N 
\left(\parfrac{}{t^i} * \right) dt^i 
+  (E*) \frac{dz}{z^2} + \mu \frac{dz }{z} \\ 
(\alpha(z), \beta(z))_{\sfF} & = g(\alpha(-z), \beta(z)) \quad 
\text{for $\alpha(z), \beta(z) \in \cM[\![z]\!]$} 
\end{align*} 
where $\{t^i\}_{i=0}^N$ is a local co-ordinate system on $\cM$, 
$\nabla^{\rm LC}$ is the Levi-Civita connection 
of $g$, and $\mu = (1-\frac{D}{2}) \id - \nabla^{\rm LC} E 
\in \End(T\cM)$ with $D$ the conformal dimension. 
This cTEP structure is equipped with the 
standard homogeneous opposite module 
\[
\sfP_{\rm std} = z^{-1} T\cM[z^{-1}] 
\]
and the standard unitary frame $\id \colon T_u \cM[\![z]\!] \cong \sfF_u$. 
\end{definition} 

Let $(\cM,*,e,g,E)$ be a Frobenius manifold such that 
the Euler multiplication $E*$ is semisimple 
with distinct eigenvalues $u_0,\dots,u_N$. 
Such a Frobenius manifold is said to be \emph{tame semisimple}. 
Then the corresponding cTEP structure is tame semisimple.  
It is known that the co-ordinate vector fields 
$\partial/\partial u_i$, $i\in\{0,\dots,N\}$, form an 
idempotent frame for $T\cM$: 
\[
\parfrac{}{u_i} * \parfrac{}{u_j} = \delta_{ij} \parfrac{}{u_i}
\]
We set:
\[
\Delta_i = g\left(\parfrac{}{u_i}, \parfrac{}{u_i}\right)^{-1}
\]
By Proposition~\ref{prop:semisimpletriv} and its proof, 
we have locally a semisimple trivialization 
$\Phiss \colon \C^{N+1} \otimes \cO[\![z]\!] 
\cong T\cM[\![z]\!]$ such that:
\[
\Phiss(e_i) = \sqrt{\Delta_i} \parfrac{}{u_i} + O(z)
\]
In view of Example~\ref{exa:Apt}, the product of 
Witten--Kontsevich tau-functions 
\[
\cT(\bq) = \prod_{i=0}^N \tau(\bq^i) 
\quad 
\text{with $\bq = (\bq^0,\dots,\bq^N)\in \C^{N+1}[\![z]\!]$}
\]
is an element of 
$\Fockanrat\big(\C^{N+1}, (1,\dots,1), \prod_{i=0}^N (-q_1^i)\big)$. 
This is the descendant potential of $(N+1)$ points\footnote
{For $N+1$ points, the descendant potential and the ancestor potential 
are the same.}. 
For each $u\in \cM$, $\Phissu$ defines a unitary isomorphism 
between $\C^{N+1}[\![z]\!]$ and $T_u\cM[\![z]\!]$,
and we have a quantized operator 
\begin{multline*} 
T_{e-\Phissu(1,\dots,1)} \circ 
\widehat{\Phissu}
\colon \\ 
\Fockanrat(\C^{N+1}, (1,\dots,1), 
\textstyle\prod_{i=0}^N (-q_1^i) ) 
\xrightarrow{\phantom{ABCD}}
\Fockanrat( T_u\cM, e, \det(-q_1*))  
\end{multline*} 
by Theorem~\ref{thm:quantization-welldefined} 
and Remark~\ref{rem:quantizedoperator-shiftisom}. 

\begin{definition}[Givental's Formula~\cite{Givental:quantization}] 
\label{def:Giventalabstractpotential}
The \emph{abstract ancestor potential} $\cAabs_u$ is 
the element of $\Fockanrat(T_u\cM, e, \det(-q_1*))$ 
given by:
\[
\cAabs_u= 
T_{e- \Phissu(1,\dots,1)} 
\widehat{\Phissu} \cT
\]
\end{definition} 
\begin{remark} 
The abstract potential $\cAabs_u$ does not depend on the choice of a 
semisimple trivialization $\Phiss$; see~\cite[Proposition 4.3]{CI:convergence}.  
Let us study what the shift isomorphism $T_{e-\Phissu(1,\dots,1)}$ 
does to $\widehat{\Phissu}\cT$. 
Note that the genus-one potential 
$\hcF^1$ of $\widehat{\Phissu} \cT$ satisfies :
\[
\hcF^1 \Bigr|_{q_0=0, q_1 =-e} =\sum_{i=0}^N  -\frac{1}{24} 
\log \left( [\Phissu^{-1} e]^i \bigr|_{z=0} \right) 
= \frac{1}{48} \sum_{i=0}^N \log \Delta^i (u)
\]
The shift isomorphism at genus one is a truncated Taylor expansion 
and this amounts to subtracting the value at the new base point $\bq = -ez$. 
Thus we can write: 
\[
\cAabs_u =e^{-\frac{1}{48} \sum_i \log \Delta^i(u)} 
\widehat{\Phissu} \cT
\]
This is the original form of Givental's formula.  
\end{remark} 

\subsubsection{A Global Section of the Fock Sheaf Associated to 
a Semisimple cTEP Structure} 
We regard the genus-$g$ ancestor potential 
$\cF^g_{\rm pt}$ of a point as a function of 
the co-ordinates $(q_0,q_1,q_2,\dots)$ 
via the Dilaton shift $q_n = y_n - \delta_{n,1}$: see Example~\ref{exa:Apt}.
When restricted to $q_0=0$, $\cF^g_{\rm pt}$ only depends on 
finitely many variables $q_1,\dots, q_{3g-2}$. 
In this section we write 
\[
\cF^g(0,q_1,q_2,\dots,q_{3g-2}) = 
\cF^g_{\rm pt}\Bigr|_{q_0=0} 
\]
making the argument explicit. 
Note that $q_1^{5g-5} \cF^g|_{q_0=0}$ is a polynomial 
for $g\ge 2$, $\cF^1|_{q_0=0} = -\frac{1}{24} \log(-q_1)$ 
and $\cF^0|_{q_0=0} = 0$ (see equation~\ref{eq:Fpt}). 

\begin{definition} 
\label{def:Giventalwavefcn}
Let $(\sfF,\nabla,(\cdot,\cdot)_{\sfF})$ be a miniversal 
cTEP structure which is tame semisimple 
over a non-empty open subset $\cMss \subset \cM$. 
Let $u_0,\dots,u_N$ be the eigenvalues of the residual part 
$\cU$, as in Proposition~\ref{prop:trivial-Frobeniusmfd}; these give local co-ordinates on $\cMss$.  
Let $\{u_i,x_n^i\}_{n\ge 1, 0\le i\le N}$ be the local co-ordinate 
system on $\LL$ associated to a semisimple 
trivialization $\Phiss$ of $\sfF$ as in Proposition~\ref{prop:semisimpletriv}. 
Define an element $\wavess = \{\Nabla^n \Css^{(g)}\}_{n,g}$ of 
$\Fock(\cMss;\sfPss)$ (see Definition~\ref{def:localFock}) 
by 
\begin{align*} 
\Nabla^3 \Css^{(0)}  &= \bY = \sum_{i=0}^N (x_1^i)^2 (du_i)^{\otimes 3}
\\ 
\Nabla \Css^{(1)} & = 
\sum_{i=0}^N d \cF^{1}_{\rm pt}(0,x_1^i)  
= - \sum_{i=0}^N 
\frac{1}{24} \frac{dx_1^i}{x_1^i}  
\\
\Css^{(g)} & = 
\sum_{i=0}^N \cF^{g}_{\rm pt}(0,x^i_1,x^i_2,\dots,x^i_{3g-2}), 
\quad g \ge 2 
\end{align*} 
and their covariant derivatives with respect to 
$\Nabla = \Nabla^{\sfPss}$. 
The global section of the Fock sheaf (Definition~\ref{def:Focksheaf}) 
over $\cMss$ given by $\wavess$ is called 
the \emph{Givental wave function}. 
This does not depend on the choice of a semisimple trivialization 
$\Phiss$. 
\end{definition} 

\begin{remark} 
It is easy to see that $\wavess$ 
satisfies the conditions in Definition~\ref{def:localFock}. 
The condition (Grading \& Filtration) 
follows from \eqref{eq:Fpt}. 
The discriminant \eqref{eq:discriminant} 
is given by $P(t,x_1) = \prod_{i=0}^N (-x_1^i)$. 
Thus the condition (Pole) follows also from \eqref{eq:Fpt}. 
\end{remark} 

\begin{remark} 
The element $\wavess$ can be identified with 
the \emph{Gromov--Witten wave function for $(N+1)$ points} 
which was introduced in Definition~\ref{def:GW-wave}. 
(The Gromov--Witten potential of $(N+1)$ points satisfies 
the necessary convergence condition stated in 
Assumption~\ref{assump:convergence}.)
\end{remark} 

\begin{remark} 
\label{rem:presentationsofGiventalwavefcn} 
Given any pseudo-opposite module $\sfP$ over an 
open subset $U\subset \cMss$, the Givental wave function 
gives rise to the element
$\wave_{\sfP} = T(\sfPss,\sfP) \wavess  \in \Fock(U;\sfP)$ 
over $U$. 
We call $\wave_{\sfP}$ 
\emph{the (local) presentation of the 
Givental wave function under $\sfP$}.  
\end{remark} 

\begin{lemma} 
\label{lem:jetofsemisimplewavefcn}
With notation as in Definition~\ref{def:Giventalwavefcn}, the equality
\[
\Nabla^n \Css^{(g)}
=  \sum_{i=0}^N 
\sum_{l_1,\dots, l_n \ge 0} 
\corr{\psi_1^{l_1}, \dots, \psi_n^{l_n} }_{g,n}^{\rm pt} 
dx^i_{l_1} \otimes \cdots \otimes dx^i_{l_n} 
\]
holds along the locus $\{x_1^i=-1,\ x_2^i=x_3^i=\cdots =0 : 0 \leq i \leq N\}
\subset \LL$, 
where we set $x_0^i = u_i$ on the right-hand side. 
In other words, we have 
\[
\For_{-z \Phissu (1,\dots,1)}  \wavess = \cT 
\quad \text{in} \quad 
\Fockanrat\big(\C^{N+1}, (1,\dots,1), 
\textstyle\prod_{i=0}^N (- q_1^i ) \big)  
\]
where $\For_{-z \Phissu(1,\dots,1)}$ is the formalization map 
(Definition~\ref{def:formalizationmap}) 
associated to the semisimple trivialization $\Phissu$ 
at $u\in \cM$. 
\end{lemma} 
\begin{proof} 
The formula $\For_{-z \Phissu(1,\dots,1)} \wavess = \cT$ 
was proved more generally for a Gromov--Witten wave function 
in Theorem~\ref{thm:jet-GW}; this lemma is a special case where 
$X$ consists of $(N+1)$ points. 
Thus it suffices to show that the former statement is equivalent 
to the latter. 
For simplicity we consider the case $N+1=\dim \cM = 1$; the general case is similar.  
Take a point $u^* \in \cMss$. 
Under the semisimple trivialization, the inverse fundamental 
solution $M$ appearing in \eqref{eq:inverse_fundsol} 
is given by 
\[
M(u,z) = e^{-(u-u^*)/z} 
\]
Therefore the flat co-ordinates $\bq$ 
associated to the unitary frame $\Phiss$ are 
given by 
\begin{equation} 
\label{eq:xandq}
\bq = [e^{-(u-u^*)/z} \bx]_+ 
\end{equation} 
(see equation~\ref{eq:formal_flat_coordinates}) with $\bx = \sum_{n=1}^\infty x_n z^n$, $\bq = \sum_{n=0}^\infty 
q_n z^n$. 
This shows that 
\[
d u = d q_0, \quad d x_n = dq_n \quad (n\ge 1)
\]
at the point $(u,\bx) = (u^*,-z) \in \LL_{u^*}$. 
The conclusion follows from the definition of the 
formalization map. 
\end{proof}

\begin{theorem} 
\label{thm:formalization-Giventalwave} 
Let $(\sfF,\nabla,(\cdot,\cdot)_{\sfF})$ 
be the tame semisimple cTEP structure associated to a tame 
semisimple Frobenius manifold $\cM$. 
Let $\wave_{\rm std}= T(\sfPss, \sfP_{\rm std}) \wavess 
\in \Fock(\cM;\sfP_{\rm std})$ 
denote the presentation of the Givental wave function 
of $(\sfF,\nabla, (\cdot,\cdot)_{\sfF})$ 
\emph{under} the standard opposite module $\sfP_{\rm std}$ 
(see Remark~\ref{rem:presentationsofGiventalwavefcn}). 
Then we have 
\[
\For_{- z e}(\wave_{\rm std}) = \cAabs_u
\]
where $\For_{-ze}$ is the formalization map at $-z e \in  \LLo_u$ 
associated with the standard unitary frame $T_u\cM[\![z]\!] \cong \sfF_u$ 
(see Definition~\ref{def:formalizationmap} and 
Lemma~\ref{lem:formalization-shiftisom}) 
and $\cAabs_u$ is the abstract potential in 
Definition~\ref{def:Giventalabstractpotential}. 
\end{theorem} 
\begin{proof} 
Combine
Theorem~\ref{thm:transformationrule-Giventalquantization},  
Lemma~\ref{lem:jetofsemisimplewavefcn},
and the definition of the abstract potential. 
\end{proof}

The quantum cohomology of $X$ is said to be 
\emph{generically semisimple} if the analytic quantum product $*$ 
(see Assumption~\ref{assump:convergence}, Part~1) 
is semisimple (i.e.~isomorphic as a ring to a direct sum of copies of $\C$)  
over an open dense subset of $\cM_{\rm A}$. 
This is equivalent to $*$ being semisimple at a single point. 
Then the Euler multiplication $E*$ (see equation~\ref{eq:Eulerfield}) 
has no repeated eigenvalues on an open dense subset 
$\cM_{\rm A}^{\rm ss}\subset \cM_{\rm A}$, and the A-model cTEP structure is 
tame semisimple over $\cM_{\rm A}^{\rm ss}$ 
(see Definition~\ref{def:tamesemisimple}). 
In particular, the Givental wave function defines a section of the A-model Fock sheaf $\Fock_X$ over $\cM_{\rm A}^{\rm ss}$. 
The following is a reformulation of a result of Teleman.

\begin{theorem}[Teleman~\cite{Teleman}]
\label{thm:Teleman} 
When the quantum cohomology of $X$ is generically semisimple, the 
Gromov--Witten wave function (Definition~\ref{def:GW-wave}) 
coincides with the Givental wave function (Definition~\ref{def:Giventalwavefcn}).  
\end{theorem} 
\begin{proof} 
Both wave functions are uniquely 
determined by their formalizations at  
$(t,-z)\in \LLo$ with $t\in \cM_{\rm A}^{\rm ss}$ 
with respect to the standard opposite module $\sfP_{\rm std}$ 
(see Example~\ref{ex:Amodel-opposite}). 
Theorem~\ref{thm:jet-GW} shows that the formalization 
of the Gromov--Witten wave function is the geometric 
ancestor potential $\cA_{X,t}$ (with $Q_1=\cdots = Q_r =1$). 
Theorem~\ref{thm:formalization-Giventalwave} shows that 
the formalization of the Givental wave function is the abstract 
ancestor potential given by Givental's formula 
(Definition~\ref{def:Giventalabstractpotential}). 
Teleman~\cite{Teleman} showed that the geometric ancestor potential 
coincides with the abstract one 
(see also~\cite[Theorems 6.4, 6.5]{CI:convergence}). 
The conclusion follows.  
\end{proof} 

\begin{remark}
\label{rem:modularity}
The Givental wave function is automatically a `modular function' 
in the following sense. 
Let $\big(\sfF,\nabla,(\cdot,\cdot)_{\sfF}\big)$ be a tame 
semisimple cTEP structure over $\cM$, and let $\pi \colon 
\tcM \to \cM$ be the universal cover.  
Let $\wave$ be the Givental wave function associated 
to $\big(\sfF,\nabla,(\cdot,\cdot)_{\sfF}\big)$. 
Suppose that we have an opposite module $\sfP$ 
for $\pi^\star\big(\sfF,\nabla,(\cdot,\cdot)_{\sfF}\big)$ 
over the universal cover $\tcM$. 
The pull-back $\pi^\star \wave$ of the Givental wave function 
is obviously invariant under the group $\Gamma = \pi_1(\cM)$ 
of deck transformations, however its presentation 
$\wave_\sfP = (\pi^\star \wave)_\sfP$ 
with respect to $\sfP$ 
is not necessarily so since the opposite module $\sfP$ 
may not be single-valued on $\cM$. Instead we have the transformation property: 
\begin{equation} 
\label{eq:deck_transformation}
\gamma^\star \wave_\sfP = 
T(\sfP, \gamma^\star \sfP) \wave_\sfP 
\end{equation}
with respect to $\gamma \in \Gamma$, since 
$\gamma^\star \wave_\sfP = \wave_{\gamma^\star \sfP}$. 
Suppose moreover that the cTEP structure 
$\big(\sfF,\nabla,(\cdot,\cdot)_\sfF\big)$ 
is the restriction of a TEP structure $\big(\cF=\cO(F),\nabla,(\cdot,\cdot)_\cF\big)$ 
to the formal neighbourhood of $z=0$ and that the 
opposite module $\sfP$ defines an extension of $\cF$ across 
$z=\infty$ (see Remark \ref{rem:opposite_extension}). 
In this case, one can rephrase equation \eqref{eq:deck_transformation}
using the $L^2$-formalism in \S \ref{subsec:global_L2}  
as follows. 
The monodromy of $\big(\cF,\nabla,(\cdot,\cdot)_\cF\big)$ 
along $\gamma$ defines a symplectic transformation 
\[
\U_\gamma \colon \cH_t \to \cH_t 
\]
where recall that $\cH_t = L^2(\{t\}\times S^1, F)$.  
Then the total potential $\cZ$ associated to $\wave_\sfP$ 
(see Definition \ref{def:transformation_L2})  
transforms under $\Gamma$ as:
  \[
    \cZ (\gamma^{-1} t) \propto \hUU_\gamma \cZ(t) \qquad
    \gamma \in \Gamma
    \]
The quantization $\hUU_\gamma$ here was defined in 
\S\ref{subsec:global_L2}.
\end{remark}

\section{The Fock Sheaf and Mirror Symmetry} 
\label{sec:mirror_symmetry}

In this section we discuss applications of our global quantization 
formalism in the context of mirror symmetry. 
We consider two cases: mirrors of toric orbifolds and mirrors of 
Calabi--Yau hypersurfaces. 
In the former case, we construct a global section 
of the B-model Fock sheaf using Givental's formula, thereby proving a higher-genus version of Ruan's Crepant Transformation Conjecture for toric orbifolds.
 In the latter case, the existence of a global section of the B-model Fock sheaf
(which corresponds under mirror symmetry to the Gromov--Witten wave function) 
is conjectural; cf.~recent work of Costello--Li~\cite{Costello--Li}. 

\subsection{The Crepant Transformation Conjecture in the Toric Case}
\label{sec:toric_mirror_symmetry}

A mirror partner of a ``Fano-like" manifold $X$ is conjecturally given 
by a so-called Landau--Ginzburg model, which is a pair 
$(\check{T}, W)$ where $\check{T}$ is a (non-compact) Calabi--Yau manifold
and $W \colon \check{T} \to \C$ is a holomorphic function. 
Suppose that $X$ has generically 
semisimple quantum cohomology and 
has a family of Landau--Ginzburg models 
$W_y \colon \check{T}_y \to \C$, $y\in \cM_{\rm B}$, as a mirror. 
The space $\cM_{\rm B}$ here parametrizes the Landau--Ginzburg 
models; we call it the \emph{B-model moduli space}. 
Under mirror symmetry, the K\"{a}hler moduli space 
$\cM_{\rm A} \cap H^2(X;\C)$ of $X$ is identified with 
a small open patch of $\cM_{\rm B}$,   
and the small quantum cohomology ring of $X$ should be identified 
with the family of Jacobi rings for $W_y$ on this patch. 
Furthermore, the A-model TEP structure 
(Example~\ref{ex:AmodelTP}) constructed from $X$
should be identified 
with the B-model TEP structure constructed from $W_y$. 
Our assumption that $X$ has generically semisimple quantum cohomology corresponds to the condition 
that, for generic $y$, all critical points of $W_y$ are isolated and non-degenerate. 
In such a situation, we can extend the total descendant 
Gromov--Witten potential $\cZ_X$ to a \emph{global section} 
of the B-model Fock sheaf over an 
extended B-model moduli space 
$\cM_{\rm B}^{\rm ext}$; this is the Givental wave function for the extended B-model TEP structure. 
Furthermore it can happen that this global section restricts, 
on another small open patch of $\cM_{\rm B}$, 
to the Gromov--Witten potential $\cZ_Y$ of another space $Y$ 
which would typically be $K$-equivalent (or derived equivalent) to $X$.  Thus the global section $\cZ_X$ of the A-model Fock sheaf for $X$ would coincide, after analytic continuation, with the global section $\cZ_Y$ of the A-model Fock sheaf for $Y$.  This gives a higher-genus version of Y.~Ruan's 
conjecture about the relationship between Gromov--Witten 
theory and crepant birational (or derived categorical) geometry. 
We illustrate this in the toric setting.

Givental~\cite{Givental:ICM} and Hori--Vafa~\cite{Hori--Vafa} have described a Landau--Ginzburg model that gives a
mirror to a toric variety.  Here the Calabi--Yau manifold $\check{T}$ is $(\C^\times)^D$ and the superpotential $W \colon \check{T} \to \C$ is
a Laurent polynomial function 
on $(\C^\times)^D$ 
with Newton polytope equal to the fan polytope of 
the toric variety. 
The B-model TEP structure in this context has been  
studied by many people, including
Sabbah~\cite{Sabbah:hypperiod}, Barannikov~\cite{Barannikov:projective}, 
Douai--Sabbah~\cite{Douai--Sabbah:I, Douai--Sabbah:II}, Coates--Iritani--Tseng~\cite{CIT:wall-crossings},  Iritani~\cite{Iritani:integral}, and
Reichelt--Sevenheck~\cite{Reichelt--Sevenheck}. 
In the rest of this section (\ref{sec:toric_mirror_symmetry}) we consider the mirrors to certain toric orbifolds, following 
closely the exposition in~\cite[\S 3]{Iritani:integral}. 

\subsubsection{Toric Orbifolds as GIT Quotients}
\label{sec:GIT}
Borisov--Chen--Smith construct toric Deligne--Mumford stacks from so-called \emph{stacky fans}~\cite{BCS}.  Let $X$ be the toric Deligne--Mumford stack corresponding to the stacky fan $(\bN;\Sigma;b_1,\ldots,b_m)$, so that:
\begin{itemize} 
\item  $\bN$ is a finitely generated abelian group; 
\item  $\Sigma$ is a rational simplicial fan in $\bN_\R = \bN \otimes_\Z \R$; 
\item  $b_1,\dots,b_m\in \bN$ are such that their images in $\bN_\R$ generate the one-dimensional cones of $\Sigma$.  
\end{itemize} 
Let $\Delta\subset \bN_\R$ denote 
the convex hull of $b_1,\dots,b_m$. 
This is called the \emph{fan polytope} of $X$. We assume that:
\begin{itemize}
\item $X$ is an orbifold, i.e.~the generic isotropy of $X$ is trivial.  
This amounts to requiring that $\bN$ is torsion-free.
\item the coarse moduli space of $X$ is projective.  
This amounts to requiring that $\Delta$ contains the origin in its strict interior 
and that $\Sigma$ admits a strictly convex piecewise-linear function.
\item $X$ is weak Fano.  
This amounts to requiring that $b_1,\ldots,b_m$ lie on the boundary 
of $\Delta$. 
\item $\Delta \cap \bN$ generates $\bN$ over $\Z$.
\end{itemize}
We now explain how to construct $X$ as a GIT quotient.

Set $\Delta \cap \bN = \{b_1,\dots,b_m, b_{m+1}, \dots, b_n\}$, with $n\ge  m$,
and define a lattice $\bL\subset \Z^{n}$ by the exact sequence 
\begin{equation} 
\label{eq:fanseq}
\xymatrix{
  0 \ar[r] & \bL \ar[r] &\Z^n \ar[r]^-{\beta} & \bN \ar[r] & 0
}
\end{equation} 
where $\beta$ is the homomorphism that sends the $i$th standard basis 
vector $e_i$ to $b_i$. 
The torus $\T:=\bL \otimes \C^\times$ acts on $\C^{n}$ 
via the inclusion $\T \subset (\C^\times)^{n}$ induced 
by $\bL \subset \Z^n$. 
We denote by $\cA_{\Sigma}$ the set of anti-cones, that is, the set of 
subsets $I\subset \{1,\dots,n\}$ 
such that $I$ contains $\{m+1,\dots,n\}$ and 
that $\{b_i : i \in \{1,\dots,n\} \setminus I\}$ 
spans a cone of the fan $\Sigma$. 
Set
\[
\cU_\Sigma = \C^{n} \setminus \bigcup_{I \notin \cA_\Sigma} \C^I
\]
where $\C^I = \{ (z_1,\dots,z_n)\in \C^n : 
\text{$z_i = 0$ for $i \notin I$}\}$. 
The toric Deligne--Mumford stack $X$ is constructed as the quotient stack: 
\begin{equation} 
\label{eq:toricstack}
X = \big[\cU_\Sigma/\T\big]
\end{equation} 

Let $M\colon \Z^{n} \to \bL^\vee = \Hom(\bL,\Z)$ 
be the map dual to the inclusion $\bL \subset \Z^{n}$. 
The vector space $\bL^\vee_\R = \bL^\vee \otimes \R$ 
is canonically identified with $H^{\le 2}_{\CR}(X)$~\cite[Remark 3.5]{Iritani:integral}. 
The \emph{extended K\"{a}hler cone} is a cone $C_X \subset \bL^\vee_\R$ 
defined by:
\[
C_{X} = 
\bigcap_{I\in \cA_{\Sigma}} M(\R_{>0}^I)
\] 
Under the identification $\bL_\R^\vee \cong H^{\le 2}_{\CR}(X)$, 
this matches with the product of the ordinary 
K\"{a}hler (or ample) cone $\Amp(X)  
\subset H^2(X;\R)$ and the rays generated by positive generators 
of the twisted sectors in $H^{\le 2}_{\CR}(X)$ and $1\in H^0(X)$ 
(see~\cite[Lemma 3.2]{Iritani:integral}). 
The \emph{extended anticanonical class} 
$- K_X^{\rm ext} :=  M(e_1+ \cdots + e_n) \in \bL^\vee$ 
projects to the usual anticanonical class $-K_X\in H^2(X)$ 
and lies in the closure $\ov{C}_{X}$ of 
$C_{X}$ by the weak Fano condition.  The space $\bL^\vee$ is the space of stability conditions for the action of $\T$ on $\C^n$, and for any stability condition $\theta$ in the extended K\"ahler cone $C_X$, we have that the GIT (stack) quotient $\big[\C^n /\!\!/_\theta \T\big]$ is equal to $X$, because $\big[\C^n /\!\!/_\theta \T\big] = \big[\cU_\Sigma/\T\big]$ as in \eqref{eq:toricstack}.

\subsubsection{Birational Toric Orbifolds Arising from Variation of GIT} 
\label{sec:variation_of_GIT}
We can have several different projective stacky fan structures 
with the same fan polytope $\Delta$. 
The corresponding toric stacks are birational and are 
related by variation of GIT.
Reversing the above construction, start now with an integral 
polytope $\Delta \subset \bN_\R$ such that the origin is 
contained in its strict interior and that $\Delta \cap \bN$ generates 
$\bN$ over $\Z$. 
Set $\Delta \cap \bN = \{b_1,\dots,b_n\}$ as before. 
These vectors define the exact sequence \eqref{eq:fanseq} and thus define an action of
$\T := \bL \otimes \C^\times$ on $\C^{n}$. 
A character $\theta \in \bL^\vee = \Hom(\T,\C^\times)$ 
of $\T$ defines a stability condition for this action. 
Set $C_{\rm eff} = M(\R_{\ge 0}^n)$, where 
$M \colon \Z^n \to \bL^\vee$ denotes the map dual 
to the inclusion $\bL \subset \Z^n$ as before; this is a strictly convex cone. 
Also define $\Wall \subset C_{\rm eff}$ to be the union of the cones 
$M(\R_{\ge 0}^I)$ for all subsets $I\subset \{1,\dots,n\}$ 
such that $\{M(e_i) : i\in I\}$ does not span $\bL_\R^\vee$ over $\R$.  The walls $\Wall$ give a wall and chamber structure on $C_{\rm eff}$; this is the \emph{secondary fan} of Gelfand--Kapranov--Zelevinsky~\cite{GKZ}.
The GIT (stack) quotient $X_\theta := \big[\C^n /\!\!/_\theta \T\big]$ is empty unless the stability parameter 
$\theta$ lies in $C_{\rm eff}$. 
If $\theta\in C_{\rm eff} \setminus \Wall$ then there are no strictly $\theta$-semistable points in $\C^n$.
Take $\theta \in C_{\rm eff} \setminus \Wall$ and 
set $\cA_\theta = \{ I \subset \{1,\dots,n\} : \theta \in M(\R_{>0}^I) \}$. 
The corresponding GIT quotient $X_\theta$ 
is the projective toric Deligne--Mumford stack given by the 
following stacky fan on $\bN$: 
\begin{itemize} 
\item $b_i$ is a specified generator of a one-dimensional cone 
if and only if $\{1,\dots,n\} \setminus \{i\} \in \cA_\theta$; 

\item a subset $\{b_i: i\in I\}$ spans a cone of the fan if and only if 
$\{1,\dots,n\} \setminus I \in \cA_\theta$. 
\end{itemize} 
Note that $\cA_\theta$ coincides with the set of anti-cones for
this fan.
The corresponding extended K\"{a}hler cone $C_\theta := 
\bigcap_{I \in \cA_\theta} M(\R_{>0}^I)$ 
is the connected component of $C_{\rm eff} \setminus \Wall$ 
containing $\theta$. 
The toric stack $X_\theta$ depends on $\theta$ only via the chamber 
$C_\theta$. 
The fan polytope of $X_\theta$ is a polytope contained in 
$\Delta$ and contains the origin in its interior. 
 
Let $\Toric(\Delta)$ denote the set of smooth projective toric 
stacks arising in this way; they are parametrized 
by connected components of $C_{\rm eff} \setminus \Wall$, that is, by maximal cones in the secondary fan.
When a chamber $C_{\theta}\subset C_{\rm eff} \setminus \Wall$ 
contains the vector $M(e_1 + \cdots + e_n)$ in its closure, 
the corresponding toric stack $X_\theta$ is weak Fano. 
In this case, the fan polytope of $X_\theta$ coincides with $\Delta$ 
and all the generators of one-dimensional cones of the stacky fan
lie in the boundary of $\Delta$~\cite[Lemma 3.3]{Iritani:integral}. 
Let $\Crep(\Delta)\subset \Toric(\Delta)$ denote the subset 
consisting of toric stacks corresponding to a chamber $C_\theta$ 
with $M(e_1 + \cdots + e_n) \in \ov{C}_\theta$. 
Toric stacks from $\Crep(\Delta)$ are all $K$-equivalent and also 
derived equivalent to each other, via a composition of Fourier--Mukai transformations~\cite{Kawamata,Coates--Iritani--Jiang--Segal}.

\subsubsection{Mirror Landau--Ginzburg Models} 
\label{sec:LG_model}
Let $X$ be a toric Deligne--Mumford stack, as in \S\ref{sec:GIT}.
The mirror of $X$ 
is given by a family of 
Laurent polynomials $W_a$ on $\check{T} = \Hom(\bN, \C^\times) 
\cong (\C^\times)^D$ 
parametrized by $a =(a_1,\dots, a_{n}) \in (\C^\times)^{n}$: 
\begin{equation} 
\label{eq:LGmodel-W}
W_a(x) = \sum_{i=1}^{n} a_i x^{b_i}
\end{equation} 
The torus $\check{T} = \Hom(\bN,\C^\times)$ acts 
on the product $(\C^\times)^{n} \times \check{T}$ by 
\begin{align*}
  (a_1,\dots,a_{n}, x) \longmapsto 
  (\lambda^{b_1} a_1,\dots, \lambda^{b_{n}} a_{n}, 
  \lambda^{-1} \cdot x) &&
  \lambda \in \check{T} 
\end{align*}
and the potential $W_a(x)$ is invariant under this action. 
The family 
of Laurent polynomials $\{W_a\}_{a\in (\C^\times)^{n}}$ therefore descends to give a family over the quotient space 
$\cM_{\rm B} := (\C^\times)^n/\check{T}$: 
\[
\xymatrix{
  ((\C^\times)^n \times \check{T}) / \check{T}  \ar[r]^-{\cW} \ar[d]_\pr & \C \\
  \cM_{\rm B} 
}
\]
where $\pr$ is the projection to the first factor and 
$\cW([a,x]) = W_a(x)$. 
Note that $\cM_{\rm B}$ is identified with 
$\Hom(\bL,\C^\times) = \bL^\vee \otimes \C^\times$ 
via the exact sequence \eqref{eq:fanseq}. 
For $y\in \cM_{\rm B}$, we write $\check{T}_y := \pr^{-1}(y) 
\cong \check{T}$ and write $W_y$ for the Laurent polynomial 
$\cW$ restricted to $\check{T}_y$.  
The parameter space $\cM_{\rm B}$ is partially compactified to 
a toric variety $\ov\cM_{\rm B}$ defined by the secondary 
fan in $\bL_\R^\vee = \Hom(\bL, \R)$. 
Note that all toric stacks from $\Crep(\Delta)$ have the 
same mirror family, 
but each of them corresponds to a different torus fixed point 
in the secondary toric variety $\ov\cM_{\rm B}$. 
We call the fixed point in $\ov{\cM}_{\rm B}$ corresponding to 
a toric stack $X\in \Crep(\Delta)$ 
the \emph{large radius limit point} for $X$ and denote it by $o_X$. 
(A toric stack $X$ from $\Toric(\Delta)\setminus \Crep(\Delta)$ 
also corresponds to a fixed point $o_X \in \ov{\cM}_{\rm B}$, but  
in this case $X$ either is non-weak-Fano or has a fan polytope different 
from $\Delta$; for such $X$, genus-zero mirror symmetry 
in the form stated below does not hold.) 

\subsubsection{The B-Model TEP Structure}
\label{sec:B-model_TEP}
We now construct the B-model TEP structure from the Landau--Ginzburg 
model. Let $\cMo_{\rm B}\subset \cM_{\rm B}$ denote the (non-empty)
Zariski open subset parametrizing non-degenerate Laurent polynomials.  
Here a Laurent polynomial $W_a$ is said to be \emph{non-degenerate}~\cite[1.19]{Kouchnirenko}  
if for every face $F \subset \Delta$ of dimension 
$0\le \dim F <D$, the Laurent polynomial 
$W_{F, a} := \sum_{i: b_i \in F} a_i x^{b_i}$ 
has no critical points in $\check{T}$. 
There is a local system $R^\vee_\Z$ of relative homology groups 
over $\cMo_{\rm B}\times \C^\times$ such that 
\[
R_{\Z, (y,z)}^\vee= H_D \left( \check{T}_y, \{ x \in \check{T}_y : 
\Re(W_y(x) /z) \ll 0\} ; \Z \right) 
\]
for $(y,z) \in \cMo_{\rm B}\times \C^\times$; see~\cite[Proposition 3.12]{Iritani:integral}.
By Morse theory and Kouchnirenko's theorem~\cite{Kouchnirenko}, 
we find that $R_{\Z, (y,z)}^\vee$ is free of rank $\Vol(\Delta)$, 
with basis given by \emph{Lefschetz thimbles} of $W_y$,  
and that the intersection pairing between the 
fibers at $(y,-z)$ and $(y,z)$
\[
I^\vee \colon R^\vee_{\Z,(y,-z)} \times R^\vee_{\Z,(y,z)} \to \Z 
\]
is perfect. 
Here $\Vol(\cdot)$ denotes a normalized volume such that 
the standard simplex has volume one. 
Dualizing, we obtain a local system $R_\Z = \Hom(R_\Z^\vee, \Z)$ 
of relative cohomology groups equipped with a perfect pairing 
$I \colon R_{\Z, (y,-z)} \times R_{\Z,(y,z)} \to \Z$. 
We write $\cR = R_\Z \otimes \cO_{\cMo_{\rm B} \times \C^\times}$ 
for the corresponding locally free sheaf; this carries
a flat connection $\nabla^{\rm GM}$ and a pairing 
$I \colon (-)^* \cR \otimes \cR \to \cO_{\cMo_{\rm B}\times \C^\times}$, 
where $(-) \colon \cMo_{\rm B}\times \C \to \cMo_{\rm B} \times \C$ 
is the map sending $(y,z)$ to $(y,-z)$ as before. 

Let $\omega$ denote the invariant holomorphic volume form 
on the torus $\check{T}$ such that $\int_{\check{T}_\R} \omega 
= (2\pi\iu)^D$ with $\check{T}_\R = \Hom(\bN,\R_{>0})$. 
An oscillatory differential form of the form 
\[
\exp(W_y(x)/z) \phi(x) \omega \quad 
\text{with} \quad \phi(x) \in \C[\check{T}_y]
\]
defines a section of $\cR$ via integration over Lefschetz thimbles. 
By requiring that these sections extend across $z=0$, we can define 
a locally free extension $\cF_{\rm B}$ of $\cR$ to 
$\cMo_{\rm B}\times \C$; this extension is denoted by $\cR^{(0)}$ in~\cite[\S 3.3.2]{Iritani:integral}. 
The flat connection $\nabla^{\rm GM}$ extends to a meromorphic 
flat connection on $\cF_{\rm B}$ with poles along $z=0$. 
The B-model TEP structure\footnote{
The shift $-\frac{D}{2} \frac{dz}{z}$ of the connection was introduced 
implicitly in~\cite[Equation (53)]{Iritani:integral} as a factor 
$(-2\pi z)^{-D/2}$ in oscillatory integrals; this also shifts the pairing 
by the factor $(2\pi\iu z)^{-D}$~\cite[Equation~56]{Iritani:integral}. 
Note that $I$ is flat with respect to 
$\nabla^{\rm GM}$ and $(\cdot,\cdot)_{\rm B}$ is flat 
with respect to $\nabla$. 
The sign factor $(-1)^{\frac{D(D-1)}{2}}$ was missing in~\cite{Iritani:integral}.  
See~\cite[footnote~16]{Iritani:periods}. 
} is given by the data  
$(\cF_{\rm B}, \nabla^{\rm B}, (\cdot,\cdot)_{\rm B})$ where: 
\begin{itemize} 
\item $\cF_{\rm B}$ is as defined above.  This is a locally free sheaf of rank $\Vol(\Delta)$ 
over $\cMo_{\rm B}\times \C$; 
\item $\nabla^{\rm B} = \nabla^{\rm GM} - \frac{D}{2} \frac{dz}{z}$; 
\item $(s_1,s_2)_{\rm B} = (-1)^{\frac{D(D-1)}{2}}
(2\pi\iu z)^{-D} I(s_1,s_2)$ is a pairing 
$(-)^* \cF_{\rm B} \otimes \cF_{\rm B} \to \cO_{\cMo_{\rm B}\times \C}$. 
\end{itemize} 
See~\cite{Sabbah:hypperiod, Douai--Sabbah:I, Douai--Sabbah:II, Reichelt--Sevenheck} 
for an algebraic construction of the B-model TEP structure via 
the Fourier--Laplace transformation of the Gauss-Manin system 
associated to $W_y$. 
The B-model TEP structure can be also described as a GKZ system 
associated to the fan polytope $\Delta$~\cite{Iritani:integral, Reichelt--Sevenheck}. 

\subsubsection{The Mirror Map and an Isomorphism of TEP Structures} 
Mirror symmetry gives an isomorphism between the A-model TEP 
structure (Example~\ref{ex:AmodelTP}) and the B-model TEP 
structure, as we now explain. 
First we recall the \emph{Galois action}~\cite[\S 2.2]{Iritani:integral} 
on the A-model TEP structure. 
In general, the A-model TEP structure of a smooth Deligne--Mumford stack $X$ 
has a discrete symmetry given by the sheaf cohomology 
$H^2(X;\Z)$ of the underlying topological stack $X$. 
The base space $H^\bullet_{\CR}(X)$ of the A-model TEP structure carries  
an action of $H^2(X;\Z)$, and the TEP structure descends to 
the quotient space $H^\bullet_{\CR}(X)/H^2(X;\Z)$. 
(This is essentially due to the Divisor Equation.) Let $(\cF_{\rm A},\nabla^{\rm A}, (\cdot,\cdot)_{\rm A})/H^2(X;\Z)$ 
denote the A-model TEP structure over $(H^\bullet_{\CR}(X)/H^2(X;\Z)) \times \C$. 
Let $X$ be a toric stack from $\Crep(\Delta)$;  
recall that there is a corresponding large radius limit point $o_X\in \ov{\cM}_{\rm B}$. 
The Mirror Theorems for toric varieties~\cite{Givental:toric} and toric Deligne--Mumford stacks~\cite{CCIT:mirror} imply, by~\cite[Proposition 4.8]{Iritani:integral}, that there exist an open neighbourhood $U_X$ of $o_X$ in $\ov{\cM}_{\rm B}$,  
a mirror map $\tau \colon U_X \cap \cMo_{\rm B} 
\to H^{\le 2}_{\CR}(X)/H^2(X;\Z)$, 
and a mirror isomorphism: 
\[
\Mir: 
(\cF_{\rm B}, \nabla^{\rm B}, (\cdot,\cdot)_{\rm B}) 
\Bigr|_{(U_X\cap \cMo_{\rm B}) \times \C} 
\cong 
(\tau\times \id)^*
\left( (\cF_{\rm A}, \nabla^{\rm A}, (\cdot, \cdot)_{\rm A} ) /H^2(X;\Z) \right) 
\] 
such that:
\[
\Mir([\exp(W_y(x)/z) \omega]) = \unit 
\]
The open set $U_X \cap \cMo_{\rm B}$ here is 
isomorphic to the punctured polydisc 
$\{(q_1,\dots,q_r)\in (\C^{\times})^r: |q_a|<\epsilon\}$ for some $\epsilon>0$
(see~\cite[Lemma 3.8]{Iritani:integral}) 
and the A-model TEP structure is convergent on the image of the mirror map. 

\subsubsection{The Extended B-model TEP Structure} 
Our global quantization formalism is based on a miniversal TP structure 
(Assumption~\ref{assump:miniversal}), but 
the B-model TEP structure just discussed is not miniversal.  
So we need to unfold it to a miniversal TEP structure. 
We use a reconstruction theorem due to 
Hertling--Manin~\cite[Theorem 2.5, Lemma 3.2]{Hertling--Manin:unfoldings} 
to show that for generic $y\in \cMo_{\rm B}$,  the germ 
$(\cF_{\rm B},\nabla^{\rm B},(\cdot,\cdot)_{\rm B})
|_{(\cM_{\rm B},y)\times \C}$ of a TEP structure at $y$ 
can be extended to a miniversal TEP structure over 
$(\cMo_{\rm B},y) \times (\C^{\Vol(\Delta)-r},0)\times \C$, where 
$r = \dim \cM_{\rm B} = n - D$. 
For this, we need to check Hertling--Manin's injectivity condition (IC) and  
generation condition (GC). 
The condition (IC) says that there exists a local section $\zeta$ near $y$ 
such that the map $T_y\cMo_{\rm B} \to \cF_{\rm B}|_{(y,0)}$ defined 
by $v \mapsto z \nabla^{\rm B}_v \zeta|_{(y,0)}$ is injective. 
The condition (GC) says that the iterated derivatives 
$z \nabla^{\rm B}_{v_1} \cdots z \nabla^{\rm B}_{v_k} \zeta|_{(y,0)}$ with 
respect to local vector fields 
$v_1,\dots,v_k\in T \cMo_{\rm B}$ generate 
the fiber $\cF_{\rm B}|_{(y,0)}$. 
We claim that (IC) and (GC) holds for $\zeta = [\exp(W_y/z) \omega]$ 
and for generic $y$. 
Since the mirror map $\tau$ is an embedding and since (IC) holds 
for the A-model TEP structure, we deduce that (IC) holds for the B-model 
TEP structure at generic $y\in\cMo_{\rm B}$. 
Since the B-model TEP structure is isomorphic to a GKZ 
system~\cite[Proof of Proposition 4.8]{Iritani:integral} and since
the GKZ system is by definition cyclic, (GC) holds. 
By the universality of the unfolding~\cite[Definition 2.3]{Hertling--Manin:unfoldings},  
these local unfoldings glue together~\cite{CI:neighbourhood} 
to give a global unfolding 
$(\cF^{\rm ext}_{\rm B},\nabla^{\rm B, ext}, (\cdot,\cdot)_{\rm B, ext})$ 
of $(\cF_{\rm B},\nabla^{\rm B},(\cdot,\cdot)_{\rm B})$ 
over an extended B-model moduli space $\cM_{\rm B}^{\rm ext}$,
which is a complex manifold of dimension 
$\Vol(\Delta)$ containing a Zariski open subset of $\cMo_{\rm B}$ 
as a submanifold. 
Moreover, by universality again, the mirror map $\tau$ 
and the mirror isomorphism $\Mir$ can be extended to a neighbourhood 
$U_X^{\rm ext}$ of $U_X \cap \cMo_{\rm B}$ in $\cM_{\rm B}^{\rm ext}$,
where $X \in \Crep(\Delta)$: 
\begin{align*} 
  \tau^{\rm ext} & \colon U_X^{\rm ext} \to H^\bullet_{\CR}(X)/H^2(X;\Z) \\ 
\Mir^{\rm ext} & \colon (\cF_{\rm B}^{\rm ext}, \nabla^{\rm B,\rm ext},
(\cdot,\cdot)_{\rm B, ext} ) \Bigr|_{U_X^{\rm ext} \times \C} 
\cong 
(\tau^{\rm ext} \times \id)^* 
\left((\cF_{\rm A}, \nabla^{\rm A}, (\cdot,\cdot)_{\rm A})/H^2(X;\Z)
\right)
\end{align*} 
(More precisely, we need here the convergence of the A-model TEP structure 
over a full-dimensional base, but this follows from the reconstruction argument:
see~\cite[Lemma 2.9]{Hertling--Manin:unfoldings} and~\cite[\S 5.5]{CIR}.)  

\subsubsection{Conclusion} 
Let $\Fock_{\rm B}$ denote the Fock sheaf over 
$\cM_{\rm B}^{\rm ext}$ associated to the 
extended B-model TEP structure 
$(\cF^{\rm ext}_{\rm B},\nabla^{\rm B, ext}, (\cdot,\cdot)_{\rm B, ext})$. 
We call it the \emph{B-model Fock sheaf}. 
Via the mirror isomorphism, $\Fock_{\rm B}$ 
restricts to the A-model Fock sheaf of $X$ over $U_X^{\rm ext}$. 
The extended B-model TEP structure 
is tame semisimple (Definition~\ref{def:tamesemisimple}) 
on an open dense subset $\cM_{\rm B,\, ss}^{\rm ext}$ 
of $\cM_{\rm B}^{\rm ext}$, 
because the Jacobi ring of $W_y$ is semisimple for a generic $y
\in \cMo_{\rm B}$~\cite[Proposition~3.10]{Iritani:integral}.  
Therefore $\Fock_{\rm B}$ admits the Givental wave function 
(Definition~\ref{def:Giventalwavefcn}) over $\cM_{\rm B,\, ss}^{\rm ext}$,  
and by Theorem~\ref{thm:Teleman} (Teleman's theorem) this coincides with 
the Gromov--Witten wave function of $X$ over $U_X^{\rm ext}$. 
This proves: 
\begin{theorem}
    \label{thm:CTC_higher_genus}
There exists a global section $\wave_{\rm B}$ 
of the B-model Fock sheaf $\Fock_{\rm B}$ 
such that for every $X \in \Crep(\Delta)$, $\wave_{\rm B}$ 
restricts to the Gromov--Witten wave function of $X$ over 
the neighbourhood $U_X^{\rm ext}$ 
of the large radius limit point $o_X$ of $X$, under the identification 
$\Fock_{\rm B}|_{U_X^{\rm ext}} \cong \Fock_X$ 
given by genus-zero mirror symmetry. 
In particular, the Gromov--Witten wave functions $\wave_X$ 
associated to $X\in \Crep(\Delta)$ coincide with each other after analytic continuation.
\end{theorem} 

This is a higher-genus version of Ruan's Crepant Transformation Conjecture. 
Note that analytic continuation for sections of a Fock sheaf makes sense 
since we have the ``Identity Theorem'' for its sections, 
just as the Identity Theorem for holomorphic functions. 
Note also that the B-model Fock sheaf $\Fock_{\rm B}$ depends only on the 
cTEP structure underlying the extended B-model TEP structure 
$(\cF^{\rm ext}_{\rm B}, \nabla^{\rm B,ext}, (\cdot,\cdot)_{\rm B,ext})$: 
as discussed in Remark \ref{rem:analyticity_in_z}, 
the analytic structure of the Fock sheaf is independent of the choice of a lift of 
the cTEP structure to a TEP structure. 
On the other hand, the lift to a TEP structure constitutes a crucial 
piece of information in the \emph{genus-zero} crepant transformation conjecture. 
It also plays a role in Corollary \ref{cor:CTC_higher_genus} below.

We can rephrase Theorem \ref{thm:CTC_higher_genus} 
in terms of the $L^2$-formalism in \S \ref{subsec:global_L2}, as follows.   
Mirror symmetry implies that 
for any $X_1$, $X_2\in \Crep(\Delta)$, the A-model TEP structures 
of $X_i$ for $i=1,2$ 
are analytically continued to each other over the B-model moduli space 
$\cM_{\rm B}^{\rm ext}$. 
Recall from Example \ref{exa:Darbouxframe_fundsol} that the 
fundamental solution $L_i = L_i(\tau,z)$ 
of the Dubrovin connection of $X_i$ 
(see equation \ref{eq:fundamentalsolution}) 
defines a Darboux frame for the A-model TEP structure of $X_i$. 
The solution $L_i$ can be analytically continued along 
any path in $\cM_{\rm B}^{\rm ext}$ 
to yield a frame $L_i$ of the B-model TEP structure 
$\big(\cF_{\rm B}^{\rm ext},\nabla^{\rm B, ext}, 
(\cdot,\cdot)_{\rm B, ext}\big)$ over the universal covering 
$\tcM_{\rm B}^{\rm ext}$ of $\cM_{\rm B}^{\rm ext}$: 
\[
L_i \colon \cH^{X_i} \xrightarrow{\phantom{A}\cong\phantom{A}} 
L^2\left(\{t\} \times S^1, 
\cF_{\rm B}^{\rm ext}\big|_{\{t\}\times S^1} \right) \qquad 
\text{with} \quad 
t\in \tcM_{\rm B}^{\rm ext}  
\]
where $\cH^{X_i}$ is Givental's symplectic vector space 
(\S \ref{subsec:Givental-symplecticvs}) for $X_i$. 
The frame $L_i$ satisfies the transversality condition 
in Definition \ref{def:Darboux_frame}(2) over an  
open dense subset of $\tcM_{\rm B}^{\rm ext}$ 
and thus gives a Darboux frame there. 
Then Givental's wave function $\wave_{\rm B} 
\in \Fock_{\rm B}$ induces an element of the Fock space 
in the $L^2$-setting (Definition \ref{def:Fockspace_L2}) 
with respect to the Darboux frame $L_i$, 
by Remark \ref{rem:Fockspace_comparison}. 
As in Definition \ref{def:transformation_L2},  
the $L^2$-Fock space element here is represented by 
a total potential $\cZ_i$ that is an analytic function 
on the Givental cone associated to $L_i$. 
Note that, via the projection to $\cH^{X_i}_+$, 
$\cZ_i$ can be identified with an analytic continuation of the total descendant 
potential $\exp(\sum_{g=0}^{\infty} \hbar^{g-1} \cF^g_{X_i,\rm an})$ 
of $X_i$ by Theorem \ref{thm:ancdec-analytic} 
(see Definition \ref{def:F^g_Xan} for $\cF^g_{X_i,\rm an}$). 
Choose a path $\gamma$  from a point 
in $U_{X_1}^{\rm ext}$ to a point in $U_{X_2}^{\rm ext}$.  
Analytic continuation along the path $\gamma$ defines 
a symplectic transformation 
\[
\U_\gamma \colon \cH^{X_1} \to \cH^{X_2}  
\]
The two Darboux frames  
$L_1$ and $L_2$ are related by $L_1 = L_2 \U_\gamma$. 
Then we have 
\begin{corollary} 
\label{cor:CTC_higher_genus} 
Let $X_1$, $X_2$ be toric Deligne--Mumford stacks 
from $\Crep(\Delta)$ and let $\cZ_i$ be 
the total descendant potential for $X_i$ for $i=1,2$. 
Let $\gamma$ be a path in $\cM_{\rm B}^{\rm ext}$ 
from a point in $U_{X_1}^{\rm ext}$ 
to a point in $U_{X_2}^{\rm ext}$ 
and let $\U_\gamma$ be the symplectic transformation 
given by parallel translation along $\gamma$. 
Regarding $\cZ_i$ as an element of the $L^2$-Fock space 
as above, we have
\[
\cZ_2 \propto \hUU_\gamma \cZ_1 
\]
under analytic continuation along the path $\gamma$. 
\end{corollary} 

\begin{remark}
This is close to the version of higher-genus Crepant 
Transformation Conjecture proposed in 
\cite[\S5]{CIT:wall-crossings},~\cite[\S10]{Coates--Ruan}. 
\end{remark} 

Considering the case $X=X_1 = X_2$, we obtain: 
\begin{corollary} 
Let $X$ be a toric Deligne--Mumford 
stack from $\Crep(\Delta)$. The total descendant 
potential $\cZ_X$ of $X$, regarded as an element of the $L^2$-Fock 
space as above, has the following modular property with respect to the group 
$\pi_1(\cM_{\rm B}^{\rm ext})$: we have 
\begin{equation} 
\label{eq:quantum_modularity}
(\gamma^{-1})^\star \cZ_X \propto \hUU_\gamma \cZ_X  
\end{equation} 
for $\gamma \in \pi_1(\cM_{\rm B}^{\rm ext})$, where 
in the left-hand side $(\gamma^{-1})^\star$ means the pull-back 
by the deck transformation $\gamma^{-1}$ of the universal 
covering 
$\tcM_{\rm B}^{\rm ext} \to \cM_{\rm B}^{\rm ext}$. 
\end{corollary}

\begin{remark} 
The symplectic transformations $\U_\gamma$ with 
$\gamma \in \pi_1(\cM^{\rm ext}_{\rm B})$ arise from the 
monodromy of the extended B-model TEP structure 
$(\cF^{\rm ext}_{\rm B}, \nabla^{\rm B,ext}, (\cdot,\cdot)_{\rm B, ext})$ 
along $\gamma$, 
and as such, they belong to a finite-dimensional group. 
Here we describe such a group precisely. 
For a given TEP structure 
$(\cF=\cO(F),\nabla,(\cdot,\cdot)_\cF)$ with base $\cM$, the monodromy 
along a loop based at $(t,z)\in \cM\times \C^\times$ 
takes values in the group: 
\[
G_{t,z} = \{ U \in GL(F_{t,z}) : \Mon \circ U = U \circ \Mon, \  
\text{$U$ preserves $[\cdot,\cdot)$} \} 
\] 
where $\Mon$ denotes the monodromy of $\nabla$ over 
the punctured $z$-plane $\{t\} \times \C^\times_z$, and $[\cdot,\cdot)$ is 
a (not necessarily symmetric) bilinear form on $F_{t,z}$ defined by 
\[
[v,w) = (v',w)_\cF 
\]
with $v,w\in F_{t,z}$ and $v'$ the parallel translate of $v$ 
along the semicircular path $[0,1] \ni \theta \mapsto e^{-\iu\pi\theta} z$. 
A transformation $U\in G_{t,z}$ can be extended to a flat 
bundle automorphism of $F$ over $\{t\} \times \C^\times$, and 
thus defines an element $\U \in Sp(\cH_t)$. Here  
$\cH_t=L^2(\{t\} \times S^1, F)$ is the symplectic vector space 
appearing in \S\ref{subsec:global_L2}. 
For the A-model TEP structure, this group $G_{t,z}$ can be 
described more concretely as: 
\[
\left\{ \U \in Sp(\cH_{\rm poly}^X) : \text{$\U$ is $\C[z,z^{-1}]$-linear}, \ 
\U\circ (z\partial_z + \mu) 
= (z\partial_z +\mu) \circ \U, \ 
\U \circ c_1(X) = c_1(X) \circ \U 
\right\} 
\]
where $\cH_{\rm poly}^X = H_X\otimes_\Q \C[z,z^{-1}]$ is a Laurent polynomial 
version of Givental's symplectic space and $\mu$ is the grading 
operator in \eqref{eq:gradingoperator}. 
This follows from the following facts: (i) the Dubrovin connection $\nabla_{z\partial_z}$ 
in the $z$-direction 
is conjugate to $z\partial_z + \mu + c_1(X)/z$ via the fundamental solution 
$L(t,z)$ \eqref{eq:fundamentalsolution}, 
see, e.g.~\cite[Proposition 2.4]{Iritani:integral}, 
(ii) $\cH_{\rm poly}^X$ is the rational structure consisting of 
sections of moderate growth (see, e.g.~\cite[\S 7.2]{Hertling:ttstar}) 
at $z=0,\infty$ with respect to the connection $z\partial_z + \mu + c_1(X)/z$, 
and (iii) $2\pi \iu c_1(X)$ is the logarithm of the unipotent part of the monodromy.  
A similar description is discussed also in \cite[Lemma 3.16]{Iritani:Ruan}. 
\end{remark}

\begin{remark} 
After shrinking $\cM_{\rm B}^{\rm ext}$ if necessary, 
we have a retraction from $\cM_{\rm B}^{\rm ext}$ 
to a Zariski open subset of $\cMo_{\rm B}$, and thus  
a natural surjective map $p \colon 
\pi_1(\cM_{\rm B}^{\rm ext}) \to \pi_1(\cMo_{\rm B})$. 
The symplectic transformation $\U_\gamma$ 
for $\gamma\in \Ker(p)$ is trivial 
and the total potential $\cZ_X$ 
is invariant under deck transformations 
$\gamma \in \Ker(p)$. Therefore the modularity 
\eqref{eq:quantum_modularity} reduces to the group 
$\pi_1(\cMo_{\rm B})$. 
\end{remark} 

Finally we remark on an implication of our
recent joint work \cite{Coates--Iritani--Jiang} 
with Jiang in this context. 
There we calculated the symplectic transformation $\U_\gamma$ 
explicitly using so-called $I$-functions and the Mellin--Barnes method. 
For a certain choice of the path $\gamma$, we showed that 
$\U_\gamma$ is induced by an equivalence 
$\FM \colon D^b(X_1) \cong D^b(X_2)$ of triangulated categories 
via the $\hGamma$-integral structure 
\cite{Iritani:integral, Katzarkov--Kontsevich--Pantev} 
on quantum cohomology. 
In other words, we have a commutative diagram of the form 
\[
\xymatrix{ 
D^b(X_1) \ar[r]^{\FM} \ar[d] & D^b(X_2) \ar[d] \\ 
\tcH^{X_1} \ar[r]^{\U_\gamma} & \tcH^{X_2} 
} 
\]
where the vertical maps are roughly speaking given by the Chern 
character followed by the multiplication by the Gamma class, 
and $\tcH^{X_i}$ is a multi-valued variant of Givental's symplectic vector 
space (see \cite{Coates--Iritani--Jiang} for more details). 
The derived equivalence $\FM$ is given as a composition of explicit 
Fourier--Mukai transformations. 
It is likely that the fundamental groupoid\footnote
{Here it is convenient to consider the fundamental \emph{groupoid}  
instead of the fundamental \emph{group}, since we are considering 
paths connecting the large radius limit points of different 
$X_1,X_2\in \Crep(\Delta)$. } of $\cMo_{\rm B}$ 
is generated by $\pi_1(U_X \cap \cMo_{\rm B})$ for toric 
stacks $X$ from $\Crep(\Delta)$ together with the classes of paths 
$\gamma$ connecting the large radius limit points for $\Crep(\Delta)$, 
which we show in \cite{Coates--Iritani--Jiang} 
to correspond to derived equivalences. 
It is easy to see that monodromy about loops in $U_X\cap \cMo_{\rm B}$ 
corresponds to tensoring line bundles in $D^b(X)$ 
\cite[Proposition 2.10(ii)]{Iritani:integral}. 
Therefore the result in \cite{Coates--Iritani--Jiang} strongly 
suggests that the symplectic transformation $\U_\gamma$ is induced by 
a derived equivalence for every path $\gamma$ 
and that the total descendant potential $\cZ_X$ 
should be `modular' with respect to the group 
$\Auteq(D^b(X))$ of autoequivalences. 
Let $\Gamma$ be the group of autoequivalences of $D^b(X)$ 
generated by 
\begin{itemize} 
\item Fourier--Mukai functors $D^b(X_1) \xrightarrow{\cong} D^b(X_2)$ 
from \cite{Coates--Iritani--Jiang, Coates--Iritani--Jiang--Segal} 
for some $X_1,X_2 \in \Crep(\Delta)$; 

\item autoequivalences of $D^b(X')$ given by tensoring 
line bundles for some $X' \in \Crep(\Delta)$. 
\end{itemize} 
An element of $\Gamma$ is of the form 
\[
(L_k\otimes) \circ \Psi_{k-1} \circ (L_{k-1}\otimes) \circ 
\cdots \circ \Psi_1\circ (L_1\otimes) \circ \Psi_0 \circ (L_0 \otimes) 
\]
where 
$\Psi_i \colon D^b(X_i) \to D^b(X_{i+1})$, $i=0,\dots,k-1$ 
are Fourier--Mukai functors 
from \cite{Coates--Iritani--Jiang,Coates--Iritani--Jiang--Segal}, 
$L_i\in \Pic(X_i)$, $i=0,\dots,k$ are line bundles, and 
$X=X_0, X_1,\dots, X_{k-1}, X_k=X$ is a sequence in $\Crep(\Delta)$. 
Then we have the following. 
\begin{corollary}
\label{cor:modularity_autoeq} 
The total descendant potential $\cZ_X$ of a toric Deligne--Mumford 
stack $X\in \Crep(\Delta)$ satisfies the modularity 
\eqref{eq:quantum_modularity} with respect to 
the subgroup $\Gamma$ of $\Auteq(D^b(X))$. 
\end{corollary} 

\begin{remark} 
A relationship between Fourier--Mukai transformations and 
analytic continuation of solutions to the GKZ system was originally 
found by Borisov--Horja~\cite{Borisov--Horja:FM}. 
\end{remark} 

\begin{remark} 
The monodromy representation gives a homomorphism 
$\U \colon \pi_1(\cMo_{\rm B}) \to Sp(\cH^X)$ and 
the $\hGamma$-integral structure gives a homomorphism 
$\Auteq(D^b(X)) \to Sp(\cH^X)$. Homological mirror 
symmetry suggests that the former map factors through 
the latter, i.e.~we expect to have the commutative diagram: 
\[
\xymatrix{
\pi_1(\cMo_{\rm B}) \ar[r]^{\U} \ar[d] & Sp(\cH^X) \\ 
\Auteq(D^b(X)) \ar[ur] 
}
\] 
\end{remark} 

\begin{remark} 
The B-model TEP structure $\cF_{\rm B}$ over $\cM_{\rm B}^\circ$ 
has a natural real structure induced from the $\Z$-structure 
$\cR_\Z^\vee$. (More precisely, the real structure is obtained by tensoring 
the $\nabla^{\rm B}$-flat local system $z^{D/2} \cR_\Z^\vee$ 
with $\R$.)
By a result of Sabbah \cite{Sabbah:FL}, 
this real structure endows $\cF_{\rm B}$ with a pure TRP structure 
in the sense of \S \ref{subsec:TRP}. 
Since the purity is an open property, this extends to a pure 
real structure on $\cF_{\rm B}^{\rm ext}$ over a small 
neighbourhood of $\cM_{\rm B}^\circ$ in $\cM_{\rm B}^{\rm ext}$. 
Using the complex-conjugate opposite module in Definition \ref{def:cc-opposite}, 
we can present the Givental wave function $\wave_{\rm B}$ 
as \emph{single-valued} (non-holomorphic) correlation functions. 
\end{remark} 

We end this section with a discussion on singularities of the 
Givental wave function. 
Around each large radius limit point, the B-model TEP structure is identified 
with the A-model TEP structure and thus extends across a normal crossing divisor 
as a logarithmic TEP structure (see Example \ref{ex:log_cTEP_A}). 
This extension was studied in details by Reichelt--Sevenheck 
\cite{Reichelt--Sevenheck}. 
The Givental wave function $\wave_{\rm B}$ extends regularly 
across these normal crossing divisors as a section of the logarithmic 
Fock sheaf in \S \ref{subsec:logarithmic} 
since so does the Gromov--Witten wave function. 
On the other hand, a result of Milanov \cite{Milanov:analyticity} 
must imply that $\wave_{\rm B}$ extends regularly  
across the non-semisimple locus 
$\cM_{\rm B}^{\rm ext} \setminus \cM_{\rm B,\, ss}^{\rm ext}$. 
The remaining important question is then:

\begin{problem} 
Study the singularities of the Givental wave function $\wave_{\rm B}$ 
along the locus $\cM_{\rm B}\setminus \cM_{\rm B}^\circ$ 
of degenerate Laurent polynomials.  
\end{problem}
\noindent 
This problem is related to the \emph{conifold gap condition} 
(see e.g.~\cite{ASYZ}) 
in the physics literature. 

\subsection{Calabi--Yau Hypersurfaces}
\label{sec:CY}

Next we consider mirror symmetry 
for Calabi--Yau manifolds. 
In this case we cannot apply Givental's formula since 
the quantum cohomology is not semisimple. 
We consider Batyrev's mirror~\cite{Batyrev:dual} 
for toric Calabi--Yau hypersurfaces. 

Let $X$ be a weak Fano toric stack such that the fan polytope 
$\Delta$ is reflexive (i.e.~the integral distance between 
each facet of $\Delta$ and the origin is one). 
Then $X$ is Gorenstein and a generic anticanonical 
section $Y\subset X$ is a quasi-smooth Calabi--Yau orbifold~\cite{Batyrev:dual}. 
Let $W_y$ be the Laurent polynomials mirror to $X$ 
from \S\ref{sec:LG_model}. 
The Batyrev mirror of $Y$ is a Calabi--Yau compactification 
$\check{Y}_y$ of the fiber $W_y^{-1}(1)$ 
inside a toric variety $\check{X}$ with fan polytope given by the dual polytope $\Delta^*$. 
To remove the ambiguity of overall scaling, 
we consider Laurent polynomials $W_y$ as in \eqref{eq:LGmodel-W} 
with vanishing constant terms, so that $a_{i}=0$ when $b_{i}=0$. 
The corresponding moduli space $\cM_B'$ of Laurent polynomials 
is defined similarly to $\cM_B$ but deleting the zero-vector 
from the set $\{b_1,\dots,b_n\}$; it can be identified 
with (an open subset of) a toric divisor in $\ov{\cM}_{\rm B}$. 
Moreover, we require that the affine hypersurface 
$W_y^{-1}(1)$ is \emph{$\Delta$-regular}~\cite[Definition 3.3]{Batyrev:VMHS}, 
which means that $W_y$ is non-degenerate (see \S\ref{sec:B-model_TEP}) 
and $1$ is not a critical value of $W_y$. 
Let $\cM_{\rm B}^{\rm reg} \subset \cM_{\rm B}'$ 
denote the non-empty Zariski-open subset parametrizing 
$\Delta$-regular hypersurfaces $W_y^{-1}(1)$.  
We use $\cM_{\rm B}^{\rm reg}$ as the base space 
of the mirror family $\{\check{Y}_y\}$. 
Note that, as in the toric case (\S\ref{sec:toric_mirror_symmetry}),
all anticanonical hypersurfaces $Y$ in toric stacks 
$X$ from $\Crep(\Delta)$ 
have the same mirror family $\{\check{Y}_y\}$.
However, they have different large radius limit points 
in the toric compactification $\ov{\cM_{\rm B}'}$.

We describe the genus-zero mirror isomorphism 
following~\cite[\S 6]{Iritani:periods}, and suggest 
the construction of a global B-model Fock sheaf. 
Define the ambient part of the cohomology group of $Y$ to be the image 
of the pullback along the inclusion map: 
$H_{\rm amb}^*(Y) :=\Image(H_{\CR}^*(X) \to H^\bullet_{\CR}(Y))$. 
The Dubrovin connection of $Y$ preserves the subsheaf 
$\cF_{\rm A}^{\rm amb} := H_{\rm amb}^*(Y) 
\otimes \cO_{\cM_{\rm A} \times \C}$ 
with fiber $H_{\rm amb}^*(Y)$~\cite[Corollary 2.5]{Iritani:periods}, 
where $\cM_{\rm A} \subset H^\bullet_{\rm amb}(Y)$ 
denotes the convergence domain 
of the quantum product as in \eqref{eq:LRLnbhd}. 
Hence by restriction to this subsheaf, the A-model TEP structure of $Y$ 
induces a TEP structure called the \emph{ambient A-model TEP structure} 
$(\cF_{\rm A}^{\rm amb}, \nabla^{\rm A}, (\cdot,\cdot)_{\rm A})$ 
of $Y$ (cf.~\cite[Definition 6.2]{Iritani:periods}).  
On the mirror side, we consider the lowest weight piece 
$W_{D-1}(H^{D-1}(W_y^{-1}(1))) = 
\gr^W_{D-1} H^{D-1}(W_y^{-1}(1))$ 
of Deligne's mixed Hodge structure on the middle 
cohomology of the affine hypersurface $W^{-1}_y(1)$. 
It has a pure Hodge structure 
of weight $D-1$.  
As explained in~\cite[\S 6.3]{Iritani:periods}, 
this can be naturally identified with the subspace 
$H_{\rm res}^{D-1}(\check{Y}_y)
\subset H^{D-1}(\check{Y}_y)$ of cohomogy classes 
obtained as the residues of meromorphic $D$-forms (with poles 
along $\check{Y}_y$) on the ambient toric variety $\check{X}$. 
It defines the \emph{residual B-model VHS}~\cite[Definition 6.5]{Iritani:periods} 
$(\cV, \nabla^{\rm GM}, F^\bullet \cV, Q)$  
over $\cM_{\rm B}^{\rm reg}$, where: 
\begin{itemize} 
\item $\cV$ is a locally free sheaf over $\cM_{\rm B}^{\rm reg}$ 
with fiber $\cV_{y} = H_{\rm res}^{D-1}(\check{Y}_y) 
\cong \gr^W_{D-1}(H^{D-1}(W_y^{-1}(1)))$;   
\item $\nabla^{\rm GM}$ is the Gauss-Manin connection; 
\item $0\subset F^{D-1}\cV \subset \cdots \subset F^1\cV \subset F^0\cV 
= \cV$ is the Hodge filtration on $\cV$ of weight $D-1$; 
\item $Q(\alpha, \beta) = (-1)^{(D-1)(D-2)/2} 
\int_{\check{Y}_y} \alpha \cup \beta$ is the intersection form. 
\end{itemize} 
The \emph{residual B-model TEP structure} 
$(\cF_{\rm B}^{\rm res}, \nabla^{\rm B}, (\cdot,\cdot)_{\rm B})$ 
over $\cM_{\rm B}^{\rm reg}$ is defined as 
follows: 
\begin{itemize} 
\item $\cF_{\rm B}^{\rm res}$ is an algebraic locally free 
sheaf over $\cM_{\rm B}^{\rm reg}\times \C$,
given by the subsheaf of $\pi^* \cV$ with the property  
\[
\pi_* \cF_{\rm B}^{\rm res}  = z^{D-1}F^0\cV [z]+ 
z^{D-2} F^1\cV[z] + \cdots + F^{D-1}\cV[z] 
\subset \cV[z] = \pi_* \pi^* \cV 
\] 
where $\pi\colon \cM_{\rm B}^{\rm reg} 
\times \C \to \cM_{\rm B}^{\rm reg}$ 
is the projection; 
\item $\nabla^{\rm B} = \pi^* \nabla^{\rm GM} 
- \frac{D-1}{2} \frac{dz}{z}$;
\item $(\alpha(-z),\beta(z))_{\rm B} = (2\pi\iu z)^{-(D-1)}  
Q(\alpha(-z), \beta(z))$.    
\end{itemize} 
As before, each toric stack $X\in \Crep(\Delta)$ defines 
a large radius limit $o_X$ in the toric compactification 
$\ov{\cM'_{\rm B}}$ of $\cM_{\rm B}'$. 
For a neighbourhood $U_X$ of $o_X$ in $\ov{\cM'_{\rm B}}$, 
we have a mirror map $\varsigma \colon U_X \cap \cM_{\rm B}^{\rm reg} 
\to H^2_{\rm amb}(Y)/i^*H^2(X;\Z)$ 
and a mirror isomorphism\footnote{In~\cite[Theorem 6.9]{Iritani:periods}, 
the mirror isomorphism was stated for the corresponding VHSs, 
but the statement here follows easily from it.}~\cite[Theorem 6.9]{Iritani:periods}: 
\[
\Mir \colon (\cF_{\rm B}^{\rm res}, \nabla^{\rm B}, (\cdot,\cdot)_{\rm B}) 
\Bigr|_{(U_X\cap \cM_{\rm B}^{\rm reg}) \times \C}
\cong (\varsigma\times \id)^* 
\left( (\cF_{\rm A}^{\rm amb}, \nabla^{\rm A}, 
(\cdot,\cdot)_{\rm A})/i^* H^2(X;\Z) \right) 
\]
where the right-hand side is the quotient by the 
Galois action from the ambient $H^2(X;\Z)$. 

As in the previous section, we use Hertling--Manin's reconstruction theorem~\cite{Hertling--Manin:unfoldings} to unfold $\cF^{\rm res}_{\rm B}$ 
to get a miniversal TEP structure.  
This is possible when:
\begin{itemize} 
\item There exists a toric stack $X\in \Crep(\Delta)$ whose 
anticanonical hypersurface $Y$ is a smooth variety (no orbifold singularities); 
in this case the ambient cohomology $H_{\rm amb}^*(Y)$ is 
generated by $H_{\rm amb}^2(Y)$ and thus the 
ambient quantum cohomology is also generated by $H^2_{\rm amb}$ 
in a neighbourhood of the large radius limit; by the mirror 
isomorphism the generation condition (GC) holds generically 
over $\cM_{\rm B}^{\rm reg}$; 
\item There exists a toric stack $X \in \Crep(\Delta)$ 
such that the map $H^2_{\rm CR}(X) \to H^2_{\rm amb}(Y)$ 
is an isomorphism; 
in this case the mirror map $\varsigma$ is locally injective and the 
injectivity condition (IC) holds generically 
over $\cM_{\rm B}^{\rm reg}$. 
\end{itemize} 
For example, these conditions hold for a hypersurface 
$Y$ in $X= \Proj^n$ (see~\cite[Theorem 8.1]{Hertling--Manin:unfoldings}). 
Under these assumptions, we have a miniversal 
unfolding of $\cF_{\rm B}^{\rm res}$ and obtain the 
corresponding B-model Fock sheaf over a complex 
manifold of dimension $\dim H_{\rm amb}^*(Y)$ 
(which contains a Zariski open subset of $\cM_{\rm B}^{\rm reg}$).  
We conjecture the following: 
\begin{conjecture} 
\label{conj:CY}
There exists a global section of the above B-model Fock sheaf 
which restricts to 
the Gromov--Witten wave function of each Calabi--Yau 
hypersurface $Y$ in $X \in \Crep(\Delta)$ over a 
neighbourhood of the large radius limit $o_X$. 
\end{conjecture} 

\begin{remark} 
The existence of a global section of the B-model Fock sheaf 
will be shown in forthcoming work by Costello--Li 
(see \cite{Costello--Li}) where they develop the mathematical 
B-model theory at higher genus. 
\end{remark} 

Sometimes we may encounter different types of 
limit points of the B-model moduli space 
$\cM_{\rm B}'$ which correspond to 
variants of Gromov--Witten theory. 
For the mirror of a quintic threefold, Chiodo--Ruan~\cite{Chiodo--Ruan} 
found that the global B-model theory (at genus zero) around 
the so-called Landau--Ginzburg point (or Gepner point) 
can be identified with a theory of $5$-spin curves (FJRW theory);
see also~\cite{CIR}. 
In particular the genus-zero Gromov--Witten theory 
and the genus-zero FJRW theory are analytic continuations
of each other. 
Chiodo--Ruan also conjecture a relationship between the higher genus theories~\cite[Conjecture 3.2.1]{Chiodo--Ruan}. 
Their conjecture, rephrased in our language, is that 
the global section in Conjecture~\ref{conj:CY} restricts to the 
FJRW wave function in a neighbourhood of the Landau--Ginzburg point.

\section{Complex-Conjugate Polarization and Holomorphic 
Anomaly} 
\label{sec:HAE}

In this section we describe how the holomorphic anomaly equation of Bershadsky--Cecotti--Ooguri--Vafa arises in our global quantization formalism, via the so-called complex-conjugate polarization.   
The holomorphic anomaly equation originally arose
in the Kodaira--Spencer theory of gravity~\cite{BCOV:HA, BCOV:KS}. 
It can be considered as a special case of the general 
anomaly equation (Theorem~\ref{thm:anomaly}), 
but strictly speaking we need to extend the holomorphic 
structure sheaf on the base $\cM$ to the real-analytic structure sheaf. 
At genus zero, the complex conjugate polarization 
gives rise to so-called $tt^*$-geometry. 

\subsection{TRP Structure and $tt^*$-Geometry} 
\label{subsec:TRP} 
In this section we work with a TRP structure 
which is a TP structure (Definition~\ref{def:TP}) 
equipped with a certain real structure.  
As usual $\cM$ denotes a complex manifold, 
$(-) \colon \cM \times \C \to \cM\times \C$ 
denotes the map sending $(t,z)$ to $(t,-z)$,
and $\pi\colon \cM \times \C \to \cM$ is the 
projection.  Let
$\gamma \colon \Proj^1\to\Proj^1$,  
$\gamma(z) = 1/\ov{z}$, denote the anti-holomorphic involution  which fixes the equator 
$S^1=\{|z|=1\}\subset \Proj^1$. 
The involution $(t,z) \mapsto (t,\gamma(z))$ on $\cM\times \Proj^1$ 
is also denoted by $\gamma$. 
For a holomorphic vector bundle $F$ over $\cM\times \C$, 
the vector bundle $\ov{\gamma^* F}$ over $\cM \times (\Proj^1\setminus \{0\})$ 
has a holomorphic structure in the $\Proj^1$-direction 
and an anti-holomorphic structure in the $\cM$-direction. 

\begin{definition}[TRP structure] 
A \emph{TRP structure} $(\cF = \cO(F), \nabla, (\cdot,\cdot)_{\cF}, \kappa)$ 
with base $\cM$ 
consists of a holomorphic vector bundle $F$ of rank $N+1$ 
over $\cM\times \C$ with the sheaf $\cF = \cO(F)$ 
of holomorphic sections, a meromorphic flat connection 
\[
\nabla \colon \cF \to \pi^* \Omega^1_{\cM} \otimes 
\cF(\cM\times \{0\})  
\]
a non-degenerate pairing 
\[
(\cdot,\cdot)_{\cF} \colon (-)^*\cF\otimes
\cF \to \cO_{\cM\times \C} 
\]
that fiberwise defines a $\C$-bilinear pairing $F_{(t,-z)}\otimes 
F_{(t,z)} \to \C$, and a real-analytic bundle map 
\[
\kappa \colon F|_{\cM \times \C^\times} \to 
\ov{\gamma^* F}|_{\cM \times \C^\times}  
\]
that fiberwise defines a $\C$-antilinear map 
$\kappa_{t,z} \colon F_{(t,z)} \to F_{(t,\gamma(z))}$, 
such that:
\begin{itemize} 
\item $(\cF,\nabla, (\cdot,\cdot)_{\cF})$ is a TP structure 
in the sense of Definition~\ref{def:TP}; 
\item $\kappa$ is an involution: $\kappa_{t,\gamma(z)} \circ \kappa_{t,z} = \id$;  
\item when restricted to $\{t\}\times \C^\times$ with $t\in \cM$, 
$\kappa$ yields an isomorphism of holomorphic vector bundles; 
in particular, we have an involution:  
\begin{equation} 
\label{eq:kappa-H0}
\kappa_t \colon H^0(\C^\times, \cO(F_t)) 
\longrightarrow  H^0(\C^\times,  \cO(\ov{\gamma^*F_t})) 
\cong H^0(\C^\times, \cO(F_t)) 
\end{equation}
where $F_t := F|_{\{t\} \times \C}$;  
\item the pairing $(\cdot,\cdot)_{\cF}$ 
is real with respect to $\kappa$, i.e.~the following diagram 
commutes: 
\begin{equation} 
\label{eq:reality-pairing} 
\begin{aligned}
  \xymatrix{
    F_{(t,-z)} \otimes F_{(t,z)} \ar[r]^-{(\cdot,\cdot)_{\cF}} \ar[d]_{\kappa_{t,-z} \otimes \kappa_{t,z}} & \C \ar[d]^{\text{complex conjugation}} \\ 
    F_{(t, -\gamma(z))} \otimes F_{(t,\gamma(z))} \ar[r]^-{(\cdot,\cdot)_{\cF}} & \C
  }
\end{aligned}
\end{equation} 
\item parallel translation by the connection $\nabla$ 
preserves $\kappa$. 
\end{itemize}  
Note that $\kappa$ defines a real involution of the bundle 
$F|_{\cM\times S^1}$, where $S^1 = \{|z| =1\}$. 
The corresponding real subbundle $F_\R = \Ker(\kappa -\id)$ 
of $F|_{\cM \times S^1}$ is equipped with a real-valued pairing 
$F_{\R,(t,-z)} \otimes_\R F_{\R, (t,z)} \to \R$ (with $z\in S^1$) 
and is flat in the $\cM$-direction. 
\end{definition} 

\begin{remark} 
A TRP structure is a TERP structure in the sense of Hertling~\cite{Hertling:ttstar} 
without ``E" i.e.~without an extension of the connection in the $z$-direction. 
``R" stands for the real structure. 
It is easy to see that a TERP(0) structure gives rise to a TRP structure 
by forgetting the connection in the $z$-direction. 
A major portion of this section \S\ref{subsec:TRP} is an adaptation 
of the framework of~\cite{Hertling:ttstar} to our setting. 
\end{remark} 

\begin{example} 
\label{exa:TERP} 
$tt^*$-geometry was discovered by 
Cecotti--Vafa~\cite{Cecotti--Vafa:top-anti-top,Cecotti--Vafa:classification} 
in their study of $N=2$ supersymmetric quantum field theory. 
There are natural TRP (or TERP) structures 
coming from geometry: the A-model and B-model. 
A TERP structure in singularity theory (the B-model) was introduced by 
Hertling~\cite{Hertling:ttstar} 
using a natural real structure on the Gauss--Manin system. 
A TERP structure in quantum cohomology (the A-model) was introduced by Iritani~\cite{Iritani:ttstar} 
using the $\hGamma$-class and the $K$-group of vector bundles. 
The real structure on quantum cohomology is different from the usual one  
coming from $H^\bullet(X,\R) \subset H^\bullet(X,\C)$.  
\end{example}

\begin{remark} 
A TRP structure is determined by the holomorphic vector bundle 
$F$ restricted to $\cM \times \{|z|\le 1\}$, 
the connection $\nabla$, the pairing $(\cdot,\cdot)_{\cF}$,
and the real subbundle $F_\R$ of $F|_{\cM\times S^1}$. 
It is given by gluing $F|_{\{|z|\le 1\}}$ 
and $\ov{\gamma^* (F|_{\{|z|\le 1\}})}$ along the circle 
via the real involution with respect to $F_\R$.  
\end{remark} 

\begin{definition}[glued bundle $\hF$] 
From a TRP structure $(\cF=\cO(F),\nabla,(\cdot,\cdot)_{\cF}, \kappa)$, one can construct 
a real-analytic complex vector bundle $\hF$ over 
$\cM\times \Proj^1$ by gluing $F$ with $\ov{\gamma^* F}$ via $\kappa$. 
The bundle $\hF$ has a fiberwise holomorphic structure 
with respect to $\pi \colon \cM \times \Proj^1 \to \cM$. 
Let $\cA_{\rm vh}(\hF)$ denote the sheaf of real-analytic 
sections of $\hF$ which are holomorphic along each fiber 
$\{t\}\times \Proj^1$ (``vh" stands for ``vertically holomorphic''). 
Let $\cA_\cM^p$ denote the sheaf of real analytic 
$p$-forms on $\cM$, and $\cA_{\cM}^p = \bigoplus_{i+j=p} \cA_{\cM}^{i,j}$ 
denote the type decomposition.   
The connection $\nabla$ on $\cF = \cO(F)$ can be extended to 
a connection $\nabla$ on $\cA_{\rm vh}(\hF)$: 
\begin{equation} 
\label{eq:nabla-hF} 
\nabla \colon \cA_{\rm vh}(\hF) \to 
\pi^*\cA^{1,0}_\cM \otimes \cA_{\rm vh}(\hF)(\cM \times \{0\}) 
\oplus 
\pi^*\cA^{0,1}_\cM \otimes \cA_{\rm vh}(\hF)(\cM \times \{\infty\})   
\end{equation} 
such that the $(0,1)$-part coincides with the $\ov\partial$-operator 
for the holomorphic bundle $F$. 
The $(1,0)$-part defines the anti-holomorphic structure of $\ov{\gamma^*F}$ 
in the $\cM$-direction. 
Since the gluing map $\kappa$ matches the pairing $(\cdot,\cdot)_{\cF}$ 
on $F$ with $\ov{\gamma^*(\cdot,\cdot)_{\cF}}$ on $\ov{\gamma^*F}$, 
there is a non-degenerate pairing 
\begin{equation} 
\label{eq:pairing-hF}
(\cdot,\cdot)_{\hF} \colon (-)^*\cA_{\rm vh}(\hF) 
\otimes_{\cA_{\cM\times \Proj^1, \rm vh}} 
\cA_{\rm vh}(\hF) \to \cA_{\cM\times \Proj^1, \rm vh} 
\end{equation} 
extending the pairing $(\cdot,\cdot)_{\cF}$ on $\cF$, 
where $\cA_{\cM \times \Proj^1, \rm vh}$ denotes the sheaf 
of real-analytic functions on $\cM\times \Proj^1$ 
which are holomorphic in the $\Proj^1$-direction. 
The pairing $(\cdot,\cdot)_{\hF}$ is $\nabla$-flat,
as in Definition~\ref{def:TP}. 
\end{definition} 

\begin{definition}[pure TRP structure]  
\label{def:pure_TRP}
A TRP structure is said to be \emph{pure} if the bundle $\hF|_{\{t\} \times \Proj^1}$ 
is trivial as a holomorphic vector bundle for every $t\in \cM$. 
(A pure TRP structure corresponds to a trTERP structure of Hertling~\cite{Hertling:ttstar} 
without ``E".) 
\end{definition} 

The involution $\kappa_t$ (see equation~\ref{eq:kappa-H0}) 
acting on the space $H^0(\C^\times, \cO(F_t))$ 
is invariant under parallel translation and 
satisfies 
\begin{align}
\label{eq:kappa_t-property} 
\kappa_t(f v ) = \ov{\gamma^*f} \cdot \kappa_t(v) &&
(\kappa_t(v), \kappa_t(w))_\cF = \ov{\gamma^* ((v,w)_\cF)}  
\end{align} 
for $f \in \cO(\C^\times)$ and $v$,~$w 
\in H^0(\C^\times, \cO(F_t))$. 
Conversely, a TRP structure is given by a TP structure 
and a translation-invariant family of involutions 
$\kappa_t$ of $H^0(\C^\times, \cO(F_t))$ 
with these properties. 
Set $\hF_t := \hF|_{\{t\}\times \Proj^1}$. 
The subspace $H^0(\C,\cO(F_t))\subset H^0(\C^\times,\cO(F_t))$ 
consists of holomorphic sections of $\hF_t$ over $\C^\times$ 
which extend to $z=0$. 
Likewise, the subspace $\kappa_t(H^0(\C, \cO(F_t))) 
\subset H^0(\C^\times, \cO(F_t))$ 
consists of holomorphic sections of $\hF_t$ over $\C^\times$ 
which extend to $z=\infty$. 
Hence \emph{a TRP structure is pure if and only if 
\begin{equation} 
\label{eq:pure-trans}
H^0(\C^\times, \cO(F_t))  = H^0(\C,\cO(F_t)) \oplus z^{-1} 
\kappa_t(H^0(\C, \cO(F_t))) 
\end{equation} 
for each $t\in \cM$}. 
One can therefore view $z^{-1} \kappa_t(H^0(\C,\cO(F_t)))$ 
as defining an opposite module for a pure TRP structure. 
It is, however, not parallel in the anti-holomorphic direction. 
\begin{remark} 
By identifying $H^0(\C^\times, \cO(F_t))$ 
with a fixed $H^0(\C^\times,\cO(F_{t_0}))$ 
by parallel translation, locally on $\cM$, a pure TRP structure is 
given by a real structure on a single infinite-dimensional symplectic 
vector space $H^0(\C^\times, \cO(F_{t_0}))$ 
such that $H^0(\C, \cO(F_t))$ and 
its complex conjugate multiplied by $z^{-1}$ 
are opposite \eqref{eq:pure-trans};
see~\cite{Iritani:ttstar} and \S\ref{subsec:Kaehler} below. 
\end{remark} 

\begin{definition}[complex conjugate opposite module] 
\label{def:cc-opposite}
Let $(\sfF,\nabla,(\cdot,\cdot)_{\sfF})$ 
denote the cTP structure associated to 
a pure TRP structure $(\cF, \nabla, (\cdot,\cdot)_{\cF}, \kappa)$, 
that is, $(\sfF,\nabla, (\cdot,\cdot)_{\sfF})$ 
is the restriction of $(\cF,\nabla,(\cdot,\cdot)_{\cF})$ 
to the formal neighbourhood of $\cM\times \{0\}$ in $\cM\times \C$. 
Let $\cA_\cM$ denote the sheaf of real-analytic functions on $\cM$ 
and set: 
\[
\AF:= \sfF \otimes_{\cO_{\cM}[\![z]\!]} \cA_{\cM}[\![z]\!]
\]
We write $\cA_{\rm vh}(\hF)(*(\cM\times \{0\}))$ for the sheaf of 
real analytic sections of $\hF$ which are meromorphic along 
each fiber with poles only along $z=0$. 
The $\cA_{\cM}[z^{-1}]$-module $\ov\AF$ 
is defined to be the push-forward of this sheaf along
$\pi \colon \cM\times \Proj^1 \to \cM$: 
\[
\ov\AF := \pi_* \left( \cA_{\rm vh}(\hF)(* (\cM\times \{0\}))\right)
\]
The purity of the TRP structure implies that
\begin{equation} 
\label{eq:opposition-cc}
\AF[z^{-1}] = \AF \oplus z^{-1} \ov\AF 
\end{equation} 
 (cf.~equation~\ref{eq:pure-trans}).  We call $z^{-1} \ov\AF$ the \emph{complex conjugate 
opposite module} or \emph{complex conjugate polarization}. 
\end{definition} 
\begin{remark} 
Note that the involution $\kappa_t$ on $H^0(\C^\times, \cO(F_t))$ 
is ill-defined on the formal version $\AF[z^{-1}]$; nonetheless 
$\ov\AF$ can be regarded as the complex conjugate of $\AF$. 
\end{remark} 

Restricting the connection \eqref{eq:nabla-hF} 
to the formal neighbourhood of $z=0$, we obtain:
\begin{align}
\label{eq:nabla-AF}
\begin{split}  
&\nabla \colon \AF[z^{-1}] \to \cA^1_\cM \otimes \AF[z^{-1}] \\ 
&\nabla \colon \AF \to \cA^{1,0}_{\cM} \otimes 
(z^{-1} \AF) \oplus \cA^{0,1}_\cM \otimes \AF \\ 
& \nabla \colon \ov\AF \to \cA^{1,0}_\cM \otimes \ov\AF 
\oplus \cA^{0,1}_\cM \otimes (z \ov\AF)  
\end{split} 
\end{align} 
The third equation means that the complex conjugate opposite module is parallel 
in the holomorphic direction, but not in the anti-holomorphic 
direction. 
Let 
\[
\Omega\colon 
\AF[z^{-1}]\otimes_{\cA_\cM}\AF[z^{-1}] 
\to \cA_\cM
\] 
denote the symplectic pairing, defined as in \eqref{eq:symplecticform}. 
Suppose that the TRP structure is pure. 
Because the pairing \eqref{eq:pairing-hF} is 
necessarily constant with respect to a holomorphic frame of 
$H^0(\Proj^1, \cO(\hF_t))$ over $\Proj^1$, 
the pairing of two elements from $z^{-1} \ov{\AF}_t$ 
has vanishing residue at $z=0$.  Thus:
\begin{equation} 
\label{eq:isotropic-cc}
\Omega(z^{-1} \ov\AF, z^{-1} \ov\AF) = 0
\end{equation} 
To summarize, we have (cf.~Definition~\ref{def:opposite}):
\begin{proposition} 
The complex conjugate opposite module $z^{-1}\ov\AF$ associated to a pure TRP 
structure is $z^{-1}$-linear, opposite to $\AF$ \eqref{eq:opposition-cc}, and
isotropic for $\Omega$ \eqref{eq:isotropic-cc}.  It is parallel in the holomorphic direction, 
but not necessarily in the anti-holomorphic direction. 
\end{proposition} 

\begin{definition}[$tt^*$-bundle] 
\label{def:ttstarbundle}
A pure TRP structure over $\cM$ defines a real-analytic complex 
vector bundle $K$ of rank $N+1 = \rank \cF$ over $\cM$ such that 
the sheaf $\cA(K)$ of real-analytic sections is given by:
\[
\cA(K) := \pi_* \cA_{\rm vh}(\hF) \cong \AF \cap \ov\AF  
\]
This bundle is equipped with: 
\begin{itemize} 
\item a complex anti-linear involution 
$\kappa \colon K \to K$ induced by $\kappa$ 
(cf.~equation~\ref{eq:kappa-H0}); 
\item a 
Hermitian metric $h(u,v) := ((-)^*\kappa(u), v)_{\hF}$, which may not be positive-definite, 
induced from the pairing \eqref{eq:pairing-hF} and $\kappa$ 
(the Hermitian metric here is complex anti-linear in the first 
variable); 
\item a one parametric family of flat connections 
\[
\nabla^{(z)}  = D - \frac{1}{z} \cC - z \tcC 
\]
induced by the connection $\nabla$ \eqref{eq:nabla-hF} on 
$\cA_{\rm vh}(\hF)$, where $D \colon \cA(K)\to \cA^1_\cM \otimes\cA(K)$ 
is a connection on $K$, $\cC\in \End(K)\otimes \cA^{1,0}_\cM$, and
$\tcC\in \End(K)\otimes \cA^{0,1}_\cM$; 
\end{itemize} 
such that\footnote
{Here $\kappa$ acts on forms by ordinary complex conjugation 
and $h$ is extended to $K$-valued one forms sesqui-linearly. 
In holomorphic co-ordinates $\{t^i\}$ on $\cM$, 
we have $D_\oi  = \kappa \circ D_i \circ \kappa$, 
$\tcC_\oi = \kappa \circ \cC_i \circ \kappa$, 
$h(\tcC_\oi u, v) = h(u, \cC_i v)$. }:  
\begin{itemize} 
\item 
the Hermitian metric $h$ is real with respect to $\kappa$: $h(\kappa(u),\kappa(v)) = \ov{h(u,v)}$; 
\item $D$ is real with respect to $\kappa$: $D = 
\kappa \circ D \circ \kappa$; 
\item  $\tcC = \kappa \circ \cC \circ \kappa$; 
\item $D$ respects the Hermitian metric $h$; 
\item $h(\tcC u, v) = h(u, \cC v)$.  
\end{itemize} 
The $(0,1)$-part of $\nabla^{(z)}$ defines a holomorphic structure 
on $K$ depending on $z$ which corresponds to the holomorphic structure 
on $F|_{\cM\times \{z\}}$; 
in particular the holomorphic structure for $z=0$ is defined by $D'' = D^{(0,1)}$ 
and coincides with that of $F_0 = F|_{\cM\times \{0\}}$. 
Therefore $D$ can be identified with the Chern connection 
for the holomorphic Hermitian bundle $(F_0,h)$ under the natural identification 
$F_0 \cong K$. 
The flatness of $\nabla^{(z)}$ implies, in terms of holomorphic local 
co-ordinates $\{t^i\}$ on $\cM$, 
\begin{align*}
& D_{\oi} \cC_j =0, \quad D_i \tcC_\oj =0, \quad [\cC_i, \cC_j] =0, 
\quad [\tcC_\oi, \tcC_\oj]=0, \quad 
[D_i,D_j] = 0, \quad [D_\oi, D_\oj] =0, 
\\ 
& D_i \cC_j -D_j \cC_i =0, \quad D_\oi \tcC_\oj - \tcC_\oj D_\oi  =0, \quad   
[D_i, D_\oj] + [\cC_i,\tcC_{\oj}] = 0.  
\end{align*}
These are the \emph{$tt^*$-equations}. 
In particular, $\cC$ gives a holomorphic section of 
$\End(F_0)\otimes \Omega^1_\cM$ 
(via the identification $K\cong F_0$). 
\end{definition} 

\begin{definition}[pure and polarized TRP structure] 
\label{def:polarized_TRP} 
A pure TRP structure is said to be \emph{polarized} if 
the Hermitian metric $h$ on $K$ above is positive definite. 
\end{definition} 

\begin{example} 
Many TERP structures coming from geometry (see Example~\ref{exa:TERP}) 
are pure and polarized. 
For the B-model, 
Sabbah~\cite{Sabbah:FL} showed that the TERP structure 
associated to a tame function on an affine variety is pure and polarized.  
For the A-model, under the identification 
between $H^0(\C^\times, \cO(F_t))$ and the Givental space $\cH$ given by the fundamental solution $L$ 
(see \S\ref{subsec:Lagrangian_TP}), we have 
\[
z^{-1} \kappa_t (H^0(\C, \cO(F_t) )) \longrightarrow \cH_- 
\]
when $t$ approaches the large radius limit. 
This implies that the A-model TERP structure is 
pure in a neighbourhood of the large radius limit point~\cite{Iritani:ttstar}. 
When we restrict ourselves to the algebraic part $\bigoplus_{p=0}^{\dim X} 
H^{p,p}(X)$ of the quantum cohomology, it is also polarized 
in a neighbourhood of the large radius limit point~(ibid.)
\end{example} 
\subsection{The Connection $\Nablacc$ on the Total Space} 
In this section we fix a pure TRP structure 
$(\cF, \nabla,(\cdot,\cdot)_\cF, \kappa)$ over $\cM$. 
The \emph{total space} $\LL$ of the TRP structure is defined to be 
the total space of the underlying cTP structure 
$(\sfF,\nabla,(\cdot,\cdot)_\sfF)$, that is, the 
total space of the infinite-dimensional vector bundle 
associated to $z\sfF$ (see \S\ref{subsec:totalspace}). 
We assume that $(\sfF,\nabla,(\cdot,\cdot)_\sfF)$ is miniversal (Assumption~\ref{assump:miniversal}), and denote by $\pr \colon \LL \to \cM$ the natural projection. 

We need to extend the structure sheaf $\bcO$ on $\LL$ 
by adding real-analytic functions on $\cM$. Set:
\[
\AO:=(\pr^{-1}\cA_\cM) \otimes_{\pr^{-1} \cO_\cM} \bcO
\]
The sheaf $\bOmega^1$ of one-forms 
on $\LL$ is also extended as 
\[
\AOmega^1 := \AOmega^{1,0} \oplus \pr^*\cA_\cM^{0,1}
\]
where $\AOmega^{1,0}$ and $\pr^*\cA_{\cM}^{0,1}$ are 
given in terms of local co-ordinates $\{t^i, x_n^i\}$ 
(see \S\ref{subsec:totalspace}) as 
\begin{align*}
  \AOmega^{1,0} = \bigoplus_{i=0}^N \AO dt^i \oplus 
  \bigoplus_{n=1}^\infty \bigoplus_{i=0}^N \AO dx_n^i &&
  \pr^*\cA_\cM^{0,1} = \bigoplus_{i=0}^N \AO d\ov{t}^i  
\end{align*}
We also write $\ATheta = 
\ATheta^{1,0} \oplus \pr^* \cT^{0,1}_\cM$ 
for the dual sheaf $\sHom_{\AO}(\AOmega^1,\AO)$, 
where $\cT^{0,1}_\cM$ denotes the sheaf of real-analytic vector fields 
of type $(0,1)$ on $\cM$. 
The gradings and (increasing) filtrations on $\bcO$, $\bOmega^1$ 
considered in \S\ref{subsec:totalspace} can be naturally extended to 
$\AO$ and $\AOmega^{1,0}$. We set 
\begin{align*} 
\AO(\pr^{-1}(U))^n & 
= \cA(U)\otimes_{\cO(U)} \bcO(\pr^{-1}(U))^n \\
\AO(\pr^{-1}(U))_l & = \cA(U)\otimes_{\cO(U)} \bcO(\pr^{-1}(U))_l 
\end{align*} 
and for $\AOmega^{1,0}$ we set, as in \eqref{eq:grading_filtration_coord},
\begin{align*}
  \deg(dt^i) = 0 &&
  \deg (dx_n^i )= 1 && 
  \filt(dt^i) = -1 &&
  \filt( dx_n^i) = n-1
\end{align*}
where $\filt(y)$ is the least number $m$ such that $y$ belongs 
to the $m$th filter.

The framework in \S\ref{sec:global_theory} generalizes easily 
to this setting. 
The dual modules $(z^n \AF)^\vee$, $\AF[z^{-1}]^\vee$ 
are defined as in \eqref{eq:dualmodules} but replacing 
$\cO_\cM$ with $\cA_\cM$;  
the pull-backs of $z^n\AF$, $\AF[z^{-1}]$, 
$(z^n\AF)^\vee$, $\AF[z^{-1}]^\vee$ under $\pr\colon \LL \to \cM$ 
are defined as in \eqref{eq:pullbacks}. 
The pull-back of the connection $\nabla$ defined in \eqref{eq:nabla-AF} 
gives a connection 
\[
\tnabla \colon \pr^*\AF \to \AOmega^{1,0} \hotimes \pr^*(z^{-1} \AF) 
\oplus \pr^*\cA_\cM^{0,1} \otimes \pr^*\AF   
\]
on $\pr^*\AF$ (cf.~equation~\ref{eq:tnabla}).  
Note that the $(0,1)$-part of $\tnabla$ is nothing but the 
$\ov\partial$-operator defining the holomorphic structure 
$\pr^*\sfF$. 
\begin{definition}[cf.~Definition~\ref{def:KS}]
\label{def:KS-cc}
The \emph{Kodaira--Spencer map} $\KS \colon \ATheta^{1,0} \to  
\pr^*\AF$ and the \emph{dual Kodaira--Spencer map} 
$\KS^* \colon \pr^*\AF^\vee \to \AOmega^{1,0}$ are 
defined by 
\begin{align*}
  \KS(v) = \tnabla_v \bx &&
  \KS^*(\varphi) = \varphi(\tnabla^{(1,0)} \bx)  
\end{align*}
where $\bx$ is the tautological section of $\pr^*(z\AF)$.  
They are simply the base changes of the 
Kodaira--Spencer maps defined previously, and are isomorphisms 
over the open subset $\LLo\subset \LL$. 
\end{definition} 

The complex conjugate opposite module $z^{-1}\ov\AF$ (Definition~\ref{def:cc-opposite}) determines connections 
$\Nablacc$ as follows.

\begin{definition}[cf.~Definition~\ref{def:Nabla}] 
\label{def:Nabla-cc}
Let $\Picc \colon \AF[z^{-1}] = \AF \oplus z^{-1}\ov{\AF} \to \AF$ 
denote the projection along $z^{-1}\ov{\AF}$. 
Set $\AOmegao^1:=\AOmega^1|_{\LLo}$, 
$\AOmegao^{1,0}:= \AOmega^{1,0}|_{\LLo}$, 
$\AThetao^1:=\ATheta^1|_{\LLo}$,  
$\AThetao^{1,0} := \ATheta^{1,0}|_{\LLo}$. 
Consider the maps 
\[
\xymatrix{\pr^*\AF \ar[rr]^-{\tnabla} && \AOmega^1 \hotimes \pr^*(z^{-1}\AF) \ar[rr]^-{\Picc} && \AOmega^1 \hotimes \pr^*(\AF)}
\]
\[
\xymatrix{ 
  \pr^*\AF^\vee \ar[rr]^-{\Picc^*} && \pr^*(z^{-1}\AF)^\vee \ar[rr]^-{\tnabla^\vee} && \AOmega^1 \otimes \pr^*\AF^\vee 
}
\]
Via the (dual) Kodaira--Spencer isomorphisms 
$\KS \colon \AThetao^{1,0} \cong \pr^*\AF$ 
and $\KS^* \colon \pr^*\AF^\vee \cong \AOmegao^{1,0}$, 
these maps induce connections 
\begin{align*} 
& \Nablacc\colon \AThetao^{1,0} \to \AOmegao^1 \hotimes \AThetao^{1,0} \\
& \Nablacc \colon \AOmegao^{1,0} \to \AOmegao^1 \otimes 
\AOmegao^{1,0} 
\end{align*} 
on the tangent and the cotangent sheaves of type $(1,0)$.  
Here the connection $\tnabla^\vee$ dual to $\tnabla$ is defined as in
\eqref{eq:tnablavee-def}, \eqref{eq:tnablavee-CD}. 
\end{definition} 

We shall see in the next section (\S\ref{subsec:Kaehler}) 
that the connection $\Nablacc$ can be viewed as 
the Chern connection on $\LLo$ associated to a certain 
K\"{a}hler metric. 
\begin{proposition} 
The connection $\Nablacc$ on $\AThetao$ is a torsion-free 
connection whose $(0,1)$-part is the $\ov\partial$-operator defining 
the holomorphic structure $\bThetao$. 
\end{proposition} 
\begin{proof} 
It is obvious from the definition that the $(0,1)$-part of $\Nablacc$ 
is the $\ov\partial$-operator. 
Torsion-freeness follows from the same argument as 
Proposition~\ref{prop:Nabla-torsionfree}. 
\end{proof} 

We next introduce the propagator $\Delta_{\sfP, \rm cc}$ 
and the background torsion $\Lambdacc$ 
associated to the complex conjugate opposite module. 
\begin{definition}[cf.~Definition~\ref{def:propagator}] 
\label{def:propagator-cc}
Let $\sfP$ be a pseudo-opposite module of $(\sfF, \nabla,(\cdot,\cdot)_\sfF)$ 
in the sense of Definition~\ref{def:opposite}. 
Write $\cA\sfP := \cA_\cM\otimes_{\cO_{\cM}} \sfP$ and 
let $\Pi_\sfP \colon \AF[z^{-1}] = \AF \oplus \cA\sfP \to \AF$ 
denote the projection along $\cA\sfP$.  
The \emph{propagator} $\Delta_{\sfP, \rm cc} = \Delta(\sfP, z^{-1}\ov\AF)$ 
between $\sfP$ and the complex conjugate opposite module 
$z^{-1} \ov\AF$ 
is a homomorphism $\AOmegao^{1,0} \otimes \AOmegao^{1,0} 
\to \AO$ defined by 
\[
\Delta_{\sfP, \rm cc}(\omega_1,\omega_2) = \Omega^\vee(\Pi_\sfP^*\varphi_1, 
\Picc^* \varphi_2 ) 
\]
where $\varphi_i := (\KS^*)^{-1} \omega_i$, $i\in\{1,2\}$. 
\end{definition} 

\noindent The propagator is symmetric $\Delta_{\sfP, \rm cc}(\omega_1,\omega_2) 
= \Delta_{\sfP, \rm cc}(\omega_2,\omega_1)$ (see the proof of 
Proposition~\ref{pro:prop-elementary}) 
and satisfies $\Delta_{\sfP, \rm cc} - \Delta_{\sfQ, \rm cc} = \Delta(\sfP,\sfQ)$ 
(see the proof of Proposition~\ref{prop:Deltasum}). 

\begin{definition}[cf.~Definition~\ref{def:background torsion}]
The (background) \emph{torsion} associated to the complex 
conjugate opposite module $z^{-1}\ov\AF$ is 
an operator $\Lambdacc \colon 
\AOmegao^{1,0} \times \AOmegao^{1,0} \to \pr^*\cA_{\cM}^{0,1}$  
defined by 
\[
\Lambdacc (\omega_1, \omega_2) = 
\Omega^\vee(\tnabla^\vee \Picc^* \varphi_1, 
\Picc^* \varphi_2) 
\]
where $\varphi_i := (\KS^*)^{-1} \omega_i$, $i\in\{1,2\}$.  
\end{definition} 

\noindent The background torsion takes values in $\pr^*\cA_{\cM}^{0,1}$ because 
$z^{-1}\ov\AF$ is parallel in the holomorphic direction.  
It is $\AO$-bilinear and symmetric 
$\Lambdacc(\omega_1,\omega_2) = 
\Lambdacc(\omega_2,\omega_1)$:
see the proof of Lemma~\ref{lem:background torsion}. 

We use tensor notation as in 
Proposition~\ref{prop:difference_conn} or 
Proposition~\ref{prop:propagator-curved}. 
Let $\{\sx^\mu\} = \{t^i, x_n^i\}$ denote an algebraic local co-ordinate system 
on $\LL$ (see \S\ref{subsec:totalspace}). We use Roman letters $i,j,k,\dots$ for the indices of 
co-ordinates $\{t^i\}$ on $\cM$, and Greek letters 
$\mu,\nu,\rho,\dots$ for the indices of co-ordinates $\{\sx^\mu\}$. 
We also use the Einstein summation convention, as previously. 
The following proposition is an analogue of  
Proposition~\ref{prop:difference_conn}(1) 
and Proposition~\ref{prop:propagator-curved}(2). 
We remark that the connection $\Nabla^\sfP$ associated to 
a pseudo-opposite module $\sfP$ (Definition~\ref{def:Nabla}) 
can be naturally extended to a connection on $\AOmegao^1$ 
(or on $\AThetao^1$) such that the $(0,1)$-part 
coincides with the $\ov\partial$-operator. 

\begin{proposition} 
\label{prop:propagatorcc} 
Let $\sfP$ be a pseudo-opposite module for $(\sfF,\nabla,(\cdot,\cdot)_{\sfF})$. 
\begin{enumerate}
\item The difference $\Nablacc-\Nabla^{\sfP}$ 
defines a map $\AOmega^{1,0} \to 
\pr^*(\cA_{\cM}^{1,0}\otimes \cA_{\cM}^{1,0})$  
given by:
\[
(\Nablacc - \Nabla^\sfP) d\sx^\nu = \Delta_{\sfP, \rm cc}^{\nu\sigma} 
C^{(0)}_{\sigma i j} dt^i \otimes dt^j
\]
\item The covariant derivative of the propagator gives: 
\begin{align*} 
\Nabla_\mu^{\sfP} \Delta_{\sfP, \rm cc}^{\nu\rho} 
& (:= 
\partial_\mu \Delta^{\nu\rho} +  
{\Gamma}^\nu_{\mu\sigma} \Delta^{\sigma\rho} 
+ {\Gamma}^\rho_{\mu\sigma} \Delta^{\nu\sigma})  
= {\Lambda_{\sfP}}_\mu^{\nu\rho} 
+ \Delta_{\sfP, \rm cc}^{\nu \sigma} C^{(0)}_{\sigma \mu \tau} 
\Delta_{\sfP, \rm cc}^{\tau \rho} \\ 
\Nabla_{\oi}^\sfP \Delta_{\sfP, \rm cc}^{\nu\rho} 
& = \ov\partial_i \Delta_{\sfP, \rm cc}^{\nu\rho} 
= - {\Lambdacc}^{\nu\rho}_{\, \oi}
\end{align*} 
where $\Lambda_\sfP$ is the torsion of $\sfP$ 
and $\Gamma^\nu_{\mu\rho}$ are the Christoffel coefficients 
of $\Nabla^{\sfP}$ as in \eqref{eq:Christoffel}. 
(In the first line, only the case $\mu=i$ yields non-vanishing 
results.)
\end{enumerate} 
\end{proposition} 
\begin{proof} 
The proof is almost the same as Propositions~\ref{prop:difference_conn},~\ref{prop:propagator-curved} and is omitted. 
\end{proof} 

\begin{definition}[cf.~equations~\ref{eq:difference1-form} and~\ref{eq:curvature-Lambda}]
\label{def:difference_one-form_cc} 
Let $\sfP$ be a pseudo-opposite module for $(\sfF, \nabla,(\cdot,\cdot)_\sfF)$. 
The \emph{difference one-form} $\omega_{\sfP, \rm cc} 
\in \pr^*\cA^{1,0}_\cM$ associated to 
$\sfP$ and the complex conjugate opposite module $z^{-1}\ov\AF$ 
is defined to be 
\[
\omega_{\sfP, \rm cc} = \frac{1}{2}  \sum_{i=0}^N  \sum_{j=0}^N  \sum_{h=0}^N
C^{(0)}_{ijh} \Delta_{\sfP, \rm cc}^{jh} dt^i  
= \frac{1}{2} \sum_{i=0}^N \Tr_{\AF_0} 
( (\Pi_{\sfP} - \Picc) \nabla_i ) 
dt^i
\] 
where $\AF_0 = \AF/z\AF$ 
(the proof of the second equality here is the same as that of
Lemma~\ref{lem:difference1-form}). 
We have $\omega_{\sfP, \rm cc} - \omega_{\sfQ, \rm cc} = 
\omega_{\sfP\sfQ}$. 
If $\sfP$ is parallel, then the two-form $\varthetacc :=d\omega_{\sfP, \rm cc}
\in \pr^*\cA_{\cM}^2$ 
does not depend on the choice of a parallel $\sfP$. 
(This follows from the same argument as Lemma~\ref{lem:curvature2form}; see Proposition~\ref{prop:curvaturecc} 
or equation~\ref{eq:varthetacc_CCtilde} 
below for an explicit formula.) 
We call $\varthetacc$ the \emph{curvature two-form} 
of the complex conjugate opposite module. 
Both the difference one-form $\omega_{\sfP, \rm cc}$ and 
the curvature two-form $\varthetacc$ are pulled back from 
$\cM$. 
\end{definition} 

Finally we give formulas for the curvature of $\Nablacc$ and its 
trace. The proofs of these are again parallel to the argument in 
Proposition~\ref{prop:curvature}, and are omitted. 
\begin{proposition}[curvature] 
\label{prop:curvaturecc}
Let $(\Nablacc)^2$ denote the curvature of $\Nablacc$ 
on the cotangent sheaf $\AOmegao^1$, which 
is an $\End(\AOmegao^{1,0})$-valued $(1,1)$-form on $\LLo$. 
\begin{enumerate}
\item The curvature $(\Nablacc)^2$ defines a map 
$\AOmegao^{1,0} \to \pr^*(\cA^{1,1}_\cM) \otimes 
\pr^*(\cA^{1,0}_\cM)$ 
given by:
\[
(\Nablacc)^2 d\sx^\nu = 
C^{(0)}_{i j h } {\Lambdacc}_{\ov{l}}^{h \nu} 
(dt^i \wedge d\ov{t}^l ) \otimes  dt^j
\] 
\item The curvature two-form $\varthetacc$ equals\footnote
{The factor $1/2$ instead of $1/4$ 
in the second line here
is not a typo; it reflects the asymmetry 
between $i$ and $j$, cf.~\eqref{eq:curvature-Lambda}.} half of the trace of 
$(\Nablacc)^2$: 
\begin{align*}
\varthetacc & = \frac{1}{2} \Tr((\Nablacc)^2)
=  \frac{1}{2} C^{(0)}_{ijh} {\Lambdacc}_{\ov{l}}^{jh} 
dt^i \wedge d\ov{t}^l  \\
& = - \frac{1}{2} \sum_{i=0}^N \sum_{j=0}^N
\Tr_{\AF_0}(\Picc \nabla_i \Picc \ov\partial_j - 
\Picc \ov\partial_j \Picc \nabla_i) 
dt^i \wedge d\ov{t}^j
\end{align*} 
In particular, $\varthetacc$ is of type $(1,1)$. 
\end{enumerate}
\end{proposition}

\begin{example} 
We give explicit formulae for the quantities 
$\Nablacc$, $\Delta_{\sfP, \rm cc}$, $\Lambdacc, 
\varthetacc$ in 
terms of local co-ordinates. 
Let $\sfP$ be a reference opposite module.  
By Proposition~\ref{prop:flat_trivialization}, $\sfP$
defines (locally) a flat trivialization of $\AF$ 
by choosing a flat frame of $z\sfP/\sfP$. We denote 
it by:
\[
\Phi \colon \C^{N+1} \otimes \cA_\cM[\![z]\!] \cong \AF
\] 
The trivialization $\Phi$ induces a trivialization $\Phi_0$ of 
the holomorphic bundle $F_0 = F|_{z=0}$; 
by purity, $\Phi_0$  
extends to a trivialization of $\hF$ which is holomorphic 
along each fiber $\{t\}\times \Proj^1$. 
Restricted to the formal neighbourhood of $z=0$, 
this gives rise to a different trivialization $\Phicc$ of 
$\AF$: 
\[
\Phicc \colon \C^{N+1} \otimes \cA_\cM[\![z]\!] \cong \AF
\]
The trivialization $\Phicc$ is only real-analytic 
and coincides with $\Phi$ along $z=0$. 
The trivialization $\Phi$ induces a Frobenius-type structure 
on the trivial bundle of rank $N+1$ 
as in Proposition~\ref{prop:flat_trivialization}, i.e.~we have 
a flat connection 
\[
\Phi^* \nabla = d - \frac{1}{z} \cC, \qquad 
\cC = \cC(t) \in \End(\C^{N+1})\otimes \Omega_\cM^1 
\]
and a constant symmetric pairing 
$g_{ij} = g(e_i,e_j)= (\Phi(e_i), \Phi(e_j))_{z\sfP/\sfP}$ 
on the trivial bundle. 
On the other hand the trivialization $\Phicc$ induces a 
$tt^*$-bundle structure 
on the trivial bundle as in Definition~\ref{def:ttstarbundle}, 
i.e.~we have a flat connection 
\[
{\Phicc}^*\nabla = D - \frac{1}{z} \cC - z \tcC, \quad 
\cC \in \End(\C^{N+1}) \otimes \Omega^1_{\cM}, 
\quad 
\tcC \in \End(\C^{N+1}) \otimes \cA^{0,1}_{\cM}   
\]
($\cC$ is the same as above) 
and a complex anti-linear involution 
$\kappa \in \End_\R(\C^{N+1})\otimes \cA_{\cM}$ 
such that $D=d + h^{-1} \partial h$ is the Chern connection associated to 
the Hermitian metric $h (u,v) = g(\kappa(u), v)$ 
and $\tcC = \kappa \circ \cC \circ \kappa$. 
We write 
\[
R = R(t,\ov{t}, z) = \id + R_1 z + R_2 z^2 + R_3 z^3 + \cdots 
: =  (\Phicc)^{-1} \circ \Phi
\]
for the gauge transformation which intertwines 
the connections $R \circ \Phi^* \nabla = {\Phicc}^*\nabla \circ R$ 
and satisfies 
$g(R(t,\ov{t},-z)u, R(t,\ov{t},z)v) = g(u,v)$ 
for $u,v\in \C^{N+1}$. 

Let $\{t^i, x_n^i\}$ denote the local co-ordinate system on $\LL$ 
associated to\footnote
{On the other hand, the co-ordinate system associated to $\Phicc$ is 
not holomorphic.} 
the trivialization $\Phi$. 
Then we have (see also Example~\ref{exa:Nabla-explicit}) 
\begin{align*} 
(\Nablacc - \Nabla) dt^h &= [K(x_1)^{-1} R_1 \cC_i \cC_j]^h 
dt^i \otimes dt^j \\ 
(\Nablacc - \Nabla) dx_n^h & 
= ([R^{-1} \cC_i \cC_j x_1]^h_{n+1} 
+ [K(x_{n+1}) K(x_1)^{-1} R_1 \cC_i \cC_j x_1]^h) 
dt^i \otimes dt^j
\end{align*} 
where $K(x_n) \in \End(\C^{N+1}) \otimes \bcO$ 
is defined by $K(x_n) e_i = \cC_i x_n$ (see Notation~\ref{nota:[]_^} for $[\cdots]_n^h$). 
The curvature of $\Nablacc$ on $\AOmegao^{1,0}$ 
is given by 
\begin{align*} 
[\Nablacc_i, \Nablacc_{\oj}] dt^h& =  
- [K(x_1)^{-1} \tcC_{\oj} \cC_i \cC x_1]^h\\ 
[\Nablacc_i, \Nablacc_{\oj}] dx_n^h 
& = [R^{-1} \tcC_{\oj} \cC_i \cC x_1]^h_n 
- [K(x_{n+1}) K(x_1)^{-1} \tcC_{\oj} \cC_i \cC x_1]^h 
\end{align*} 
with all the other components being zero. In particular 
we have:
\begin{equation} 
\label{eq:varthetacc_CCtilde}
\varthetacc = \frac{1}{2} \Tr((\Nablacc)^2) 
= - \frac{1}{2} \sum_{i=0}^N \sum_{j=0}^N \Tr(\tcC_{\oj} \cC_i) dt^i \wedge d\ov{t}^j
\end{equation}
Let $\{\varphi_n^i\}$ denote the frame of 
$\pr^*\AF[z^{-1}]^\vee$ given by the trivialization 
$\Phi$ (see equation~\ref{eq:frame_varphi}). 
As in \S\ref{subsubsec:Giventalpropagator}, we set 
\begin{align*}
\begin{split} 
V^{(n,j),(m,i)}_{\rm cc} &:= 
\Omega^\vee(\varphi_n^j, \Picc^* (\varphi_m^i|_{\AF})) 
= (-1)^n g^{j u} \left[R^{-1} [R e_u z^{-n-1}]_+\right]_m^i  
\end{split} 
\end{align*}
where $(g^{ij})$ is the matrix inverse to $(g_{ij})$. 
The $V^{(n,j),(m,i)}_{\rm cc}$ depend real-analytically on $t$. 
Explicit formulae for the propagators 
$\Delta_{\sfP, \rm cc}(dt^a \otimes dt^b)$, 
$\Delta_{\sfP, \rm cc}(dt^a\otimes dx_n^b)$, 
$\Delta_{\sfP, \rm cc}(dx_m^a \otimes dx_n^b)$ 
are given by the same formulae as in Remark~\ref{rem:propagator-explicit} 
with $V^{(n,i),(m,j)}$ there replaced with $V_{\rm cc}^{(n,i),(m,j)}$. 
Using  
\begin{align*} 
\Omega^\vee(\tnabla \Picc^* (\varphi_n^i|_{\AF}), 
\Picc^* (\varphi_m^u|_{\AF})) 
&= (-1)^{n+m+1} 
g( \tcC R_m e^u, R_n e^i)  
\end{align*} 
where $e^i = \sum_{j=0}^N g^{ij} e_j$, 
we obtain the following explicit formula for $\Lambdacc$: 
\begin{align*} 
& \Lambdacc(dt^a \otimes dt^b) 
= - g(\tcC K(x_1)^{\dagger -1} e^a, K(x_1)^{\dagger -1} e^b)  \\ 
& \Lambdacc(dt^a \otimes dx_n^b) 
= (-1)^n g(\tcC K(x_1)^{\dagger -1} e^a, R_n e^b) 
- g(\tcC K(x_1)^{\dagger -1} e^a, 
K(x_1)^{\dagger -1} K(x_{n+1})^\dagger e^b) \\ 
& \Lambdacc(dx_n^a \otimes dx_m^b) 
= (-1)^{n+m+1} g( \tcC R_n e^a, R_m e^b) \\ 
& \qquad + (-1)^m  g(\tcC K(x_1)^{\dagger -1} 
K(x_{n+1})^\dagger e^a, R_m e^b)   
+ (-1)^n g(\tcC R_n e^a, K(x_1)^{\dagger -1} K(x_{m+1})^\dagger e^b) \\
& \qquad - g( \tcC K(x_1)^{\dagger -1} K(x_{n+1})^\dagger e^a, 
K(x_1)^{\dagger -1} K(x_{m+1})^\dagger e^b) 
\end{align*} 
where $K(x_1)^\dagger$ is the adjoint of $K(x_1)$ with 
respect to the complex bilinear pairing $g$. 
\end{example} 

\subsection{K\"{a}hler Geometry of the Total Space} 
\label{subsec:Kaehler}
We now introduce a (pseudo-)K\"{a}hler metric on the $L^2$-subspace 
$L^2(\LLo)$ and identify $\Nablacc$ with the Chern connection. 
The propagator and the Yukawa coupling also have descriptions 
in terms of the K\"{a}hler metric. 
Recently David and Strachan~\cite{David--Strachan} have considered an extension of
$tt^*$-geometry to the big phase space; their construction 
seems to be closely related to ours.

Let $(\cF = \cO(F), \nabla, (\cdot,\cdot)_\cF, \kappa)$ be 
a pure TRP structure over $\cM$ 
and let $(\sfF, \nabla,(\cdot,\cdot)_\sfF)$ denote the 
corresponding cTP structure. 
Since $\cF=\cO(F)$ is defined over $\cM \times \C$, 
the total space $\LL$ of the TRP structure has a canonical $L^2$-subspace 
$L^2(\LL)$ as follows (cf.~Remark~\ref{rem:L2-neighbourhood}). 
Let $\F_t \subset \sfF_t$ denote the subspace consisting of 
elements in $\sfF_t$ which extend to holomorphic sections of $F$ 
over the unit disc $\{t\} \times \D$ (where $\D = \{z\in \C: |z|<1\}$) 
and have $L^2$-boundary values over $S^1$ (vector-valued Hardy space). 
In other words, $\bx\in \sfF_t$ lies in $\F_t$ if and only if 
it has a square summable expansion $\bx = \sum_{n=0}^\infty x_n z^n$ 
for some (and hence any) trivialization of $F|_{\{t\}\times \C}$. 
Then $L^2(\LL)$ consists of $(t,\bx) \in \LL$ such that $\bx\in z\F_t$. 
This has the structure of a complex Hilbert manifold (the total 
space of a Hilbert vector bundle over $\cM$). 
We let 
\[
\cH_t := L^2(\{t\} \times S^1, F)
\] 
denote the space of $L^2$-sections of $F|_{\{t\}\times S^1}$. 
We define the symplectic form on $\cH_t$ as 
\[
\Omega_t(v,w) = \frac{1}{2\pi\iu} 
\int_{S^1} (v(-z), w(z))_{\cF} \, dz
\]
(cf.~equation~\ref{eq:symplecticform}).
The pair $(\cH_t,\Omega_t)$ is an analogue of the Givental space
(\S\ref{subsec:Givental-symplecticvs}) for the TRP structure. 
The involution $\kappa$ of the TRP structure induces an involution 
$\kappa_{\cH_t}$ on $\cH_t$ (cf.~equation~\ref{eq:kappa-H0}). 
We have (cf.~equation~\ref{eq:kappa_t-property}) 
\[
\kappa_{\cH_t}(f v) = \ov{\gamma^*f} \cdot \kappa_{\cH_t}(v) 
\]
for $f \in L^\infty(S^1;\C)$. 
Note that parallel translation using the flat connection $\nabla$ 
identifies all the triples $(\cH_t,\Omega_t, \kappa_{\cH_t})$
for $t$ in a simply-connected open subset $U$ of $\cM$. 
We work locally on $\cM$ and write $(\cH, \Omega,\kappa_\cH)$ for nearby 
$(\cH_t,\Omega_t,\kappa_{\cH_t})$, $t\in U$, 
identified with each other. 
The reality of the pairing $(\cdot,\cdot)_\cF$ (see equation~\ref{eq:reality-pairing}) 
implies that the symplectic form is pure-imaginary 
with respect to the \emph{shifted} involution
$\tkappa_\cH :=z^{-1} \kappa_\cH$ 
(note that we still have $\tkappa_\cH \circ \tkappa_\cH =\id$): 
\[
\Omega(\tkappa_\cH(v), \tkappa_\cH(w)) = - \ov{\Omega(v,w)}. 
\]
The subspace $\F_t \subset \cH_t \cong \cH$ is 
Lagrangian with respect to $\Omega$. 
The family $t\mapsto \F_t$ of subspaces of $\cH$ should be viewed as 
a \emph{semi-infinite period map} 
(see \S\ref{subsec:Lagrangian_TP}, where the semi-infinite 
subspace is denoted by $T_t \subset \cH$) which takes values 
in the semi-infinite Grassmannian $\Gr_{\frac{\infty}{2}}(\cH)$.  
Locally one can imbed the total space $L^2(\LLo)$ into $\cH$  
via the semi-infinite period map: 
\[
\iota \colon L^2(\LLo)|_U  = \bigcup_{t\in U} (z\F_t)^\circ 
\looparrowright \cH 
\]
where $(z\F_t)^\circ = L^2(\LLo_t) = z\F_t \cap \LLo_t$ is a ``Zariski open"  
subset of $z\F_t=L^2(\LL_t)$. 
The derivative of $\iota$ defines an isomorphism: 
\begin{equation} 
\label{eq:KS-L2}
d \iota \colon T_{(t,\bx)} L^2(\LLo)|_U \cong \F_t \subset \cH  
\end{equation} 
which corresponds to the Kodaira--Spencer map (Definitions~\ref{def:KS},~\ref{def:KS-cc}). 
Note that $\F_t$ is identified with the tangent space of $L^2(\LLo)$ 
at $(t,\bx)$. 
\begin{definition} 
We define a non-degenerate sesqui-linear pairing $\bh$ on $\cH$ by:
\[
\bh(v,w) = - \Omega(\tkappa_\cH (v), w) = - \Omega(z^{-1} \kappa_\cH(v),w)
\]
This is Hermitian and purely imaginary; one can easily check that:
\begin{align*} 
& \bh(v,w) = \ov{\bh(w,v)} && \bh(\alpha v, w) = \ov{\alpha} \bh(v,w)  
\quad (\alpha \in \C) \\ 
& \bh(\tkappa_{\cH}(v), \tkappa_{\cH}(w)) = -\ov{\bh(v,w)} && 
\bh(zv, zw) = - \bh(v,w)
\end{align*} 
Pulling back $\bh$ along the local immersion $\iota \colon L^2(\LLo)|_U 
\looparrowright \cH$ gives a Hermitian metric $\bh$ on $L^2(\LLo)$. 
Thus $L^2(\LLo)$ has the structure of a (pseudo)-K\"{a}hler Hilbert 
manifold. 
\end{definition} 
\begin{remark} 
The pairing $\bh$ on $\cH$ is indefinite of signature $(\infty,\infty)$. 
The metric $\bh$ restricted to $L^2(\LLo)$ is non-degenerate under 
purity---this follows from the $\bh$-orthogonal decomposition 
\eqref{eq:purity-L2} below---and is also of signature $(\infty,\infty)$. 
\end{remark} 

Purity of the TRP structure implies (cf.~equation~\ref{eq:pure-trans}) that:
\begin{equation} 
\label{eq:purity-L2} 
\cH = \F_t \oplus \tkappa_\cH(\F_t)
\end{equation} 
The family $t\mapsto \tkappa_\cH(\F_t)$ defines 
an $L^2$-version of the 
complex conjugate opposite module $z^{-1}\ov\AF$ from \eqref{eq:opposition-cc}. 
Note that $\F_t$ and $\tkappa_\cH(\F_t)$ are orthogonal 
to each other with respect to $\bh$. 
In particular the projection $\Picc \colon \cH \to \F_t$ along $\tkappa_\cH(\F_t)$ 
(which is an $L^2$-version of $\Picc$ in Definition~\ref{def:Nabla-cc}) 
is the orthogonal projection to $\F_t$. 
Therefore $\Nablacc$ on $L^2(\LLo)$ can be identified with the induced connection 
on the immersed submanifold $L^2(\LLo)|_U\looparrowright \cH$ 
via the orthogonal projection. 
This implies the following proposition. 
\begin{proposition} 
The connection $\Nablacc$ on $L^2(\LLo)$ (Definition~\ref{def:Nabla-cc}) 
is the Chern connection associated to the Hermitian metric $\bh$. 
\end{proposition} 

\begin{corollary} 
\label{cor:curvature_cc_pi} 
The curvature two-form $\varthetacc$ is a pure-imaginary $(1,1)$-form. 
\end{corollary} 
\begin{proof} 
Recall from Proposition~\ref{prop:curvaturecc} and \eqref{eq:varthetacc_CCtilde} that $\varthetacc$ 
is the half of the trace of the curvature of $\Nablacc$.
\end{proof} 

Let $\sfP$ be a parallel pseudo-opposite module for the 
underlying cTP structure $(\sfF,\nabla, (\cdot,\cdot))$ 
over $U$. 
We assume here that $\sfP$ is \emph{compatible with the given 
$L^2$-structure} on $\sfF$, namely that:  
\begin{itemize} 
\item every element of $\sfP_t\subset \sfF_t[z^{-1}]$ 
extends to a holomorphic section of $F|_{\{t\}\times \D^*}$ 
over the unit punctured disc $\D^* =\{z: 0<|z|<1\}$ 
and has an $L^2$-boundary value along $S^1$.  
Thus $\sfP_t$ is a subspace of $\cH_t$; 
\item the $L^2$-closure $\Proj_t$ of $\sfP_t$ in $\cH_t$ 
is complementary to $\F_t$, i.e.~$\cH_t = \F_t \oplus \Proj_t$ 
(as an algebraic direct sum, not necessarily orthogonal). 
\end{itemize} 
The same notion already appeared in Example \ref{exa:compatible_with_L2}. 
Since $\sfP$ is parallel, it gives rise to a constant Lagrangian 
subspace $\Proj$ in $\cH \cong \cH_t$. 

\begin{definition} 
Let $\Pi_{\sfP} \colon \cH \to \F_t$ denote the projection 
along $\Proj$. 
We define the complex-antilinear endomorphism $\varkappa_\sfP 
\colon TL^2(\LLo)|_U \to TL^2(\LLo)|_U$ 
by 
$\varkappa_\sfP(v) = (d\iota)^{-1} \Pi_{\sfP} \tkappa_\cH ( d\iota ( v) )$:
\[
\xymatrix{
  \varkappa_\sfP \colon 
  T_{(t,\bx)} L^2(\LLo) \cong \F_t \ar[rr]^-{\tkappa_\cH} &&
  \tkappa_\cH(\F_t) \subset \cH \ar[rr]^-{\Pi_{\sfP}} &&
  \F_t \cong T_{(t,\bx)}L^2(\LLo) 
}
\]
where $d\iota \colon T_{(t,\bx)}L^2(\LLo) \cong \F_t$ 
is the Kodaira--Spencer map \eqref{eq:KS-L2}. 
\end{definition} 
\begin{remark} 
In general, $\varkappa_\sfP$ is neither an isomorphism nor an involution. 
It is easy to see that:
\begin{itemize} 
\item $\varkappa_\sfP$ is an isomorphism 
if and only if $\Proj \oplus \tkappa_\cH(\F_t) =\cH$; 
\item $\varkappa_\sfP$ is an involution if and only if $\Proj$ is real, 
i.e.~$\tkappa_\cH(\Proj) = \Proj$. 
\end{itemize} 
Let us prove the second statement. 
Note that $v= \varkappa_\sfP(w)$ if and only if $v -\tkappa_\cH(w) \in \Proj$. 
The ``if" part of the statement is obvious. 
Every $p\in \Proj$ can be written as $p = v - \tkappa_\cH(w)$ for 
some $v,w\in \F_t$ by purity \eqref{eq:purity-L2}.  
Then we have $v = \varkappa_\sfP(w)$. 
If $\varkappa_\sfP$ is an involution, we have $w = \varkappa_\sfP(v)$ and 
thus $\tkappa_\cH(p) = \tkappa_\cH(v) - w$ lies in $\Proj$.  
The ``only if" part follows. 
It would be interesting to study parallel pseudo-opposite modules 
$\sfP$ such that $\varkappa_\sfP$ is an involution. 
\end{remark} 

\begin{remark}
If $\sfP$ is an opposite module (i.e.~is closed under $z^{-1}$), then
$\varkappa_\sfP$ cannot be an isomorphism. 
Moreover if the flat trivialization (Proposition~\ref{prop:flat_trivialization}) 
of $\sfF_t$ given by $\sfP$ extends to a smooth trivialization 
of the bundle $F$ over the closed disc $\{t\} \times \{|z|\le 1\}$,  then
$\varkappa_\sfP$ is Hilbert--Schmidt and hence compact. 
\emph{Proof}: Let $v\in \cH$ be a vector of unit length. 
The $L^2$-distance $\dist(z^{-n}v, \Proj) = \dist(v, z^n \Proj)$ 
goes to zero as $n\to \infty$ because 
$\bigcup_{n\ge 0} z^n \Proj$ is dense in $\cH$ 
(see Lemma~\ref{lem:fullrange}).  
Similarly, $\dist(z^{-n}v,\tkappa_\cH(\F_t))\to 0$ as $n\to \infty$. 
These together imply that the distance between the unit spheres 
in $\Proj$ and in $\tkappa_\cH(\F_t)$ is zero 
(because $\|z^{-n}v \|=1$ for all $n$).   
Therefore we cannot have $\Proj \oplus \tkappa_\cH(\F_t) = \cH$. 
To see the latter statement, note that $\varkappa_\sfP$ can be 
viewed as a Hankel operator associated to the gauge transformation 
from the trivialization given by $z^{-1}\ov\AF_t$ 
to the trivialization given by $\sfP_t$. 
Consequently $\varkappa_\sfP$ is Hilbert--Schmidt 
if the gauge transformation extends smoothly 
to the circle $S^1$. 
\end{remark} 

\begin{remark} 
In the Calabi-Yau B-model~\cite{Witten:background, ABK}, 
the middle cohomology $H^3(X,\C)$ of a Calabi-Yau threefold $X$ 
(equipped with the intersection form) 
is a symplectic vector space to be quantized. 
The so-called ``real polarization'' in this context is 
a Lagrangian subspace $P$ of $H^3(X,\C)$ 
with $\overline{P} = P$. 
The above two remarks say that a real polarization 
in the infinite-dimensional setting would be a rather exotic object: 
at least it is not given by an opposite module. 
\end{remark}

The following proposition gives an interpretation of the 
propagator $\Delta_{\sfP,\rm cc}$ and the Yukawa coupling $\bY$ 
in terms of K\"{a}hler geometry. 

\begin{proposition} 
Let $\bh^\vee$ denote the dual Hermitian metric on the 
cotangent bundle of $L^2(\LLo)$. Then we have 
\begin{align*} 
\Delta_{\sfP, \rm cc}(\omega_1,\omega_2) 
& = - \bh^\vee(\varkappa_\sfP^* \omega_1, \omega_2) \\ 
\bY( u, \varkappa_\sfP v,w) &= (\Nabla^{\sfP}_u\bh) (v,w) 
\end{align*} 
for cotangent vectors $\omega_1, \omega_2 \in T^*L^2(\LLo)$ 
and tangent vectors $u,v,w\in T L^2(\LLo)$. 
\end{proposition} 
\begin{proof} 
We identify the tangent space $T_{(t,\bx)}L^2(\LLo)$ 
with $\F_t \subset \cH$ in the proof. 
Let $\cH'$, $\F_t'$ denote the topological duals of $\cH$, $\F_t$.  
The dual symplectic form $\Omega^\vee$ and the dual Hermitian form 
$\bh^\vee$ are defined on $\cH'$ and are 
related by $\bh^\vee (\omega_1,\omega_2) = - \Omega^\vee(
\tkappa_\cH^*\omega_1,\omega_2)$ for $\omega_i \in \cH'$. 
Thus we have 
\[
\Delta_{\sfP, \rm cc}(\omega_1,\omega_2) = 
\Omega^\vee(\Pi_{\sfP}^* \omega_1, \Picc^* \omega_2) 
= - \bh^\vee(\tkappa_\cH^* \Pi_\sfP^* \omega_1, \Picc^* \omega_2) 
\]
for $\omega_1,\omega_2 \in \F_t'$. 
The right-hand side equals $-\bh^\vee(
(\tkappa_\cH^*\Pi_\sfF^*\omega_1)|_\F, \omega_2) 
= -\bh^\vee(\varkappa_\sfP \omega_1, \omega_2)$ 
by the $\bh$-orthogonal decomposition \eqref{eq:purity-L2}.  
\end{proof} 

\subsection{Holomorphic Anomaly Equation} 
We now consider correlation functions under the complex conjugate 
opposite module $z^{-1}\ov\AF$ from Definition~\ref{def:cc-opposite}. 
We show that they satisfy the Bershadsky--Cecotti--Ooguri--Vafa holomorphic anomaly equation, 
and use this to define the Fock space for $z^{-1}\ov\AF$. 
Throughout the section we fix a pure TRP structure 
$(\cF, \nabla,(\cdot,\cdot)_\cF, \kappa)$ over $\cM$. 
The associated cTP structure is denoted by 
$(\sfF,\nabla,(\cdot,\cdot)_\sfF)$. 
We denote an algebraic local co-ordinate system on the total space $\LL$ 
by $\{\sx^\mu\} = \{t^i,x_n^i\}$, as usual. 

\begin{definition}[cf.~Definition~\ref{def:under-curvedopposite}]
\label{def:correlationfunctions_under_cc}
Let $\sfP$ be a parallel pseudo-opposite module for the cTP structure 
$(\sfF,\nabla,(\cdot,\cdot)_\sfF)$, and let $\wave 
= \{C^{(g)}_{\mu_1,\dots,\mu_n}\} \in \Fock(U;\sfP)$ 
be a Fock space element. 
We define a set of completely symmetric tensors 
\[
\wave_\cc = \left\{C^{(g)}_{\cc; \mu_1,\dots,\mu_n} d\sx^{\mu_1} \otimes 
\cdots \otimes d\sx^{\mu_n} 
\in (\AOmegao^{1,0})^{\otimes n}(\pr^{-1}(U)) 
: \text{$n\ge 0$, $g\ge 0$, $2g-2+n>0$}\right\}
\] 
via the Feynman rule in Definition~\ref{def:transformation} 
\[
C^{(g)}_{\cc;\mu_1,\dots,\mu_n} = \sum_{\Gamma} 
\frac{1}{\Aut(\Gamma)} \Cont_{\Gamma} (\wave, \Delta_{\sfP,\rm cc})_{
\mu_1,\dots,\mu_n} 
\]
and the propagator $\Delta_{\sfP,\rm cc}$ in 
Definition~\ref{def:propagator-cc}. 
We call $\wave_\cc = \{C^{(g)}_{\cc;\mu_1,\dots,\mu_n}\}$ 
the \emph{correlation functions under the complex conjugate 
opposite module $z^{-1} \ov\AF$ corresponding to $\wave$} and write:
\[
\wave_\cc = T\left(\sfP,z^{-1} \ov\AF\right) \wave
\]
The corresponding \emph{jet potential} is defined by 
\[
\cW_{\cc}(\sx,\sy) = \sum_{g=0}^\infty \sum_{n=\max(3-2g,0)}^\infty 
\frac{\hbar^{g-1}}{n!} C^{(g)}_{\cc;\mu_1,\dots,\mu_n}(\sx)  
\sy^{\mu_1} \cdots \sy^{\mu_n}   
\]
and we have 
\[
\exp\big(\cW_\cc(\sx,\sy)\big) = \exp\left(
\frac{\hbar}{2} \Delta^{\mu\nu} 
(\sfP, z^{-1} \ov\AF) \partial_{\sy^\mu} \partial_{\sy^\nu} 
\right) 
\exp\big(\cW(\sx,\sy)\big) 
\]
where $\cW(\sx,\sy)$ is the jet potential associated to $\wave$. 
\end{definition} 

\begin{remark} 
The correlation functions $C^{(g)}_{\cc;\mu_1,\dots,\mu_n}$ are holomorphic 
in $\{x_n^i  : n\ge 1, 0\le i\le N\}$ and are real analytic in $t^0,\dots,t^N$. 
Note that $\mu_1,\dots,\mu_n$ are holomorphic indices. 
\end{remark} 

\begin{proposition}[cf.~Definition~\ref{def:localFock}] 
\label{prop:HAE} 
Let $\sfP$ be a parallel pseudo-opposite module for the cTP structure 
$(\sfF,\nabla,(\cdot,\cdot))$ and 
let $\wave=\{C^{(g)}_{\mu_1,\dots,\mu_n}\}$ 
be an element of $\Fock(U;\sfP)$.  
The correlation functions $\{C^{(g)}_{\cc;\mu_1,\dots,\mu_n}\}$ 
under $z^{-1} \ov\AF$ corresponding to $\wave$ satisfy the 
following properties: 
\begin{description} 
\item[(Yukawa)] $C^{(0)}_{\cc;\mu\nu\rho} d\sx^\mu \otimes 
d \sx^\nu \otimes d\sx^\rho$ 
is the Yukawa coupling $\bY$; 
\item[(Jetness)] $\Nablacc_{\mu_1}
C^{(g)}_{\cc;\mu_2,\dots,\mu_n} = C^{(g)}_{\cc;\mu_1,\dots,\mu_n}$, 
where we use notation as in \eqref{eq:covariant-derivative}; 
\item[(Holomorphic Anomaly)] 
\begin{equation} 
\label{eq:HAE}
0  = \partial_{\ov\mu_1} C^{(g)}_{\cc;\mu_2,\dots,\mu_n} 
+ \frac{1}{2} 
\sum_{\substack{\{2,\dots,n\} = I \sqcup J \\ 
k+ l = g}} 
C^{(k)}_{\cc;\mu_I, \alpha} {\Lambdacc}_{\, \ov\mu_1}^{\alpha\beta}
C^{(l)}_{\cc;\mu_J, \beta}  
+ \frac{1}{2} 
C^{(g-1)}_{\cc;\mu_2,\dots,\mu_n,\alpha,\beta} 
{\Lambdacc}^{\alpha\beta}_{\, \ov{\mu}_1}
\end{equation} 
\item[(Grading \& Filtration)]
$C_{\cc;\mu_1,\dots,\mu_n} d\sx^{\mu_1} \otimes \cdots 
\otimes d\sx^{\mu_n} \in 
((\AOmega^{1,0})^{\otimes n}(\pr^{-1}(U)^\circ))^{2-2g}_{3g-3}$;  

\item[(Pole)]
$P (C_{\cc;\mu} d\sx^\mu)$ extends to a regular $(1,0)$-form on 
$\pr^{-1}(U)$, and for $g\ge 2$ we have:
\[
C^{(g)}_{\cc} \in P^{-(5g-5)} \cA(U)[x_1,x_2,P x_3, P^2 x_4, \dots, 
P^{3g-4} x_{3g-2}] 
\]
where $P = P(t,x_1)$ is the discriminant \eqref{eq:discriminant}. 
\end{description} 
\end{proposition}
\begin{proof} 
The proof is similar to that of Lemmas~\ref{lem:jetness}--\ref{lem:pole} 
and Theorem~\ref{thm:anomaly}.  (Yukawa) is obvious 
from the Feynman rule: $C^{(0)}_{\cc; \mu\nu\rho} = C^{(0)}_{\mu\nu\rho}$. 
(Jetness) follows from the argument of
Lemma~\ref{lem:jetness}, using Proposition~\ref{prop:propagatorcc} instead 
of Proposition~\ref{prop:difference_conn}. 
To establish (Holomorphic Anomaly), we differentiate with respect to $\ov{t}^i$ 
the Feynman rule expressing $C^{(g)}_{\cc;\mu_1,\dots,\mu_n}$ 
in terms of $\{C^{(h)}_{\nu_1,\dots,\nu_m}\}$ and 
$\Delta = \Delta( \sfP, z^{-1} \ov\AF)$. 
The only non-holomorphic objects in the Feynman rule are propagators, 
and we have that
$\ov\partial_i \Delta^{\mu\nu} = - {\Lambdacc}_{\oi}^{\mu\nu}$ 
by Proposition~\ref{prop:propagatorcc}(2). 
Therefore $\ov\partial_i C^{(g)}_{\cc;\mu_1,\dots,\mu_n}$ 
can be written as the sum over graphs with one distinguished internal edge,  
on which the propagator is replaced 
with $-{\Lambdacc}_{\oi}^{\mu\nu}$. The second and 
the third terms in equation \eqref{eq:HAE} correspond (respectively) 
to the cases where the distinguished edge separates and does not separate the graph.
(Grading \& Filtration) follows from the argument of 
Lemma~\ref{lem:GrFil}. Here we need to establish an analogue 
of Proposition~\ref{prop:grading-filtration-prop} for 
$\Delta = \Delta(\sfP,z^{-1} \ov\AF)$, but this is straightforward. 
(Pole) follows from the argument of Lemma~\ref{lem:pole}.  
\end{proof} 

\begin{remark} 
Since $\Nablacc$ is flat in the holomorphic direction, 
the condition (Jetness) is compatible with the symmetry 
of the correlation functions (see also Remark~\ref{rem:Fockspace}).  
Note that the holomorphic anomaly equation \eqref{eq:HAE} is non-trivial only when 
the index $\ov\mu_1$ corresponds to one of the co-ordinates 
$\{\ov{t^0},\dots,\ov{t^N}\}$ on $\cM$. 
\end{remark} 

We deduce an important consequence of the holomorphic anomaly 
equation for the genus-one one-point correlation function, 
which is similar to Proposition~\ref{prop:curvature_condition}. 
\begin{proposition}[curvature condition] 
\label{prop:curvature_condition_cc} 
The genus-one one-point function under the complex conjugate 
opposite module $z^{-1} \ov\AF$ satisfies:
\[
d (C^{(1)}_{\cc;\mu} d\sx^\mu) = \varthetacc 
\]
\end{proposition} 
\begin{proof} 
(Jetness) implies that $\Nablacc_\mu C^{(1)}_\nu$ is symmetric 
in $\mu$ and $\nu$. 
Thus $C^{(1)}_{\cc;\mu} d\sx^\mu$ is $\partial$-closed. 
The holomorphic anomaly equation gives 
$\partial_{\oi} C^{(1)}_{\cc,\mu_1} = -\frac{1}{2} C^{(0)}_{\mu_1\alpha\beta}
{\Lambdacc}_{\oi}^{\alpha\beta}$, and this implies that
$\ov\partial ( C^{(1)}_{\cc;\mu} d\sx^\mu ) = \varthetacc$ 
in view of Proposition~\ref{prop:curvaturecc}(2). 
\end{proof} 

\begin{definition}[Fock space for the complex conjugate opposite module] 
The local Fock space $\Fock(U;z^{-1} \ov\AF)$ for the TRP structure 
$(\cF,\nabla, (\cdot,\cdot)_{\cF},\kappa)$ consists of collections 
\[
\left\{C^{(g)}_{\cc; \mu_1,\dots,\mu_n} d\sx^{\mu_1} 
\otimes \cdots \otimes d \sx^{\mu_n} 
\in (\AOmega^{1,0})^{\otimes n} (\pr^{-1}(U)^\circ) : 
\text{$g\ge 0$, $n\ge 0$, $2g-2 +n >0$}\right\}
\]
of completely symmetric tensors satisfying the conditions 
(Yukawa), (Jetness), (Holomorphic Anomaly), (Grading \& Filtration),
and (Pole) listed in Proposition~\ref{prop:HAE}. 
\end{definition} 

Note that Definition~\ref{def:correlationfunctions_under_cc} defines 
a transformation rule 
\[
T(\sfP,z^{-1} \ov\AF) \colon 
\Fock(U;\sfP) \to \Fock(U;z^{-1} \ov\AF) 
\]
for a parallel pseudo-opposite module $\sfP$ over $U$. 

\begin{proposition} 
Let $\sfP$ be a parallel pseudo-opposite module. 
The transformation rule $T(\sfP, z^{-1} \ov\AF)$ defines a bijection 
between $\Fock(U;\sfP)$ and $\Fock(U;z^{-1} \AF)$. 
The inverse map is given by a transformation rule $T(z^{-1} \ov\AF, \sfP)$ 
defined in terms of the propagator $\Delta_{\cc,\sfP} = -\Delta_{\sfP,\cc}$ 
and the Feynman rule similarly to Definition~\ref{def:correlationfunctions_under_cc}. 
\end{proposition} 
\begin{proof} 
Let $\wave_\cc$ be an element of $\Fock(U;z^{-1}\ov\AF)$. 
It suffices to show that $\wave := T(z^{-1} \ov\AF, \sfP) \wave_\cc$ 
satisfies the conditions for elements in $\Fock(U;\sfP)$ in 
Definition~\ref{def:localFock}. (It is clear from the definition 
that the transformation rules $T(\sfP,z^{-1} \ov\AF)$, 
$T(z^{-1} \ov\AF, \sfP)$ are inverse to each other.) 
The conditions (Yukawa), (Jetness), (Grading \& Filtration), and
(Pole) can be checked using the argument in 
Lemmas~\ref{lem:jetness}--\ref{lem:pole}. 
It suffices to show that each correlation function in $\wave$ 
is holomorphic. Writing $\wave_\cc = \{C^{(g)}_{\cc;\mu_1,\dots,\mu_n}\}$ 
and $\wave = \{C^{(g)}_{\mu_1,\dots,\mu_n}\}$, 
we have the following Feynman rule: 
\begin{align}
\label{eq:Feynmanrule_Ccc_zeropoint} 
\begin{split}  
C^{(1)}_{\cc; \mu} & = C^{(1)}_\mu + (\omega_{\sfP,\cc})_\mu \\
C^{(g)}_{\cc} & = C^{(g)} + \sum_{\Gamma} 
\frac{1}{|\Aut(\Gamma)|} \Cont_\Gamma( \wave, \Delta_{\sfP,\cc}) 
\qquad \text{ for $g\ge 2$} 
\end{split} 
\end{align} 
where the trivial graph (with one genus-$g$ vertex) is removed 
from the summation in the second line. The curvature condition 
in Proposition~\ref{prop:curvature_condition_cc} together with 
\eqref{eq:Feynmanrule_Ccc_zeropoint} 
shows that $\ov\partial( C^{(1)}_\mu d\sx^\mu) =0$. 
Hence $C^{(1)}_\mu$ is holomorphic.  
Suppose by induction that $C^{(h)}$ is holomorphic 
for all $h<g$ for some $g\ge 2$. 
We differentiate the above Feynman rule \eqref{eq:Feynmanrule_Ccc_zeropoint} 
for $C^{(g)}_\cc$ with respect to $\ov{t}^i$. 
Using the argument in the proof of Proposition~\ref{prop:HAE} 
and the induction hypothesis, we find that the differentiation of the 
second term (in the right-hand side) gives the negative of 
the second and the third 
terms of the holomorphic anomaly equation \eqref{eq:HAE}. 
By the assumed holomorphic anomaly equation, 
we obtain $\ov\partial_i C^{(g)} =0$. This completes 
the induction steps, and the proof. 
\end{proof} 

\begin{remark} 
We have, as in Proposition~\ref{prop:cocycle}, 
\begin{align*} 
T(\sfP_1,z^{-1} \ov\AF) & = T(\sfP_2,z^{-1} \ov\AF) 
\circ T(\sfP_1,\sfP_2) \\
T(z^{-1}\ov\AF, \sfP_1) & = T(\sfP_2,\sfP_1) \circ 
T(z^{-1} \ov\AF, \sfP_2) 
\end{align*} 
for parallel pseudo-opposite modules $\sfP_1,\sfP_2$. 
\end{remark}

\begin{lemma} 
\label{lem:real_potential}
Let $\sfP$ be a parallel pseudo-opposite module. 
Consider the difference one-form $\omega = \omega_{\sfP,\cc} 
\in \cA_\cM^{1,0}$ from 
Definition~\ref{def:difference_one-form_cc}. 
Locally on $\cM$ there exists a real-valued function $u = u_{\sfP}$ such that 
$\partial u = \omega$. The function $u$ is unique 
up to a real constant.   
\end{lemma} 
\begin{proof} 
Recall from Corollary~\ref{cor:curvature_cc_pi} 
that the curvature two-form $\varthetacc = d \omega$ is 
a purely imaginary $(1,1)$-form. Therefore $\partial \omega = 0$ 
and $\ov\partial \omega = \varthetacc$. 
Hence we can locally find a complex-valued function 
$u$ with $\partial u = \omega$. 
Also we have 
\[
\ov\partial \partial (\Im u) = \Re(\ov\partial \partial u) = 
\Re(\ov \partial \omega) = \Re ( \varthetacc) = 0 
\]
since $\ov\partial \partial$ is a purely imaginary operator. 
Therefore $\Im u$ is a pluri-harmonic function. 
We can locally find an anti-holomorphic function $f$ such that 
$\Im f  = \Im u$. Replacing $u$ with $u-f$, we obtain a 
real-valued function $u$ satisfying $\partial u = \omega$. 
The ambiguity in $u$ is a real-valued anti-holomorphic 
function, and hence is a real constant. 
\end{proof} 

\begin{definition} 
\label{def:half_density_metric}
Let $u = u_\sfP$ be the real-valued function in Lemma~\ref{lem:real_potential} 
associated to a parallel pseudo-opposite module $\sfP$. 
We call $h_{\sfP} = \exp(u)$ the \emph{half-density metric} 
associated to $\sfP$. 
This is a locally defined function, 
unique up to multiplication by a real positive constant. 
\end{definition} 

The curvature two-form $\varthetacc$ can be viewed as the curvature of 
the half-density line bundle ``$\det(\bOmegao^1)^{1/2}$'' 
by Proposition~\ref{prop:curvaturecc}. 
From the equation 
\[
\ov\partial \partial \log h_{\sfP} = \varthetacc
\]
we may view $h_{\sfP}$ as the Hermitian metric on 
``$\det(\bOmegao^1)^{1/2}$''. 
Let $C^{(1)}_\mu$ be a genus-one one-point function 
under a parallel pseudo-opposite module $\sfP$. 
Locally we can integrate this to obtain a multi-valued genus-one 
potential $C^{(1)}$; see Remark~\ref{rem:Fockspace}(2). 
Let $C_{\cc;\mu}^{(1)}$ be the corresponding genus-one 
one point function under the complex conjugate opposite module 
$z^{-1} \ov\AF$. 
We can regard $\exp(C^{(1)})$ as a section of the line bundle 
``$\det(\bOmegao^1)^{1/2}$'' and define $\exp(C^{(1)}_\cc)$ 
to be the norm:
\[
\left\|\exp(C^{(1)}) \right\|^2 := \left| \exp(C^{(1)}) \right|^2  h_\sfP
\] 
In fact we have 
\[
\partial \log\left\|\exp(C^{(1)}) \right\|^2 = \omega_{\sfP,\cc} + C^{(1)}_\mu 
d\sx^\mu = C^{(1)}_{\cc;\mu} d\sx^\mu 
\]
and:
\begin{equation} 
\label{eq:genus_one_HAE}
\ov\partial \partial \log\left\|\exp(C^{(1)}) \right\|^2 = \varthetacc
\end{equation}
The latter equation is similar to the holomorphic anomaly 
equation at genus one considered in~\cite{BCOV:HA}. 

\begin{remark} 
We defined the half-density metric for a TRP structure equipped 
with a parallel pseudo-opposite module. In particular, this
defines a CDV\footnote{CDV structures are named after Cecotti, Dubrovin and Vafa~\cite{Cecotti--Vafa:top-anti-top,Cecotti--Vafa:classification,
Dubrovin:top-anti-top, Dubrovin:2DTFT}.}~structure~\cite[Definition~1.2]{Hertling:ttstar}, 
which is a mixture of a Frobenius manifold structure 
and a $tt^*$-structure. 
It would be interesting to study singularity and monodromy of 
the half-density metric for CDV structures which arise from 
singularity theory and quantum cohomology. 
For example, for quantum cohomology equipped with the $\hGamma$-real 
structure~\cite{Iritani:integral, Iritani:ttstar},  
is the half-density metric a single-valued function 
around the large radius limit point? 
\end{remark}

\appendix
\section{Opposite Subspaces in the $L^2$-picture.}
\label{sec:L2pair} 
In this Appendix we collect some facts about opposite subspaces in the $L^2$-picture.  
Let $\cH =L^2(S^1,\C^{N+1})$ be the Hilbert space of 
$\C^{N+1}$-valued square integrable functions on $S^1$. 
(This corresponds to the Givental space in the main body of the text.) 
Let $\cH_+ \subset \cH$ denote the 
subspace consisting of boundary values of holomorphic functions 
$\D \to \C^{N+1}$ on the unit disc $\D = \{z: |z|<1\}$; 
cf.~\eqref{eq:cH+-}. 

\begin{lemma} 
\label{lem:fullrange}
Let $\Proj \subset \cH$ be a closed subspace 
such that $z^{-1} \Proj \subset \Proj$ and 
$\Proj \oplus \cH_+ = \cH$ (the direct sum here is not necessarily orthogonal).  
Then $\bigcup_{n\ge 0} z^n \Proj$ 
is dense in $\cH$. 
\end{lemma} 
\begin{proof} 
Let $V$ be the closure of $\bigcup_{n\ge 0} z^n \Proj$. 
Then $V$ is a $z^{\pm 1}$-invariant subspace: $z V = V$. 
By a vector-valued version of Wiener's theorem~\cite{Srinivasan}, 
$V$ is of the form 
\[
V= \{Q(z) f(z) : f \in \cH \}
\] 
for a measurable function $Q \colon S^1 \to \End(\C^{N+1})$ 
such that $Q(z)$ is an orthogonal projector for each $z\in S^1$, 
i.e.~$Q(z)^2 = Q(z)$, $Q(z)^* = Q(z)$. 
On the other hand, $V \cap \cH_+$ is a $z$-invariant subspace of 
$\cH_+$: 
$z(V\cap \cH_+) \subset V \cap \cH_+$. 
The Beurling--Lax theorem (see e.g.~\cite{Halmos})  
tells us that it is of the form: 
\[
V \cap \cH_+ = \{ T(z) f(z) : f\in \cH_+ \} 
\]
where $T:S^1 \to \End(\C^{N+1})$ is a measurable function 
with the following properties: (1) $T$ is a boundary value of a 
holomorphic function $T\colon \D \to \End(\C^{N+1})$; 
(2) there exists a subspace $U$ of $\C^{N+1}$ such that, 
for each $z \in S^1$, $T(z)|_U$ is an isometry on $U$ and $T(z)|_{U^\perp}=0$. 
Arguing as in the proof of Proposition~\ref{prop:flat_trivialization}(i), we see that $V \cap \cH_+\supset z\Proj\cap \cH_+ \cong \cH_+/z \cH_+
\cong \C^{N+1}$. Therefore $U= \C^{N+1}$. 
Because $V\cap \cH_+ \subset V$ and every element $g\in V$ satisfies 
$Q(z) g(z) = g(z)$, we have $Q(z) T(z) f(z) = T(z)f(z)$ 
for all $f\in \cH_+$.  
This implies that $Q(z)=\id$ and $V = \cH$. 
\end{proof} 

\begin{lemma}
\label{lem:A2}
Let $\Proj\subset \cH$ be as in the previous lemma. Then
$\bigcap_{n\ge 0} z^{-n}\Proj =\{0\}$. 
\end{lemma}
\begin{proof} 
Suppose that there is a non-zero vector 
$x\in \bigcap_{n\ge 0} z^{-n}\Proj$. 
Write $z^n x = a_n + b_n$ where $a_n \in \cH_+$ 
and $b_n \in \cH_- :=(\cH_+)^\perp$ (cf.~equation~\ref{eq:cH+-}). 
We have $b_n \to 0$ in the norm topology as $n\to \infty$. 
The projection $\cH \to \cH_-$ along $\cH_+$ 
induces an isomorphism $\Proj \to \cH_-$. 
Let $f\colon \cH_- \to \Proj$ be 
the inverse isomorphism. We have $z^n x = (a_n + b_n - f(b_n)) + f(b_n)$ 
with $a_n +b_n-f(b_n) \in \cH_+$, $f(b_n) \in \Proj$,
and $z^n x\in \Proj$. Therefore $f(b_n) = z^n x$ and 
so $\|f(b_n)\| = \|z^n x\| = \|x\|$. 
This contradicts the fact that $\lim_{n\to \infty} b_n = 0$. 

\end{proof} 

Lemmas~\ref{lem:fullrange} and~\ref{lem:A2} together imply that the pair 
$(\cH_+, \Proj)$ satisfies $\ov{\bigcup_{n\ge 0} z^{-n} \cH_+} =  
\ov{\bigcup_{n\ge 0} z^n \Proj} = \cH$ and  
$\bigcap_{n\ge 0} z^n \cH_+ = \bigcap_{n\ge 0} z^{-n} \Proj 
=\{0\}$. 
A pair of complementary subspaces with these properties 
is studied in~\cite{Ball--Raney} under the name ``dual shift-invariant pair".

\bibliographystyle{plain}
\bibliography{bibliography}

\end{document}